\documentclass[twoside]{report}
\usepackage{amssymb, l}
\usepackage[arrow,matrix,tips,curve] {xy}
\input amssym.def
\usepackage{float}

\Title{ Unification types and union splittings in intermediate logics} \ShortAuthor{W.
Dzik, S.Kost and P. Wojtylak} \LongAuthor{
\author{WOJCIECH DZIK} \address{Institute of  Mathematics, Silesian University, Bankowa 14,  Katowice 40-007, Poland; wdzik@wdzik.pl}
\author{S{\L}AWOMIR KOST} \address{Institute of  Computer Science, University of Opole, Oleska 48, Opole 45-052, Poland;  skost@uni.opole.pl}
\author{PIOTR WOJTYLAK} \address{Institute of  Computer Science, University of Opole, Oleska 48, Opole 45-052, Poland; ;  pwojtylak@uni.opole.pl} }

\begin{document}
\begin{paper}
\begin{abstract} Following a characterization \cite{dkw} of  locally tabular logics with finitary (or unitary) unification by their Kripke models we determine  the unification types of some intermediate logics (extensions of  {\sf INT}).  There are exactly four maximal logics with nullary unification ${\mathsf L}(\mathfrak R_{2}+)$, \ ${\mathsf L}(\mathfrak R_{2})\cap{\mathsf L}(\mathfrak F_{2})$, \ ${\mathsf L}(\mathfrak G_{3})$ \ and \ ${\mathsf L}(\mathfrak G_{3}+)$ and they are tabular. There are only two minimal logics with  hereditary finitary unification: {\sf L}($\mathbf F_{un}$), the least logic with hereditary unitary unification, and  {\sf L}( $\mathbf F_{pr}$)  the least logic with hereditary  projective approximation; they are locally tabular. Unitary and non-projective logics need additional variables for mgu's of some unifiable formulas, and unitary logics with projective approximation are exactly projective. None of locally tabular intermediate logics has infinitary unification.  Logics with finitary, but not hereditary finitary, unification are rare and scattered among the majority of those with nullary unification, see the example of  $\mathsf H_3\mathsf B_2$ and its extensions.

\end{abstract}

\section{Introduction.}\label{Intro}
Unification, in general, is concerned with finding a substitution that makes two terms equal.
Unification  in logic is the study of substitutions under which a formula becomes provable in a a given logic {\sf L}. In this case the substitutions are called the unifiers of the formula in {\sf L} ({\sf L}-unifiers). If an {\sf L}-unifier for a formula $A$ exists,  $A$ is called unifiable in {\sf L}. An {\sf L}-unifier $\sigma$  for $A$ can be more general than the other {\sf L}-unifier $\tau$, in symbols $\sigma \preccurlyeq \tau$; the pre-order $\preccurlyeq$ of substitutions gives rise to four unification types:  $1$,  $\omega$,   $\infty$,  and $0$,  from the ''best'' to the ''worst'',  see \cite{BaSny,BaGhi}. Unification is unitary, or it has the type $1$,  if there is a most general unifier (mgu)  for every unifiable formula. Unification is finitary or infinitary if, for every unifiable formula, there is a (finite or infinite) basis of unifiers. Nullary unification means that no such basis of unifiers exists at all.

Silvio Ghilardi  introduced unification in propositional (intuitionistic \cite{Ghi2} and modal \cite{Ghi3}) logic.  In \cite{Ghi2} he showed  that unification in {\sf INT} is finitary,   but in {\sf KC}  it is unitary and any  intermediate logic with unitary unification contains {\sf KC}.  Dzik \cite{dzSpl} uses the particular splitting of the lattice of intermediate logics by the pair ({\sf L}($\mathfrak{F}_{2}$),{\sf KC}), where {\sf L}($\mathfrak{F}_{2}$) is the logic determined by the `2-fork frame' $\mathfrak {F}_{2}$ depicted  in Figure \ref{8fames}, to give location of logics with finitary but not unitary unification: they all  are included in {\sf L}($\mathfrak{F}_{2})$. In Wro$\acute{\rm n}$ski \cite{Wro1,Wro2}, see also \cite{dw1}, it is shown  that unification in any intermediate logic {\sf L} is projective iff {\sf L} is an extension of {\sf  LC} (that is it is one of G\"{o}del-Dummett logics); projective implies unitary unification.  In Ghilardi \cite{Ghi5} first examples of intermediate logics with nullary unification are given. Iemhoff \cite{IemRoz} contains a proof-theoretic account of unification in fragments of intuitionistic logics.
 Many papers concern unification in modal logics, see e.g. \cite{Ghi3,Jer,Balb1,dw2,Kost}, and also in intuitionistic predicate logic, see \cite{dw4}. No (modal or intermediate) logic with infinitary unification has been found so far and it is expected that no such logic exists.

Generally,  similar results on unification types in transitive modal logics and corresponding  intermediate logics are given in  \cite{dkw}.

In \cite{Ghi5}  Ghilardi  studied unification in intermediate logics of finite slices (or finite depths). He applied his method,  based on Category Theory, of finitely presented projective objects (see \cite{Ghi1}) and duality, and characterized injective objects in finite posets.   He gave some positive and negative criteria for unification to be finitary. From these criteria it follows, for instance, that  bounded  depth  axioms $\mathsf{H_n }$ plus bounded width axioms $\mathsf{B_k }$ keep unification finitary. It also follows that there are  logics without finitary unification.\footnote{Ghilardi's original notation of  frames, as well as our notation of frames in \cite{dkw}, was quite different. All frames depicted  in this paper represent finite po-sets.} He considered, among others, the following frames:

 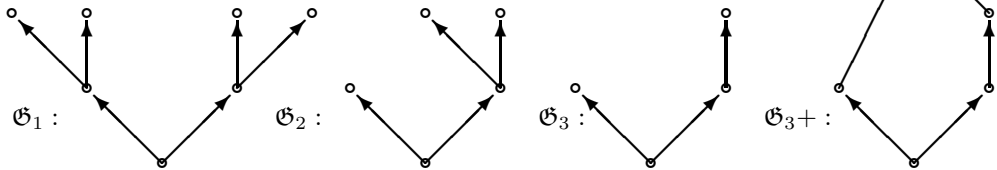
\begin{figure}[H]
\unitlength1cm
\begin{picture}(0,2.2)
\thicklines
\put(0,0.5){$\mathfrak{ G}_1:$}
\put(2,0){\vector(-1,1){0.9}}
\put(2,0){\vector(1,1){0.9}}
\put(1,1){\vector(-1,1){0.9}}
\put(1,1){\vector(0,1){0.9}}
\put(3,1){\vector(1,1){0.9}}
\put(3,1){\vector(0,1){0.9}}

\put(0,2){\circle{0.1}}
\put(1,2){\circle{0.1}}
\put(3,2){\circle{0.1}}
\put(4,2){\circle{0.1}}
\put(1,1){\circle{0.1}}
\put(3,1){\circle{0.1}}
\put(2,0){\circle{0.1}}

\put(3.5,0.5){$\mathfrak{G }_2:$}
\put(5.5,0){\vector(-1,1){0.9}}
\put(5.5,0){\vector(1,1){0.9}}
\put(6.5,1){\vector(0,1){0.9}}
\put(6.5,1){\vector(-1,1){0.9}}
\put(4.5,1){\circle{0.1}}
\put(6.5,2){\circle{0.1}}
\put(5.5,2){\circle{0.1}}
\put(5.5,0){\circle{0.1}}
\put(6.5,1){\circle{0.1}}

\put(7,0.5){$\mathfrak{G}_3:$}
\put(8.5,0){\vector(-1,1){0.9}}
\put(8.5,0){\vector(1,1){0.9}}
\put(9.5,1){\vector(0,1){0.9}}
\put(7.5,1){\circle{0.1}}
\put(9.5,2){\circle{0.1}}
\put(8.5,0){\circle{0.1}}
\put(9.5,1){\circle{0.1}}

\put(10,0.5){${\mathfrak{G}_{3}}+:$}
\put(12,3){\circle{0.1}}
\put(13,2){\circle{0.1}}
\put(11,1){\circle{0.1}}
\put(13,1){\circle{0.1}}
\put(12,0){\circle{0.1}}
\put(13,2){\vector(-1,1){0.9}}
\put(13,1){\vector(0,1){0.9}}
\put(11,1){\vector(1,2){0.9}}
\put(12,0){\vector(1,1){0.9}}
\put(12,0){\vector(-1,1){0.9}}
\end{picture}\\
\caption{Ghilardi's Frames} \label{GF}
\end{figure}
\noindent Since $\mathsf L(\mathfrak{G}_1)$, the logic of $\mathfrak{G}_1$,  coincides with  $\mathsf{H}_3\mathsf{B}_2$, it has finitary unification by  \cite{Ghi5}.
Theorem 9, p.112 of \cite{Ghi5}) says that, if $\mathfrak{G}_3$ is a frame of any  intermediate logic  with finitary unification, then  $\mathfrak{G}_2$ is a frame of this logic, as well. It means, in particular, that $\mathsf L(\mathfrak{G}_3)$ has not finitary unification.
 (the unification type of $\mathsf L(\mathfrak{G}_2)$ and  $\mathsf L(\mathfrak{G}_3)$ was not determined). Ghilardi announced that `attaching a final point everywhere' provide examples in which unification is nullary. Thus, $\mathsf L({\mathfrak{G}_3}+)$ has nullary unification.\footnote{The frame received from  $\mathfrak{F}$, by adding a top (=final) element is denoted by ${\mathfrak F}+$.} He also showed that replacing one of maximal elements  in $\mathfrak{G}_3$ with any finite (rooted) po-set $\mathfrak P$, gives a frame of logic without finitary unification,  see Figure \ref{NU}.
 \begin{figure}[H]
\unitlength1cm
\begin{picture}(0,2)
\thicklines

\put(4,0.5){$\mathfrak{G}_{3\mathfrak P}:$}
\put(6.5,0){\vector(-1,1){0.9}}
\put(6.5,0){\vector(1,1){0.9}}
\put(7.5,1){\vector(0,1){0.9}}
\put(5.4,1.1){$\mathfrak P$}
\put(7.5,2){\circle{0.1}}
\put(6.5,0){\circle{0.1}}
\put(7.5,1){\circle{0.1}}
\put(5.5,1.2){\circle{0.7}}

\end{picture}\\
\caption{Frames of Logics with Nullary Unification} \label{NU}
\end{figure}
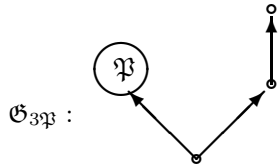
Hence, there are infinitely many intermediate logics without finitary (by \cite{dkw}: with nullary) unification.
In \cite{dkw} we gave necessary and sufficient conditions for finitary (or unitary) unification  in locally tabular logics  solely in terms of mappings between (bounded) Kripke models. Our approach was entirely different from that in \cite{Ghi5}. A simpler variant of the conditions characterizes logics with projective approximation.  Then we applied the conditions to determine the unification types of logics (intermediate or modal) given by relatively simple frames. In particular,
  we studied tabular modal  and intermediate logics determined by the frames in Figure \ref{8fames}.

\begin{figure}[H]
\unitlength1cm
\begin{picture}(3,2)
\thicklines
\put(0,0.5){$\mathfrak L_1:$}
\put(1,0){\circle{0.1}}
\put(2.5,0.5){$\mathfrak L_2:$}
\put(3.5,0){\circle{0.1}}
\put(3.5,0){\line(0,1){0.9}}
\put(3.5,1){\circle{0.1}}
\put(3.5,0){\vector(0,1){0.9}}

\put(5,0.5){$\mathfrak L_3:$}
\put(6,0){\vector(0,1){0.9}}
\put(6,1){\vector(0,1){0.9}}
\put(6,1){\circle{0.1}}
\put(6,2){\circle{0.1}}
\put(6,0){\circle{0.1}}

\put(7,0.5){$\mathfrak{F}_{2}:$}
\put(8,1){\circle{0.1}}
\put(9,0){\circle{0.1}}
\put(10,1){\circle{0.1}}
\put(9,0){\vector(1,1){0.9}}
\put(9,0){\vector(-1,1){0.9}}

\put(10.5,0.5){${\mathfrak{R}_{2}}:$}
\put(12,0){\vector(-1,1){0.9}}
\put(12,0){\vector(1,1){0.9}}
\put(13,1){\vector(-1,1){0.9}}
\put(11,1){\circle{0.1}}
\put(12,2){\circle{0.1}}
\put(12,0){\circle{0.1}}
\put(13,1){\circle{0.1}}
\put(11,1){\vector(1,1){0.9}}

\end{picture}\\

\unitlength1cm
\begin{picture}(5,3)
\thicklines

\put(0,0.5){$\mathfrak{G}_3:$}
\put(2,0){\vector(-1,1){0.9}}
\put(2,0){\vector(1,1){0.9}}
\put(3,1){\vector(0,1){0.9}}
\put(1,1){\circle{0.1}}
\put(3.1,2){\circle{0.1}}
\put(2,0){\circle{0.1}}
\put(3,1){\circle{0.1}}

\put(3.5,0.5){${\mathfrak{G}_{3}}+:$}
\put(5.5,3){\circle{0.1}}
\put(6.5,2){\circle{0.1}}
\put(4.5,1){\circle{0.1}}
\put(6.5,1){\circle{0.1}}
\put(5.5,0){\circle{0.1}}
\put(6.5,2){\vector(-1,1){0.9}}
\put(6.5,1){\vector(0,1){0.9}}
\put(4.5,1){\vector(1,2){0.9}}
\put(5.5,0){\vector(1,1){0.9}}
\put(5.5,0){\vector(-1,1){0.9}}

\put(7.2,0.5){$\mathfrak{F}_{3}:$}
\put(8,1){\circle{0.1}}
\put(9,0){\circle{0.1}}
\put(10,1){\circle{0.1}}
\put(9,1){\circle{0.1}}
\put(9,0){\vector(1,1){0.9}}
\put(9,0){\vector(-1,1){0.9}}
\put(9,0){\vector(0,1){0.9}}

\put(10.4,0.2){${\mathfrak{R}_{3}}:$}
\put(11,1){\circle{0.1}}
\put(12,0){\circle{0.1}}
\put(12,2){\circle{0.1}}
\put(13,1){\circle{0.1}}
\put(12,1){\circle{0.1}}
\put(12,0){\vector(1,1){0.9}}
\put(12,0){\vector(-1,1){0.9}}
\put(12,0){\vector(0,1){0.9}}
\put(11,1){\vector(1,1){0.9}}
\put(12,1){\vector(0,1){0.9}}
\put(13,1){\vector(-1,1){0.9}}

\end{picture}\\

\caption{Frames of \cite{dkw}} \label{8fames}
\end{figure}
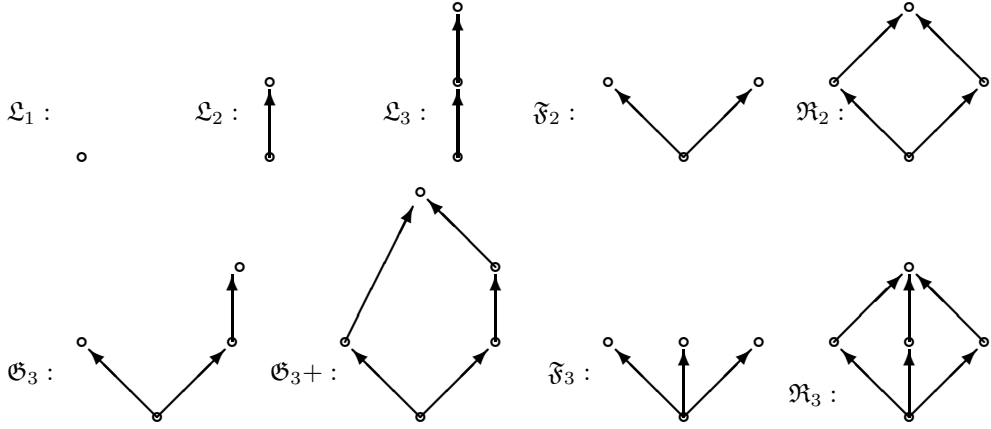
\noindent We proved  that unification  in the modal (as well as intermediate) logics of the frames $\mathfrak L_1, \mathfrak L_2, \mathfrak L_3,{\mathfrak{R}_{2}}$ and ${\mathfrak{R}_{3}}$ is unitary,   in (the logic of) $\mathfrak{F}_{2}$ and $\mathfrak{F}_{3}$ it is finitary and in $\mathfrak{G}_3$ and $\mathfrak{G}_{3}+$ it is nullary. We have also considered $n$-forks ${\mathfrak{F}_{n}}$ and $n$-rhombuses ${\mathfrak{R}_{n}}$, for any $n\geq 2$, see Figure  \ref{FRF}. We showed that the logic of any fork (including the infinite `fork frame' ${\mathfrak{F}_{\infty}}$) has projective approximation, and hance it has finitary unification.  The logic of any rhombus (including  ${\mathfrak{R}_{\infty}}$) has unitary unification.

 \begin{figure}[H]
\unitlength1cm
\begin{picture}(3,2)
\thicklines
\put(2,0){${\mathfrak{F}_{n}}:$}
\put(2,1){\circle{0.1}}
\put(5,1){\circle{0.1}}
\put(4,1){\circle{0.1}}
\put(6,1){\circle{0.1}}
\put(3,1){\circle{0.1}}

\put(4,0){\vector(1,1){0.9}}
\put(4,0){\vector(-1,1){0.9}}
\put(4,0){\vector(0,1){0.9}}
\put(4,0){\vector(2,1){1.9}}
\put(4,0){\vector(-2,1){1.9}}

\put(1,1){\circle{0.1}}
\put(4,0){\circle{0.1}}
\put(2,1){\circle{0.1}}
\put(1.5,1){\circle{0.1}}
\put(7,1){\circle{0.1}}
\put(6.5,1){\circle{0.1}}
\put(7.5,1){\circle{0.1}}

\put(7,0){$\mathfrak{R}_n={\mathfrak{F}_{n}}+:$}
\put(8,1){\circle{0.1}}
\put(11,1){\circle{0.1}}
\put(10,1){\circle{0.1}}
\put(12,1){\circle{0.1}}
\put(9,1){\circle{0.1}}

\put(10,0){\vector(1,1){0.9}}
\put(10,0){\vector(-1,1){0.9}}
\put(10,0){\vector(0,1){0.9}}
\put(10,0){\vector(2,1){1.9}}
\put(10,0){\vector(-2,1){1.9}}

\put(10,0){\circle{0.1}}
\put(12.5,1){\circle{0.1}}
\put(13,1){\circle{0.1}}
\put(10,2){\circle{0.1}}

\put(9,1){\vector(1,1){0.9}}
\put(11,1){\vector(-1,1){0.9}}
\put(10,1){\vector(0,1){0.9}}
\put(8,1){\vector(2,1){1.9}}
\put(12,1){\vector(-2,1){1.9}}

\end{picture}
\caption{$n$-Fork and $n$-Rhombus Frames, for $n\geq 1$.}\label{FRF}
\end{figure}
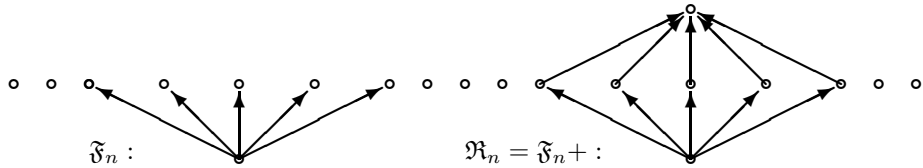
\noindent

Still many questions about unification of intermediate logics and location of particular types remained open. Here is a summary  of the results  in the present paper.\\
 1) We give another proof  that our conditions (see Theorem \ref{main}) are necessary and sufficient  for finitary\slash unitary unification, as well as for projective approximation  (Theorem \ref{retraction}) in locally tabular intermediate logics.  Variants of the  frames in Figure \ref{8fames} are considered and we determine the unification types of their logics. In particular,
 we prove that unification in $\mathsf L(\mathfrak{G}_2)$ is finitary and  though (we know that) it is also finitary in $\mathsf L(\mathfrak{F}_{3})$, it is nullary in  their intersection  $\mathsf L(\mathfrak{G}_2)\cap\mathsf L(\mathfrak{F}_{3})$ .\\
2) It turns out that intermediate logics with unitary unification  are either projective (hence they are extensions of {\sf LC}) or they need new variables for mgu's of some unifiable formulas. It means that any (non-projective) logic with unitary unification has a unifiable formula $A(x_1,\dots,x_n)$ which do not have any mgu in $n$-variables (but its mgu's must introduce additional variables -- like in filtering unification). The same result for transitive modal logics is proved  in \cite{dkw}.\\
3) We prove that a locally tabular intermediate logic with infinitary unification does not exist and  we think that no intermediate logic has infinitary unification.\\
4) We claim (and give some evidences) that 'most of' intermediate logics have nullary unification. For instance, logics of the following frames are nullary:
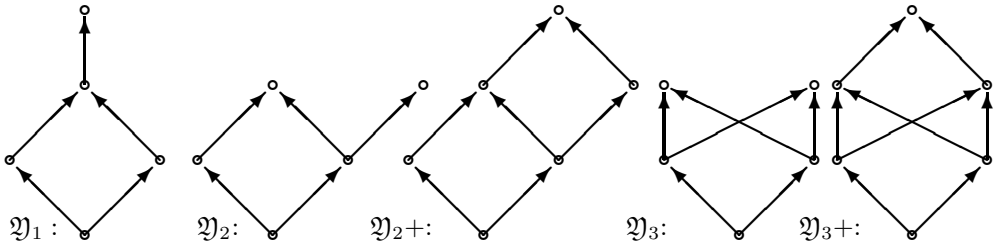
\begin{figure}[H]
\unitlength1cm
\thicklines
\begin{picture}(0,3)
\put(0,0){$\mathfrak Y_{1}:$}
\put(1,0){\vector(-1,1){0.9}}
\put(1,0){\vector(1,1){0.9}}
\put(2,1){\vector(-1,1){0.9}}
\put(0,1){\circle{0.1}}
\put(1,2){\circle{0.1}}
\put(1,0){\circle{0.1}}
\put(2,1){\circle{0.1}}
\put(0,1){\vector(1,1){0.9}}
\put(1,3){\circle{0.1}}
\put(1,2){\vector(0,1){0.9}}

\put(2.5,0){$\mathfrak{Y}_2$:}
\put(3.5,0){\vector(-1,1){0.9}}
\put(3.5,0){\vector(1,1){0.9}}
\put(4.5,1){\vector(1,1){0.9}}
\put(2.5,1){\circle{0.1}}
\put(3.5,2){\circle{0.1}}
\put(3.5,0){\circle{0.1}}
\put(4.5,1){\circle{0.1}}
\put(5.5,2){\circle{0.1}}
\put(2.5,1){\vector(1,1){0.9}}
\put(4.5,1){\vector(-1,1){0.9}}

\put(4.8,0){$\mathfrak{Y}_2+$:}
\put(6.3,0){\vector(-1,1){0.9}}
\put(6.3,0){\vector(1,1){0.9}}
\put(7.3,1){\vector(1,1){0.9}}
\put(5.3,1){\circle{0.1}}
\put(6.3,2){\circle{0.1}}
\put(6.3,0){\circle{0.1}}
\put(7.3,1){\circle{0.1}}
\put(8.3,2){\circle{0.1}}
\put(5.3,1){\vector(1,1){0.9}}
\put(7.3,1){\vector(-1,1){0.9}}
\put(7.3,3){\circle{0.1}}
\put(6.3,2){\vector(1,1){0.9}}
\put(8.3,2){\vector(-1,1){0.9}}

\put(8.2,0){$\mathfrak{Y}_3$:}
\put(8.7,2){\circle{0.1}}
\put(10.7,2){\circle{0.1}}
\put(8.7,1){\circle{0.1}}
\put(10.7,1){\circle{0.1}}
\put(9.7,0){\circle{0.1}}
\put(8.7,1){\vector(0,1){0.9}}
\put(10.7,1){\vector(0,1){0.9}}
\put(8.7,1){\vector(2,1){1.9}}
\put(10.7,1){\vector(-2,1){1.9}}
\put(9.7,0){\vector(1,1){0.9}}
\put(9.7,0){\vector(-1,1){0.9}}

\put(10.5,0){${\mathfrak{Y}_3}+$:}
\put(11,2){\circle{0.1}}
\put(13,2){\circle{0.1}}
\put(11,1){\circle{0.1}}
\put(13,1){\circle{0.1}}
\put(12,0){\circle{0.1}}
\put(12,3){\circle{0.1}}
\put(11,1){\vector(0,1){0.9}}
\put(13,1){\vector(0,1){0.9}}
\put(11,1){\vector(2,1){1.9}}
\put(13,1){\vector(-2,1){1.9}}
\put(12,0){\vector(1,1){0.9}}
\put(12,0){\vector(-1,1){0.9}}
\put(11,2){\vector(1,1){0.9}}
\put(13,2){\vector(-1,1){0.9}}

\end{picture}
\caption{Frames of Logics with Nullary Unification}\label{MNU}
\end{figure}

Intermediate logics with nullary unification  can be found 'almost everywhere'.
Extensions of finitary\slash unitary logics may have nullary unification, intersections of finitary logics may be nullary. We cannot put apart logics with finitary\slash unitary unification from those with the nullary one.

5) In structurally complete logics \footnote{We consider  rules  $r\!\!:\!\!{A}\slash{B}$, where $A, B$ play the role of formula schemata, i.e. $r$ enables us to derive $\varepsilon(B)$ from $\varepsilon(A)$, for any substitution $\varepsilon$.
The rule is said to be {\it admissible} in an intermediate logic {\sf L} (or  {\sf L}-admissible),  if $\vdash_{\sf L} \varepsilon(A)$ implies  $\vdash_{\sf L} \varepsilon(B)$, for any substitution
$\varepsilon$, that is any {\sf L}-unifier for $A$ must be an {\sf L}-unifier for $B$. The rule  is {\it {\sf L}-derivable} if $A\vdash_{\sf L}B$.
A logic {\sf L}  is {\it structurally complete} if every its admissible rule is derivable (the reverse inclusion always holds). {\it Hereditary structural completeness} of {\sf L} means that any extension of {\sf L} is structurally complete.}
the situation is somehow similar. A.Citkin (see Tzitkin \cite{Tsitkin})   characterized  hereditary structurally complete logics (instead of structurally complete) 
and showed that a  logic {\sf L} is hereditary structurally complete iff {\sf L} omits (i.e. {\sf L} is falsified in)  the following frames:
\begin{figure}[H]
\unitlength1cm
\thicklines
\begin{picture}(0,2.2)
\put(0,0){$\mathfrak C_{1}:$}
\put(0,1){\circle{0.1}}
\put(1,0){\circle{0.1}}
\put(2,1){\circle{0.1}}
\put(1,1){\circle{0.1}}
\put(1,0){\vector(1,1){0.9}}
\put(1,0){\vector(-1,1){0.9}}
\put(1,0){\vector(0,1){0.9}}

\put(2.5,0){$\mathfrak C_{2}:$}
\put(2.5,1){\circle{0.1}}
\put(3.5,0){\circle{0.1}}
\put(4.5,1){\circle{0.1}}
\put(3.5,1){\circle{0.1}}
\put(3.5,0){\vector(1,1){0.9}}
\put(3.5,0){\vector(-1,1){0.9}}
\put(3.5,0){\vector(0,1){0.9}}
\put(3.5,2){\circle{0.1}}
\put(2.5,1){\vector(1,1){0.9}}
\put(3.5,1){\vector(0,1){0.9}}
\put(4.5,1){\vector(-1,1){0.9}}

\put(5,0){$\mathfrak C_{3}:$}
\put(6,0){\vector(-1,1){0.9}}
\put(6,0){\vector(1,1){0.9}}
\put(7,1){\vector(0,1){0.9}}
\put(5,1){\circle{0.1}}
\put(7,2){\circle{0.1}}
\put(6,0){\circle{0.1}}
\put(7,1){\circle{0.1}}

\put(7.5,0){$\mathfrak C_{4}:$}
\put(8.5,0){\line(-1,1){0.9}}
\put(8.5,0){\vector(-1,1){0.9}}
\put(8.5,0){\vector(1,1){0.9}}
\put(9.5,1){\vector(0,1){0.9}}
\put(7.5,1){\circle{0.1}}
\put(9.5,2){\circle{0.1}}
\put(8.5,0){\circle{0.1}}
\put(9.5,1){\circle{0.1}}
\put(8.5,3){\circle{0.1}}
\put(9.5,2){\vector(-1,1){0.9}}
\put(7.5,1){\vector(1,2){0.9}}

\put(10.5,0){$\mathfrak C_{5}:$}
\put(11.5,0){\vector(-1,1){0.9}}
\put(11.5,0){\vector(1,1){0.9}}
\put(12.5,1){\vector(-1,1){0.9}}
\put(10.5,1){\circle{0.1}}
\put(11.5,2){\circle{0.1}}
\put(11.5,0){\circle{0.1}}
\put(12.5,1){\circle{0.1}}
\put(10.5,1){\vector(1,1){0.9}}
\put(10.5,2){\circle{0.1}}
\put(12.5,2){\circle{0.1}}
\put(10.5,1){\vector(0,1){0.9}}
\put(12.5,1){\vector(0,1){0.9}}

\end{picture}
\caption{Citkin's Frames}\label{TF}
\end{figure}
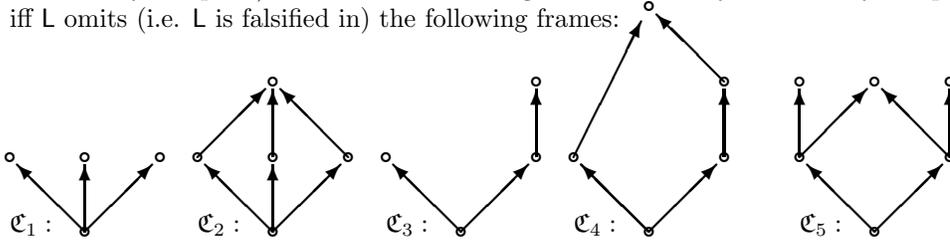
We consider logics with {\it hereditary finitary unification} that is logics  all their extensions have either finitary or unitary unification. We prove that  there are exactly four maximal logics with nullary unification:  $\mathsf L (\mathfrak Y_1)$,  $\mathsf L(\mathfrak R_2)\cap \mathsf L(\mathfrak F_2)$, $\mathsf L(\mathfrak G_3)$ and  $ \mathsf L(\mathfrak G_3+)$. Thus, an intermediate logic has hereditary finitary unification if it omits  $\mathfrak Y_1$, $\mathfrak G_3$, $\mathfrak G_3+$ and one of the frames  $\{\mathfrak R_2,\mathfrak F_2\}$. This characterization is not optimal as, for instance, omitting $\mathfrak F_2$ the logic  omits $\mathfrak G_3$; omitting $\mathfrak R_2$ it omits $\mathfrak G_3+$ and $\mathfrak Y_1$.

There is no correlation between structural completeness and finitary unification. In particular, since $\mathfrak C_1 = \mathfrak F_3$  the logic of $\mathfrak C_1$ has projective approximation (and therefore it is finitary), since $\mathfrak C_2 = \mathfrak R_3$, $\mathfrak C_2$ is unitary and we will  show that the fifth $\mathsf L(\mathfrak C_5)$ is finitary  but  not hereditary finitary. The remaining frames {$\mathfrak C_{3}$} and {$\mathfrak C_{4}$} coincide with ${\mathfrak{G}_{\sf 3}}$ and ${\mathfrak{G}_{\sf 3}}+$ and their logics have nullary unification.\\
6) Two additional classes of logics emerge here: logics with {\it hereditary unitary unification} and logics with {\it hereditary projective approximation}. We show that an intermediate logic {\sf L} has hereditary unitary unification iff {\sf L} omits the frames $\mathfrak Y_1$, $\mathfrak F_2$  and    $\mathfrak G_3+$. A logic {\sf L} has hereditary projective approximation iff {\sf L} omits the frames $\mathfrak R_2$ and $\mathfrak G_3$. Thus, {\sf L} has hereditary finitary unification iff either {\sf L} has hereditary unitary unification or {\sf L} has hereditary projective characterization.
  Logics with hereditary projective approximation can be characterized by  frames $\mathfrak L_d+\mathfrak F_n$, for any $d,n\geq 0$ (that is forks on chains), whereas logics with hereditary unitary unification by $\mathfrak L_d+\mathfrak R_n$, for any $d,n\geq 0$ (that is rhombuses on chains); see Figure \ref{hpa}.

 \begin{figure}[H]
\unitlength1cm
\begin{picture}(3,3.5)
\thicklines
\put(0,1){$\mathbf{H}_{pa}$:}
\put(0,3){\circle{0.1}}
\put(3,3){\circle{0.1}}
\put(2,3){\circle{0.1}}
\put(4,3){\circle{0.1}}
\put(1,3){\circle{0.1}}
\put(2,3){\circle{0.1}}
\put(2,2){\vector(1,1){0.9}}
\put(2,2){\vector(-1,1){0.9}}
\put(2,2){\vector(0,1){0.9}}
\put(2,2){\vector(2,1){1.9}}
\put(2,2){\vector(-2,1){1.9}}
\put(2,2){\circle{0.1}}
\put(2,1.5){\circle{0.1}}
\put(2,1){\circle{0.1}}
\put(2,1.25){\circle{0.1}}
\put(2,1.75){\circle{0.1}}
\put(2,0){\vector(0,1){0.9}}
\put(2,0){\circle{0.1}}

\put(8,1){$\mathbf{H}_{un}$:}
\put(8,3){\circle{0.1}}
\put(11,3){\circle{0.1}}
\put(10,3){\circle{0.1}}
\put(12,3){\circle{0.1}}
\put(9,3){\circle{0.1}}

\put(10,2){\vector(1,1){0.9}}
\put(10,2){\vector(-1,1){0.9}}
\put(10,2){\vector(0,1){0.9}}
\put(10,2){\vector(2,1){1.9}}
\put(10,2){\vector(-2,1){1.9}}

\put(10,2){\circle{0.1}}
\put(10,1.75){\circle{0.1}}
\put(10,1.5){\circle{0.1}}
\put(10,1.25){\circle{0.1}}
\put(10,1){\circle{0.1}}
\put(10,0){\circle{0.1}}

\put(10,0){\vector(0,1){0.9}}

\put(10,1){\circle{0.1}}
\put(10,0){\circle{0.1}}
\put(10,4){\circle{0.1}}
\put(9,3){\vector(1,1){0.9}}
\put(11,3){\vector(-1,1){0.9}}
\put(10,3){\vector(0,1){0.9}}
\put(8,3){\vector(2,1){1.9}}
\put(12,3){\vector(-2,1){1.9}}
\end{picture}
\caption{Frames of Logics with Hereditary Finitary Unification.}\label{hpa}
\end{figure}
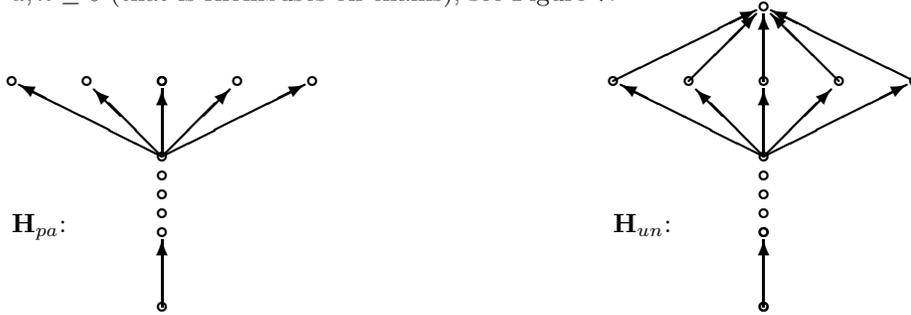
\noindent
$\mathsf L(\mathbf{H}_{pa})$  is the least intermediate logic with hereditary projective approximation and  $\mathsf L(\mathbf{H}_{un})$ is the least logic with hereditary unitary unification. The logics $\mathsf L(\mathbf{H}_{pa})$ and $\mathsf L(\mathbf{H}_{un})$  are locally tabular and  they are (the only) minimal logics with hereditary finitary unification. We have $\mathsf L(\mathsf L(\mathbf{H}_{pa})\cup\mathsf L(\mathbf{H}_{un}))=\mathsf{LC}$ as, it is  proved that, any unitary intermediate logic with projective approximation is projective.

\section{Basic Concepts.}\label{BC}

\subsection{Intermediate Logics.}\label{IL}
We consider the standard language of intuitionistic propositional logic $\{\rightarrow,\lor,\land,\bot\}$ where $\leftrightarrow,\neg,\top$ are defined in the usual way.
 Let $\mathsf{Var}=\{x_1,x_2,\dots\}$ be the set of propositional variables and  $\mathsf{Fm}$ be the set of (intuitionistic) formulas,  denoted by  $A,B,C,\dots$
 For any $n\geq 0$, let $\mathsf{Fm^n}$, be the set of formulas in the variables $\{x_1,\dots,x_n\}$, that is $A\in \mathsf{Fm^n}\Leftrightarrow  \mathsf{Var}(A)\subseteq\{x_1,\dots,x_n\}\Leftrightarrow A=A(x_1,\dots,x_n).$

Substitutions $\alpha,\beta,\dots$  are  finite mappings;  for each $\alpha$ there are $k,n\geq 0$ such that $\alpha\colon\{x_1,\dots,x_n\}\to \mathsf{Fm}^k$. The extension of  $\alpha$ to an endomorphism of   $\mathsf{Fm}$ is also denoted by $\alpha$. Thus, $\alpha(A)$ means the substitution of a formula $A$. Let $\alpha\circ\tau$ be the composition of the substitutions, that is a substitution such that $\alpha\circ\tau(A)=\alpha(\tau(A))$, for any $A$.

 An {\it intermediate logic} {\sf L} is any set of  formulas  containing the intuitionistic  logic {\sf INT}, closed under the modus ponens rule MP and closed under substitutions.\footnote{Intermediate logics may be regarded as fragments of  transitive modal logics (or extensions of {\sf S4}, or {\sf Grz}); the intuitionistic variable $x_i$ is meant as $\Box^+ x_i$ and $A\rightarrow B=\Box^+(\neg A\lor B)$.}   All intermediate logics form, under inclusion, a (complete distributive) lattice where inf$\{\mathsf L_i\}_{i\in I}=\bigcap_{i\in I}\mathsf L_i$. Let $\mathsf L(X)$, for any set $X$ of formulas, mean the least intermediate logic containing $X$. Given two intermediate logics {\sf L} and {\sf L'}, we say {\sf L'} is {\it an extension of} {\sf L} if $\mathsf L\subseteq\mathsf L'$.
The least intermediate logic is  {\sf INT}. Consistent logics are proper subsets of $\mathsf{Fm}$.
We will refer to the following  list of formulas\slash logics:
\begin{figure}[H]
$$\begin{array}{ll} \mathsf{ LC}: (x_1\rightarrow  x_2)\lor (x_2\rightarrow  x_1); \qquad \qquad \mathsf{ KC}:  \neg x \lor \neg \neg x;& \\
\mathsf{ SL}:  (( \neg \neg x\rightarrow  x)\rightarrow (\neg x \lor \neg\neg x)) \rightarrow  (\neg x \lor \neg \neg x):  &\mathsf{ } \\
\mathsf{PWL}: (x_2\to x_1)\lor\bigl(((x_1\to x_2)\to x_1)\to x_1\bigr);&{}\\
\mathsf{H_n } : \ \mathsf{H}_1 = x_1 \lor \neg x_1,\qquad  \mathsf{H}_{n+1} = x_{n+1} \lor (x_{n+1} \rightarrow  \mathsf{H}_n);  &\mathsf{ }\\
\mathsf{B_n}: \bigwedge_{i=1}^{n+1}\Bigl(\bigl(x_i\rightarrow\bigvee_{j\not=i}x_j\bigr)\rightarrow\bigvee_{j\not=i}x_j\Bigr)\rightarrow \bigvee_{i=1}^{n+1}x_i.&\mathsf{ }
\end{array}$$\caption{Intermediate Logics.}\label{ILs}
\end{figure}

{\sf KC}  is  called the logic of weak excluded middle or Jankov logic or de~Morgan logic (see \cite{Ghi2}). {\sf SL} is  Scott logic and {PWL} is the logic of weak law of Peirce, see \cite{Esakia}.

We define the {\it consequence relation} $\vdash_{\mathsf L}$, for any given  intermediate   logic $\mathsf L$, admitting  only the  rule $\mathsf{MP}$  in derivations. Then we prove the {\it deduction theorem} $$X,A\vdash_{\mathsf L}B \quad\Leftrightarrow\quad X\vdash_{\mathsf L}A\rightarrow B.\leqno{(DT)}$$
The relation  of  $\mathsf L-equivalent$ formulas,
$$ A=_{\mathsf L} B \qquad \Leftrightarrow\qquad \vdash_{\mathsf L} A\leftrightarrow B,$$
leads to  the standard {\it Lindenbaum-Tarski algebra}.  The relation
$=_{\mathsf L}$ extends to substitutions,
$ \varepsilon=_{\mathsf L} \mu$ means that  $\varepsilon(A)=_{\mathsf L}   \mu(A)$,  for each formula $A$.  We define  a {\it pre-order} (that is a  reflexive and transitive relation) on the set of substitutions:
$$ \varepsilon\preccurlyeq_{\mathsf L} \mu \qquad \Leftrightarrow \qquad
\bigl(\alpha\circ\varepsilon=_{\mathsf L}   \mu, \mbox{ for some  $\alpha$}\bigr).\footnote{Sometimes the reverse pre-order is used; in this case $\mu  \preccurlyeq \varepsilon\Leftrightarrow
(\alpha\circ\varepsilon=_{\mathsf L}   \mu, \mbox{ for some  $\alpha$})$.}$$

 Note that  $\varepsilon\preccurlyeq_{\mathsf L} \mu \land \mu\preccurlyeq_{\mathsf L} \varepsilon$ does not yield $\varepsilon=_{\mathsf L} \mu$.  If $\varepsilon\preccurlyeq_{\mathsf L} \mu$, we say that $\varepsilon$ is  {\it more general}  than $\mu$. If it is not misleading, we
omit the subscript $_{\mathsf L}$ and write $=$ and $\preccurlyeq$, instead of $=_{\mathsf L}$ and $\preccurlyeq_{\mathsf L}$, correspondingly.\\

A {\it frame} $\mathfrak F=(W,R,w_0)$ consists of a non-empty set $W$, a pre-order $R$ on $W$ and a {\it root}  $w_0\in W$ such that $w_0Rw$, for any $w\in W.$ {For any set $U$, let $P(U)=\{V:V\subseteq U\}$.} Let $n$ be a natural number.
Any $n$-{\it model} $\mathfrak{M}^n=(W,R,w_0,V^n)$, over the frame $(W,R,w_0)$, contains a valuation $V^n:W\to P(\{x_1,\dots,x_n\})$ which is monotone:
$$u R w\quad \Rightarrow\quad V^n(u)\subseteq V^n(w), \quad \mbox{for each } u,w\in W.$$

Thus, $n$-models, are (bounded) variants of  usual Kripke models $\mathfrak{M}=(W,R,w_0,V)$ where all variables are valuated; $V:W\to P(\mathsf{Var})$. Given $\mathfrak{M}^n$ and $\mathfrak{M}^k$ (for $n\not=k$),   we do not assume that $\mathfrak{M}^n$ and $\mathfrak{M}^k$ have anything in common.  In particular, we do not assume that there is any model $\mathfrak{M}$ such that $\mathfrak{M}^n$ and $\mathfrak{M}^k$ are  its fragments. If $\mathfrak{M}^k=(W,R,w_0,V^k)$ and $n\leq k$, then $\mathfrak{M}^k\!\!\upharpoonright_n$ is the restriction of $\mathfrak{M}^k$  to the $n$-model. Thus, $\mathfrak{M}^k\!\!\upharpoonright_n=(W,R,w_0,V^n)$ is the $n$-model over the same frame as $\mathfrak{M}^k$ in which  $V^n(w)=V^k(w)\cap\{x_1,\dots,x_n\}$, for each $w\in W$. We say $(W,R,w_0)$ is a po-frame, and $(W,R,w_0,V^n)$ is a po-model, if the relation $R$ is a partial order.

Let $\mathfrak{F}=(W,\leq,w_0)$ be a finite po-frame. We define {the {\it depth}, $d_{\mathfrak F}(w)$, of any element $w\in W$ in  $\mathfrak F$}. We let $d_{\mathfrak F}(w)=1$ if $w$ is a $\leq$-maximal element ($\leq$-maximal elements are also called end elements) and  $d_{\mathfrak F}(w)=i+1$ if all elements in $\{u\in W\colon w<u\}$ are of the depth at most $i$ and there is  at least one element $u>w$  of the depth $i$. The depth of the root, $d_{\mathfrak F}(w_0)$, is  the   depth of  the frame $\mathfrak F$ (or any $n$-model over $\mathfrak F$).

 Let $\mathfrak F=(W,\leq_W,w_0)$ and $\mathfrak G=(U,\leq_U,u_0)$ be two disjoint (that is $W\cap U=\emptyset$) po-frames. The join $\mathfrak F +\mathfrak G$ of the frames is the frame $(W\cup U,\leq,w_0)$ where
$$x\leq y\qquad \Leftrightarrow \qquad x\leq_W y \quad \mbox{or} \quad x\leq_U y\quad  \mbox{or} \quad (x\in W\land y\in U).$$
If $\mathfrak F$ and $\mathfrak G$ are not disjoint, we take their disjoint isomorphic copies and the join of the copies is  called the join of $\mathfrak F$ and $\mathfrak G$ (it is also denoted by $\mathfrak F +\mathfrak G$). Thus, the join of frames is defined up to an isomorphism.
The join is associative (up to an isomorphism) and it is not commutative. Instead of $\mathfrak F +\mathfrak L_1$ and $\mathfrak L_1 +\mathfrak G$, where $\mathfrak L_1$ is one-element frame (see Figure \ref{8fames}), we write $\mathfrak F+$ and $+\mathfrak G$, correspondingly.

Let $(W,R,w_0,V^n)$ be any $n$-model.
The subsets   $\{V^n(w)\}_{w\in W}$ of $\{x_1,\dots,x_n\}$ are usually  given  by their characteristic functions $\mathfrak{f}_w^n\colon\{x_1,\dots,x_n\}\to \{0,1\}$ or binary strings $\mathfrak{f}_w^n=i_1\dots i_n$, where $i_k\in\{0,1\}$. Thus, $n$-models may also appear in the   form $(W,R,w_0,\{V^n(w)\}_{w\in W})$, or  $(W,R,w_0,\{\mathfrak{f}_w^n\}_{w\in W})$.  $n$-Models are usually depicted as graphs whose nodes are labeled with binary strings.

The forcing relation $\mathfrak{M}^n\Vdash_wA$, for any $w\in W$ and  $A\in \mathsf{Fm}^n$, is defined as usual
$$\mathfrak{M}^n\Vdash_wx_i\quad\Leftrightarrow\quad x_i\in V^n(w),\qquad \mbox{ for any } i\leq n;\qquad\qquad\qquad\qquad\qquad\qquad\qquad\qquad\qquad$$
$$\mathfrak{M}^n\Vdash_w\bot,\quad \mbox{for none } w\in W;\qquad\qquad\qquad\qquad\qquad\qquad\qquad\qquad\qquad\qquad\qquad\qquad$$
$$\mathfrak{M}^n\Vdash_w(A\rightarrow B)\quad\Leftrightarrow\quad \forall_{u}\bigl(wRu\quad\mbox{and}\quad \mathfrak{M}^n\Vdash_uA\quad\Rightarrow\quad\mathfrak{M}^n\Vdash_uB\bigr);\qquad\qquad\qquad\qquad\qquad\qquad$$
$$\mathfrak{M}^n\Vdash_w(A\lor B)\quad\Leftrightarrow\quad \bigl(\mathfrak{M}^n\Vdash_wA\quad \mbox{or}\quad\mathfrak{M}^n\Vdash_wB\bigr);\qquad\qquad\qquad\qquad\qquad\qquad\qquad\qquad\qquad\qquad$$
$$\mathfrak{M}^n\Vdash_w(A\land B)\quad\Leftrightarrow\quad \bigl(\mathfrak{M}^n\Vdash_wA\quad \mbox{and}\quad\mathfrak{M}^n\Vdash_wB\bigr).\qquad\qquad\qquad\qquad\qquad\qquad\qquad\qquad\qquad\qquad$$
\begin{lemma}\label{pMm}
If $u R w$ and $\mathfrak{M}^n\Vdash_u A$, then $\mathfrak{M}^n\Vdash_w A, \quad \mbox{for any } u,w\in W \mbox{ and any} A\in \mathsf{Fm}^n.$\end{lemma}
 Let   $(W)_w=\{u\in W\colon wRu\}$, for any $w\in W$. The subframe of $(W,R,w_0)$ {\it generated by $w$} is $((W)_w,R\upharpoonright_{(W)_w},w)$; the submodel of $\mathfrak{M}^n$ {\it generated by $w$} is
$$(\mathfrak{M}^n)_w=((W)_w,R\!\upharpoonright_{(W)_w},w,V^n\!\!\upharpoonright_{(W)_w}).$$
We write $\mathfrak{M}^n\Vdash A$ if $\mathfrak{M}^n\Vdash_{w_0}A$ and we  obviously have $(\mathfrak{M}^n)_w\Vdash A\Leftrightarrow\mathfrak{M}^n\Vdash_wA$.  For any $n$-model, we put $\mathsf{Th}(\mathfrak{M}^n)=\{A\in \mathsf{Fm}^n\colon\mathfrak{M}^n\Vdash A\}$. Given two $n$-models $\mathfrak{M}^n$ and $\mathfrak{N}^n$, we say they are  {\it  equivalent}, in symbols $\mathfrak M^n\thicksim\mathfrak N^n$, if $\mathsf{Th}(\mathfrak{M}^n)=\mathsf{Th}(\mathfrak{N}^n).$

Let $(W,\leq,w_0,V^n)$ and  $(W,\leq,w_0,V'^{\ n})$ be $n$-models over the same po-frame, we say  they are (mutual) {\it variants} if $V(w)=V'(w)$ for each $w\not=w_0$.

Let $\mathbf{F}$ be a class of frames and $\mathbf{M}^n(\mathbf{F})$, for any $n\geq 0$,    be the class of $n$-models  over the frames $\mathbf{F}$; we write  $\mathbf{M}^n$, instead of $\mathbf{M}^n(\mathbf{F})$, if there is no danger of confusion. The intermediate logic determined by $\mathbf{F}$ is denoted by $\mathsf{L}(\mathbf{F})$. Thus, if $A\in\mathsf{Fm^n}$, then
$$A\in\mathsf{L}(\mathbf{F}) \quad \Leftrightarrow \quad (\mathfrak{M}^n\Vdash A, \mbox{ for every } \mathfrak{M}^n\in \mathbf{M}^n). $$
We say that  {\bf F} are {\it frames of an intermediate logic {\sf L}} if $\mathsf{L}\subseteq\mathsf{L}(\mathbf{F})$  and {\sf L} {\it omits a frame}  $\mathfrak{F}$ if $\mathfrak{F}$ is not a frame of {\sf L}. A logic {\sf L} is  {\it Kripke complete} if $\mathsf{L}=\mathsf{L}(\mathbf{F})$ for some  $\mathbf{F}$.
 The logic $\mathsf{L}(\mathbf{F})$ is said to be {\it tabular} if $\mathbf{F}$ is a finite family  of finite frames.
 {\sf L} is {\it Halld\'{e}n complete} ({\it H}-complete) if for any formulas $A,B$ with $\mathsf{Var}(A)\cap\mathsf{Var}(B)=\emptyset$ we have
$$ \mathsf{L}\vdash A\lor B \quad \Rightarrow \quad \mathsf{L}\vdash A \quad\mbox{ or }\quad  \mathsf{L}\vdash B.$$
\begin{theorem}\label{hcl}  Let {\bf F} be  finite. Then  $\mathsf{L}(\mathbf F)$ is {\it H}-complete iff $\mathsf{L}(\mathbf F)=\mathsf{L}(\mathfrak F)$ for some $\mathfrak F\in \mathbf F$.
\end{theorem}
A  logic {\sf L} is  {\it locally tabular} if  \ {\sf Fm$^n\slash\!\!=_{\mathsf L}$} is finite, for each $n\geq 0$. Tabular logics are locally tabular but not vice versa.  For each locally tabular logic {\sf L} there exists a family $\mathbf{F}$ of finite frames such that  $\mathsf{L}=\mathsf{L}(\mathbf{F})$. Thus, locally tabular logics have the {\it finite model property} but, again, the converse is false. A logic {\sf L} is said to be in the $n$-{\it slice} if $\mathsf{L}=\mathsf{L}(\mathbf{F})$ for a family {\bf F} of finite po-frames such that $d(\mathfrak{F})\leq n$, for any $\mathfrak{F}\in\mathbf F$.
\begin{theorem}\label{lf6}  Suppose that the family {\bf F} consists of finite frames. Then  $\mathsf{L}(\mathbf F)$ is locally tabular iff $\mathbf{M}^n\slash\!\!\thicksim$ is finite, for each $n$.
\end{theorem}
\begin{proof} $(\Rightarrow)$ Using finitely many (up to equivalence) formulas we do not distinguish infinitely many models. $(\Leftarrow)$ is obvious.
\end{proof}
\begin{corollary}\label{fp} (i) If {\sf L} and {\sf L'} are locally tabular intermediate logics, then their intersection $\mathsf L\cap\mathsf L'$ is also a locally tabular intermediate logic;\\
(ii) any extension of any locally tabular intermediate logic is locally tabular.
\end{corollary}
\begin{proof}
(i) Let {\sf L=$\mathsf L({\mathbf F})$} and {\sf L'=L({\bf G})} for some classes {\bf F,G} of finite frames. Then $\mathsf{ L}\cap\mathsf{L'}=\mathsf{L}(\mathbf{F}\cup\mathbf{G})$ and  $\mathbf{M}^n(\mathbf{F}\cup\mathbf{G}) =
\mathbf{M}^n(\mathbf{F}) \ \cup \ \mathbf{M}^n(\mathbf{G})$ .Thus,  $\mathbf{M}^n(\mathbf{F}\cup\mathbf{G})\slash\!\!\thicksim$ \ is finite if \ $\mathbf{M}^n(\mathbf{F})\slash\!\!\thicksim$ \ and $\mathbf{M}^n(\mathbf{G})\slash\!\!\thicksim$ \ are finite. \
(ii) is obvious.\end{proof}
Let us characterize po-frames of the logics in Figure \ref{ILs}. Thus,
{\sf LC}-frames are chains and we let  $\mathfrak L_d$, for any natural number $d\geq 1$, be the chain on $\{1,2,\dots ,d\}$ with the reverse (natural) ordering $\geq$, where $d$ is the root and $1$ is the top (=greatest) element. Finite {\sf KC}-frames have top elements. $\mathsf H_n$-Frames are of the depth $\leq n$ and $\mathsf H_n\mathsf B_m$-frames have (additionally) $m$-bounded branching, that is each point has at most $m$ immediate successors.
To get {\sf PWL}-frames we need  unrooted frames; {\sf PWL}-frames are
$$\mathfrak F_n+ \mathfrak I_{n_1}+\cdots+\mathfrak I_{n_k} \footnote{$\mathfrak F_n+ \mathfrak I_{n_1}+\cdots+\mathfrak I_{n_k}$ denotes the vertical union with $\mathfrak F_n$  on the top and $\mathfrak I_{n_k}$ on the bottom},\quad \mbox{where $n\geq 0$ and $n_1,\dots,n_k\geq 1$;} $$  where $\mathfrak I_n$ is the frame with the identity relation on an $n$-element set (and we agree that $\mathfrak F_0=\mathfrak L_1$ and $\mathfrak F_1=\mathfrak L_2$). Note that the frames in Figure \ref{hpa} are {\sf PWL}-frames and hence $\mathsf L({\mathbf H}_{pa})$ and $\mathsf L({\mathbf H}_{un})$ are extensions of {\sf PWL}.

There are three pretabular intermediate logics, see \cite{Maks72}: {\sf LC} of G\"odel and Dummett, given by all chains $\mathfrak L_n$, {\sf LJ} of Jankov, given by all $n$-forks $\mathfrak F_n$, and  {\sf LH} of Hosoi, given by all rhombuses $\mathfrak R_n$; see Figure \ref{FRF}.

A pair of logics  $(\mathsf L_1,\mathsf L_2)$ is a {\it splitting pair} of the lattice of (intermediate) logics if $\mathsf L_2\not\subseteq \mathsf L_1$ and, for any intermediate logic $\mathsf L$, either  $\mathsf L \subseteq\mathsf L_1$, or $\mathsf L_2 \subseteq\mathsf L$.\footnote{In the same way, one can define a splitting pair in any complete lattice.} Then   we say $\mathsf L_1$  splits the lattice and $\mathsf L_2$ is the splitting (logic) of the lattice, see  \cite{ZWC}.
 Jankov \cite{Jankov} {\it characteristic formula} of a finite rooted frame $\mathfrak F$ is denoted by $\chi (\mathfrak F)$. \footnote{Jankov originally defined $\chi (\mathfrak F)$ for any subdirectly irreducible finite Heyting algebra. By duality, finite rooted frames are tantamount to finite s.i. algebras and hence we proceed as if $\chi (\mathfrak F)$ were defined for frames.}
\begin{theorem}\label{Jankov} The pair $(\mathsf L(\mathfrak F),\mathsf L(\chi (\mathfrak F))$ is a splitting pair, for any finite  frame $\mathfrak F$.  Thus, for any intermediate logic {\sf L} and any finite frame $\mathfrak F$, the logic {\sf L} omits  $\mathfrak F$ iff $\chi (\mathfrak F) \in \mathsf L$.
\end{theorem}
For instance  $\mathsf {KC} =\mathsf {L}(\{\chi ({\mathfrak F_2}) \}$ is the splitting logic.
 If $\{\mathsf L_i\}_{i\in  I}$ is a family of splitting logics, then $\mathsf L(\bigcup_{i\in I}\mathsf L_i)$ is called {\it a union splitting}.  For instance,   $\mathsf {LC} =\mathsf {L}(\{\chi ({\mathfrak F_2}),  \chi ({\mathfrak R_2}) \})$ is a union splitting but not a splitting.
\begin{corollary}\label{Jankov2} If $\{(\mathsf L'_{i},\mathsf L_{i})\}_{i\in I}$ is a family of splitting pairs  and $\mathsf L=\mathsf L(\bigcup_{i\in I}\mathsf L_i)$, then $\mathsf L$ is a union splitting and, for any intermediate logic $\mathsf L'$, either  $\mathsf L' \subseteq\mathsf L'_i$ for some $i\in I$, or $\mathsf L \subseteq\mathsf L'$.\end{corollary}

\subsection{The Problem of Unification.}\label{UP}
 A  substitution $\varepsilon$  is a  \emph{unifier} for a formula $A$ in a logic $\mathsf L$ (an  $\mathsf L$-\emph{unifier} for $A$) if   $ \varepsilon(A)\in\mathsf{L}$.
In any intermediate logic, the set of unifiable formulas  coincides with the set of consistent formulas.
A set $\Sigma$ of {\sf L}-unifiers for $A$ is said to be {\it complete}, if for each {\sf L}-unifier $\mu$ of $A$, there is a unifier $\varepsilon\in \Sigma$ such that $\varepsilon\preccurlyeq_{\sf L}\mu$.

The unification type of {\sf L} is $1$ (in other words, unification in {\sf L} is {\it unitary}) if the set of unifiers of any unifiable formula $A$ contains a least, with respect to $\preccurlyeq_L$, element called  {\it a most general unifier} of $A$, (an mgu of $A$). In other words, unification in {\sf L} is unitary if each unifiable formula has a one-element complete set of unifiers.
The unification type of {\sf L} is $\omega$ (unification in {\sf L} is {\it finitary}), if it is not $1$ and each unifiable formula has a finite complete set of unifiers.
 The unification type of {\sf L} is $\infty$ (unification in {\sf L} is {\it  infinitary})  if it is not $1$, nor $ \omega$, and each unifiable formula has a minimal (with respect to inclusion) complete set of unifiers.
The unification type of {\sf L} is $0$ (unification in {\sf L} is {\it nullary}) if there is a unifiable formula which has no minimal complete set
of unifiers.
In a similar way one  defines the unification type of any {\sf L}-unifiable formula. The unification type of the logic is the worst unification type of its unifiable formulas.\\

 Ghilardi \cite{Ghi2} introduced projective unifiers and formulas; an  $\mathsf L$-{unifier} $\varepsilon$ for $A$  is called \emph{ projective} if $A \vdash_{\mathsf L}\varepsilon(x) \leftrightarrow x$, for each variable $x$ (and consequently $A \vdash_{\mathsf L}\varepsilon(B) \leftrightarrow B$, for each $B$).
A formula $A$ is said to be {\it projective} in $\mathsf L$ (or $\mathsf L$-projective) if it has a projective  unifier in $\mathsf L$. It is said that a logic $\mathsf L$ enjoys {\it projective unification} if each {\sf L}-unifiable formula is $\mathsf L$-projective. An $\mathsf L$-projective formula may have many non-equivalent in $\mathsf L$-projective unifiers and  each {\sf L}-projective unifier is its  mgu:
\begin{lemma}\label{proj}
If $\varepsilon$ is an {\sf L}-projective unifier for $A$  and $\sigma$ is any {\sf L}-unifier for $A$, then $\sigma\circ\varepsilon=_\mathsf{L}\sigma$.
\end{lemma}
Thus,  projective unification implies unitary unification. If $A\in \mathsf{Fm^n}$ is {\sf L}-projective, then $A$ has a projective unifier $\varepsilon\colon\{x_1,\dots,x_n\}\to \mathsf{Fm}^n$ that is a mgu {\it preserving the variables of $A$} (which is not always the case with unitary unification). In contrast to unitary unification, projective unification is also monotone:
\begin{lemma}\label{mon}
If $A$ is $\mathsf L$-projective and $\mathsf L\subseteq \mathsf L'$, then $A$ is $\mathsf L'$-projective.
\end{lemma}
Ghilardi \cite{Ghi2} gives a semantical characterization of projective formulas. The condition (ii) is called {\it the extension property}.\footnote{More specifically, the theorem says that the class of models of a projective formula enjoys extension property.}:
\begin{theorem}\label{niu2}  Let $\mathbf{F}$ be a class of finite po-frames and $\mathsf{L}=\mathsf{L}(\mathbf{F})$. The followings   are equivalent:\\
(i) $ A$ is {\sf L}-projective;\\
(ii) for every $n$-model $\mathfrak{M}^n=(W,\leq,w_0,V^n)$ over a po-frame $(W,\leq,w_0)$ of the logic {\sf L}:\\ if
$(\mathfrak{M}^n)_w\Vdash A$  for each $w\not=w_0$,
then  $\mathfrak{N}^n\Vdash A$ for some variant $\mathfrak{N}^n$ of $\mathfrak{M}^n$.
\end{theorem}
 Wro\'{n}ski \cite{Wro1,Wro2} proved that
\begin{theorem}\label{projj}
 An intermediate logic  {\sf L} has projective unification iff \
{\sf LC} $\subseteq$ {\sf L}.
\end{theorem}
There are unitary logics  which are not projective. Following Ghilardi  and Sachetti \cite{Ghisac},  unification in {\sf L} is said to be \emph{filtering} if given two unifiers, for any formula $A$, one can find a unifier that is more general than both of them. Unitary unification is filtering.  If unification is filtering, then every unifiable formula either has an mgu  or no basis of
unifiers exists (unification is nullary). It is known, see e.g. \cite{dzSpl}, that
\begin{theorem}\label{fil}
 Unification in any intermediate logic {\sf L} is filtering iff \
{\sf KC} $\subseteq$ {\sf L}.
\end{theorem}
If  $\varepsilon,\sigma\colon\{x_1,\dots,x_n\}\to\mathsf{Fm}^k$ are unifiers of a formula $A(x_1,\dots,x_n)$ in (any extension of) {\sf KC}, then, as a more general unifier for $\varepsilon,\sigma$ the following substitution $\mu$ can be taken (where $y$ is a fresh variable, i.e.$y\not\in\mathsf{Fm^k}$):
$$\mu(x_i)\qquad=\qquad(\varepsilon(x_i)\land \neg y) \quad \lor \quad (\sigma(x_i)\land \neg\neg y),\qquad \mbox{for $i=1,\dots,n$.}$$
Thus, unifiers in filtering unification {\it introduce new variables}.  We have, see  \cite{dzSpl, Ghi2},
\begin{theorem}\label{kc}  {\sf KC}  is the least intermediate logic  with unitary unification. All extensions of  {\sf KC} have nullary or unitary unification. All intermediate logics with finitary unification are included in  {\sf L}($\mathfrak{F}_{2}$), the logic determined by the `fork frame' $\mathfrak {F}_{2}$ see Figure \ref{8fames}.  ({\sf L}($\mathfrak{F}_{2}$),{\sf KC})  is a splitting pair of the lattice of intermediate logics. \end{theorem}
Logics with finitary and unitary unification are  separated by the splitting ({\sf L}($\mathfrak{F}_{2}$),{\sf KC}). Let us agree that having  {\it good unification}  means either unitary, or finitary one. Given a logic {\sf L} with good unification,  it has  unitary or finitary unification depending only on that if {\sf L} contains {\sf KC} or not. Our aim would be to distinguish logics with good unification  from those with nullary one. We show in later that locally tabular intermediate logics with infinitary unification do not exist at all.
 Let us notice that the splitting generated by ({\sf L}($\mathfrak{F}_{2}$),{\sf KC}) is irrelevant for logics with nullary unification; there are extensions of {\sf KC}, as well as sublogics of {\sf L}($\mathfrak{F}_{2}$), that have nullary unification.

A logic {\sf L} is said to have {\it projective approximation} if, for each formula $A$ one can find a finite set $\Pi(A)$ of {\sf L}-projective formulas such that:\\
(i) \ $ \mathsf{Var}(B)\subseteq \mathsf{Var}(A)$ and $B\vdash_\mathsf{L}A$, for each $B\in \Pi(A)$;\\
(ii) each {\sf L}-unifier of $A$ is an {\sf L}-unifier of some $B\in\Pi(A)$.\footnote{Ghilardi \cite{Ghi1,Ghi2}, instead of assuming $\Pi(A)$ is finite, postulates   $deg(B)\leq deg(A)$, for each $B\in \Pi(A)$, from which it follows that $\Pi(A)$ is finite. The condition $deg(B)\leq deg(A)$ is relevant for logics with disjunction property, like {\sf INT}, but is irrelevant for  locally tabular logics where $\mathsf{Var}(B) \subseteq \mathsf{Var}(A)$ is sufficient. We decided, therefore, to modify slightly Ghilardi's formulations preserving,  we hope, his ideas.  }

If a finite $\Pi(A)$ exists we can assume that all $B\in\Pi(A)$ are maximal (with respect to $\vdash_{\sf L}$) {\sf L}-projective formulas fulfilling (i). But, even if there is  finitely many maximal {\sf L}-projective formulas  fulfilling (i), we cannot be sure  (ii) is fulfilled.

\begin{theorem}\label{praprox} Each  logic with projective approximation has finitary (or unitary) unification.
\end{theorem}
Logics with projective approximation play a similar role for finitary unification as projective logics do for  unitary unification, even though projective approximation is not monotone.
 Ghilardi  \cite{Ghi2} proved that
\begin{theorem}\label{int} Intuitionistic propositional logic {\sf INT} enjoys projective approximation and hence unification in {\sf INT} is finitary.
\end{theorem}

\section{Intuitionistic Kripke $n$-Models.}\label{km}

\subsection{p-Morphisms.}\label{pM}
Let $(W,R,w_0,V^n)$ and $(U,S,u_0,V'^n)$ be $n$-models. A mapping $p\colon W{\to} U$, from $W$ \underline{onto} $U$, is said to be a {\it p-morphisms of their frames},
$p\colon (W,R,w_0)\to (U,S,u_0), \mbox{if}$\\
\indent(i) $wRv\Rightarrow p(w)Sp(v), \quad\mbox{for any } w,v\in W$;\\
\indent (ii) $p(w)Sa\Rightarrow \exists_{v\in W}\bigl(wRv\land p(v)=a\bigr), \quad \mbox{for any }w\in W \ \mbox{and } \ a\in U$;\\
\indent (iii) $p(w_0)=u_0$.\\
 {\it A p-morphism of $n$-models}, $p\colon (W,R,w_0,V^n)\to (U,S,u_0,V'^n)$ fulfills (additionally)\\
\indent (iv) $V^n(w)=V'^n(p(w))$, for any $w\in W$.

\noindent If $p\colon\mathfrak{M}^n\to\mathfrak{N}^n$ is a p-morphism, then $\mathfrak{N}^n$ is called a p-morphic image (or reduct, see \cite{ZWC}) of $\mathfrak{M}^n$ and we write $p(\mathfrak{M}^n)=\mathfrak{N}^n$. Reducing $\mathfrak{M}^n$ (by a p-morphism), we  preserve its logical properties. In particular, $p(\mathfrak{M}^n)\thicksim\mathfrak{M}^n$ as
\begin{lemma}\label{pM0}
If $p\colon \mathfrak{M}^n\to \mathfrak{N}^n$ is a p-morphism of $n$-models, $w\in W$ and  $A\in\mathsf{Fm}^n$, then
$$\mathfrak{M}^n\Vdash_{w}A\quad\Leftrightarrow\quad p(\mathfrak{M}^n)\Vdash_{p(w)}A.$$
\end{lemma}
 p-Morphisms are also used in modal logic. The above property is generally valid which means it also holds for modal models and modal formulas and {it can be shown without assuming that $R$ is a pre-order  and  $V^n$ is monotone.}
 \begin{example}\label{pMe} Let $\mathfrak{M}^n=(W,R,w_0,V^n)$ be an $n$-model  in which the pre-order $R$ is not a partial order. Let $w\thickapprox v\Leftrightarrow wRv\land vRw$, for any $w,v\in W$. Then $\thickapprox$ is an equivalence relation on $W$ and one can easily show that the canonical mapping $p(w)=[w]_\thickapprox$, for any $w\in W$, is a p-morphism from $\mathfrak{M}^n$ onto the quotient model
 $$\mathfrak{M}^n\slash\!\!\thickapprox\quad =\quad \bigl(W\slash\!\!\thickapprox,R\slash\!\!\thickapprox,[w_0]_\thickapprox,V^n\!\!\slash\!\thickapprox\bigr).$$
Reducing all $R$-clusters  to single points, we receive an equivalent $n$-model over a po-set;
and hence po-sets (not pre-orders)  are  often taken as intuitionistic frames. \hfill\qed
 \end{example}

If a p-morphism $p\colon\mathfrak{M}^n\to\mathfrak{N}^n$ is one-to-one, then  $w R v\Leftrightarrow p(w) S p(v),$ for any $w,v\in W$ which means $p$ is {\it an isomorphism} and, if there is an isomorphism between the $n$-models, we write $\mathfrak{M}^n\equiv\mathfrak{N}^n$. It is usual to identify isomorphic objects.

\subsection{Bisimulations.}\label{biss}

Bisimulations (between Kripke frames) were introduced by K.Fine \cite{fine}, by imitating Ehrenfeucht games.
They found many applications. In particular, S.Ghilardi \cite{Ghi2} used  bounded bisimulation to characterize projective formulas. We  show that bisimulations are closely related to p-morphisms. In our approach  we follow A.Patterson \cite{Pat}.

 A binary relation $B$ on $W$ is {\it a bisimulation of the frame}  $(W,R,w_0)$  if
$$wBv\Rightarrow\forall_{w'}\exists_{v'}(wRw'\Rightarrow vRv'\land w'Bv')\land\forall_{v'}\exists_{w'}(vRv'\Rightarrow wRw'\land w'Bv').$$
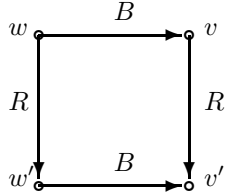
\begin{figure}[H]
 \unitlength1cm
\begin{picture}(3,2)
\thicklines

\put(5,2){\vector(1,0){1.9}}
\put(5,2){\vector(0,-1){1.9}}
\put(7,2){\vector(0,-1){1.9}}
\put(5,0){\vector(1,0){1.9}}
\put(5,0){\circle{0.1}}
\put(5,2){\circle{0.1}}
\put(7,0){\circle{0.1}}
\put(7,2){\circle{0.1}}

\put(4.6,2){\mbox{$w$}}
\put(7.2,2){\mbox{$v$}}
\put(4.6,0){\mbox{$w'$}}
\put(7.2,0){\mbox{$v'$}}
\put(6,2.2){\mbox{$B$}}
\put(6,0.2){\mbox{$B$}}
\put(4.6,1){\mbox{$R$}}
\put(7.2,1){\mbox{$R$}}

\end{picture}\caption{Bisimulation}\label{bis}\end{figure}
\noindent Note that $wBv\Rightarrow\forall_{w'}\exists_{v'}(wRw'\Rightarrow vRv'\land w'Bv')$  suffices if $B$ is symmetric.
 {\it A bisimulation of the $n$-model} $(W,R,w_0,V^n)$ additionally fulfils  $V^n(w)=V^n(v)$ if $wBv$.
\begin{lemma}\label{pMf}
(i)  If $B$ is a bisimulation of  $\mathfrak{M}^n$, then $B\!\!\upharpoonright_{(W)_w}$ is a bisimulation of $(\mathfrak{M}^n)_w$;\\
(ii) if $B$ is a bisimulation of  $(\mathfrak{M}^n)_w$, then $B$ is a bisimulation of $\mathfrak{M}^n$;  for any $w\in W$.
\end{lemma}

\begin{lemma}\label{pM4} If $B$ is a  bisimulation (of a frame or an $n$-model), then  the least equivalence relation $B^\star$ containing $B$ is also a bisimulation.
\end{lemma}
\begin{proof} A proof of this lemma  can be found in \cite{Pat}. Let us only specify properties of bisimulations which are useful here.

\noindent(i) \quad $\{(w,w)\colon w\in W\}$ is a bisimulation.

\noindent (ii)\quad $B$ is a bisimulation $\Rightarrow$ $B^{-1}$ is  a bisimulation.

\noindent (iii)\quad $\forall_i(B_i$ is a bisimulation) $\Rightarrow \quad \bigcup_iB_i$ is a bisimulation.

\noindent (iv)\quad $B$ is a bisimulation $\Rightarrow$ the transitive closure of $B$ is a bisimulation.
\hfill\qed\end{proof}

Suppose that $B$ is an equivalence  bisimulation of an $n$-model $\mathfrak{M}^n=(W,R,w_0,V^n)$. Let us define  $\mathfrak{M}^n\slash B=(W\slash B,R\slash B,[w_0]_B,[V]^n)$ where $W\slash B=\{[w]_B\colon w\in W\}$, and $[V]^n([w]_B)=V^n(w)$ for any $w\in W$, and
$$[w]_B\ R\slash B \ [v]_B \quad\Leftrightarrow\quad \exists_{w'v'}\bigl(wBw'\land vBv'\land w'Rv'\bigr).$$
\begin{theorem}\label{pM7} If $B$ is an equivalence  bisimulation of  an $n$-model $\mathfrak{M}^n$, then $\mathfrak{M}^n\slash B$ is an $n$-model and the canonical mapping $[\ ]_B\colon W\to W\slash B$ is a p-morphism of the $n$-models.\end{theorem}
\begin{proof} We should show that $R\slash B$ is a pre-order. If $w=v$, one can take $w'=v'=w$ (in the definition of $R\slash B$) to show $[w]_B\ R\slash B \ [w]_B$. Thus, $R\slash B$ is reflexive.

 Suppose that $[w]_B\ R\slash B\ [v]_B\ R\slash B\ [u]_B$, for some $w,v,u\in W$. Then $wBw'\land vBv'\land w'Rv'$ and  $vBv''\land uBu''\land v''Ru''$, for some $w',v',v'',u''\in W$. But $B$ is an equivalence,  hence $v''Bv'$ and, by $v''Ru''$, we get $v'Ru'\land u''Bu'$, for some $u'\in W$, as $B$ is a bisimulation. By transitivity of $R$, we have $w'Ru'$ and $uBu'$ as $B$ is an equivalence relation. Thus, $[w]_B\ R\slash B\ [u]_B$; the relation $R\slash B$ is transitive.

 There remains to show that the canonical mapping is a p-morphism.\\
(i) If $wRv$, then $[w]_B\ R\slash B\ [v]_B$, by the definition of $R\slash B$.

\noindent (ii) Suppose that $[w]_B\ R\slash B\ [v]_B$, for some $w,v\in W$. Then $wBw'$, and $vBv'$, and $w'Rv'$, for some $w',v'\in W$. As $B$ is a bisimulation, $wRv''\land v''Bv'$, for some $v''\in B$. Thus, $wRv''$ and $[v'']_B=[v]_B$, as required.

The conditions (iii) and (iv) are obviously fulfilled.\hfill\qed\end{proof}
\begin{theorem}\label{pMp}
If $B$ and $B'$ are  equivalence  bisimulations of  an $n$-model $\mathfrak{M}^n=(W,R,w_0,V^n)$ and $B'\subseteq B$, then there is a p-morphism $q\colon\mathfrak{M}^n\slash B'\to\mathfrak{M}^n\slash B$ such that the diagram in Figure \ref{pms} commutes.\end{theorem}

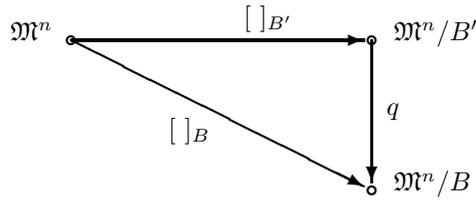
\begin{figure}[H]
\unitlength1cm
\begin{picture}(4,1.5)
\thicklines

\put(9,2){\vector(0,-1){1.9}}

\put(5,2){\vector(2,-1){3.9}}
\put(5,2){\vector(1,0){3.9}}
\put(9,0){\circle{0.1}}
\put(5,2){\circle{0.1}}
\put(9,2){\circle{0.1}}
\put(9.3,2){\mbox{$\mathfrak{M}^n\slash B'$}}
\put(4.2,2){\mbox{$\mathfrak{M}^n$}}
\put(9.3,0){\mbox{$\mathfrak{M}^n\slash B$}}
\put(9.2,1){\mbox{$q$}}
\put(7.3,2.2){\mbox{$[\ ]_{B'}$}}
\put(6.3,0.7){\mbox{$[\ ]_{B}$}}

\end{picture}
\caption{Comparison of Bisimulations.}\label{pms}
\end{figure}
\begin{proof} Let us define $q([w]_{B'})=[w]_B$ and notice that the mapping is well-defined and maps $W\slash B'$ onto $W\slash B$. We should only cheque that $q$ is a p-morphism. Note that the conditions (i),(iii) and (iv) are quite obvious.

(ii) Suppose that $q([w]_{B'})R\slash B\ [u]_B$. By the definition of $R\slash B$, there are $w',u'$ such that $wBw'Ru'Bu$. Since $B$ is a bisimulation and $wBw'Ru'$ there is an $u''$ such that $wRu''Bu'$. Thus, $[w]_{B'}R\slash B'\ [u'']_{B'}$ and $q([u'']_{B'})=[u'']_B=[u]_B$ as required.
\hfill\qed\end{proof}

\begin{theorem}\label{pMr}
If $p:\mathfrak{M}^n\to \mathfrak{N}^n$  is a  p-morphism of $n$-models, then
$$wBv\quad\Leftrightarrow\quad p(w)=p(v)$$
is an equivalence bisimulation of the $n$-model $\mathfrak{M}^n$, and $\mathfrak{M}^n\slash B\equiv\mathfrak{N}^n$.\end{theorem}
\begin{proof}Let $wBv$ and $wRw'$ for some $w,w',v\in W$ (see Figure \ref{bis}). Then $p(w)=p(v)$ and $p(w)Sp(w')$, where $S$ is the accessibility relation in $\mathfrak{N}^n$. Thus, $p(v)Sp(w')$. Since $p$ is a p-morphism, $vRv'$ and $p(v')=p(w')$, for some $v'\in W$. Thus, $vRv'$ and $w'Bv'$.

In the same way one shows $wBv$ and $vRv'$ give us $wRw'$ and $w'Bv'$, for some $w'$, and we obviously have $V^n(w)=V^n(v)$ if $wBv$.

The $n$-models $\mathfrak{M}^n\slash B$ and $\mathfrak{N}^n$ are isomorphic as the mapping $i([w]_B)=p(w)$ is well defined, one-to-one and p-morphic.
\hfill\qed\end{proof}

Bisimulations preserve such properties of frames as reflexivity, symmetry, transitivity; consequently, p-morphic images preserve these properties, as well. There are, however, some properties  which are not preserved by p-morphisms.
 \begin{example}\label{pMex}
Let $W=\{u_i\colon i\geq 0\}\cup \{v_i\colon i\geq 0\}\cup\{w_0\}$ and a partial order $R$ on $W$, and a bisimulation $B$ on $W$, are defined as in the following picture (see Figure \ref{asym})
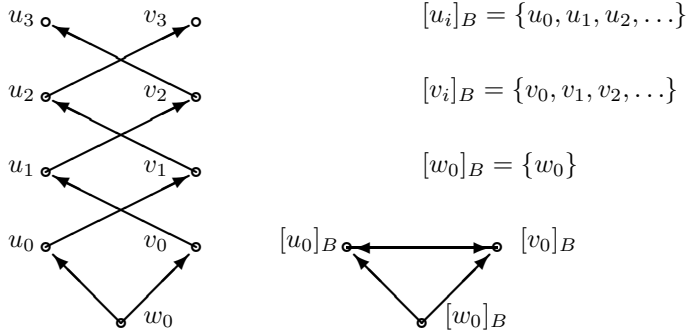
\begin{figure}[H]
 \unitlength1cm
\begin{picture}(3,4)
\thicklines
\put(2,0){\circle{0.1}}
\put(1,1){\circle{0.1}}
\put(3,1){\circle{0.1}}
\put(1,2){\circle{0.1}}
\put(3,2){\circle{0.1}}
\put(1,3){\circle{0.1}}
\put(3,3){\circle{0.1}}
\put(1,4){\circle{0.1}}
\put(3,4){\circle{0.1}}
\put(2,0){\vector(-1,1){0.9}}
\put(2,0){\vector(1,1){0.9}}
\put(3,1){\vector(-2,1){1.9}}
\put(1,1){\vector(2,1){1.9}}
\put(2.3,0){\mbox{$w_0$}}
\put(0.5,1){\mbox{$u_0$}}
\put(2.3,1){\mbox{$v_0$}}
\put(0.5,2){\mbox{$u_1$}}
\put(2.3,2){\mbox{$v_1$}}
\put(0.5,3){\mbox{$u_2$}}
\put(2.3,3){\mbox{$v_2$}}
\put(0.5,4){\mbox{$u_3$}}
\put(2.3,4){\mbox{$v_3$}}

\put(3,2){\vector(-2,1){1.9}}
\put(1,2){\vector(2,1){1.9}}
\put(3,3){\vector(-2,1){1.9}}
\put(1,3){\vector(2,1){1.9}}
\put(6,0){\circle{0.1}}
\put(5,1){\circle{0.1}}
\put(7,1){\circle{0.1}}
\put(6,0){\vector(-1,1){0.9}}
\put(6,0){\vector(1,1){0.9}}
\put(5,1){\vector(1,0){1.9}}
\put(7,1){\vector(-1,0){1.9}}
\put(6,4){\mbox{$[u_i]_B=\{u_0,u_1,u_2,\dots\}$}}
\put(6,3){{\mbox{$[v_i]_B=\{v_0,v_1,v_2,\dots\}$}}}
\put(6,2){\mbox{$[w_0]_B=\{w_0\}$}}
\put(6.3,0){\mbox{$[w_0]_B$}}
\put(4.1,1){\mbox{$[u_0]_B$}}
\put(7.3,1){\mbox{$[v_0]_B$}}\end{picture}\caption{Weak Asymmetry is not Preserved.}\label{asym}\end{figure}
\noindent Thus, a p-morphic image of a partial order is not a partial order (only pre-order).  \hfill\qed \end{example} Note that the set $W$ in the above Example is infinite which is essential as
\begin{corollary}\label{fin} Any p-morhic image of any finite po-frame is a po-frame.
\end{corollary}

\subsection{p-Irreducible $n$-Models.}\label{pirm}

 An $n$-model $\mathfrak{M}^n$ is said to be {\it p-irreducible} if each  p-morphism $p\colon \mathfrak{M}^n\to \mathfrak{N}^n$, for any $n$-model $\mathfrak{N}^n$, is an isomorphism. Thus, any  p-morphic image of any irreducible $n$-model is its isomorphic copy.\footnote{ The concept of p-irreducibility, in contrast to other concepts in this Section, would  make no sense for frames.}
Irreducible $n$-models are  po-sets, see Example \ref{pMe}, and we show any $n$-model can be  reduced to a p-irreducible one.
\begin{theorem}\label{Irr}
For each $n$-model $\mathfrak{M}^n$ there exists a p-irreducible $n$-model $\mathfrak{N}^n$ which is a p-morphic image of $\mathfrak{M}^n$ (and $\mathfrak{N}^n$ is unique up to $\equiv$).
\end{theorem}
\begin{proof} Let $\mathfrak{M}^n=(W,R,w_0,\{\mathfrak{f}^n_w\}_{w\in W})$ and $B$ be the least equivalence  on $W$ containing  $\bigcup\{B_i\colon B_i \ \mbox{is a bisimulation on } \mathfrak{M}^n\}.$ By Lemma \ref{pM4}, $B$ is the greatest bisimulation on $\mathfrak{M}^n$. Take $\mathfrak{N}^n=\mathfrak{M}^n\!\!\slash B$, see Theorem \ref{pM7}.
Since the composition of any two p-morphisms is a p-morphism, any p-morphic image $\mathfrak{N'}^n$ of $\mathfrak{N}^n$ would be a p-morphic image of  $\mathfrak{M}^n$. Thus, by maximality of $B$, we would get, by Theorem \ref{pMp}, an isomorphism $p'\colon\mathfrak{N'}^n\equiv\mathfrak{N}^n$ which means $\mathfrak{N}^n$ is p-irreducible.

The uniqueness of $\mathfrak{N}^n$ also follows; if $\mathfrak{N'}^n$ were another p-irreducible p-morphic image of $\mathfrak{M}^n$, we would get by Theorems \ref{pMp} and \ref{pMr}, a p-morphism  $p'\colon\mathfrak{N'}^n\to\mathfrak{N}^n$ which would mean  that $\mathfrak{N'}^n$ and $\mathfrak{N}^n$ are isomorphic. \hfill\qed\end{proof}
The following theorem could give another characterization of p-irreducible $n$-models.

\begin{theorem}\label{pM5} If an $n$-model $\mathfrak{M}^n$ is p-irreducible, then for any $n$-model $\mathfrak{N}^n$ there is at most one p-morphism $p\colon \mathfrak{N}^n\to \mathfrak{M}^n$.
\end{theorem}
\begin{proof} Let $\mathfrak{M}^n=(W,R,w_0,V^n)$ be p-irreducible and $p,q\colon \mathfrak{N}^n\to \mathfrak{M}^n$ be two (different) p-morphisms for some  $\mathfrak{N}^n=(U,S,u_0,V'^n)$. Take $B=\{(p(v),q(v))\colon v\in V\}$ and let us show  $B$ is a bisimulation on $\mathfrak{M}^n$. This would be a contradiction as, if $B^\star$ were  the least equivalence relation containing $B$ (see Lemma \ref{pM4}),  $[\ ]_{B^\star}\colon  \mathfrak{M}^n\to\mathfrak{M}^n\slash B^\star$ would be a non-isomorphic p-morphism, see Theorem \ref{pM7}, and it would mean that $\mathfrak{M}^n$ were not p-irreducible.

Let $p(v)Rw$, for some $v\in V$ and $w\in W$. As $p$ is a p-morphism, $p(v')=w$ and $vSv'$ for some $v'\in V$. Then $q(v)Rq(v')$, as $q$ is a p-morphism, and  $wBq(v')$ as $w=p(v')$.

Similarly, if $q(v)Rw$, for some $v\in V$ and $w\in W$, then $q(v')=w$ and $vSv'$, for some $v'\in V$, and hence $p(v)Rp(v')$ and  $p(v')Bw$ (as $w=q(v'))$.
  \hfill\qed\end{proof}
\begin{theorem}\label{pM6} If  $\mathfrak{M}^n$ is p-irreducible, then $(\mathfrak{M}^n)_w$ is p-irreducible for each $w\in W$.\end{theorem}
\begin{proof} Let $\mathfrak{M}^n=(W,R,w_0,V^n)$ and suppose  $(\mathfrak{M}^n)_w$ is not p-irreducible for some $w\in W$. By Theorem \ref{pMr}, there is a (non-trivial) bisimulation $B$ on $(\mathfrak{M}^n)_w$. Since (by Lemma \ref{pMf}) $B$ is  a bisimulation of $\mathfrak{M}^n$, if we extend $B$ (see Lemma \ref{pM4}) to an equivalence bisimulation $B^\star$ of $\mathfrak{M}^n$, we  get a (non-isomorphic) p-morphism of $\mathfrak{M}^n$, see Theorem  \ref{pM7}.
Thus, $\mathfrak{M}^n$ is not p-irreducible.     \hfill\qed\end{proof}

\subsection{Finite $n$-Models.}\label{Fin}
It follows  from Example \ref{pMe} that, without loosing any generality, we can confine ourselves to frames\slash $n$-models defined over partial orders (not pre-orders). So, in what follows, we assume that all frames\slash $n$-model are (defined over) po-sets even though we (sometimes) keep the notation $\mathfrak{M}^n=(W,R,w_0,V^n)$.
We examine here specific properties of finite $n$-models such as Corollary \ref{fin}.

\begin{theorem}\label{lfi2}
If $\mathfrak{M}^n$ is a finite $n$-model, then one can define $\Delta(\mathfrak{M}^n)\in \mathsf{Fm}^n$  (called the {\it character} of $\mathfrak{M}^n$)\footnote{The explicit definition of the character can be found in many papers; for instance, see Ghilardi \cite{Ghi2}, p.869. The idea of characterizing finite structures by formulas is due to Jankov \cite{Jankov} but the character should not be missed with the characteristic formula of a frame. If we consider $n$-models of a given locally tabular logic {\sf L}, where there is only finitely many (up to $=_{\sf L}$) formulas in $n$-variables, one could define the character of any finite $n$-model as the conjunction of the formulas (out of the finitely many) which are true in the model.} such that \
$\mathfrak{N}^n\Vdash \Delta(\mathfrak{M}^n) \quad\Leftrightarrow\quad \mathsf{Th}(\mathfrak{M}^n)\subseteq\mathsf{Th}(\mathfrak{N}^n)$, \ for any $n$-model $\mathfrak{N}^n$.
\end{theorem}
The next theorem is due to Patterson \cite{Pat}:\begin{theorem}\label{pat}
If $\{\mathsf{Th}((\mathfrak{M}^n)_w)\}_{w\in W}$ is finite (which is the case when $\mathfrak{M}^n$ is finite), then $$\mathsf{Th}(\mathfrak{M}^n)\subseteq\mathsf{Th}(\mathfrak{N}^n)\quad\Leftrightarrow\quad\mathfrak N^n\thicksim(\mathfrak M^n)_w, \ \mbox{for some} \ w\in W,\quad \mbox{for any $n$-model $\mathfrak{N}^n$}.$$
\end{theorem}
\begin{proof} The implication $(\Leftarrow)$ is obvious by Lemma \ref{pMm}. Let us prove $(\Rightarrow)$. If not all of $\mathsf{Th}(\mathfrak{N}^n)$ is true at $(\mathfrak M^n)_w$, we pick   $A_w\in\mathsf{Th}(\mathfrak{N}^n)$ such that $A_w\not\in\mathsf{Th}((\mathfrak{M}^n)_w)$ or $A_w=\top$ otherwise. As $\{\mathsf{Th}((\mathfrak{M}^n)_w)\colon w\in W\}$ is finite, we  take $A=\bigwedge A_w$  and notice  $\mathfrak{M}^n\Vdash_wA$ means that  $\mathsf{Th}(\mathfrak{N}^n)\subseteq \mathsf{Th}((\mathfrak M^n)_w)$.

If a formula not in $\mathsf{Th}(\mathfrak{N}^n)$ is true at $(\mathfrak M^n)_w$, we pick  $B_w\not\in\mathsf{Th}(\mathfrak{N}^n)$ such that $B_w\in\mathsf{Th}((\mathfrak{M}^n)_w)$ (or $B_w=\bot$ if $\mathsf{Th}(\mathfrak{N}^n)\supseteq \mathsf{Th}((\mathfrak M^n)_w$)), for each $w\in W$. Take $B=\bigvee B_w$  and notice  $\mathfrak{M}^n\not\Vdash_wB$ yields  $\mathsf{Th}(\mathfrak{N}^n)\supseteq \mathsf{Th}((\mathfrak M^n)_w)$.

Clearly, $(A\Rightarrow B)\not\in\mathsf{Th}(\mathfrak{N}^n)$. Thus, $(A\Rightarrow B)\not\in\mathsf{Th}(\mathfrak{M}^n)$ and hence $\mathfrak{M}^n\Vdash_wA$ and $\mathfrak{M}^n\not\Vdash_w B$, for some $w\in W$, and this means that   $\mathsf{Th}(\mathfrak{N}^n)=\mathsf{Th}((\mathfrak M^n)_w)$.
\end{proof}
\begin{theorem}\label{GB}
If $\{\mathsf{Th}((\mathfrak{M}^n)_w)\}_{w\in W}$ is finite, then  the greatest bisimulation $B$ of  $\mathfrak{M}^n$ is:
$$wBv \quad\Leftrightarrow\quad (\mathfrak{M}^n)_w\thicksim(\mathfrak{M}^n)_v.$$
\end{theorem}
\begin{proof}  Let $wBv\land wRw'$. Then $\mathsf{Th}((\mathfrak{M}^n)_v)=\mathsf{Th}((\mathfrak{M}^n)_{w})\subseteq\mathsf{Th}((\mathfrak{M}^n)_{w'})$ and, by Theorem \ref{pat}, $w'Bv'\land vRv'$ for some $v'$. Thus, $B$ is a bisimulation as $B$ is symmetric.

Let $wB'v$ and $B'$ be a bisimulation  of  $\mathfrak{M}^n$. By Theorem \ref{pM7}, there is a p-morphism $p\colon\mathfrak{M}^n\to\mathfrak{M}^n\slash B'$ such that $p(w)=p(v)$. Hence, by Lemma \ref{pM0},
$(\mathfrak{M}^n)_w\thicksim(\mathfrak{M}^n)_v$ which means $wBv$. Thus, we have showed $B'\subseteq B$.
\end{proof}
\begin{corollary}\label{FMbis} If $\{\mathsf{Th}((\mathfrak{M}^n)_w)\}_{w\in W}$ is finite, then there is a p-morphism from $\mathfrak{M}^n$ onto the $n$-model:
$$\Bigl(\{\mathsf{Th}((\mathfrak{M}^n)_w)\}_{w\in W},\ \subseteq ,\ \mathsf{Th}(\mathfrak{M}^n),\ \{\{x_1,\dots,x_n\}\cap\mathsf{Th}((\mathfrak{M}^n)_w)\}_{w\in W}\Bigr).$$
\end{corollary}
\begin{proof} By the above Theorem and by Theorem \ref{pM7}.\end{proof}
\begin{corollary}\label{FM}
 $\mathfrak{M}^n$ is finitely reducible (which means there is a p-morphism $p\colon\mathfrak{M}^n\to\mathfrak{N}^n$ for some finite $n$-model $\mathfrak{N}^n$) if and only if $\{\mathsf{Th}((\mathfrak{M}^n)_w)\}_{w\in W}$ is finite.
\end{corollary}
\begin{corollary}\label{lf3i}
 Let $\mathfrak{M}^n$ and  $\mathfrak{N}^n$ be finite (or finitely reducible) $n$-models. Then $\mathfrak{M}^n\thicksim\mathfrak{N}^n$ if and only if
 $\mathfrak{M}^n$ and $\mathfrak{N}^n$ have a common p-morphic  image.\end{corollary}
 \begin{proof} Let $\mathfrak{M}^n=(W,R,w_0,V^n)$ and  $\mathfrak{N}^n=(U,S,u_0,V'^n)$. It suffices to notice   that $\mathfrak{M}^n\thicksim\mathfrak{N}^n$ yields, by Theorem  \ref{pat}, $\{\mathsf{Th}((\mathfrak{M}^n)_w)\}_{w\in W}=\{\mathsf{Th}((\mathfrak{N}^n)_u)\}_{u\in U}$.\end{proof}
\begin{corollary}\label{lf4i}
If $\mathfrak{M}^n=(W,R,w_0,V^n)$ and $\mathfrak{N}^n=(U,S,u_0,V'^n)$ are finite  and $\mathfrak{M}^n\thicksim\mathfrak{N}^n$, then\\
(i) for every $w\in W$ there is an element $u\in U$ such that $(\mathfrak{M}^n)_{w}\thicksim(\mathfrak{N}^n)_{u}$;
\\ (ii) for every $u\in U$ there is an element  $w\in W$ such that $(\mathfrak{M}^n)_{w}\thicksim(\mathfrak{N}^n)_{u}$.
\end{corollary}
\begin{proof} Let $p$ and $q$ be p-morphisms from $\mathfrak{M}^n$ and $\mathfrak{N}^n$, correspondingly, onto a common p-morphic image. By Lemma \ref{pM0},  $(\mathfrak{M}^n)_{w}\thicksim(\mathfrak{N}^n)_{u}$ if $p(w)=q(u)$.
\end{proof}
\subsection{$\sigma$-Models.}\label{sM}
This is the key notion and it was  defined by Ghilardi  \cite{Ghi2}.
Let  $\sigma:\{x_1,\dots,x_n\}\to \mathsf{Fm^k}$, for  $k,n\geq 0$. For any  $\mathfrak{M}^k=(W,R,w_0,V^k)$,  let $\sigma(\mathfrak{M}^k)=(W,R,w_0,V^n)$ where
$$x_i\in V^n(w)\quad\Leftrightarrow\quad \mathfrak{M}^k\Vdash_w\sigma(x_i), \quad \mbox{ for any $w\in W$ \ and \ $i=1,\dots,n$}.$$
\begin{lemma}\label{sigma0} For every $w\in W$ \ and every \ $A\in \mathsf{Fm^n}$, we have
$$\sigma(\mathfrak{M}^k)\Vdash_wA \quad\Leftrightarrow\quad \mathfrak{M}^k\Vdash_w\sigma(A) .$$
\end{lemma}
\begin{lemma}\label{sigmai}
(i) $\mathfrak{M}^k$ and $\sigma(\mathfrak{M}^k)$ are models over the same frame;\\
 (ii) $\sigma((\mathfrak{M}^k)_w)=(\sigma(\mathfrak{M}^k))_w$, \quad for every $w\in W$;\\
 (iii) if  $\mathsf{Th}(\mathfrak{M}^k)\subseteq\mathsf{Th}(\mathfrak{N}^k)$,\ then\ $\mathsf{Th}(\sigma(\mathfrak{M}^k))\subseteq\mathsf{Th}(\sigma(\mathfrak{N}^k))$.
	\end{lemma}
 \begin{proof} We get  (i) and (ii) by the definition of $\sigma(\mathfrak{M}^k)$. As concerns (iii):\\
 $\sigma(\mathfrak{M}^k)\Vdash A \ \Leftrightarrow  \ \mathfrak{M}^k\Vdash\sigma(A) \ \Rightarrow \ \mathfrak{N}^k\Vdash\sigma(A)\ \Leftrightarrow \ \sigma(\mathfrak{N}^k)\Vdash A$.
\end{proof}
\begin{lemma}\label{sigma2}
If $p\colon\mathfrak{M}^k\to \mathfrak{N}^k$ is a p-morphism of $k$-models, then $p\colon\sigma(\mathfrak{M}^k)\to \sigma(\mathfrak{N}^k)$ is also a p-morphism of $n$-models and hence $p(\sigma(\mathfrak{M}^k))=\sigma(p(\mathfrak{M}^k))$ (see Figure. \ref{ps}).
\end{lemma}
\begin{figure}[H]
\unitlength1cm
\begin{picture}(2.5,2.5)
\thicklines

\put(4.3,0){\mbox{$\sigma(\mathfrak{M}^n)$}}
\put(4.5,2){\mbox{$\mathfrak{M}^k$}}
\put(4.6,1.9){\vector(0,-1){1.5}}
\put(4.2,1){\mbox{$\sigma$}}
\put(8.5,0){\mbox{$\sigma(\mathfrak{N}^k)$}}
\put(8.6,1.9){\vector(0,-1){1.5}}
\put(8.5,2){\mbox{$\mathfrak{N}^k$}}
\put(8.8,1){\mbox{$\sigma$}}
\put(6.7,2.2){\mbox{$p$}}
\put(6.7,0.2){\mbox{$p$}}
\put(5.5,2){\vector(1,0){2.5}}
\put(5.5,0){\vector(1,0){2.5}}

\end{picture}
\caption{p-Morphic images of $\sigma$-models.}\label{ps}
\end{figure}
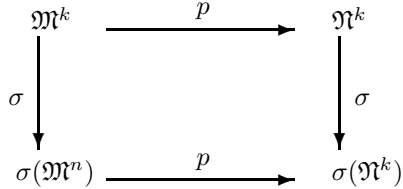
The above does not mean that $\sigma$-models are closed under p-morphic images. Two
(counter)examples below show that they may be not.
\begin{example}\label{Kost}
Let $\sigma(x_1)=x_2 \lor (x_2 \to (x_1\lor \neg x_1)).$ The  $1$-model over the two-element chain (in Figure \ref{ex1}) cannot be any $\sigma$-model as to falsify $\sigma(x_1)$ at the root one needs at least three elements in the chain.
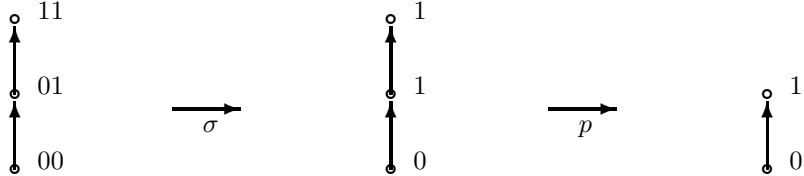
\begin{figure}[H]

\unitlength1cm
\begin{picture}(3,2)
\thicklines
\put(1,0){\vector(0,1){0.9}}
\put(1,1){\vector(0,1){0.9}}
\put(1,1){\circle{0.1}}
\put(1,2){\circle{0.1}}
\put(1,0){\circle{0.1}}
\put(1.3,0){\mbox{$00$}}
\put(1.3,2){\mbox{$11$}}
\put(1.3,1){\mbox{$01$}}

\put(3.1,0.8){\vector(1,0){0.9}}
\put(3.5,0.5){$\sigma$}

\put(6,0){\vector(0,1){0.9}}
\put(6,1){\vector(0,1){0.9}}
\put(6,1){\circle{0.1}}
\put(6,2){\circle{0.1}}
\put(6,0){\circle{0.1}}
\put(6.3,0){\mbox{$ 0$}}
\put(6.3,2){\mbox{$ 1$}}
\put(6.3,1){\mbox{$ 1$}}

\put(8.1,0.8){\vector(1,0){0.9}}
\put(8.5,0.5){$p$}

\put(11,0){\circle{0.1}}
\put(11.3,0){\mbox{$ 0$}}
\put(11,1){\circle{0.1}}
\put(11.3,1){\mbox{$1$}}
\put(11,0){\vector(0,1){0.9}}

\end{picture}
\caption{The First Counterexample.}\label{ex1}
\end{figure}

Let $\sigma(x)=\neg\neg x\lor \neg x$ (we write $x$ instead of $x_1$). Models and the p-morphism are defined in Figure \ref{ex2}. The  $1$-model over a two-element chain cannot be any $\sigma$-model as to falsify $\sigma(x)$ at the root one needs at least two end elements above the root.

\begin{figure}[H]
\unitlength1cm
\begin{picture}(5,2)
\thicklines
\put(0,1){\circle{0.1}}
\put(1,0){\circle{0.1}}
\put(2,1){\circle{0.1}}
\put(1.3,0){\mbox{$0$}}
\put(0.3,1){\mbox{$1$}}
\put(2.3,1){\mbox{$0$}}
\put(1,0){\vector(1,1){0.9}}
\put(1,0){\vector(-1,1){0.9}}

\put(3.1,0.8){\vector(1,0){0.9}}
\put(3.5,0.5){$\sigma$}

\put(6,0){\vector(-1,1){0.9}}
\put(6,0){\vector(1,1){0.9}}
\put(5,1){\circle{0.1}}
\put(6,0){\circle{0.1}}
\put(7,1){\circle{0.1}}
\put(6.3,0){\mbox{$0$}}
\put(5.2,1){\mbox{$1$}}
\put(7.3,1){\mbox{$1$}}

\put(8.6,0.8){\vector(1,0){0.9}}
\put(9,0.5){$p$}

\put(11,0){\circle{0.1}}
\put(11.3,0){\mbox{$ 0$}}
\put(11,0){\line(0,1){0.9}}
\put(11,1){\circle{0.1}}
\put(11.3,1){\mbox{$ 1$}}
\put(11,0){\vector(0,1){0.9}}

\end{picture}
\caption{The Second Counterexample.}\label{ex2}
\end{figure}
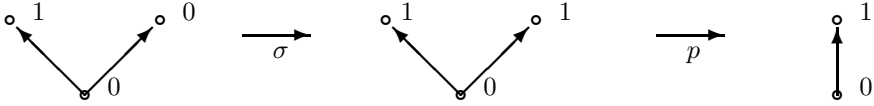
\end{example}
Nowhere  (but Theorem  \ref{lfi2}) we have used the fact that valuations of any $n$-model are restricted to the $n$-initial variables. It would  make no change in our argument if we replaced (everywhere) valuations $V^n$ with $V$, valuations of all variables. Thus,  all results (but Theorem  \ref{lfi2}) of this section remain valid for usual Kripke models.
\section{Locally Tabular Logics.}\label{LDF}
For any class  {\bf F} of frames, let {\it sm({\bf F})} be the least class (of  frames) containing {\bf F} and closed under  generated subframes and p-morphic images.
\begin{lemma}\label{lf8} \indent\indent\indent\indent \indent\indent\indent \indent \indent \qquad $\mathsf{L}(sm(\mathbf{F}))=\mathsf{L}(\mathbf{F}).$\end{lemma}\begin{proof} By Lemma \ref{pM0}  and Lemma \ref{pMm}\end{proof}
Extending any class of frames with generated subframes and p-morphic images does not change the logic but it enables us to characterize extensions of $\mathsf{L}(\mathbf{F})$:\footnote{The following theorem resembles (not without reasons)  characterizations, see \cite{rw,PW}, of extensions of logics given by  logical matrices.}
\begin{theorem}\label{lf7}  Let $\mathbf{F}$ be a class of finite frames and $\mathsf L=\mathsf{L}(\mathbf{F})$ be locally tabular.   If {\sf L'} is an intermediate logic  such that  $\mathsf{L}\subseteq \mathsf L'$, then $\mathsf{L'}=\mathsf{L}(\mathbf{G})$, for some $\mathbf{G}\subseteq sm(\mathbf{F})$.
\end{theorem}
\begin{proof} Let $\mathbf{G}=\{\mathfrak{F}\in sm(\mathbf{F})\colon \mathsf{L'}\subseteq \mathsf{L}(\mathfrak{F})\}$. Clearly, $\mathsf{L'}\subseteq\mathsf{L}(\mathbf{G})$. We need to show the reverse inclusion. So, assume $A\not\in\mathsf{L'}$ and show $A\not\in\mathsf{L}(\mathfrak{F})$ for some $\mathfrak{F}\in \mathbf{G}$.

Suppose that $A=A(x_1,\dots,x_k)$, for some $k\geq 0$, and let $A_0,\dots,A_j$ be all (non-equivalent in {\sf L}) formulas in $\mathsf{Fm^k}\cap\mathsf{L'}$. Let
$$B=\bigwedge_{i=0}^jA_i\ \rightarrow\ A.$$
If $B\in \mathsf{L}(\mathbf{F})$, then $B\in\mathsf{L'}$ and it would give $A\in \mathsf{L'}$, a contradiction. Thus, we have  $B\not\in \mathsf{L}(\mathbf{F})$. There is a $k$-model $\mathfrak{M}^k=(W,R,w_0,V^k)$ over a frame from {\bf F} such that $\mathfrak{M}^k\Vdash_wA_i$, for all $i\leq j$, and $\mathfrak{M}^k\not\Vdash_wA$, for some $w\in W$.  Let $p\colon\mathfrak{M}^k\to\mathfrak{N}^k$ be a p-morphism from $\mathfrak{M}^k$ onto a $p$-irreducible $k$-model $\mathfrak{N}^k$, see Theorem \ref{Irr}. We take the frame of $(\mathfrak{N}^k)_{p(w)}$ as our $\mathfrak{F}$. Let $\mathfrak{F}=(U,\leq,p(w))$. Since $(\mathfrak{N}^k)_{p(w)}$ is a $k$-model over $\mathfrak{F}$, we have $A\not\in\mathsf{L}(\mathfrak{F})$. There remains to show that
$\mathsf{L'}\subseteq \mathsf{L}(\mathfrak{F})$.

Suppose that $C\not\in\mathsf{L}(\mathfrak{F})$ for some $C\in \mathsf{L'}$. Let $C=C(x_1,\dots,x_n)$ and let $\mathfrak{N}^n$ be an $n$-model over $\mathfrak{F}$ such that $\mathfrak{N}^n\not\Vdash C$. We  define a substitution $\varepsilon\colon\{x_1,\dots,x_n\}\to\mathsf{Fm^k}$ taking
$\varepsilon(x_i)=\bigvee\{\Delta((\mathfrak{N}^k)_u)\colon \mathfrak{N}^n\Vdash_ux_i\}$, for any $i\leq n$. Then we have $\mathfrak{N}^k\Vdash_v \varepsilon(x_i)\Leftrightarrow$\\ $\exists_{u\in U}\bigl(\mathfrak{N}^k\Vdash_v\Delta((\mathfrak{N}^k)_u)\land \mathfrak{N}^n\Vdash_ux_i\bigr)\Leftrightarrow \exists_{u\in U}\bigl(\mathsf{Th}((\mathfrak{N}^k)_u)\subseteq\mathsf{Th}((\mathfrak{N}^k)_v) \land\mathfrak{N}^n\Vdash_ux_i\bigr) \Leftrightarrow\exists_{u\in U}(u\leq v\land\mathfrak{N}^n\Vdash_ux_i)\Leftrightarrow\mathfrak{N}^n\Vdash_vx_i$, for any $i\leq n$ and  $v\in U$. Note that the last but one equivalence needs Corollary \ref{FMbis}. This shows $\mathfrak{N}^k\Vdash_{v} \varepsilon(C)\Leftrightarrow\mathfrak{N}^n\Vdash_v C$, for any $v\in U$ and hence we get $\mathfrak{N}^k\not\Vdash_{p(w)} \varepsilon(C)$, that is $\mathfrak{M}^k\not\Vdash_w \varepsilon(C)$,
which cannot happen as $\varepsilon(C)$ is one of the $A_i$'s and must be true at $(\mathfrak{M}^k)_w$.
\end{proof}

\subsection{Substitutions in Locally Tabular Logics.}\label{sub}
Let $\mathbf{F}$ be a class of finite frames, $\mathsf L=\mathsf{L}(\mathbf{F})$ be locally tabular and $\mathbf{M}^n=\mathbf{M}^n(\mathbf F)$, for any $n\geq 0$. Assume, additionally, that  {\bf F} is closed under generated subframes and p-morphic images, that is {\it sm}({\bf F})={\bf F},  see Lemma \ref{lf8}.    For any $\sigma:\{x_1,\dots,x_n\}\to \mathsf{Fm^k}$,  define $H_\sigma\colon\mathbf{M}^k\to\mathbf{M}^n$  putting $H_\sigma(\mathfrak{M}^k)=\sigma(\mathfrak{M}^k)$, for each $\mathfrak{M}^k$.\footnote{Ghilardi wrote $\sigma(u)$ for any Kripke model $u$ and hence we have $\sigma(\mathfrak{M}^k)$. We should, perhaps, wrote $\sigma\colon\mathbf{M}^k\to\mathbf{M}^n$ but we think it could be misleading as we already have $\sigma:\{x_1,\dots,x_n\}\to \mathsf{Fm^k}$ and $\sigma:\mathsf{Fm}\to \mathsf{Fm}$.  Talking about the mapping $\sigma$, it would be unclear if we had in mind a mapping between formulas  or models. For this reason we decided to introduce $H_\sigma$, to replace $\sigma$, though it could be seen as an excessive reaction.}
\begin{lemma}\label{lfs}  Suppose that $\varepsilon,\sigma:\{x_1,\dots,x_n\}\to \mathsf{Fm^k}$. Then \ $\varepsilon=_{\sf L}\sigma$ \ iff \ $H_\sigma\thicksim H_\varepsilon$.\footnote{where $H_\sigma\thicksim H_\varepsilon$ obviously means $H_\sigma(\mathfrak{M}^k)\thicksim H_\varepsilon(\mathfrak{M}^k)$ for each $\mathfrak{M}^k$.}\end{lemma}
\begin{proof}   $(\Rightarrow)$ is obvious. $(\Leftarrow)$. Let  $H_\sigma(\mathfrak{M}^k)\thicksim H_\varepsilon(\mathfrak{M}^k)$,  for any    $\mathfrak{M}^k\in\mathbf{M}^k$. Then\\
$\mathfrak{M}^k\vdash \sigma(A)\Leftrightarrow\sigma(\mathfrak{M}^k)\vdash A\Leftrightarrow\varepsilon(\mathfrak{M}^k)\vdash A\Leftrightarrow\mathfrak{M}^k\vdash \varepsilon(A).$ Thus, $\vdash_{\sf L}\sigma(A)\leftrightarrow\varepsilon(A)$, for any $A\in\mathsf{Fm^n}$, which shows $\varepsilon=_{\sf L}\sigma$.\end{proof}
The assumptions that the frames  {\bf F} are finite  and {\sf L}({\bf F}) is locally tabular do not play any role in the above Lemma but they   are essential in the subsequent theorem, to prove that  the conditions (i)-(iii) of Lemma \ref{sigmai} characterize  substitutions:
\begin{theorem}\label{nsigmai}
Let $H\colon\mathbf{M}^k\to\mathbf{M}^n$. Then $H\thicksim H_\sigma$, for some $\sigma:\{x_1,\dots,x_n\}\to \mathsf{Fm^k}$ if and only if $H$ fulfills the following conditions:\\
(i)  the $n$-model $H(\mathfrak{M}^k)$ has the same frame  as the $k$-model $\mathfrak{M}^k$, for any   $\mathfrak{M}^k\in\mathbf{M}^k$;\\
(ii) $H((\mathfrak{M}^k)_w)\thicksim(H(\mathfrak{M}^k))_w$, \ for any $\mathfrak{M}^k=(W,R,w_0,V^k)\in\mathbf{M}^k$ and  $w\in W$;\\
 (iii) if  $\mathfrak{N}^k\thicksim\mathfrak{M}^k$,\ then\ $H(\mathfrak{N}^k)\thicksim H(\mathfrak{M}^k)$, for any $\mathfrak{M}^k,\mathfrak{N}^k\in\mathbf{M}^k.$
 \end{theorem}
\begin{proof}  $(\Leftarrow)$ follows from Lemma \ref{sigmai}. The conditions (i)-(iii) of Lemma \ref{sigmai} seem to be stronger than the above ones, but they are not (see Theorem \ref{pat}).

To prove $(\Rightarrow)$ we assume  $H\colon\mathbf{M}^k\to\mathbf{M}^n$ fulfills the above (i)--(iii). Let
$$\sigma(x_i)=\bigvee\{\Delta(\mathfrak{N}^k)\colon  \ \mathfrak{N}^k\in\mathbf{M}^k \ \land \ H(\mathfrak{N}^k)\Vdash x_i\},\qquad \mbox{ for $i=1,\dots,n$}.$$ By  Theorem \ref{lf6}, we can claim that we have defined $\sigma\colon \{x_1,\dots,x_n\}\to \mathsf{Fm}^k$. For  any  $k$-model $\mathfrak{M}^k=(W,R,w_0,V^k)\in \mathbf{M}^k$, we  have
$$\sigma(\mathfrak{M}^k)\Vdash_{w}x_i \quad \Leftrightarrow \quad \sigma((\mathfrak{M}^k)_w)\Vdash x_i\quad \Leftrightarrow \quad (\mathfrak{M}^k)_w\Vdash\sigma(x_i)\quad \Leftrightarrow  $$ $$\exists_{\mathfrak{N}^k}\bigl((\mathfrak{M}^k)_w\Vdash\Delta(\mathfrak{N}^k) \land  H(\mathfrak{N}^k)\Vdash x_i\bigr) \  \Leftrightarrow \ H((\mathfrak{M}^k)_w)\Vdash x_i \ \Leftrightarrow\ H(\mathfrak{M}^k)\Vdash_{w}x_i,$$ for any $i=1,\dots,n$ and any $w\in W$.
 Hence  $\sigma(\mathfrak{M}^k)\thicksim H(\mathfrak{M}^k)$.
 \end{proof}
The above theorem is useful to define substitutions. However, the condition (iii) is hard to check if  there is too many p-morphisms between models. So,  we would prefer a variant of \ref{nsigmai}, given below,  concerning p-irreducible models. The closure of {\bf F} under p-morphic images is not necessary  for the above theorem (it suffices the closure under generated subframes) but it is necessary for the subsequent theorem.

Let $\mathbf{M}^n_{ir}$, for any $n\geq 0$, be the class of p-irreducible $n$-models  over the frames $\mathbf{F}$. According to Theorem \ref{Irr}, for any $A\in \mathsf{Fm}^n$
$$A\in\mathsf{L} \quad \Leftrightarrow \quad (\mathfrak{M}^n\Vdash A, \mbox{ for every } \mathfrak{M}^n\in \mathbf{M}^n_{ir} ). $$
\begin{theorem}\label{nsi}
If $H\colon\mathbf{M}^k_{ir}\to\mathbf{M}^n$ fulfills\\
(i) the $n$-model $H(\mathfrak{M}^k)$ has the same frame  as the $k$-model $\mathfrak{M}^k$, for any   $\mathfrak{M}^k\in\mathbf{M}^k_{ir}$;\\
(ii) $H((\mathfrak{M}^k)_w)\thicksim(H(\mathfrak{M}^k))_w$, \ for any $\mathfrak{M}^k=(W,R,w_0,V^k)\in\mathbf{M}^k_{ir}$ and any  $w\in W$;\\
 (iii) if  $\mathfrak{N}^k\equiv\mathfrak{M}^k$,\ then\ $H(\mathfrak{N}^k)\thicksim H(\mathfrak{M}^k)$, for any $\mathfrak{M}^k,\mathfrak{N}^k\in\mathbf{M}^k_{ir};$\\
then there is exactly one (up to $=_{\sf L}$) substitution  $\sigma:\{x_1,\dots,x_n\}\to \mathsf{Fm^k}$ such that $H(\mathfrak{M}^k)\thicksim H_\sigma(\mathfrak{M}^k)$, for each $\mathfrak{M}^k\in\mathbf{M}^k_{ir}.$
\end{theorem}
\begin{proof}  We proceed  in the same way as above. In particular, we define $$\sigma(x_i)=\bigvee\{\Delta(\mathfrak{N}^k)\colon  \ \mathfrak{N}^k\in\mathbf{M}^k_{ir} \ \land \ H(\mathfrak{N}^k)\Vdash x_i\},\qquad \mbox{ for $i=1,\dots,n$}$$ and prove $H(\mathfrak{M}^k)\thicksim H_\sigma(\mathfrak{M}^k)$, for any $\mathfrak{M}^k\in\mathbf{M}^k_{ir}$. The crucial step in our argument $$\exists_{\mathfrak{N}^k}\bigl((\mathfrak{M}^k)_w\Vdash\Delta(\mathfrak{N}^k) \land  H(\mathfrak{N}^k)\Vdash x_i\bigr) \  \Rightarrow \ H((\mathfrak{M}^k)_w)\Vdash x_i$$ follows from the fact that, if $(\mathfrak{M}^k)_w\thicksim(\mathfrak{N}^k)_u$, for some $u$, then $(\mathfrak{M}^k)_w$ and $(\mathfrak{N}^k)_u$ are p-irreducible by Theorem \ref{pM6} and hence $(\mathfrak{M}^k)_w\equiv(\mathfrak{N}^k)_u$ by Corollary \ref{lf3i}. Thus, by (iii), we have $H((\mathfrak{M}^k)_w)\Vdash x_i$ if $H(\mathfrak{N}^k)\Vdash x_i$.

The uniqueness of $\sigma$ follows from Lemma \ref{lfs} (and Theorem \ref{Irr}). \end{proof}
Suppose we need $H\colon\mathbf{M}^k_{ir}\to\mathbf{M}^n$ fulfilling the above (i)--(iii). Let $\mathfrak{M}^k=(W,\leq,w_0,V^k)$ be a p-irreducible $k$-model. We should have $H(\mathfrak{M}^k)=(W,\leq,w_0,V^n)$, for some $V^n$, which means only the valuations $V^n$ are to be defined. By Theorem \ref{pM6}, $(\mathfrak{M}^k)_w$ is p-irreducible, for any $w\in W$, and hence we could define $H((\mathfrak{M}^k)_w)$ inductively with respect to $d_{\mathfrak F}(w)$, where $\mathfrak F=(W,\leq,w_0)$. First, we define $H((\mathfrak{M}^k)_w)$ for $n$-models over 1-element reflexive frames, that is we define $V^n(w)$ for end elements $w\in W$; any subset of $\{x_1,\dots,x_k\}$ can be taken as $V^n(w)$.  Then, assuming $H((\mathfrak{M}^k)_u)$ has been defined for any $u>w$, and hence we have $V^n(u)$ for any $u>w$, we define $V^n(w)$. The only restriction is monotonicity, that is we should have $V^n(w)\subseteq V^k(u)$ for any $u>w$. The condition (ii) would be satisfied as we would even have $H((\mathfrak{M}^k)_w)=(H(\mathfrak{M}^k))_w$.
The condition (iii) should not produce any harm if we
define $H$ on the equivalence classes $[\mathfrak{M}^k]_\equiv$. We should get $H\colon\mathbf{M}^k_{ir}\slash\!\!\equiv \ \to \ \mathbf{M}^n\slash\!\!\thicksim$. Note that it would be much easier to satisfy this condition than to satisfy the corresponding (iii) of Theorem \ref{nsigmai} where we should get $H\colon\mathbf{M}^k\slash\!\!\thicksim \ \to \ \mathbf{M}^n\slash\!\!\thicksim$.\\

Since {\sf L} is locally tabular,
any formula $A\in\mathsf{Fm}^n$ is a disjunction of the characters of all $A$-{\it models} (that is $n$-models $\mathfrak{M}^n\in\mathbf M^n$ such that $A$ is true at $\mathfrak{M}^n$):
$$A=_{\sf L}\bigvee\{\Delta(\mathfrak{M}^n)\colon \Delta(\mathfrak{M}^n)\to A\in \mathsf{L}\}=_{\sf L}\bigvee\{\Delta(\mathfrak{M}^n)\colon \mathfrak{M}^n\Vdash A\}.\leqno (\star)$$ It also shows that each {\sf L}-consistent (that is $\not=_{\sf L}\bot$) formula   is {\sf L}-unifiable.

\begin{lemma}\label{n1i} A substitution $\sigma\colon \{x_1,\dots,x_n\}\to \mathsf{Fm}^k$ is an {\sf L}-unifier for $A\in\mathsf{Fm}^n$ iff
$\sigma(\mathfrak{M}^k)\Vdash A$ for each $k$-model $\mathfrak{M}^k$. In other words, $\sigma$ is a unifier for $A$ iff all $\sigma$-models are $A$-models. \end{lemma}
\begin{proof}  $\sigma(\mathfrak{M}^k)\Vdash A$ iff $\mathfrak{M}^k\Vdash\sigma(A)$,
for each  $\mathfrak{M}^k$. Thus, $\sigma(\mathfrak{M}^k)\Vdash A$ for each  $\mathfrak{M}^k$ iff $\mathfrak{M}^k\Vdash\sigma(A)$ for each  $\mathfrak{M}^k$. But $\sigma$ is an {\sf L}-unifier for $A$ iff $\mathfrak{M}^k\Vdash\sigma(A)$ for each  $\mathfrak{M}^k$.\end{proof}
Accordingly, $\sigma$ is a unifier for $A$ iff the disjunction $\bigvee\{\Delta(\mathfrak{M}^n)\colon \mathfrak{M}^n\Vdash A\}$ in $(\star)$ contains the characters of all $\sigma$-models. To put it in other words:
\begin{corollary}\label{in2} A substitution $\sigma\colon \{x_1,\dots,x_n\}\to \mathsf{Fm}^k$ is an {\sf L}-unifier for a formula $A\in\mathsf{Fm}^n$ iff $A_\sigma\to A\in \mathsf{L}$, where
$A_\sigma=\bigvee\{\Delta(\sigma(\mathfrak{M}^k))\colon \mathfrak{M}^k\in \mathbf{M}^k\}.$
\end{corollary}
Thus, each substitution $\sigma$ is an {\sf L}-unifier of some formulas $A\in\mathsf{Fm}^n$ and among these formulas we can find the strongest (or  the smallest in the Lindenbum algebra) formula $A_\sigma$. Using Lemma \ref{n1i} and Lemma \ref{pat} we can now characterize all unifiers of $A_\sigma$.
\begin{corollary}\label{n3i} Let $\sigma\colon \{x_1,\dots,x_n\}\to \mathsf{Fm}^k$. A substitution $\tau\colon \{x_1,\dots,x_n\}\to \mathsf{Fm}^m$ is an {\sf L}-unifier for $A_\sigma$ iff all $\tau$-models are (equivalent to)  $\sigma$-models, that is iff\\
(iv) for every $\mathfrak{M}^m\in\mathbf{M}^m$ there is an $\mathfrak{M}^k\in\mathbf{M}^k$ such that $\tau(\mathfrak{M}^m) \thicksim \sigma(\mathfrak{M}^k)$.
\end{corollary}
The above condition (iv) can be also written down as the inclusion
 $$\tau(\mathbf{M}^m)\slash\!\!\thicksim \quad \subseteq \quad \sigma(\mathbf{M}^k)\slash\!\!\thicksim.$$
There remains to characterize the relation $\preccurlyeq$ (of being a more general substitution) in terms of $\sigma$-models.
\begin{lemma}\label{n4i} A substitution $\tau\colon \{x_1,\dots,x_n\}\to \mathsf{Fm}^m$ is more general than $\sigma\colon \{x_1,\dots,x_n\}\to \mathsf{Fm}^k$ (that is $\tau\preccurlyeq\sigma$) iff there is  $F\colon\mathbf{M}^k\to\mathbf{M}^m$, fulfilling (i)--(iii) of Theorem \ref{nsigmai}, such that\\
(v) \ $\tau(F(\mathfrak{M}^k))\thicksim\sigma(\mathfrak{M}^k)$, for every  $\mathfrak{M}^k\in \mathbf{M}^k$.\\
Consequently, if $\tau\preccurlyeq\sigma$, then all $\sigma$-models are (equivalent to) $\tau$-models and hence
$$\tau(\mathbf{M}^m)\slash\!\!\thicksim \quad \supseteq \quad \sigma(\mathbf{M}^k)\slash\!\!\thicksim.$$
\end{lemma}
\begin{proof} Suppose that $\nu\circ\tau=_{\sf L}\sigma$ for some  $\nu\colon \{x_1,\dots,x_m\}\to \mathsf{Fm}^k$ and  $F=H_\nu$. Then,
 $\tau(F(\mathfrak{M}^k))=\tau(\nu(\mathfrak{M}^k))$ and, for each $B\in \mathsf{Fm}^n$,\  $\tau(\nu(\mathfrak{M}^k))\Vdash  B \ \Leftrightarrow \ \nu(\mathfrak{M}^k)\Vdash  \tau(B) \ \Leftrightarrow \ \mathfrak{M}^k\Vdash \nu(\tau(B)) \ \Leftrightarrow \ \mathfrak{M}^k\Vdash \sigma(B) \ \Leftrightarrow \ \sigma(\mathfrak{M}^k)\Vdash B.$ Hence we get (v).

The reverse implication is shown in the same way. We get $\nu\colon \{x_1,\dots,x_m\}\to \mathsf{Fm}^k$ such that $F(\mathfrak{M}^k)=\nu(\mathfrak{M}^k)$, for each $\mathfrak{M}^k$, by Lemma \ref{n1i}. Then, using (v), we show $\mathfrak{M}^k\Vdash \nu(\tau(B))\ \Leftrightarrow \ \mathfrak{M}^k\Vdash \sigma(B)$ for each $B\in \mathsf{Fm}^n$ and this yields $\nu\circ\tau=_{\sf L}\sigma$.\end{proof}

\subsection{Locally Tabular Logics with Finitary Unification.}\label{FU}
Now, we give a short proof of the main result of \cite{dkw} characterizing locally tabular intermediate logics with finitary (or unitary) unification.

 \begin{theorem}\label{main} A locally tabular intermediate logic  {\sf L}  has finitary (or unitary)  unification if and only if for every $n\geq 0$ there exists a number $m\geq 0$ such that for every $k\geq 0$ and every  $\sigma\colon \{x_1,\dots,x_n\}\to \mathsf{Fm}^k$  there are   $G:\mathbf{M}^m\to\mathbf{M}^n$ and $F:\mathbf{M}^k\to\mathbf{M}^m$ such that \\
(i) $G$ preserves the frame  of any $m$-model $\mathfrak{M}^m=(W,R,w_0,V^m)\in\mathbf{M}^m$;\\
\indent	 $F$ preserves the frame  of any $k$-model $\mathfrak{M}^k=(W',R',w'_0,V^k)\in\mathbf{M}^k$;\\
	 (ii) $G((\mathfrak{M}^m)_w)\thicksim(G(\mathfrak{M}^m))_w$,   for any  $\mathfrak{M}^m=(W,R,w_0,V^m)\in\mathbf{M}^m$ and  $w\in W$;\\
\indent $F((\mathfrak{M}^k)_{w'})\thicksim(F(\mathfrak{M}^k))_{w'}$, for any $\mathfrak{M}^k=(W',R',w_0',V^k)\in\mathbf{M}^k$ and  $w'\in W'$;\\
(iii) $\mathfrak{N}^m\thicksim \mathfrak{M}^m \ \Rightarrow \ G(\mathfrak{N}^m)\thicksim G(\mathfrak{M}^m), \ \mbox{ for any } \ \mathfrak{M}^m,\mathfrak{N}^m\in \mathbf{M}^m$;\\
\indent \indent	 $\mathfrak{N}^k\thicksim \mathfrak{M}^k \ \Rightarrow \ F(\mathfrak{N}^k)\thicksim F(\mathfrak{M}^k), \ \mbox{ for any } \ \mathfrak{M}^k,\mathfrak{N}^k\in \mathbf{M}^k$;\\
(iv) for every $\mathfrak{M}^m\in\mathbf{M}^m$ there is  $\mathfrak{M}^k\in\mathbf{M}^k$ such that $G(\mathfrak{M}^m) \thicksim \sigma(\mathfrak{M}^k)$;\\
(v) $G(F(\mathfrak{M}^k))\thicksim\sigma(\mathfrak{M}^k)$, for every  $\mathfrak{M}^k\in \mathbf{M}^k$; see Figure \ref{v} below.
\end{theorem}
\begin{figure}[H]
\unitlength1cm
\begin{picture}(5,2)
\thicklines

\put(8,2){\vector(0,-1){1.9}}

\put(8,2){\vector(-1,-1){1.9}}
\put(8,0){\vector(-1,0){1.9}}
\put(8,0){\circle{0.1}}
\put(6,0){\circle{0.1}}
\put(8,2){\circle{0.1}}
\put(8.3,2){\mbox{$\mathbf{M}^k\!\slash\!\!\thicksim$}}
\put(5,0){\mbox{$\mathbf{M}^n\!\slash\!\!\thicksim$}}
\put(8.3,0){\mbox{$\mathbf{M}^m\!\slash\!\!\thicksim$}}
\put(8.1,1){\mbox{$F$}}
\put(7.1,0.1){\mbox{$G$}}
\put(6.7,1.2){\mbox{$H_\sigma$}}

\end{picture}
\caption{Condition (v)}\label{v}
\end{figure}
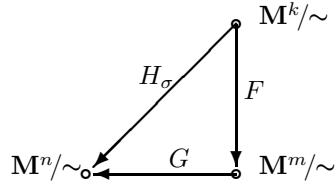
\begin{proof}
Since {\sf L} is locally tabular, $\mathsf{Fm}^n$ is finite (up to $=_{\sf L}$), for each $n$, and each $\sigma\colon \{x_1,\dots,x_n\}\to \mathsf{Fm}^k$ is a unifier of $A_\sigma\in\mathsf{Fm}^n$, see Corollary \ref{in2}. Thus, if {\sf L} has finitary (or unitary) unification, one can find, for every $n\geq 0$, a number $m$ such that each {\sf L}-unifiable formula $A\in\mathsf{Fm}^n$ (and $A_\sigma$ in particular) has a complete set of unifiers among $\tau\colon \{x_1,\dots,x_n\}\to \mathsf{Fm}^m$. Suppose $\tau\colon \{x_1,\dots,x_n\}\to \mathsf{Fm}^m$ is a unifier for $A_\sigma$ and $\tau\preccurlyeq\sigma$.
Take $G=H_\tau$ and let $F$ be the mapping  determined by Lemma \ref{n4i}. Then the conditions (i)--(v) are obviously fulfilled.

On the other hand, using the conditions (i)--(v) one shows that the set of  {\sf L}-unifiers $\tau\colon \{x_1,\dots,x_n\}\to\mathsf{Fm}^m$, for any unifiable $A\in\mathsf{Fm}^n$,   is complete. Since  $\mathsf L$ is locally tabular,  there is only finitely many (up to $=_{\sf L}$) such unifiers $\tau$ and hence unification is finitary (unless unitary).
\end{proof}
The above theorem do not need the assumption that {\bf F} is closed under p-morphic images; it suffices only the closure under generated subframes. But {\bf F}={\it sm}({\bf F}) is needed for the following version of \ref{main} where we refer to Theorem \ref{nsi} (instead of \ref{nsigmai}). For this version of \ref{main} we must, however, modify, the condition (v) as we cannot assume that $F(\mathfrak{M}^k)$ is p-irreducible even if $\mathfrak{M}^k$ is p-irreducible.
\begin{theorem}\label{main2} A locally tabular intermediate logic  {\sf L}  has finitary (or unitary)  unification if and only if for every $n\geq 0$ there exists a number $m\geq 0$ such that for every $k\geq 0$ and every  $\sigma\colon \{x_1,\dots,x_n\}\to \mathsf{Fm}^k$  there are   $G:\mathbf{M}^m_{ir}\to\mathbf{M}^n$ and $F:\mathbf{M}^k_{ir}\to\mathbf{M}^m$ such that \\
(i) $G$ preserves the frame  of any $m$-model $\mathfrak{M}^m=(W,R,w_0,V^m)\in\mathbf{M}^m_{ir}$;\\
\indent	 $F$ preserves the frame  of any $k$-model $\mathfrak{M}^k=(W',R',w'_0,V^k)\in\mathbf{M}^k_{ir}$;\\
	 (ii) $G((\mathfrak{M}^m)_w)\thicksim(G(\mathfrak{M}^m))_w$,   for any  $\mathfrak{M}^m=(W,R,w_0,V^m)\in\mathbf{M}^m_{ir}$ and  $w\in W$;\\
\indent $F((\mathfrak{M}^k)_{w'})\thicksim(F(\mathfrak{M}^k))_{w'}$, for any $\mathfrak{M}^k=(W',R',w_0',V^k)\in\mathbf{M}^k_{ir}$ and  $w'\in W'$;\\
(iii) $\mathfrak{N}^m\equiv \mathfrak{M}^m \ \Rightarrow \ G(\mathfrak{N}^m)\thicksim G(\mathfrak{M}^m), \ \mbox{ for any } \ \mathfrak{M}^m,\mathfrak{N}^m\in \mathbf{M}^m_{ir}$;\\
\indent \indent	 $\mathfrak{N}^k\equiv \mathfrak{M}^k \ \Rightarrow \ F(\mathfrak{N}^k)\thicksim F(\mathfrak{M}^k), \ \mbox{ for any } \ \mathfrak{M}^k,\mathfrak{N}^k\in \mathbf{M}^k_{ir}$;\\
(iv) for every $\mathfrak{M}^m\in\mathbf{M}^m_{ir}$ there is  $\mathfrak{M}^k\in\mathbf{M}^k_{ir}$ such that $G(\mathfrak{M}^m) \thicksim \sigma(\mathfrak{M}^k)$;\\
(v) if $\mathfrak{N}^m\thicksim F(\mathfrak{M}^k)$, then $G(\mathfrak{N}^m)\thicksim\sigma(\mathfrak{M}^k)$, for every  $\mathfrak{M}^k\in \mathbf{M}^k_{ir}$ and $\mathfrak{N}^m\in \mathbf{M}^m_{ir}$.
\end{theorem}

In  \cite{dkw}  the above theorems is applied to determine the unification types of locally tabular logics. We use them, as well, in the present paper.

 Let us recall that $\mathbf H_{un}$ (see Figure \ref{hpa}) consists of all frames $\mathfrak L_d+\mathfrak R_s$ , where $s,d\geq 0$.\footnote{There is no $\mathfrak L_0$ (or it would be the empty frame), so we should have here: all logics determined by  $\mathfrak R_s$ and $\mathfrak L_d+\mathfrak R_s$, where $s\geq 0$ and $d\geq 1$, if we  agree that $\mathfrak R_0=\mathfrak F_0=\mathfrak L_1$.} The logic {\sf L}($\mathbf H_{un}$) is locally tabular as it extends {\sf PWL}, see \cite{Esakia}.
\begin{theorem}\label{wpl} For any  $\mathbf F\subseteq \mathbf H_{un}$, the logic   {\sf L}($\mathbf F$)   has  unitary unification.\end{theorem}
\begin{proof}Let $n\geq 1$, and  $m=n\cdot 2^n+1$ and $\sigma\colon \{x_1,\dots,x_n\}\to \mathsf{Fm}^k$, for some $k$.  We need   $G:\mathbf{M}^m_{ir}\to\mathbf{M}^n$ and $F:\mathbf{M}^k_{ir}\to\mathbf{M}^m$ fulfilling (i)-(v) of Theorem \ref{main2}.  Let  $code(\mathfrak{f}^k_0)=\mathfrak{f}^n_0\mathfrak{g}^n_2\mathfrak{g}^n_3\dots,\mathfrak{g}^n_{2^n}1\quad \mbox{(the concatenation of  $\mathfrak{f}^k_0,\mathfrak{g}^n_i$'s and the suffix $1$)}$
includes  all  $\mathfrak{g}^n_i$'s (in any order where repetitions are allowed) such that\\
\begin{figure}[H]
\begin{center} \unitlength1cm
	\begin{minipage}[c][10mm][b]{10mm}
		
		\begin{picture}(3,2)
		\thicklines
		\linethickness{0.3mm}
		\put(0.2,0){\circle{0.1}}
		\put(0.5,0){\mbox{$\mathfrak{g}^n_i$}}
		\put(0.2,1){\circle{0.1}}
		\put(0.5,1){\mbox{$\mathfrak{f}^n_0$}}
		\put(0.2,0){\vector(0,1){0.9}}
		\end{picture}
		
	\end{minipage}
$\quad = \quad \sigma \Bigl($
	\begin{minipage}[c][10mm][b]{10mm}
		
		\begin{picture}(3,2)
		\thicklines
		\linethickness{0.3mm}
		\put(0.2,0){\circle{0.1}}
		\put(0.5,0){\mbox{$\mathfrak{g}^k_i$}}
		\put(0.2,1){\circle{0.1}}
		\put(0.5,1){\mbox{$\mathfrak{f}^k_0$}}
		\put(0.2,0){\vector(0,1){0.9}}
		\end{picture}
		
	\end{minipage}
$\Bigr), \quad \mbox{ for some } \mathfrak{g}_i^k .$ Then we define $F(\mathfrak{M}^{k})$:
\end{center}

\end{figure}
\begin{figure}[H]
\unitlength1cm
\begin{picture}(5,4)
\thicklines

\put(0,1){$\mathfrak{M}^k$:}
\put(2,3){\circle{0.1}}
\put(1,3){\circle{0.1}}
\put(0.5,3){\circle{0.1}}
\put(0.75,3){\circle{0.1}}
\put(1.75,3){\circle{0.1}}
\put(1.25,3){\circle{0.1}}
\put(0.25,3){\circle{0.1}}
\put(0,3){\circle{0.1}}
\put(1,2){\vector(1,1){0.9}}
\put(1,2){\vector(-1,1){0.9}}
\put(1,2){\vector(0,1){0.9}}
\put(1,2){\circle{0.1}}
\put(1,1.75){\circle{0.1}}
\put(1,1.5){\circle{0.1}}
\put(1,1.25){\circle{0.1}}
\put(1,1){\circle{0.1}}
\put(1,0){\circle{0.1}}
\put(1,0){\vector(0,1){0.9}}
\put(1,1){\circle{0.1}}
\put(1,0){\circle{0.1}}
\put(1,4){\circle{0.1}}
\put(2,3){\vector(-1,1){0.9}}
\put(1,3){\vector(0,1){0.9}}
\put(0,3){\vector(1,1){0.9}}
\put(1,2){\vector(-1,2){0.4}}
\put(1,2){\vector(1,2){0.4}}
\put(1.5,3){\circle{0.1}}
\put(0.5,3){\vector(1,2){0.4}}
\put(1.5,3){\vector(-1,2){0.4}}
\put(1.2,4){$\mathfrak{f}^k_0$}
\put(2.2,3){$\mathfrak{f}^k_{s}$}
\put(1.4,2){$\mathfrak{f}^k_{s+1}$}
\put(1.2,0){$\mathfrak{f}^k_{s+d+1}$}
\put(0,3.5){$\mathfrak{f}^k_{1}$}

\put(3.5,1){$\sigma(\mathfrak{M}^k)$:}
\put(6,3){\circle{0.1}}
\put(5,3){\circle{0.1}}
\put(4.5,3){\circle{0.1}}
\put(4.75,3){\circle{0.1}}
\put(5.25,3){\circle{0.1}}
\put(4.25,3){\circle{0.1}}
\put(4,3){\circle{0.1}}
\put(5,2){\vector(1,1){0.9}}
\put(5,2){\vector(1,1){0.9}}
\put(5,2){\vector(0,1){0.9}}
\put(5,2){\circle{0.1}}
\put(5,1.75){\circle{0.1}}
\put(5,1.5){\circle{0.1}}
\put(5,1.25){\circle{0.1}}
\put(5,1.75){\circle{0.1}}
\put(5,1){\circle{0.1}}
\put(5,0){\circle{0.1}}
\put(5,0){\vector(0,1){0.9}}
\put(5,1){\circle{0.1}}
\put(5,0){\circle{0.1}}
\put(5,4){\circle{0.1}}
\put(6,3){\vector(-1,1){0.9}}
\put(5,3){\vector(0,1){0.9}}
\put(4,3){\vector(1,1){0.9}}
\put(5,2){\vector(1,2){0.4}}
\put(5,2){\vector(-1,2){0.4}}
\put(5,2){\vector(-1,1){0.9}}
\put(5.5,3){\circle{0.1}}
\put(4.5,3){\vector(1,2){0.4}}
\put(5.5,3){\vector(-1,2){0.4}}
\put(5.2,4){$\mathfrak{f}^n_0$}
\put(6.2,3){$\mathfrak{f}^n_{s}$}
\put(5.4,2){$\mathfrak{f}^n_{s+1}$}
\put(5.2,0){$\mathfrak{f}^n_{s+d+1}$}
\put(3.6,3){$\mathfrak{f}^n_{1}$}

\put(7.5,1){$F(\mathfrak{M}^{k})$:}
\put(10,3){\circle{0.1}}
\put(9,3){\circle{0.1}}
\put(8.5,3){\circle{0.1}}
\put(8.75,3){\circle{0.1}}
\put(9.25,3){\circle{0.1}}
\put(8.25,3){\circle{0.1}}
\put(8,3){\circle{0.1}}
\put(9,2){\vector(1,1){0.9}}
\put(9,2){\vector(-1,1){0.9}}
\put(9,2){\vector(0,1){0.9}}
\put(9,2){\circle{0.1}}
\put(9,1.75){\circle{0.1}}
\put(9,1.5){\circle{0.1}}
\put(9,1.25){\circle{0.1}}
\put(9,1){\circle{0.1}}
\put(9,0){\circle{0.1}}
\put(9,0){\vector(0,1){0.9}}
\put(9,1){\circle{0.1}}
\put(9,0){\circle{0.1}}
\put(9,4){\circle{0.1}}
\put(10,3){\vector(-1,1){0.9}}
\put(9,3){\vector(0,1){0.9}}
\put(8,3){\vector(1,1){0.9}}
\put(9,2){\vector(-1,2){0.4}}
\put(9,2){\vector(1,2){0.4}}
\put(9.5,3){\circle{0.1}}
\put(8.5,3){\vector(1,2){0.4}}
\put(9.5,3){\vector(-1,2){0.4}}
\put(9.2,4){$\mathfrak{f}^n_0\mathfrak{g}^n_2\dots\mathfrak{g}^n_{2^n}1=code(\mathfrak f^k_0)$}
\put(10.2,3){$\mathfrak{f}^n_{s}\mathfrak{g}^n_2\dots\mathfrak{g}^n_{2^n}0$}
\put(9.4,2){$\mathfrak{f}^n_{s+1}0\cdots0$}
\put(9.2,0){$\mathfrak{f}^n_{s+d+1}0\cdots0$}
\put(6.5,3.5){$\mathfrak{f}^n_{1}\mathfrak{g}^n_2\dots\mathfrak{g}^n_{2^n}0$}
\end{picture}
\caption{}\label{hpa2}
\end{figure}
\noindent

\noindent One checks the conditions (i)-(iii) for the  mapping $F$.  Note that $F(\mathfrak{M}^{k})$ could be not  p-irreducible, for some p-irreducible $\mathfrak{M}^{k}$, but one could only collapse elements of the depth $2$ or elements below $s+1$ that is elements in the `leg' of the model.\\
A.  We put $G(\mathfrak{M}^{m})= \mathfrak{M}^{m}\!\!\upharpoonright n$ if $\mathfrak{M}^{m}\thicksim F(\mathfrak{M}^{k})$ for some $\mathfrak{M}^{k}$. Let $\mathfrak{M}^{m}=(W,R,w_0,V^m)$ and $\mathfrak{M}^{k}=(U,S,u_0,V^k)$. The conditions (i)-(v) for $G$ seem to be obvious; as concerns (ii), for every $w\in W$ there is  $u\in U$ such that $(\mathfrak{M}^{m})_w\thicksim F((\mathfrak{M}^{k})_u)$.
To show (iv) and (v), let us note that $F(\mathfrak{M}^{k})\!\!\upharpoonright n=\sigma(\mathfrak{M}^{k})$, see Figure \ref{hpa2}. Thus, $G(F(\mathfrak{M}^{k}))=\sigma(\mathfrak{M}^{k})$ and $G(\mathfrak{M}^{m})=\mathfrak{M}^{m}\!\!\upharpoonright n\thicksim F(\mathfrak{M}^{k})\!\!\upharpoonright n=\sigma(\mathfrak{M}^{k})$ if $\mathfrak{M}^{m}\thicksim F(\mathfrak{M}^{k})$.\\
B. There remains to define $G(\mathfrak{M}^{m})$ for $\mathfrak{M}^{m}\in \mathbf M^m_{ir}$ non-equivalent to any $F(\mathfrak{M}^{k})$. Our definition is inductive (to secure (ii)) and we should take into account that for some $w\in W$ and some $\mathfrak{M}^{k}$ we may have $(\mathfrak{M}^{m})_w\thicksim F(\mathfrak{M}^{k})$ and hence we have defined $G((\mathfrak{M}^{m})_w)=(\mathfrak{M}^{m})_w\!\!\upharpoonright n$. Since (iii) is  obvious  and (v) is irrelevant in this case, we have to bother only about (i) and (iv). Thus, the frame of  the $n$-model  $G(\mathfrak{M}^{m})$  should be the same as the frame of $\mathfrak{M}^{m}$ and $G(\mathfrak{M}^{m})$ should always be equivalent with a $\sigma$-model.
Let $\mathfrak{M}^{m}$ be one of the following $m$-models:
\begin{figure}[H]
\unitlength1cm
\begin{picture}(5,4)
\thicklines

\put(0.5,4){\circle{0.1}}
\put(2,3){\vector(0,1){0.9}}
\put(2,4){\circle{0.1}}
\put(2,3){\circle{0.1}}
\put(0.7,4){$\mathfrak{f}^m_0$}
\put(2.2,4){$\mathfrak{f}^m_{0}$}
\put(2.2,3){$\mathfrak{f}^m_{1}$}
\put(0.7,1){where $s,d\geq 1$}

\put(3.6,3){$\mathfrak{f}^m_{1}$}
\put(6,3){\circle{0.1}}
\put(5,3){\circle{0.1}}
\put(4.5,3){\circle{0.1}}
\put(4.75,3){\circle{0.1}}
\put(5.25,3){\circle{0.1}}
\put(4.25,3){\circle{0.1}}
\put(4,3){\circle{0.1}}
\put(5,2){\vector(1,1){0.9}}
\put(5,2){\vector(1,1){0.9}}
\put(5,2){\vector(0,1){0.9}}
\put(5,2){\circle{0.1}}
\put(5,4){\circle{0.1}}
\put(6,3){\vector(-1,1){0.9}}
\put(5,3){\vector(0,1){0.9}}
\put(4,3){\vector(1,1){0.9}}
\put(5,2){\vector(1,2){0.4}}
\put(5,2){\vector(-1,2){0.4}}
\put(5,2){\vector(-1,1){0.9}}
\put(5.5,3){\circle{0.1}}
\put(4.5,3){\vector(1,2){0.4}}
\put(5.5,3){\vector(-1,2){0.4}}
\put(5.2,4){$\mathfrak{f}^m_0$}
\put(6.2,3){$\mathfrak{f}^m_{s}$}
\put(5.4,2){$\mathfrak{f}^m_{s+1}$}

\put(7.6,3){$\mathfrak{f}^m_{1}$}
\put(10,3){\circle{0.1}}
\put(9,3){\circle{0.1}}
\put(8.5,3){\circle{0.1}}
\put(8.75,3){\circle{0.1}}
\put(9.25,3){\circle{0.1}}
\put(8.25,3){\circle{0.1}}
\put(8,3){\circle{0.1}}
\put(9,2){\vector(1,1){0.9}}
\put(9,2){\vector(-1,1){0.9}}
\put(9,2){\vector(0,1){0.9}}
\put(9,2){\circle{0.1}}
\put(9,1.75){\circle{0.1}}
\put(9,1.5){\circle{0.1}}
\put(9,1.25){\circle{0.1}}
\put(9,1){\circle{0.1}}
\put(9,0){\circle{0.1}}
\put(9,0){\vector(0,1){0.9}}
\put(9,1){\circle{0.1}}
\put(9,0){\circle{0.1}}
\put(9,4){\circle{0.1}}
\put(10,3){\vector(-1,1){0.9}}
\put(9,3){\vector(0,1){0.9}}
\put(8,3){\vector(1,1){0.9}}
\put(9,2){\vector(-1,2){0.4}}
\put(9,2){\vector(1,2){0.4}}
\put(9.5,3){\circle{0.1}}
\put(8.5,3){\vector(1,2){0.4}}
\put(9.5,3){\vector(-1,2){0.4}}
\put(9.2,4){$\mathfrak{f}^m_0$}
\put(10.2,3){$\mathfrak{f}^m_{s}$}
\put(9.4,2){$\mathfrak{f}^m_{s+1}$}
\put(9.2,0){$\mathfrak{f}^m_{s+d+1}$}
\end{picture}
\caption{}\label{hpa3}
\end{figure}
\noindent B1. If $(\mathfrak{M}^{m})_{0}$\footnote{Let us agree that, in Figure \ref{hpa3}, the vertex $i$ is  labeled with $\mathfrak f^m_i$.} is not equivalent with any $F(\mathfrak{M}^{k})$, we take any valuation $\mathfrak f^k$, let $\sigma(\mathfrak L_1,\{\mathfrak f^k\})=(\mathfrak L_1,\{\mathfrak f^n\})$, for some $\mathfrak f^n$, and put  $G(\mathfrak{M}^{m})\thicksim(\mathfrak L_1,\{\mathfrak f^n\})$ (which uniquely determines $G(\mathfrak{M}^{m})$).
Then (i)--(iv) are clearly fulfilled and (v) is irrelevant. We do the same if $0$ is the only vertex $i$ for which $(\mathfrak{M}^{m})_i\thicksim F(\mathfrak{M}^{k})$, with some $\mathfrak{M}^{k}$; we only take $\mathfrak f^k=\mathfrak f^k_0$.\\
B2. We can defined $G((\mathfrak{M}^{m})_{i})$, if $1\leq i\leq s$, using  A.  or B1. Thus, if $(\mathfrak{M}^{m})_{i}=(\mathfrak{L}_{2},\{\mathfrak{f}^m_0,\mathfrak{f}^m_1\})$ and
$\mathfrak{f}^m_0=\mathfrak{f}^n_0\mathfrak{g}^n_2\dots\mathfrak{g}^n_{2^n}1=code(\mathfrak f^k_0)$ and
$\mathfrak{f}^m_1=\mathfrak{f}^n_1\mathfrak{g}^n_2\dots\mathfrak{g}^n_{2^n}0$, for some p-irreducible $\mathfrak{M}^{k}=(\mathfrak{L}_{2},\{\mathfrak{f}^k_0,\mathfrak{f}^k_1\})$ (and we have $\sigma(\mathfrak{M}^{k})=(\mathfrak{L}_{2},\{\mathfrak{f}^n_0,\mathfrak{f}^n_1\})$), we put $G((\mathfrak{M}^{m})_{i})=\sigma(\mathfrak{M}^{k})=(\mathfrak{L}_{2},\{\mathfrak{f}^n_0,\mathfrak{f}^n_1\})$. Note that $\mathfrak{f}^n_1$ must be one of  $\mathfrak{g}^n_i$'s in $code(\mathfrak f^k_0)$.
Otherwise, $G((\mathfrak{M}^{m})_{i})\thicksim\sigma(\mathfrak{L}_{1},\{\mathfrak{f}^k_0\})$,  where  $G((\mathfrak{M}^{m})_{0})=\sigma((\mathfrak{L}_{1},\{\mathfrak{f}^k_0\})$).
Thus, in each case we have $G((\mathfrak{M}^{m})_{i})=\sigma(\mathfrak{M}^{k})$, for some $k$-model $\mathfrak{M}^{k}$ over  $\mathfrak{L}_{2}$.\footnote{We have added the suffix $0$ or $1$ at the end of $code(\mathfrak f^k_0)$, see Figure \ref{hpa2}, to be sure that $F(\mathfrak{M}^{k})$ is p-irreducible if $\mathfrak{M}^{k}$ is a p-irreducible model over $\mathfrak{L}_{2}$.}

\noindent B3. The worst case is   $G((\mathfrak{M}^{m})_{s+1})$. By B2. and  the definition of $code(\mathfrak f^k_0)$ (included in $\mathfrak f^m_0$), we can assume that, if $1\leq i\leq s$, then $G((\mathfrak{M}^{m})_{i})=G(\mathfrak L_2,\{\mathfrak f^m_0,\mathfrak f^m_i\})=\sigma(\mathfrak{M}^{k}_i)$ for some $k$-model $\mathfrak{M}^{k}_i=(\mathfrak L_2,\{\mathfrak f^k_0,\mathfrak f^k_i\})$; note that we  have the same $\mathfrak f^k_0$ in each of the models $\mathfrak{M}^{k}_i$. There are different  $\mathfrak f^k_0$'s with the same $code(\mathfrak f^k_0)$; but we can always decide which $\mathfrak f^k_0$ is chosen - for instance, our choose could be given according to the lexicographical order  in which all valuations are gven. Then, assuming $(\mathfrak{M}^{m})_{{s+1}}$ is not equivalent with any $F(\mathfrak{M}^{k})$, we take the $k$-valuation $0\cdots0$ and define $G((\mathfrak{M}^{m})_{{s+1}})$ as it is shown in Figure \ref{hpa5} below.
\begin{figure}[H]
\unitlength1cm
\begin{picture}(5,4)
\thicklines

\put(0,3){$\sigma\Bigl(\ \ \mathfrak{f}^k_1$}
\put(2.5,3){\circle{0.1}}
\put(1.5,3){\circle{0.1}}
\put(1,3){\circle{0.1}}
\put(0.5,3){\circle{0.1}}
\put(1.5,2){\vector(1,1){0.9}}
\put(1.5,2){\vector(1,1){0.9}}
\put(1.5,2){\vector(0,1){0.9}}
\put(1.5,2){\circle{0.1}}
\put(1.5,4){\circle{0.1}}
\put(2.5,3){\vector(-1,1){0.9}}
\put(1.5,3){\vector(0,1){0.9}}
\put(0.5,3){\vector(1,1){0.9}}
\put(1.5,2){\vector(1,2){0.4}}
\put(1.5,2){\vector(-1,2){0.4}}
\put(1.5,2){\vector(-1,1){0.9}}
\put(2,3){\circle{0.1}}
\put(2,3){\vector(-1,2){0.4}}
\put(1,3){\vector(1,2){0.4}}
\put(1.7,4){$\mathfrak{f}^k_0$}
\put(2.7,3){$\mathfrak{f}^k_s\ \Bigr)\ =$}
\put(1.9,2){$0\cdots0$}

\put(6,3){\circle{0.1}}
\put(5,3){\circle{0.1}}
\put(4.5,3){\circle{0.1}}
\put(4,3){\circle{0.1}}
\put(5,2){\vector(1,1){0.9}}
\put(5,2){\vector(1,1){0.9}}
\put(5,2){\vector(0,1){0.9}}
\put(5,2){\circle{0.1}}
\put(5,4){\circle{0.1}}
\put(6,3){\vector(-1,1){0.9}}
\put(5,3){\vector(0,1){0.9}}
\put(4,3){\vector(1,1){0.9}}
\put(5,2){\vector(1,2){0.4}}
\put(5,2){\vector(-1,2){0.4}}
\put(5,2){\vector(-1,1){0.9}}
\put(5.5,3){\circle{0.1}}
\put(4.5,3){\vector(1,2){0.4}}
\put(5.5,3){\vector(-1,2){0.4}}
\put(5.2,4){$\mathfrak{f}^n_0$}
\put(6.2,3){$\mathfrak{f}^n_s$}
\put(5.4,2){$\mathfrak{f}^n$}
\put(4.22,3){$\mathfrak{f}^n_1$}

\put(10.9,3.3){$\mathfrak{f}^n_1$}
\put(13,3){\circle{0.1}}
\put(12,3){\circle{0.1}}
\put(11.5,3){\circle{0.1}}
\put(11,3){\circle{0.1}}
\put(12,2){\vector(1,1){0.9}}
\put(12,2){\vector(-1,1){0.9}}
\put(12,2){\vector(0,1){0.9}}
\put(12,2){\circle{0.1}}
\put(12,1.75){\circle{0.1}}
\put(12,1.5){\circle{0.1}}
\put(12,1.25){\circle{0.1}}
\put(12,1){\circle{0.1}}
\put(12,0){\circle{0.1}}
\put(12,0){\vector(0,1){0.9}}
\put(12,1){\circle{0.1}}
\put(12,0){\circle{0.1}}
\put(12,4){\circle{0.1}}
\put(13,3){\vector(-1,1){0.9}}
\put(12,3){\vector(0,1){0.9}}
\put(11,3){\vector(1,1){0.9}}
\put(12,2){\vector(-1,2){0.4}}
\put(12,2){\vector(1,2){0.4}}
\put(12.5,3){\circle{0.1}}
\put(11.5,3){\vector(1,2){0.4}}
\put(12.5,3){\vector(-1,2){0.4}}

\put(7.8,3.3){$\mathfrak{f}^m_1$}
\put(10,3){\circle{0.1}}
\put(9,3){\circle{0.1}}
\put(8.5,3){\circle{0.1}}
\put(8,3){\circle{0.1}}
\put(9,2){\vector(1,1){0.9}}
\put(9,2){\vector(-1,1){0.9}}
\put(9,2){\vector(0,1){0.9}}
\put(9,2){\circle{0.1}}
\put(9,1.75){\circle{0.1}}
\put(9,1.5){\circle{0.1}}
\put(9,1.25){\circle{0.1}}
\put(9,1){\circle{0.1}}
\put(9,0){\circle{0.1}}
\put(9,0){\vector(0,1){0.9}}
\put(9,1){\circle{0.1}}
\put(9,0){\circle{0.1}}
\put(9,4){\circle{0.1}}
\put(10,3){\vector(-1,1){0.9}}
\put(9,3){\vector(0,1){0.9}}
\put(8,3){\vector(1,1){0.9}}
\put(9,2){\vector(-1,2){0.4}}
\put(9,2){\vector(1,2){0.4}}
\put(9.5,3){\circle{0.1}}
\put(8.5,3){\vector(1,2){0.4}}
\put(9.5,3){\vector(-1,2){0.4}}

\put(9.2,0){$\mathfrak{f}^m_{s+d+1}$}
\put(9.2,4){$\mathfrak{f}^m_0$}
\put(10,3.3){$\mathfrak{f}^m_s$}
\put(9.4,2){$\mathfrak{f}^m_{s+1}$}
\put(7.1,2){$G\Bigl($}
\put(10.3,2){$\Bigr)\ =$}

\put(12.2,1){$\mathfrak{f}^n$}
\put(12.2,0){$\mathfrak{f}^n$}
\put(12.2,4){$\mathfrak{f}^n_0$}
\put(13,3.3){$\mathfrak{f}^n_s$}
\put(12.4,2){$\mathfrak{f}^n$}

\end{picture}
\caption{}\label{hpa5}
\end{figure}
\noindent B4. There remains to define   $G((\mathfrak{M}^{m})_{{s+i+1}})$ assuming that $i\geq 1$, and  $(\mathfrak{M}^{m})_{{s+i+1}}$ is not equivalent with any  $F(\mathfrak{M}^{k})$ and  $G((\mathfrak{M}^{m})_{{s+i}})$ has already been defined. Then we approach as it is shown in Figure \ref{hpa6}
\begin{figure}[H]
\unitlength1cm
\begin{picture}(5,4)
\thicklines

\put(0.2,3.2){ $\mathfrak{f}^m_1$}
\put(0,2){$G\Bigl( $}
\put(2.5,3){\circle{0.1}}
\put(1.5,3){\circle{0.1}}
\put(1,3){\circle{0.1}}
\put(0.5,3){\circle{0.1}}
\put(1.5,2){\vector(1,1){0.9}}
\put(1.5,2){\vector(1,1){0.9}}
\put(1.5,2){\vector(0,1){0.9}}
\put(1.5,2){\circle{0.1}}
\put(1.5,4){\circle{0.1}}
\put(2.5,3){\vector(-1,1){0.9}}
\put(1.5,3){\vector(0,1){0.9}}
\put(0.5,3){\vector(1,1){0.9}}
\put(1.5,2){\vector(1,2){0.4}}
\put(1.5,2){\vector(-1,2){0.4}}
\put(1.5,2){\vector(-1,1){0.9}}
\put(2,3){\circle{0.1}}
\put(2,3){\vector(-1,2){0.4}}
\put(1,3){\vector(1,2){0.4}}
\put(1.7,4){$\mathfrak{f}^m_0$}
\put(2.5,3.2){$\mathfrak{f}^m_s$}
\put(2.7,2){$\ \ \Bigr)\ =$}
\put(1.9,2){$\mathfrak{f}^m_{s+1}$}
\put(1.5,1.75){\circle{0.1}}
\put(1.5,1.5){\circle{0.1}}
\put(1.5,1){\circle{0.1}}
\put(1.5,1){\vector(0,1){0.4}}
\put(1.5,1){\circle{0.1}}
\put(1.9,1){$\mathfrak{f}^m_{s+i}$}

\put(6,3){\circle{0.1}}
\put(5,3){\circle{0.1}}
\put(4.5,3){\circle{0.1}}
\put(4,3){\circle{0.1}}
\put(5,2){\vector(1,1){0.9}}
\put(5,2){\vector(1,1){0.9}}
\put(5,2){\vector(0,1){0.9}}
\put(5,2){\circle{0.1}}
\put(5,4){\circle{0.1}}
\put(6,3){\vector(-1,1){0.9}}
\put(5,3){\vector(0,1){0.9}}
\put(4,3){\vector(1,1){0.9}}
\put(5,2){\vector(1,2){0.4}}
\put(5,2){\vector(-1,2){0.4}}
\put(5,2){\vector(-1,1){0.9}}
\put(5.5,3){\circle{0.1}}
\put(4.5,3){\vector(1,2){0.4}}
\put(5.5,3){\vector(-1,2){0.4}}
\put(5.2,4){$\mathfrak{f}^n_0$}
\put(6,3.2){$\mathfrak{f}^n_s$}
\put(5.4,2){$\mathfrak{f}^n_{s+1}$}
\put(3.8,3.2){$\mathfrak{f}^n_1$}
\put(5,1.75){\circle{0.1}}
\put(5,1.5){\circle{0.1}}
\put(5,1){\circle{0.1}}
\put(5,1){\vector(0,1){0.4}}
\put(5,1){\circle{0.1}}
\put(5.4,1){$\mathfrak{f}^n_{s+i}$}

\put(10.9,3.3){$\mathfrak{f}^n_1$}
\put(13,3){\circle{0.1}}
\put(12,3){\circle{0.1}}
\put(11.5,3){\circle{0.1}}
\put(11,3){\circle{0.1}}
\put(12,2){\vector(1,1){0.9}}
\put(12,2){\vector(-1,1){0.9}}
\put(12,2){\vector(0,1){0.9}}
\put(12,2){\circle{0.1}}
\put(12,1.75){\circle{0.1}}
\put(12,1.5){\circle{0.1}}
\put(12,1.25){\circle{0.1}}
\put(12,1){\circle{0.1}}
\put(12,0){\circle{0.1}}
\put(12,0){\vector(0,1){0.9}}
\put(12,1){\circle{0.1}}
\put(12,0){\circle{0.1}}
\put(12,4){\circle{0.1}}
\put(13,3){\vector(-1,1){0.9}}
\put(12,3){\vector(0,1){0.9}}
\put(11,3){\vector(1,1){0.9}}
\put(12,2){\vector(-1,2){0.4}}
\put(12,2){\vector(1,2){0.4}}
\put(12.5,3){\circle{0.1}}
\put(11.5,3){\vector(1,2){0.4}}
\put(12.5,3){\vector(-1,2){0.4}}

\put(7.8,3.3){$\mathfrak{f}^m_1$}
\put(10,3){\circle{0.1}}
\put(9,3){\circle{0.1}}
\put(8.5,3){\circle{0.1}}
\put(8,3){\circle{0.1}}
\put(9,2){\vector(1,1){0.9}}
\put(9,2){\vector(-1,1){0.9}}
\put(9,2){\vector(0,1){0.9}}
\put(9,2){\circle{0.1}}
\put(9,1.75){\circle{0.1}}
\put(9,1.5){\circle{0.1}}
\put(9,1.25){\circle{0.1}}
\put(9,1){\circle{0.1}}
\put(9,0){\circle{0.1}}
\put(9,0){\vector(0,1){0.9}}
\put(9,1){\circle{0.1}}
\put(9,0){\circle{0.1}}
\put(9,4){\circle{0.1}}
\put(10,3){\vector(-1,1){0.9}}
\put(9,3){\vector(0,1){0.9}}
\put(8,3){\vector(1,1){0.9}}
\put(9,2){\vector(-1,2){0.4}}
\put(9,2){\vector(1,2){0.4}}
\put(9.5,3){\circle{0.1}}
\put(8.5,3){\vector(1,2){0.4}}
\put(9.5,3){\vector(-1,2){0.4}}

\put(9.2,1){$\mathfrak{f}^m_{s+i}$}
\put(9.2,0){$\mathfrak{f}^m_{s+d+1}$}
\put(9.2,4){$\mathfrak{f}^m_0$}
\put(10,3.3){$\mathfrak{f}^m_s$}
\put(9.4,2){$\mathfrak{f}^m_{s+1}$}
\put(7.1,2){$G\Bigl($}
\put(10.3,2){$\Bigr)\ =$}

\put(12.2,1){$\mathfrak{f}^n_{s+i}$}
\put(12.2,0){$\mathfrak{f}^n_{s+i}$}
\put(12.2,0.5){$\mathfrak{f}^n_{s+i}$}
\put(12.2,4){$\mathfrak{f}^n_0$}
\put(13,3.3){$\mathfrak{f}^n_s$}
\put(12.4,2){$\mathfrak{f}^n_{s+1}$}

\end{picture}
\caption{}\label{hpa6}
\end{figure}
\noindent Thus, we have $G(\mathfrak{M}^{m})\thicksim G((\mathfrak{M}^{m})_{{s+i}})$.
\end{proof}
\begin{theorem}\label{c5} The logic $\mathsf L(\mathfrak C_5)$ (see Figure \ref{TF}) has finitary unification.\end{theorem}\begin{proof} Let $\mathbf F=\{\mathfrak L_1, \mathfrak L_2, \mathfrak F_2, \mathfrak L_3, +\mathfrak F_2, \mathfrak R_2, \mathfrak Y_2, \mathfrak Y_3, \mathfrak C_5 \}$ (see Figure \ref{GF} and \ref{8fames}) and let $\mathsf{L}=\mathsf{L}(\mathbf{F})$. Since $\mathbf{F}=sm(\mathfrak C_5)$, we conclude {\sf L} is the logic of  $\mathfrak C_5$; see Lemma \ref{lf8}. Let $n\geq 1$ and $m=n+n\cdot 4^n+2$. Take any  $\sigma\colon\{x_1,\dots,x_n\}\to \mathsf{Fm}^k$ and show there are $F\colon \mathbf M^k\to\mathbf M^m$ and $G\colon \mathbf M^m\to\mathbf M^n$  fulfilling the conditions (i)--(v) of  Theorem \ref{main2}.

 Let $\mathfrak f^k$ be any $k$-valuation  and let $\sigma(\mathfrak L_1,\mathfrak f^k)=(\mathfrak L_1,\mathfrak f^n)$ for some $\mathfrak f^n$. We define the $m$-valuation $code(\mathfrak f^k)$  a bit different way than in the proof of Theorem \ref{wpl}. Let\\
\indent $code(\mathfrak{f}^k)=\mathfrak{f}^n\mathfrak{g}^n_1\mathfrak h^n_1\dots\mathfrak{g}^n_{2^n}\mathfrak{h}^n_{2^n}11\quad \mbox{(the concatenation of $\mathfrak f^n$, all $\mathfrak{g}^n_i,\mathfrak{h}^n_i$'s and  $11$)}$\\
where  $\mathfrak{g}^n_i,\mathfrak{h}^n_i$'s are all $n$-valuations such that
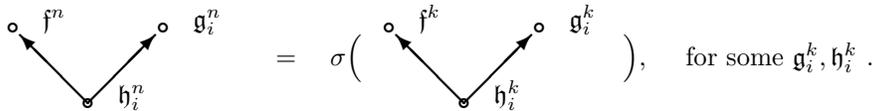
\begin{figure}[H]
\unitlength1cm
\begin{picture}(3,1)
\thicklines

		\put(0,1){\circle{0.1}}
        \put(1,0){\circle{0.1}}
		\put(0.4,1){\mbox{$\mathfrak{f}^n$}}
		\put(2,1){\circle{0.1}}
		\put(2.4,1){\mbox{$\mathfrak{g}^n_i$}}
		\put(1,0){\vector(1,1){0.9}}
         \put(1,0){\vector(-1,1){0.9}}
		\put(1.4,0){\mbox{$\mathfrak{h}^n_i$}}
		
\put(3.5,0.5){\mbox{$ = \quad \sigma \Bigl($}}

			\put(5,1){\circle{0.1}}
        \put(6,0){\circle{0.1}}
        \put(6.4,0){\mbox{$\mathfrak{h}_i^k$}}
		\put(5.4,1){\mbox{$\mathfrak{f}^k$}}
		\put(7,1){\circle{0.1}}
		\put(7.4,1){\mbox{$\mathfrak{g}^k_i$}}
		\put(6,0){\vector(1,1){0.9}}
         \put(6,0){\vector(-1,1){0.9}}
		\put(8,0.5){ $\Bigr), \quad \mbox{ for some } \mathfrak{g}_i^k,\mathfrak{h}_i^k$ .}
\end{picture}
\caption{The definition of $code(\mathfrak f^k)$} \label{code}\end{figure}
We could make $code(\mathfrak f^k)$ unique, for any $\mathfrak f^k$. For instance, let us begin with the sequence $\mathfrak{g}_1^n\mathfrak{h}_1^n\dots \mathfrak{g}_{2^n}^n\mathfrak{h}_{2^n}^n$ containing all possible pairs $\mathfrak{g}_i^n\mathfrak{h}_i^n$ of $n$-valuations (regarded as binary strings) put in a lexicographical order. The sequence contains $4^n$ bites. Then, step by step, we would remain the pair $\mathfrak{g}_i^n\mathfrak{h}_i^n$ or replace it with $\mathfrak{f}^n\mathfrak{f}^n$ depending  if it fulfills (or not) the equation in Figure \ref{code}, for some  $\mathfrak{g}_i^k,\mathfrak{h}_i^k$. Eventually, we add $\mathfrak f^n$ at the beginning and $11$ at the end of the sequence receiving an $m$-valuation.

Now, we are ready to define $F(\mathfrak{M}^k)$, for any p-irreducible $k$-model $\mathfrak{M}^k$ over any frame in $\mathbf F$. So, let   $\mathfrak{M}^k$ be one of the following $k$-models:
\begin{figure}[H]
\unitlength1cm
\begin{picture}(3,2)
\thicklines

\put(0,0){\circle{0.1}}
\put(0.3,0){\mbox{$\mathfrak{f}^k_0$}}
\put(2,0){\circle{0.1}}
\put(2.3,0){\mbox{$\mathfrak{f}^k_1$}}
\put(2,1){\circle{0.1}}
\put(2.3,1){\mbox{$\mathfrak{f}^k_0$}}
\put(2,0){\vector(0,1){0.9}}

\put(4,1){\circle{0.1}}
\put(5,0){\circle{0.1}}
\put(6,1){\circle{0.1}}
\put(5.3,0){\mbox{$\mathfrak{f}^k_1$}}
\put(4.3,1){\mbox{$\mathfrak{f}^k_{0'}$}}
\put(6.3,1){\mbox{$\mathfrak{f}^k_0$}}
\put(5,0){\vector(1,1){0.9}}
\put(5,0){\vector(-1,1){0.9}}

\put(8,0){\vector(0,1){0.9}}
\put(8,1){\vector(0,1){0.9}}
\put(8,1){\circle{0.1}}
\put(8,2){\circle{0.1}}
\put(8,0){\circle{0.1}}
\put(8.3,0){\mbox{$\mathfrak{f}^k_2$}}
\put(8.3,1.8){\mbox{$\mathfrak{f}^k_0$}}
\put(8.3,1){\mbox{$\mathfrak{f}^k_1$}}

\put(10,2){\circle{0.1}}
\put(11,1){\circle{0.1}}
\put(11,0){\circle{0.1}}
\put(12,2){\circle{0.1}}
\put(11.3,1){\mbox{$\mathfrak{f}^k_1$}}
\put(11.3,0){\mbox{$\mathfrak{f}^k_2$}}
\put(10.3,2){\mbox{$\mathfrak{f}^k_{0'}$}}
\put(12.3,2){\mbox{$\mathfrak{f}^k_0$}}
\put(11,1){\vector(1,1){0.9}}
\put(11,1){\vector(-1,1){0.9}}
\put(11,0){\vector(0,1){0.9}}

\end{picture}\\

\unitlength1cm
\begin{picture}(5,3)
\thicklines

\put(1,0){\vector(-1,1){0.9}}
\put(1,0){\vector(1,1){0.9}}
\put(0,1){\vector(1,1){0.9}}
\put(2,1){\vector(-1,1){0.9}}
\put(0,1){\circle{0.1}}
\put(1,2){\circle{0.1}}
\put(1,0){\circle{0.1}}
\put(2,1){\circle{0.1}}
\put(1.3,0){\mbox{$\mathfrak{f}^k_2$}}
\put(0,1.4){\mbox{$\mathfrak{f}^k_{1'}$}}
\put(1.1,2){\mbox{$\mathfrak{f}^k_0$}}
\put(2,1.4){\mbox{$\mathfrak{f}^k_1$}}

\put(4.5,0){\vector(-1,1){0.9}}
\put(4.5,0){\vector(1,1){0.9}}
\put(5.5,1){\vector(1,1){0.9}}
\put(3.5,1){\circle{0.1}}
\put(4.5,2){\circle{0.1}}
\put(4.5,0){\circle{0.1}}
\put(5.5,1){\circle{0.1}}
\put(6.5,2){\circle{0.1}}
\put(3.5,1){\vector(1,1){0.9}}
\put(5.5,1){\vector(-1,1){0.9}}
\put(4.8,0){\mbox{$\mathfrak{f}^k_2$}}
\put(3.5,1.4){\mbox{$\mathfrak{f}^k_{1'}$}}
\put(4.1,2){\mbox{$\mathfrak{f}^k_{0'}$}}
\put(5.5,1.4){\mbox{$\mathfrak{f}^k_1$}}
\put(6.1,2){\mbox{$\mathfrak{f}^k_{0}$}}

\put(7.7,2){\circle{0.1}}
\put(9.7,2){\circle{0.1}}
\put(7.7,1){\circle{0.1}}
\put(9.7,1){\circle{0.1}}
\put(8.7,0){\circle{0.1}}
\put(7.7,1){\vector(0,1){0.9}}
\put(9.7,1){\vector(0,1){0.9}}
\put(7.7,1){\vector(2,1){1.9}}
\put(9.7,1){\vector(-2,1){1.9}}
\put(8.7,0){\vector(1,1){0.9}}
\put(8.7,0){\vector(-1,1){0.9}}
\put(9,0){\mbox{$\mathfrak{f}^k_2$}}
\put(7.3,1){\mbox{$\mathfrak{f}^k_{1'}$}}
\put(7.7,2.2){\mbox{$\mathfrak{f}^k_{0'}$}}
\put(9.8,1){\mbox{$\mathfrak{f}^k_1$}}
\put(9.7,2.2){\mbox{$\mathfrak{f}^k_{0}$}}

\put(11.5,0){\vector(-1,1){0.9}}
\put(11.5,0){\vector(1,1){0.9}}
\put(12.5,1){\vector(-1,1){0.9}}
\put(10.5,1){\circle{0.1}}
\put(11.5,2){\circle{0.1}}
\put(11.5,0){\circle{0.1}}
\put(12.5,1){\circle{0.1}}
\put(10.5,1){\vector(1,1){0.9}}
\put(10.5,2){\circle{0.1}}
\put(12.5,2){\circle{0.1}}
\put(10.5,1){\vector(0,1){0.9}}
\put(12.5,1){\vector(0,1){0.9}}
\put(12,0){\mbox{$\mathfrak{f}^k_2$}}
\put(10.8,1){\mbox{$\mathfrak{f}^k_{1'}$}}
\put(10.5,2.2){\mbox{$\mathfrak{f}^k_{0'}$}}
\put(11.8,1){\mbox{$\mathfrak{f}^k_1$}}
\put(11.5,2.2){\mbox{$\mathfrak{f}^k_{0}$}}
\put(12.5,2.2){\mbox{$\mathfrak{f}^k_{0''}$}}

\end{picture}\\
\caption{} \label{c}
\end{figure}
and let $\sigma(\mathfrak{M}^k)$ be, correspondingly;
\begin{figure}[H]
\unitlength1cm
\begin{picture}(3,2)
\thicklines

\put(0,0){\circle{0.1}}
\put(0.3,0){\mbox{$\mathfrak{f}^n_0$}}
\put(2,0){\circle{0.1}}
\put(2.3,0){\mbox{$\mathfrak{f}^n_1$}}
\put(2,1){\circle{0.1}}
\put(2.3,1){\mbox{$\mathfrak{f}^n_0$}}
\put(2,0){\vector(0,1){0.9}}

\put(4,1){\circle{0.1}}
\put(5,0){\circle{0.1}}
\put(6,1){\circle{0.1}}
\put(5.3,0){\mbox{$\mathfrak{f}^n_1$}}
\put(4.3,1){\mbox{$\mathfrak{f}^n_{0'}$}}
\put(6.3,1){\mbox{$\mathfrak{f}^n_0$}}
\put(5,0){\vector(1,1){0.9}}
\put(5,0){\vector(-1,1){0.9}}

\put(8,0){\vector(0,1){0.9}}
\put(8,1){\vector(0,1){0.9}}
\put(8,1){\circle{0.1}}
\put(8,2){\circle{0.1}}
\put(8,0){\circle{0.1}}
\put(8.3,0){\mbox{$\mathfrak{f}^n_2$}}
\put(8.3,1.8){\mbox{$\mathfrak{f}^n_0$}}
\put(8.3,1){\mbox{$\mathfrak{f}^n_1$}}

\put(10,2){\circle{0.1}}
\put(11,1){\circle{0.1}}
\put(11,0){\circle{0.1}}
\put(12,2){\circle{0.1}}
\put(11.3,1){\mbox{$\mathfrak{f}^n_1$}}
\put(11.3,0){\mbox{$\mathfrak{f}^n_2$}}
\put(10.3,2){\mbox{$\mathfrak{f}^n_{0'}$}}
\put(12.3,2){\mbox{$\mathfrak{f}^n_0$}}
\put(11,1){\vector(1,1){0.9}}
\put(11,1){\vector(-1,1){0.9}}
\put(11,0){\vector(0,1){0.9}}

\end{picture}\\

\unitlength1cm
\begin{picture}(5,3)
\thicklines

\put(1,0){\vector(-1,1){0.9}}
\put(1,0){\vector(1,1){0.9}}
\put(0,1){\vector(1,1){0.9}}
\put(2,1){\vector(-1,1){0.9}}
\put(0,1){\circle{0.1}}
\put(1,2){\circle{0.1}}
\put(1,0){\circle{0.1}}
\put(2,1){\circle{0.1}}
\put(1.3,0){\mbox{$\mathfrak{f}^n_2$}}
\put(0,1.4){\mbox{$\mathfrak{f}^n_{1'}$}}
\put(1.1,2){\mbox{$\mathfrak{f}^n_0$}}
\put(2,1.4){\mbox{$\mathfrak{f}^n_1$}}

\put(4.5,0){\vector(-1,1){0.9}}
\put(4.5,0){\vector(1,1){0.9}}
\put(5.5,1){\vector(1,1){0.9}}
\put(3.5,1){\circle{0.1}}
\put(4.5,2){\circle{0.1}}
\put(4.5,0){\circle{0.1}}
\put(5.5,1){\circle{0.1}}
\put(6.5,2){\circle{0.1}}
\put(3.5,1){\vector(1,1){0.9}}
\put(5.5,1){\vector(-1,1){0.9}}
\put(4.8,0){\mbox{$\mathfrak{f}^n_2$}}
\put(3.5,1.4){\mbox{$\mathfrak{f}^n_{1'}$}}
\put(4.1,2){\mbox{$\mathfrak{f}^n_{0'}$}}
\put(5.5,1.4){\mbox{$\mathfrak{f}^n_1$}}
\put(6.1,2){\mbox{$\mathfrak{f}^n_{0}$}}

\put(7.7,2){\circle{0.1}}
\put(9.7,2){\circle{0.1}}
\put(7.7,1){\circle{0.1}}
\put(9.7,1){\circle{0.1}}
\put(8.7,0){\circle{0.1}}
\put(7.7,1){\vector(0,1){0.9}}
\put(9.7,1){\vector(0,1){0.9}}
\put(7.7,1){\vector(2,1){1.9}}
\put(9.7,1){\vector(-2,1){1.9}}
\put(8.7,0){\vector(1,1){0.9}}
\put(8.7,0){\vector(-1,1){0.9}}
\put(9,0){\mbox{$\mathfrak{f}^n_2$}}
\put(7.3,1){\mbox{$\mathfrak{f}^n_{1'}$}}
\put(7.7,2.2){\mbox{$\mathfrak{f}^n_{0'}$}}
\put(9.8,1){\mbox{$\mathfrak{f}^n_1$}}
\put(9.7,2.2){\mbox{$\mathfrak{f}^n_{0}$}}

\put(11.5,0){\vector(-1,1){0.9}}
\put(11.5,0){\vector(1,1){0.9}}
\put(12.5,1){\vector(-1,1){0.9}}
\put(10.5,1){\circle{0.1}}
\put(11.5,2){\circle{0.1}}
\put(11.5,0){\circle{0.1}}
\put(12.5,1){\circle{0.1}}
\put(10.5,1){\vector(1,1){0.9}}
\put(10.5,2){\circle{0.1}}
\put(12.5,2){\circle{0.1}}
\put(10.5,1){\vector(0,1){0.9}}
\put(12.5,1){\vector(0,1){0.9}}
\put(12,0){\mbox{$\mathfrak{f}^n_2$}}
\put(10.8,1){\mbox{$\mathfrak{f}^n_{1'}$}}
\put(10.5,2.2){\mbox{$\mathfrak{f}^n_{0'}$}}
\put(11.8,1){\mbox{$\mathfrak{f}^n_1$}}
\put(11.5,2.2){\mbox{$\mathfrak{f}^n_{0}$}}
\put(12.5,2.2){\mbox{$\mathfrak{f}^n_{0''}$}}

\end{picture}\\
\caption{} \label{cc1}
\end{figure}

\noindent We define $F(\mathfrak M^k)$ as one of the following $m$-models:
\begin{figure}[H]
\unitlength1cm
\begin{picture}(3,2)
\thicklines

\put(0,0){\circle{0.1}}
\put(0,0.3){\mbox{$code(\mathfrak{f}^k_0)$}}
\put(2,0){\circle{0.1}}
\put(2.3,0){\mbox{$\mathfrak{f}^n_10\cdots01$}}
\put(2,1){\circle{0.1}}
\put(2,1.3){\mbox{$code(\mathfrak{f}^k_0)$}}
\put(2,0){\vector(0,1){0.9}}

\put(4,1){\circle{0.1}}
\put(5,0){\circle{0.1}}
\put(6,1){\circle{0.1}}
\put(5.3,0){\mbox{$\mathfrak{f}^n_10\cdots01$}}
\put(4,1.3){\mbox{$code(\mathfrak{f}^k_{0'})$}}
\put(6,1.3){\mbox{$code(\mathfrak{f}^k_0)$}}
\put(5,0){\vector(1,1){0.9}}
\put(5,0){\vector(-1,1){0.9}}

\put(8,0){\vector(0,1){0.9}}
\put(8,1){\vector(0,1){0.9}}
\put(8,1){\circle{0.1}}
\put(8,2){\circle{0.1}}
\put(8,0){\circle{0.1}}
\put(8.3,0){\mbox{$\mathfrak{f}^n_20\cdots0$}}
\put(8.3,1.8){\mbox{$code(\mathfrak{f}^k_0)$}}
\put(8.3,1){\mbox{$\mathfrak{f}^n_10\cdots01$}}

\put(10,2){\circle{0.1}}
\put(11,1){\circle{0.1}}
\put(11,0){\circle{0.1}}
\put(12,2){\circle{0.1}}
\put(11.3,1){\mbox{$\mathfrak{f}^n_10\cdots01$}}
\put(11.3,0){\mbox{$\mathfrak{f}^n_20\cdots0$}}
\put(10.3,2){\mbox{$code(\mathfrak{f}^k_{0'})$}}
\put(12,2.3){\mbox{$code(\mathfrak{f}^k_0)$}}
\put(11,1){\vector(1,1){0.9}}
\put(11,1){\vector(-1,1){0.9}}
\put(11,0){\vector(0,1){0.9}}

\end{picture}\\

\unitlength1cm
\begin{picture}(5,3)
\thicklines

\put(1,0){\vector(-1,1){0.9}}
\put(1,0){\vector(1,1){0.9}}
\put(0,1){\vector(1,1){0.9}}
\put(2,1){\vector(-1,1){0.9}}
\put(0,1){\circle{0.1}}
\put(1,2){\circle{0.1}}
\put(1,0){\circle{0.1}}
\put(2,1){\circle{0.1}}
\put(1.3,0){\mbox{$\mathfrak{f}^n_20\cdots0$}}
\put(0.3,0.9){\mbox{$\mathfrak{f}^n_{1'}0\cdots01$}}
\put(0.8,2.2){\mbox{$code(\mathfrak{f}^k_0)$}}
\put(2,1.4){\mbox{$\mathfrak{f}^n_10\cdots01$}}

\put(4.5,0){\vector(-1,1){0.9}}
\put(4.5,0){\vector(1,1){0.9}}
\put(5.5,1){\vector(1,1){0.9}}
\put(3.5,1){\circle{0.1}}
\put(4.5,2){\circle{0.1}}
\put(4.5,0){\circle{0.1}}
\put(5.5,1){\circle{0.1}}
\put(6.5,2){\circle{0.1}}
\put(3.5,1){\vector(1,1){0.9}}
\put(5.5,1){\vector(-1,1){0.9}}
\put(4.8,0){\mbox{$\mathfrak{f}^n_20\cdots0$}}
\put(3.7,0.9){\mbox{$\mathfrak{f}^n_{1'}0\cdots01$}}
\put(3.6,2.2){\mbox{$code(\mathfrak{f}^k_{0'})$}}
\put(5.5,0.6){\mbox{$\mathfrak{f}^n_10\cdots01$}}
\put(5.3,2.2){\mbox{$code(\mathfrak{f}^k_{0})$}}

\put(7.7,2){\circle{0.1}}
\put(9.7,2){\circle{0.1}}
\put(7.7,1){\circle{0.1}}
\put(9.7,1){\circle{0.1}}
\put(8.7,0){\circle{0.1}}
\put(7.7,1){\vector(0,1){0.9}}
\put(9.7,1){\vector(0,1){0.9}}
\put(7.7,1){\vector(2,1){1.9}}
\put(9.7,1){\vector(-2,1){1.9}}
\put(8.7,0){\vector(1,1){0.9}}
\put(8.7,0){\vector(-1,1){0.9}}
\put(9,0){\mbox{$\mathfrak{f}^n_20\cdots0$}}
\put(6.3,1.2){\mbox{$\mathfrak{f}^n_{1'}0\cdots01$}}
\put(7,2.2){\mbox{$code(\mathfrak{f}^k_{0'})$}}
\put(8.2,0.9){\mbox{$\mathfrak{f}^n_10\cdots01$}}
\put(8.2,1.9){\mbox{$code(\mathfrak{f}^k_{0})$}}

\put(11.5,0){\vector(-1,1){0.9}}
\put(11.5,0){\vector(1,1){0.9}}
\put(12.5,1){\vector(-1,1){0.9}}
\put(10.5,1){\circle{0.1}}
\put(11.5,2){\circle{0.1}}
\put(11.5,0){\circle{0.1}}
\put(12.5,1){\circle{0.1}}
\put(10.5,1){\vector(1,1){0.9}}
\put(10.5,2){\circle{0.1}}
\put(12.5,2){\circle{0.1}}
\put(10.5,1){\vector(0,1){0.9}}
\put(12.5,1){\vector(0,1){0.9}}
\put(11.8,0){\mbox{$\mathfrak{f}^n_20\cdots0$}}
\put(10.8,1){\mbox{$\mathfrak{f}^n_{1'}0\cdots01$}}
\put(9.4,2.2){\mbox{$code(\mathfrak{f}^k_{0'})$}}
\put(12.3,0.5){\mbox{$\mathfrak{f}^n_10\cdots01$}}
\put(10.9,2.2){\mbox{$code(\mathfrak{f}^k_{0})$}}
\put(12.3,2.2){\mbox{$code(\mathfrak{f}^k_{0''})$}}

\end{picture}\\
\caption{} \label{cc2}
\end{figure}

\noindent The conditions (i)-(iii) of Theorem \ref{main2}, as concerns the mapping $F$, are obviously fulfilled. Note that it is important that all binary strings  have the size $m$ and we get intuitionistic models, but it is of no importance what is included in $code(\mathfrak f^k)$. There remains to define the mapping $G$. Let
$$G(\mathfrak{M}^m)=\left\{
                      \begin{array}{ll}
                        \mathfrak{M}^m\!\!\upharpoonright n\quad & \hbox{if $\mathfrak{M}^m\thicksim F(\mathfrak{M}^k)$ for some $\mathfrak{M}^k$;} \\
                        ? & \hbox{otherwise;}
                      \end{array}   \right.$$
where the $n$-model $\mathfrak{M}^m\!\!\upharpoonright n$ results from  $\mathfrak{M}^m$ by the restriction of all valuations $\mathfrak{f}^{m}_i$ of $\mathfrak{M}^m$ to $\{x_1,\dots,x_n\}$. Note that $\mathfrak{M}^m\thicksim F(\mathfrak{M}^k)$ yields $\mathfrak{M}^m\!\!\upharpoonright n\thicksim F(\mathfrak{M}^k)\!\!\upharpoonright n$ and hence
$G(\mathfrak{M}^m)\thicksim \sigma(\mathfrak{M}^k)$ by the definition of $F$ and $G$. Thus, (v) is fulfilled. If an  $m$-model $\mathfrak{M}^m$ is equivalent with some  $F(\mathfrak{M}^k)$, then any  its generated submodel $(\mathfrak{M}^m)_w $ is  equivalent with $F((\mathfrak{M}^k)_u)$, for some $u$; see Corollary \ref{lf4i}. Thus, (i)-(v) are fulfilled if the first line of the definition $G(\mathfrak{M}^m)$ applies.

There remains  to define $G(\mathfrak{M}^m)$ assuming $\mathfrak{M}^m$ is not equivalent with any  $F(\mathfrak{M}^k)$. Our definition is inductive with respect to the depth of $\mathfrak{M}^m$ to secure  (ii). Since (v) is irrelevant here and (iii) is quite obvious, we should require  $G(\mathfrak{M}^m)$ and $\mathfrak{M}^m$ were defined over the same frame,   and $G(\mathfrak{M}^m)$ should  be equivalent with a $\sigma$-model.

\underline{$\mathfrak{L}_1$-Models.} If $\mathfrak{M}^m=(\mathfrak{L}_1,\mathfrak f^m)$ is  equivalent with some   $F(\mathfrak{M}^k)$, then we easily get $\mathfrak{M}^m=F(\mathfrak{L}_1,\mathfrak f^k)$, for some $\mathfrak f^k$, and hence  $\mathfrak f^m=code(\mathfrak f^k)$. Thus, we have
 $$G(\mathfrak M^m)=(\mathfrak{L}_1,code(\mathfrak f^k)\!\!\upharpoonright n)=\sigma(\mathfrak{L}_1,\mathfrak f^k).$$
If $\mathfrak f^m\not=code(\mathfrak f^k)$, for any $\mathfrak f^k$, we could take $G(\mathfrak{M}^m)=\sigma(\mathfrak{M}^k)$, for any $k$-model $\mathfrak{M}^k$  over $\mathfrak{L}_1$. We can even take the same $\mathfrak{M}^k$, for each $\mathfrak{M}^m$.  So, we choose some $\mathfrak s^k$ and agree that, if $(\mathfrak{L}_1,\mathfrak f^m)$ is not equivalent with any $F(\mathfrak M^k)$, we have
 $$G(\mathfrak M^m)=\sigma(\mathfrak{L}_1,\mathfrak s^k)=(\mathfrak{L}_1,code(\mathfrak s^k)\!\!\upharpoonright n).$$
\underline{$\mathfrak{F}_2$- and $\mathfrak{L}_2$-Models.}
Let  $\mathfrak{M}^m=(\mathfrak{F}_2,\{\mathfrak{f}^m_0,\mathfrak{f}^m_{0'},\mathfrak{f}^m_1\})$ be a p-irreducible $m$-model over $\mathfrak{F}_2$. Assume that  $\mathfrak{M}^m\thicksim F(\mathfrak{M}^k)$, for some $\mathfrak{M}^k$. Then $\mathfrak{M}^m\equiv F(\mathfrak{M}^k)$ as the suffices $11$,$01$ and $00$ prevent any reduction of  $F(\mathfrak{M}^k)$, for any p-irreducible $\mathfrak{M}^k$ of the depth $3$, to a model over $\mathfrak F_2$. Thus $\mathfrak{M}^k=(\mathfrak{F}_2,\{\mathfrak{f}^k_0,\mathfrak{f}^k_{0'},\mathfrak{f}^k_1\})$, and $\mathfrak{f}^m_0=code(\mathfrak{f}^k_0)=\mathfrak{f}^n_0\mathfrak{g}^n_1\mathfrak h^n_1\dots\mathfrak{g}^n_{2^n}\mathfrak{h}^n_{2^n}11$, and    $\mathfrak{f}^m_{0'}=code(\mathfrak{f}^k_{0'})=\mathfrak f^n_{0'}\mathfrak{g'}^n_1\mathfrak {h'}^n_1\dots\mathfrak{g'}^n_{2^n}\mathfrak{h'}^n_{2^n}11$, and $\mathfrak{f}^m_1=\mathfrak f_1^n0\cdots01$ where $\sigma(\mathfrak M^k)=(\mathfrak{F}_2,\{\mathfrak{f}^n_0,\mathfrak{f}^n_{0'},\mathfrak{f}^n_1\})=G(\mathfrak M^m)$. By the definition of $code()$, see the Figure \ref{code}, the pair $\mathfrak f^n_0\mathfrak f^n_1$ must occur in  $code(\mathfrak{f}^k_{0'})$ and $\mathfrak f^n_{0'}\mathfrak f^n_1$ in
$code(\mathfrak{f}^k_{0})$. We would like to preserve this condition even if $\mathfrak{M}^m$ is not equivalent  with any $F(\mathfrak{M}^k)$.

Let $\mathfrak{M}^m$ be not equivalent with any $F(\mathfrak{M}^k)$. We have  $G(\mathfrak{L}_1,\mathfrak{f}^m_0)=\sigma(\mathfrak{L}_1,\mathfrak{f}^k_0)=(\mathfrak{L}_1,\mathfrak{f}^n_0) $ and $G(\mathfrak{L}_1,\mathfrak{f}^m_{0'})=\sigma(\mathfrak{L}_1,\{\mathfrak{f}^k_{0'}\})=(\mathfrak{L}_1,\mathfrak{f}^n_{0'}) $,  for some $\mathfrak{f}^k_0,\mathfrak{f}^k_{0'}$ and some $\mathfrak{f}^n_0,\mathfrak{f}^n_{0'}$.  Take any  $\mathfrak{f}^k_1$ such that $(\mathfrak{F}_2,\{\mathfrak{f}^k_0,\mathfrak{f}^k_{0'},\mathfrak{f}^k_1\})$ is an intuitionistic model and define
$$ G(\mathfrak{M}^m)=\sigma((\mathfrak{F}_2,\{\mathfrak{f}^k_0,\mathfrak{f}^k_{0'},\mathfrak{f}^k_1\})=(\mathfrak{F}_2, \{\mathfrak{f}^n_0,\mathfrak{f}^n_{0'},\mathfrak{f}^n_1\}), \quad \mbox{for some } \mathfrak{f}^n_1.$$
We obviously have $\mathfrak f^n_0\mathfrak f^n_1$  in  $code(\mathfrak{f}^k_{0'})$ and $\mathfrak f^n_{0'}\mathfrak f^n_1$ in $code(\mathfrak{f}^k_{0})$. For the choice of $\mathfrak{f}^k_1$,  use the lexicographical order which would also be sufficient to satisfy the condition (iii) of Theorem \ref{main2} saying that isomorphic models have isomorphic images.

Let $\mathfrak{M}^m=(\mathfrak{L}_2,\{\mathfrak{f}^m_0,\mathfrak{f}^m_1\})$ be a p-irreducible $m$-model over $\mathfrak{L}_2$ and   $\mathfrak{M}^m\thicksim F(\mathfrak M^k)$, for some $\mathfrak M^k$. We have two possibilities: either $\mathfrak M^k$ is a model over $\mathfrak L_2$ or it is a model over $\mathfrak F_2$. In the first case,  $G(\mathfrak{M}^m)=\sigma(\mathfrak{M}^k)=\sigma(\mathfrak{L}_2,\{\mathfrak{f}^k_0,\mathfrak{f}^k_1\})=
(\mathfrak{L}_2,\{\mathfrak{f}^k_0,\mathfrak{f}^k_1\})$, for some $\mathfrak{f}^k_0,\mathfrak{f}^k_1,\mathfrak{f}^n_0,\mathfrak{f}^n_1$, and $\mathfrak{f}^m_0=code(\mathfrak{f}^k_0)$. It is clear that
$\mathfrak f^n_0\mathfrak f^n_1$  occurs in  $code(\mathfrak{f}^k_{0})$. The second possibilities is  worse. We have $\mathfrak{M}^k=(\mathfrak{F}_2,\{\mathfrak{f}^k_0,\mathfrak{f}^k_{0'},\mathfrak{f}^k_1\})$ and $\mathfrak{f}^m_0=code(\mathfrak{f}^k_0)=code(\mathfrak{f}^k_{0'})$ (which does not yield $\mathfrak{f}^k_0=\mathfrak{f}^k_{0'}$). Since $\sigma(\mathfrak{M}^k)=(\mathfrak{F}_2,\{\mathfrak{f}^n_0,\mathfrak{f}^n_{0'},\mathfrak{f}^n_1\})$ and $\mathfrak{f}^n_0=\mathfrak{f}^n_{0'}$, we take $G(\mathfrak{M}^m)=(\mathfrak{L}_2,\{\mathfrak{f}^n_0,\mathfrak{f}^n_1\})=\mathfrak{M}^m\!\!\upharpoonright n\thicksim\sigma(\mathfrak{M}^k)$ and have $\mathfrak f^n_0\mathfrak f^n_1$  in  $code(\mathfrak{f}^k_{0})$.

Suppose that $\mathfrak{M}^m=(\mathfrak{L}_2,\{\mathfrak{f}^m_0,\mathfrak{f}^m_1\})$ is not equivalent with  any $F(\mathfrak M^k)$. Then   $G(\mathfrak{L}_1,\mathfrak{f}^m_0)=\sigma(\mathfrak{L}_1,\mathfrak{f}^k_0)=(\mathfrak{L}_1,\mathfrak{f}^n_0), $  for some $\mathfrak{f}^k_0,\mathfrak{f}^n_0$. Let us take any $k$-valuation $\mathfrak{f}^k_1$ such that $(\mathfrak{L}_2,\{\mathfrak{f}^k_0,\mathfrak{f}^k_1\})$ is an intuitionistic model and define
$ G(\mathfrak{M}^m)=\sigma((\mathfrak{L}_2,\{\mathfrak{f}^k_0,\mathfrak{f}^k_1\})=(\mathfrak{L}_2, \{\mathfrak{f}^n_0,\mathfrak{f}^n_1\}), \quad \mbox{for some } \mathfrak{f}^n_1.$
We obviously have $\mathfrak f^n_0\mathfrak f^n_1$  in   $code(\mathfrak{f}^k_{0})$.

\underline{$\mathfrak{C}_5$- and $\mathfrak{Y}_3$-Models.}
Let $\mathfrak{M}^m$ be a p-irreducible model over $\mathfrak{C}_5$ (or over $\mathfrak{Y}_3$)  and assume  it is not equivalent with any $F(\mathfrak{M}^k)$. We have  (see Figure \ref{c}-\ref{cc2})\\ $G((\mathfrak{M}^m)_1)\thicksim\sigma(\mathfrak{M}^k)=(\mathfrak F_2,\{\mathfrak{f}^n_0,\mathfrak{f}^n_{0'},\mathfrak{f}^n_1\})$, and \\ $G((\mathfrak{M}^m)_{1'})\thicksim\sigma(\mathfrak{M'}^k)=(\mathfrak F_2,\{\mathfrak{f}^n_0,\mathfrak{f}^n_{0''},\mathfrak{f}^n_{1'}\})$ (or $\sigma(\mathfrak{M'}^k)=(\mathfrak F_2,\{\mathfrak{f}^n_0,\mathfrak{f}^n_{0'},\mathfrak{f}^n_{1'}\})$)\\ for some $\mathfrak{M}^k=(\mathfrak F_2,\{\mathfrak{f}^k_0,\mathfrak{f}^k_{0''},\mathfrak{f}^k_{1'}\})$ and  $\mathfrak{M'}^k=(\mathfrak F_2,\{\mathfrak{f}^k_0,\mathfrak{f}^k_{0''},\mathfrak{f}^k_{1'}\})$ (or $(\mathfrak F_2,\{\mathfrak{f}^k_0,\mathfrak{f}^k_{0'},\mathfrak{f}^k_{1'}\})$), which are $k$-model over $\mathfrak F_2$.
We have $\mathfrak{f}^n_{0'}\mathfrak{f}^n_{1'}$ and $\mathfrak{f}^n_{0''}\mathfrak{f}^n_{1}$ (or $\mathfrak{f}^n_{0'}\mathfrak{f}^n_{1}$) in $code(\mathfrak{f}^k_0)$ (see Figure \ref{code}). Then, for some $\mathfrak{g}^k_1,\mathfrak{h}^k_1,\mathfrak{g}^k_2,\mathfrak{h}^k_2$, we can define  $\mathfrak{C}_5$-models such that

\begin{figure}[H]
\unitlength1cm
\begin{picture}(3,2)
\thicklines
\put(0,1){\mbox{$\sigma\Bigl($}}
\put(2,0){\vector(-1,1){0.9}}
\put(2,0){\vector(1,1){0.9}}
\put(3,1){\vector(-1,1){0.9}}
\put(1,1){\circle{0.1}}
\put(2,2){\circle{0.1}}
\put(2,0){\circle{0.1}}
\put(3,1){\circle{0.1}}
\put(1,1){\vector(1,1){0.9}}
\put(1,2){\circle{0.1}}
\put(3,2){\circle{0.1}}
\put(1,1){\vector(0,1){0.9}}
\put(3,1){\vector(0,1){0.9}}
\put(2.3,0){\mbox{$0\cdots0$}}
\put(1.3,1){\mbox{$\mathfrak{h}^k_{1}$}}
\put(0.5,2){\mbox{$\mathfrak{g}^k_{1}$}}
\put(3.3,1){\mbox{$\mathfrak{h}^k_2$}}
\put(2.3,2){\mbox{$\mathfrak{f}^k_{0}$}}
\put(3.3,2){\mbox{$\mathfrak{g}^k_{2}$}}
\put(4,1){\mbox{$\Bigr)\quad=$}}

\put(7,0){\vector(-1,1){0.9}}
\put(7,0){\vector(1,1){0.9}}
\put(8,1){\vector(-1,1){0.9}}
\put(6,1){\circle{0.1}}
\put(7,2){\circle{0.1}}
\put(7,0){\circle{0.1}}
\put(8,1){\circle{0.1}}
\put(6,1){\vector(1,1){0.9}}
\put(6,2){\circle{0.1}}
\put(8,2){\circle{0.1}}
\put(6,1){\vector(0,1){0.9}}
\put(8,1){\vector(0,1){0.9}}
\put(7.6,0){\mbox{$\mathfrak{f}^n_2$}}
\put(8.5,1){\mbox{$\mathfrak{f}^n_{1}$\qquad, for some  $\mathfrak f_2^n$.}}
\put(5.4,2){\mbox{$\mathfrak{f}^n_{0'}$}}
\put(5.3,1){\mbox{$\mathfrak{f}^n_{1'}$}}
\put(7.3,2){\mbox{$\mathfrak{f}^n_{0}$}}
\put(8.3,2){\mbox{$\mathfrak{f}^n_{0''}\  (\mbox{or}\ \mathfrak{f}^n_{0'})$}}

\end{picture}\\
\caption{} \label{cc5}
\end{figure}
The last $n$-model over $\mathfrak C_5$ can be taken as $G(\mathfrak{M}^m)$. In the case of $\mathfrak Y_3$ we have $\mathfrak f^n_{0'}=\mathfrak f^n_{0''}$ and we can reduce the $\mathfrak C_5$-model to a $\mathfrak Y_3$-model taking $0'=0''$.

\underline{$\mathfrak{R}_2$- and $\mathfrak Y_2$-Models.}
Let $\mathfrak{M}^m$ be a p-irreducible model over $\mathfrak{R}_2$  and assume  it is not equivalent with any $F(\mathfrak{M}^k)$. By our inductive hypothesis, we have  (see  Figure \ref{c}-\ref{cc2})\\
$G((\mathfrak{M}^m)_1)\thicksim\sigma(\mathfrak{M}^k)=(\mathfrak F_2,\{\mathfrak{f}^n_0,\mathfrak{f}^n_{0'},\mathfrak{f}^n_{1'}\})$,  $G((\mathfrak{M}^m)_{1'})\thicksim\sigma(\mathfrak{M'}^k)=(\mathfrak F_2,\{\mathfrak{f}^n_0,\mathfrak{f}^n_{0''},\mathfrak{f}^n_{1}\})$;\\ for some $\mathfrak{M}^k=(\mathfrak F_2,\{\mathfrak{f}^k_0,\mathfrak{f}^k_{0'},\mathfrak{f}^k_{1'}\})$ and  $\mathfrak{M'}^k=(\mathfrak F_2,\{\mathfrak{f}^k_0,\mathfrak{f}^k_{0''},\mathfrak{f}^k_{1}\})$.
We have $\mathfrak{f}^n_{0'}=\mathfrak{f}^n_{0}=\mathfrak{f}^n_{0''}$ and $\mathfrak{f}^n_{0}\mathfrak{f}^n_{1'}$, $\mathfrak{f}^n_{0}\mathfrak{f}^n_{1}$ occur in $code(\mathfrak{f}^k_0)$ (see Figure \ref{code}). Then for some $\mathfrak{g}^k_1,\mathfrak{h}^k_1,\mathfrak{g}^k_2,\mathfrak{h}^k_2,\mathfrak f_2^n$

\begin{figure}[H]
\unitlength1cm
\begin{picture}(3,2)
\thicklines
\put(0,1){\mbox{$\sigma\Bigl($}}
\put(2,0){\vector(-1,1){0.9}}
\put(2,0){\vector(1,1){0.9}}
\put(3,1){\vector(-1,1){0.9}}
\put(1,1){\circle{0.1}}
\put(2,2){\circle{0.1}}
\put(2,0){\circle{0.1}}
\put(3,1){\circle{0.1}}
\put(1,1){\vector(1,1){0.9}}
\put(1,2){\circle{0.1}}
\put(3,2){\circle{0.1}}
\put(1,1){\vector(0,1){0.9}}
\put(3,1){\vector(0,1){0.9}}
\put(2.3,0){\mbox{$0\cdots0$}}
\put(1.3,1){\mbox{$\mathfrak{h}^k_{1}$}}
\put(0.5,2){\mbox{$\mathfrak{g}^k_{1}$}}
\put(3.3,1){\mbox{$\mathfrak{h}^k_2$}}
\put(2.3,2){\mbox{$\mathfrak{f}^k_{0}$}}
\put(3.3,2){\mbox{$\mathfrak{g}^k_{2}$}}
\put(4,1){\mbox{$\Bigr)\quad=$}}

\put(7,0){\vector(-1,1){0.9}}
\put(7,0){\vector(1,1){0.9}}
\put(8,1){\vector(-1,1){0.9}}
\put(6,1){\circle{0.1}}
\put(7,2){\circle{0.1}}
\put(7,0){\circle{0.1}}
\put(8,1){\circle{0.1}}
\put(6,1){\vector(1,1){0.9}}
\put(6,2){\circle{0.1}}
\put(8,2){\circle{0.1}}
\put(6,1){\vector(0,1){0.9}}
\put(8,1){\vector(0,1){0.9}}
\put(7.6,0){\mbox{$\mathfrak{f}^n_2$}}
\put(8.5,1){\mbox{$\thicksim$}}
\put(7.5,1){\mbox{$\mathfrak{f}^n_{1}$}}
\put(5.4,2){\mbox{$\mathfrak{f}^n_{0'}$}}
\put(5.3,1){\mbox{$\mathfrak{f}^n_{1'}$}}
\put(7.3,2){\mbox{$\mathfrak{f}^n_{0}$}}
\put(8.3,2){\mbox{$\mathfrak{f}^n_{0''}$}}

\put(10,0){\vector(-1,1){0.9}}
\put(10,0){\vector(1,1){0.9}}
\put(11,1){\vector(-1,1){0.9}}
\put(9,1){\circle{0.1}}
\put(10,2){\circle{0.1}}
\put(10,0){\circle{0.1}}
\put(11,1){\circle{0.1}}
\put(9,1){\vector(1,1){0.9}}
\put(10.6,0){\mbox{$\mathfrak{f}^n_2$}}
\put(11.2,1){\mbox{$\mathfrak{f}^n_{1}$}}
\put(9.3,1){\mbox{$\mathfrak{f}^n_{1'}$}}
\put(10.3,2){\mbox{$\mathfrak{f}^n_{0}$}}

\end{picture}\\
\caption{} \label{cc6}
\end{figure}
\noindent The received $n$-model over $\mathfrak R_2$ can be taken as $G(\mathfrak M^m)$. A similar argument (as we used above for $\mathfrak R_2$-models) applies for models  over $\mathfrak{Y}_2$.

\underline{$+\mathfrak{F}_2$- and $\mathfrak{L}_3$-Models.}
Let $\mathfrak{M}^m$ be a p-irreducible model over $\mathfrak{L}_3$, or over $+\mathfrak{F}_2$,  and assume that it is not equivalent with any $F(\mathfrak{M}^k)$. We can also assume $G((\mathfrak{M}^m)_1)$ has been defined (see Figure \ref{c}-\ref{cc2}). If $G((\mathfrak{M}^m)_1)=\sigma(\mathfrak{M}^k)$, for some $\mathfrak{M}^k$, then $\mathfrak{M}^k$ is a model over $\mathfrak L_2$, or over $\mathfrak F_2$,  and we can take $G(\mathfrak{M}^m)=\sigma(\mathfrak N^k)$ where $\mathfrak N^k$ is an extension of  $\mathfrak{M}^k$ with any $k$-valuation $\mathfrak{f}^k_2$ (the extended model must be an intuitionistic model). Then $\mathfrak N^k$ is a model over $\mathfrak{L}_3$, or over $+\mathfrak{F}_2$, as required. The only  problem  rises  if $G((\mathfrak{M}^m)_1)$ is equivalent with some $\sigma(\mathfrak{M}^k)$ but it is not  any $\sigma$-model. This can happen if $(\mathfrak{M}^m)_1$ is a model over $\mathfrak{L}_2$ and $\mathfrak{M}^k$ over $\mathfrak{F}_2$. Then extending $\mathfrak{M}^k$ with $\mathfrak{f}^k_2$ we get a model $\mathfrak{N}^k$ over $+\mathfrak{F}_2$ such that $\sigma(\mathfrak{N}^k)$ is equivalent with an extension of $G((\mathfrak{M}^m)_1)$ to a model over $\mathfrak L_3$;  and the model over $\mathfrak L_3$ is $G(\mathfrak{M}^m)$.
\end{proof}
\subsection{Projective Formulas in Locally Tabular Logics.}\label{PFLTL}
Projective formulas  are useful in the area of unification. By Ghilardi \cite{Ghi2}, any consistent Rasiowa-Harrop formula (e.g. $\neg B$) is projective in {\sf INT} and, consequently, in any intermediate logic, see Lemma \ref{mon}.  One can easily produce other examples of projective formulas. Although  conjunctions of projective formulas may be not projective, for instance $(x_1\rightarrow x_2\lor x_3)\land x_1$ is not, we have
\begin{lemma}\label{niu1} If $A=\bigwedge_{i=1}^s(B_i\leftrightarrow z_{i})$ for some distinct variables $z_1,\dots,z_s$ which do not occur in the  formulas $B_1,\dots,B_s$, where $s\geq 1$, then $A$ is projective in {\sf INT}.\end{lemma}
\begin{proof} Note that $\varepsilon\colon z_1\slash B_1\cdots z_s\slash B_s$ is a projective unifier for $A$.\end{proof}
For locally tabular logics, we can show
\begin{theorem}\label{n5}
Let $\sigma\colon\{x_1,\dots,x_n\}\to \mathsf{Fm}^k$  and  $\varepsilon\colon\{x_1,\dots,x_n\}\to \mathsf{Fm}^n$, for some $k,n\geq 0$. Then $\varepsilon$ is a projective unifier for $A_\sigma$ in the logic {\sf L} if and only if \\
(iv) for every $\mathfrak{M}^n\in\mathbf{M}^n$ there is a $\mathfrak{M}^k\in\mathbf{M}^k$ such that $\varepsilon(\mathfrak{M}^n) \thicksim \sigma(\mathfrak{M}^k)$;\\
(v) $\varepsilon(\sigma(\mathfrak{M}^k))\thicksim\sigma(\mathfrak{M}^k)$, for every  $\mathfrak{M}^k\in \mathbf{M}^k$, see Figure \ref{p1}.
\end{theorem}
\begin{figure}
\unitlength1cm
\begin{picture}(5,2)
\thicklines

\put(8,2){\vector(0,-1){1.9}}

\put(8,2){\vector(-1,-1){1.9}}
\put(8,0){\vector(-1,0){1.9}}
\put(8,0){\circle{0.1}}
\put(6,0){\circle{0.1}}
\put(8,2){\circle{0.1}}
\put(8.3,2){\mbox{$\mathbf{M}^k\!\slash\!\!\thicksim$}}
\put(5,0){\mbox{$\mathbf{M}^n\!\slash\!\!\thicksim$}}
\put(8.3,0){\mbox{$\mathbf{M}^n\!\slash\!\!\thicksim$}}
\put(8.1,1){\mbox{$\sigma$}}
\put(7.1,0.1){\mbox{$\varepsilon$}}
\put(6.7,1.2){\mbox{$\sigma$}}

\end{picture}
\caption{Projective Unifiers}\label{p1}
\end{figure}
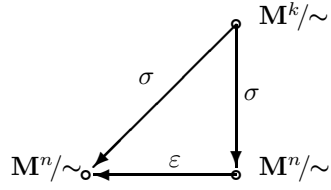
\begin{proof} Assume (iv) and (v). Then $\varepsilon$ is a unifier for $A_\sigma$,  by Corollary \ref{n3i}. We need to show $A_\sigma\vdash\varepsilon(x_i)\leftrightarrow x_i$, for each $i$. Let $\mathfrak{M}^k$ be any $k$-model. Then
$\varepsilon(\sigma(\mathfrak{M}^k))\thicksim\sigma(\mathfrak{M}^k)$ by (v). Hence, see Lemma \ref{sigma0}, we get
$$\sigma(\mathfrak{M}^k)\Vdash\varepsilon(x_i) \ \Leftrightarrow \ \varepsilon(\sigma(\mathfrak{M}^k))\Vdash x_i \ \Leftrightarrow \
\sigma(\mathfrak{M}^k)\Vdash x_i.$$
By Lemma \ref{sigma0}, $\sigma$ is a unifier for $\varepsilon(x_i)\leftrightarrow x_i$ and , by Corollary \ref{in2}, $A_\sigma\vdash\varepsilon(x_i)\leftrightarrow x_i$.

If $\varepsilon$ is a projective unifier for $A_\sigma$, we get (iv) and (v) using Lemma \ref{n1i} and  \ref{proj}.\end{proof}
\begin{corollary}\label{n7} A substitution $\varepsilon\colon\{x_1,\dots,x_n\}\to \mathsf{Fm}^n$ is a projective unifier for $A_\varepsilon$ in  {\sf L} iff $\varepsilon\circ\varepsilon=_\mathsf{L}\varepsilon$.\\
\end{corollary}
\begin{corollary}\label{n8} If $\varepsilon\colon\{x_1,\dots,x_n\}\to \mathsf{Fm}^n$ is an {\sf L}-projective unifier for a formula $A\in \mathsf{Fm^n}$,  then $A=_{\sf L}\bigwedge_{i=1}^n\bigl(x_i\leftrightarrow\varepsilon(x_i)\bigr).$\end{corollary}

Suppose that  $m=n$ and  $F=H_\sigma$  in Theorem \ref{main}. Then, by (iv) and (v),  $G$ would generate a retraction from $\mathbf{M}^n\!\slash\!\!\thicksim$ \ {onto} \ $\sigma(\mathbf{M}^k)\slash\!\!\thicksim$. This simplified version of Theorem \ref{main} characterizes  logics with projective approximation (see Theorem \ref{praprox}).
\begin{theorem}\label{retraction} The logic  {\sf L} has projective approximation iff  for every $\sigma\colon \{x_1,\dots,x_n\}\to \mathsf{Fm}^k$, where $n,k\geq 0$,  there is a mapping $G:\mathbf{M}^n\to\mathbf{M}^n$ such that:\\
(i) $G$ preserves the frame  of any $n$-model $\mathfrak{M}^n=(W,R,w_0,V^n)$;\\
(ii) $G((\mathfrak{M}^n)_w)\thicksim(G(\mathfrak{M}^n))_w, \mbox{ for every } \quad w\in W; $\\
(iii) $\mathfrak{N}^n\thicksim \mathfrak{M}^n\quad\Rightarrow \quad G(\mathfrak{N}^n)\thicksim G(\mathfrak{M}^n), \quad \mbox{ for every } \quad \mathfrak{M}^n,\mathfrak{N}^n\in \mathbf{M}^n$;\\
(iv) for every $n$-model $\mathfrak{M}^n$, there is a $k$-model $\mathfrak{M}^k$ such that $G(\mathfrak{M}^n)\thicksim\sigma(\mathfrak{M}^k)$;\\
(v) for every $k$-model $\mathfrak{M}^k$, we have $G(\sigma(\mathfrak{M}^k))\thicksim\sigma(\mathfrak{M}^k)$.
\end{theorem}
\begin{proof} Suppose that  {\sf L} has projective approximation and let $\sigma\colon \{x_1,\dots,x_n\}\to \mathsf{Fm}^k$, for  $k,n\geq 0$. By Corollary \ref{in2}, $\sigma$ is a unifier for $A_\sigma\in\mathsf{Fm}^n$. Let $B\in \Pi(A_\sigma)$ be a projective formula in $\mathsf{Fm}^n$ such that $B\vdash_\mathsf{L}A_\sigma$ and $\vdash_\mathsf{L}\sigma(B)$. If $\varepsilon\colon \{x_1,\dots,x_n\}\to \mathsf{Fm}^n$ is a projective unifier for $B$, then $\sigma\circ\varepsilon=_\mathsf{L}\sigma$ and $\vdash_\mathsf{L}\varepsilon(A_\sigma)$, by Lemma \ref{proj}. Thus, if one takes $G=H_\varepsilon$, one easily shows (i)-(v) using Lemma \ref{n1i} and Corollary \ref{n4i}.

Assume (i)-(v) and let $\sigma\colon\{x_1,\dots,x_n\}\to \mathsf{Fm}^k$ be a  unifier for a formula $A\in\mathsf{Fm}^n$. By Lemma \ref{nsigmai}, we have $G\thicksim H_\varepsilon$ for some $\varepsilon\colon\{x_1,\dots,x_n\}\to \mathsf{Fm}^n$. Then, by Lemma \ref{n5}, $\varepsilon$ is a projective unifier for $A_\sigma$. Thus, we have $A_\sigma$  projective, and $A_\sigma\vdash A$, and $A_\sigma\in\mathsf{Fm}^n$, and $\sigma$ is a unifier for $A_\sigma$. For each $A$ one can define $\Pi(A)$ so that to show projective approximation in  {\sf L}.
\end{proof}
In  \cite{dkw}, using the above Theorem, we have shown that any logic determined by frames of the depth $\leq 2$  has projective approximation. Thus, {\sf L}({$\mathfrak{F}_m$}) for any $m$-fork  $\mathfrak{F}_m$ (see Figure \ref{FRF}) has finitary unification. Now we extend this result to the logic determined by all frames from $\mathbf H_{pr}$ (that is frames $\mathfrak L_d+\mathfrak F_m$ , where $m,d\geq 0$,  see  Figure \ref{hpa}).
The logic {\sf L}($\mathbf H_{pa}$) extends {\sf PWL}, see \cite{Esakia}, hence it is locally tabular.
\begin{theorem}\label{lmk} For any $\mathbf F\subseteq \mathbf H_{pr}$,  the logic {\sf L}($\mathbf F$)   has projective approximation.
\end{theorem}
\begin{proof}
Let $\mathbf F\subseteq \mathbf H_{pr}$. Since $\mathbf H_{pr}$ is closed on p-morphic images and generated subframes, we can assume {\it sm}({\bf F})={\bf F}; see Lemma \ref{lf8}. For any $\sigma\colon\{x_1,\dots,x_n\}\to \mathsf{Fm}^k$, where $k\geq 0$, we need $G\colon\mathbf{M}^n\to\mathbf{M}^n$ fulfilling (i)-(v).

 A. If $\mathfrak{M}^n=(W,R,w_0,V^n)\in\mathbf{M}^n$ is equivalent with a $\sigma$-model (i.e. $\mathfrak{M}^n\thicksim\sigma(\mathfrak{M}^k)$, for some $\mathfrak{M}^k$),  we take
      $G(\mathfrak{M}^n)=\mathfrak{M}^n$. Then
(i),(iv) and (v) are obvious.

\noindent (ii) If $\mathfrak{M}^n$ is equivalent with a $\sigma$-model, then so is $(\mathfrak{M}^n)_w$, for any  $w\in W$ (see Corollary \ref{lf4i} and Lemma \ref{sigmai}(iii), and hence $G((\mathfrak{M}^n)_w)=(\mathfrak{M}^n)_w=(G(\mathfrak{M}^n))_w$.

\noindent (iii) If $\mathfrak{M}^n$ is equivalent with a $\sigma$-model and $\mathfrak{M}^n\thicksim\mathfrak{N}^n$, then $\mathfrak{N}^n$ is equivalent with the same $\sigma$-model and hence $G(\mathfrak{M}^n)=\mathfrak{M}^n\thicksim\mathfrak{N}^n=G(\mathfrak{N}^n)$.

 B. Suppose that $\mathfrak{M}^n$ is not equivalent with any $\sigma$-model. Assume $\mathfrak{M}^n$
is p-irreducible and define $G(\mathfrak{M}^n)$ inductively with respect to the depth of $\mathfrak{M}^n$. We could not bother about (v) as it is irrelevant, our inductive approach secure (ii) and (iii), see Theorem \ref{nsi}. Our task is to preserve (i) and (iv), that is $G(\mathfrak{M}^n)$ should be defined over the  frame of $\mathfrak{M}^n$ and $G(\mathfrak{M}^n)$ should be equivalent with some $\sigma$-model.

 B1. Let $\mathfrak{M}^n$ be an $n$-model over  1-element chain $\mathfrak L_1$, that is
$\mathfrak{M}^n=(\mathfrak L_1,\mathfrak{f}^n)$ where $\mathfrak{f}^n=i_1\dots i_n$ with $i_j\in\{0,1\}$. All $n$-models over $\mathfrak L_1$ are p-irreducible and we have (see \cite{dkw}):
{ any  $n$-model over   $\mathfrak L_1$  equivalent with some $\sigma$-model is a $\sigma$-model.}
We can take any $k$-model  $\mathfrak{M}^k=(\mathfrak L_1,\mathfrak{f}^k)$  over   $\mathfrak L_1$ and define $G(\mathfrak{M}^n)=\sigma(\mathfrak{M}^k).$

B2. Let  $\mathfrak{M}^n$ be an $n$-model over  $m$-fork $\mathfrak F_m$, see Figure \ref{mfork}, for $m\geq 1$;
\begin{figure}[H]
\unitlength1cm
\begin{picture}(3,1.1)
\thicklines
\put(4,1){\circle{0.1}}
\put(7,1){\circle{0.1}}
\put(6,1){\circle{0.1}}
\put(8,1){\circle{0.1}}
\put(5,1){\circle{0.1}}
\put(6,0){\circle{0.1}}
\put(7.3,1.2){\mbox{$\cdots$}}
\put(6.4,0){\mbox{$m$}}
\put(4,1.2){\mbox{$0$}}
\put(5,1.2){\mbox{$1$}}
\put(6,1.2){\mbox{$2$}}
\put(8,1.2){\mbox{$(m-1)$}}
\put(7,1.2){\mbox{$3$}}
\put(6,0){\vector(1,1){0.9}}
\put(6,0){\vector(-1,1){0.9}}
\put(6,0){\vector(0,1){0.9}}
\put(6,0){\vector(2,1){1.9}}
\put(6,0){\vector(-2,1){1.9}}
\end{picture}
\caption{$m$-Fork}\label{mfork}
\end{figure}
\noindent By our inductive approach,   $G((\mathfrak{M}^n)_i)=(\mathfrak L_1,\mathfrak{g}^n_i)=\sigma(\mathfrak L_1,\mathfrak{f}^k_i)$, for some $\mathfrak{g}^n_i$ and $\mathfrak{f}^k_i$, if $i<m$.   We  define $\mathfrak{M}^k$, see Figure \ref{abc},  and put
$G(\mathfrak{M}^n)=\sigma(\mathfrak{M}^k).$
\begin{figure}[H]
\unitlength1cm
\begin{picture}(3,1.5)
\thicklines

\put(2,1){\circle{0.1}}
\put(1,1){\circle{0.1}}
\put(3,1){\circle{0.1}}
\put(0,1){\circle{0.1}}
\put(1,0){\circle{0.1}}
\put(2.4,1.2){\mbox{$\cdots$}}
\put(1.6,0){\mbox{$\mathfrak{f}^n_m$}}
\put(0,1.2){\mbox{$\mathfrak{f}^n_0$}}
\put(1,1.2){\mbox{$\mathfrak{f}^n_1$}}
\put(3,1.2){\mbox{$\mathfrak{f}^n_{m-1}$}}
\put(2,1.2){\mbox{$\mathfrak{f}^n_2$}}
\put(1,0){\vector(1,1){0.9}}
\put(1,0){\vector(-1,1){0.9}}
\put(1,0){\vector(0,1){0.9}}
\put(1,0){\vector(2,1){1.9}}
\put(3,0.5){\vector(1,0){1.9}}
\put(3.8,0){\mbox{$G$}}
\put(0,0){\mbox{$\mathfrak M^n=$}}

\put(7,1){\circle{0.1}}
\put(6,1){\circle{0.1}}
\put(8,1){\circle{0.1}}
\put(5,1){\circle{0.1}}
\put(6,0){\circle{0.1}}
\put(7.3,1.2){\mbox{$\cdots$}}
\put(6.6,0){\mbox{$?$}}
\put(4.6,1.2){\mbox{$\mathfrak{g}^n_0$}}
\put(5.6,1.2){\mbox{$\mathfrak{g}^n_1$}}
\put(7.8,1.2){\mbox{$\mathfrak{g}^n_{m-1}$}}
\put(6.6,1.2){\mbox{$\mathfrak{g}^n_2$}}
\put(6,0){\vector(1,1){0.9}}
\put(6,0){\vector(-1,1){0.9}}
\put(6,0){\vector(0,1){0.9}}
\put(6,0){\vector(2,1){1.9}}
\put(9,0.5){\vector(-1,0){1.5}}
\put(8.2,0.2){\mbox{$\sigma$}}

\put(11,1){\circle{0.1}}
\put(10,1){\circle{0.1}}
\put(12,1){\circle{0.1}}
\put(9,1){\circle{0.1}}
\put(10,0){\circle{0.1}}
\put(11.4,1.2){\mbox{$\cdots$}}
\put(10.6,0){\mbox{$0\cdots0$}}
\put(9.2,1.2){\mbox{$\mathfrak{f}^k_0$}}
\put(10,1.2){\mbox{$\mathfrak{f}^k_1$}}
\put(12,1.2){\mbox{$\mathfrak{f}^k_{m-1}$}}
\put(11,1.2){\mbox{$\mathfrak{f}^k_2$}}
\put(10,0){\vector(1,1){0.9}}
\put(10,0){\vector(-1,1){0.9}}
\put(10,0){\vector(0,1){0.9}}
\put(10,0){\vector(2,1){1.9}}
\put(11.8,0.2){\mbox{$=\mathfrak M^k$}}
\end{picture}
\caption{}\label{abc}
\end{figure}
\noindent We should be more careful about defining the mapping $G$ on isomorphic $n$-models. Our mapping should be factorized by $\equiv$ and hence we should get the factorized mapping $G\colon\mathbf{M}^n_{ir}\slash\!\!\equiv \to\mathbf{M}^n\slash\!\!\equiv$. Thus, from any $\equiv$-class, we  take only one $\mathfrak{M}^n$, define as above $G(\mathfrak{M}^n)$ and then extend $G$, on other members of the equivalence class, taking as the values  for $G$ the appropriate isomorphic copy of $G(\mathfrak{M}^n)$.

B3. Eventually, suppose that the frame of $\mathfrak{M}^n$ is an $m$-fork extended with an $d$-element chain 'leg`, where $m,d\geq 0$. Our subsequent definition of $G$ is inductive with respect to $d$. So, we assume that $d\geq 1$ and $G((\mathfrak{M}^n)_{m+d-1})$ has already been defined, see Figure \ref{abcd}. We define
$G(\mathfrak{M}^n)$ extending  $G((\mathfrak{M}^n)_{m+d-1})$ with $\mathfrak{g}^n_{m+d}=\mathfrak{g}^n_{m+d-1}$. So, we get $G(\mathfrak{M}^n)\thicksim G((\mathfrak{M}^n)_{m+d-1})$ which guarantees the condition (iv).
\begin{figure}[H]
\unitlength1cm
\begin{picture}(3,3.5)
\thicklines

\put(2,3){\circle{0.1}}
\put(1,3){\circle{0.1}}
\put(3,3){\circle{0.1}}
\put(0,3){\circle{0.1}}
\put(1,2){\circle{0.1}}
\put(2.4,3.2){\mbox{$\cdots$}}
\put(1.6,2){\mbox{$\mathfrak{f}^n_m$}}
\put(0,3.2){\mbox{$\mathfrak{f}^n_0$}}
\put(1,3.2){\mbox{$\mathfrak{f}^n_1$}}
\put(3,3.2){\mbox{$\mathfrak{f}^n_{m-1}$}}
\put(2,3.2){\mbox{$\mathfrak{f}^n_2$}}
\put(1,2){\vector(1,1){0.9}}
\put(1,2){\vector(-1,1){0.9}}
\put(1,2){\vector(0,1){0.9}}
\put(1,2){\vector(2,1){1.9}}
\put(0,1){\mbox{$\mathfrak M^n=$}}
\put(1,1.5){\circle{0.1}}
\put(1,1){\circle{0.1}}
\put(1,0.5){\circle{0.1}}
\put(1,0){\circle{0.1}}
\put(1.3,1.5){\mbox{$\mathfrak{f}^n_{m+1}$}}
\put(1.3,0.5){\mbox{$\mathfrak{f}^n_{m+d-1}$}}
\put(1.3,0){\mbox{$\mathfrak{f}^n_{m+d}$}}
\put(1,1.6){\vector(0,1){0.4}}
\put(1,0.1){\vector(0,1){0.4}}

\put(3.1,1){\mbox{$G((\mathfrak M^n)_{m+d-1})=$}}
\put(7,31){\circle{0.1}}
\put(6,3){\circle{0.1}}
\put(8,3){\circle{0.1}}
\put(5,3){\circle{0.1}}
\put(6,2){\circle{0.1}}
\put(7.3,3.2){\mbox{$\cdots$}}
\put(6.6,2){\mbox{$\mbox{$\mathfrak{g}^n_{m}$}$}}
\put(4.6,3.2){\mbox{$\mathfrak{g}^n_0$}}
\put(5.6,3.2){\mbox{$\mathfrak{g}^n_1$}}
\put(7.8,3.2){\mbox{$\mathfrak{g}^n_{m-1}$}}
\put(6.6,3.2){\mbox{$\mathfrak{g}^n_2$}}
\put(6,2){\vector(1,1){0.9}}
\put(6,2){\vector(-1,1){0.9}}
\put(6,2){\vector(0,1){0.9}}
\put(6,2){\vector(2,1){1.9}}
\put(6,1.5){\circle{0.1}}
\put(6,1){\circle{0.1}}
\put(6,0.5){\circle{0.1}}
\put(6,1.6){\vector(0,1){0.4}}
\put(6.3,1.5){\mbox{$\mathfrak{g}^n_{m+1}$}}
\put(6.3,0.5){\mbox{$\mathfrak{g}^n_{m+d-1}$}}

\put(11,3){\circle{0.1}}
\put(10,3){\circle{0.1}}
\put(12,3){\circle{0.1}}
\put(9,3){\circle{0.1}}
\put(10,2){\circle{0.1}}
\put(11.4,3.2){\mbox{$\cdots$}}
\put(10.6,2){\mbox{$\mathfrak{g}^n_m$}}
\put(9.2,3.2){\mbox{$\mathfrak{g}^n_0$}}
\put(10,3.2){\mbox{$\mathfrak{g}^n_1$}}
\put(12,3.2){\mbox{$\mathfrak{g}^n_{m-1}$}}
\put(11,3.2){\mbox{$\mathfrak{g}^n_2$}}
\put(10,2){\vector(1,1){0.9}}
\put(10,2){\vector(-1,1){0.9}}
\put(10,2){\vector(0,1){0.9}}
\put(10,2){\vector(2,1){1.9}}
\put(10,1.5){\circle{0.1}}
\put(10,1){\circle{0.1}}
\put(10,0.5){\circle{0.1}}
\put(10,0){\circle{0.1}}
\put(10,1.6){\vector(0,1){0.4}}
\put(10,0.1){\vector(0,1){0.4}}
\put(8.3,1){\mbox{$G(\mathfrak M^n)=$}}
\put(10.3,1.5){\mbox{$\mathfrak{g}^n_{m+1}$}}
\put(10.3,0.5){\mbox{$\mathfrak{g}^n_{m+d-1}$}}
\put(10.3,0){\mbox{$\mathfrak{g}^n_{m+d}=\mathfrak{g}^n_{m+d-1}$}}
\end{picture}
\caption{}\label{abcd}
\end{figure}
So, we have $G\colon\mathbf{M}^n_{ir} \to\mathbf{M}^n$. Using Theorem \ref{nsi}, we extend the mapping to $G\colon\mathbf{M}^n \to\mathbf{M}^n$ preserving the conditions (i)-(iii), see Lemma \ref{sigmai}. It is an easy task to check that the conditions (iv) and (v) are also preserved by the extension.
\end{proof}

Theorems \ref{main} and \ref{retraction} can also be used to  show that certain locally tabular logics neither have finitary\slash unitary unification, nor projective approximation. Many examples follow in the next section. Here we only give
\begin{theorem}\label{L8i} The logics ${\mathsf L}(\mathfrak G_2)$ and ${\mathsf L}(\mathfrak G_2)\cap{\mathsf L}(\mathfrak C_{5})$ (see Figure \ref{GF} and   \ref{TF}) have finitary unification but they do not have projective approximation.
\end{theorem}
\begin{proof} Let us prove ${\mathsf L}(\mathfrak G_2)\cap{\mathsf L}(\mathfrak C_{5})$ has finitary unification. We  extend the family {\bf F} (in our proof of  Theorem \ref{c5}) with the frames $\mathfrak G_2$ and $\mathfrak G_3$.
The definition of $F$, see Figure \ref{c}-\ref{cc2}, needs the following two (additional) clauses:
\begin{figure}[H]
\unitlength1cm
\begin{picture}(3,2)
\thicklines

\put(1.5,0){\vector(-1,1){0.9}}
\put(1.5,0){\vector(1,1){0.9}}
\put(2.5,1){\vector(0,1){0.9}}
\put(0.5,1){\circle{0.1}}
\put(2.5,2){\circle{0.1}}
\put(1.5,0){\circle{0.1}}
\put(2.5,1){\circle{0.1}}
\put(1.8,0){\mbox{$\mathfrak{f}^k_2$}}
\put(0.6,1.1){\mbox{$\mathfrak{f}^k_{0'}$}}
\put(2.6,2){\mbox{$\mathfrak{f}^k_0$}}
\put(2.6,1.1){\mbox{$\mathfrak{f}^k_1$}}
\put(0,1){\mbox{$\sigma\Bigl($}}
\put(3,1){\mbox{$\Bigr)\ = $}}
\put(0.9,1.8){\mbox{$\mathfrak{M}^k=$}}

\put(5,0){\vector(-1,1){0.9}}
\put(5,0){\vector(1,1){0.9}}
\put(6,1){\vector(0,1){0.9}}
\put(4,1){\circle{0.1}}
\put(6,2){\circle{0.1}}
\put(5,0){\circle{0.1}}
\put(6,1){\circle{0.1}}
\put(5.3,0){\mbox{$\mathfrak{f}^n_2$}}
\put(4.1,1.1){\mbox{$\mathfrak{f}^n_{0'}$}}
\put(6.1,2){\mbox{$\mathfrak{f}^n_0$}}
\put(6.3,1.1){\mbox{$\mathfrak{f}^n_1$}}

\put(8,1){\mbox{$F(\mathfrak{M}^k)\ = $}}
\put(11,0){\vector(-1,1){0.9}}
\put(11,0){\vector(1,1){0.9}}
\put(12,1){\vector(0,1){0.9}}
\put(10,1){\circle{0.1}}
\put(12,2){\circle{0.1}}
\put(11,0){\circle{0.1}}
\put(12,1){\circle{0.1}}

\put(10.5,1){\mbox{$\mathfrak{f}^n_{1}0\cdots01$}}
\put(11.3,0){\mbox{$\mathfrak{f}^n_20\cdots0$}}
\put(9.7,1.4){\mbox{$code(\mathfrak{f}^k_{0'})$}}
\put(11.4,2.1){\mbox{$code(\mathfrak{f}^k_{0})$}}

\end{picture}\\

\unitlength1cm
\begin{picture}(5,2.2)
\thicklines
\put(1.5,0){\vector(-1,1){0.9}}
\put(1.5,0){\vector(1,1){0.9}}
\put(2.5,1){\vector(0,1){0.9}}
\put(0.5,1){\circle{0.1}}
\put(2.5,2){\circle{0.1}}
\put(1.5,0){\circle{0.1}}
\put(2.5,1){\circle{0.1}}
\put(1.8,0){\mbox{$\mathfrak{f}^k_2$}}
\put(0.6,1.1){\mbox{$\mathfrak{f}^k_{0''}$}}
\put(2.6,2){\mbox{$\mathfrak{f}^k_0$}}
\put(2.6,1.1){\mbox{$\mathfrak{f}^k_1$}}
\put(0,1){\mbox{$\sigma\Bigl($}}
\put(3,1){\mbox{$\Bigr)\ = $}}
\put(2.5,1){\vector(-1,1){0.9}}
\put(1.5,2){\circle{0.1}}
\put(1,1.8){\mbox{$\mathfrak{f}^k_{0'}$}}

\put(5,0){\vector(-1,1){0.9}}
\put(5,0){\vector(1,1){0.9}}
\put(6,1){\vector(0,1){0.9}}
\put(4,1){\circle{0.1}}
\put(6,2){\circle{0.1}}
\put(5,0){\circle{0.1}}
\put(6,1){\circle{0.1}}
\put(5.3,0){\mbox{$\mathfrak{f}^n_2$}}
\put(4.1,1.1){\mbox{$\mathfrak{f}^n_{0''}$}}
\put(6.1,2){\mbox{$\mathfrak{f}^n_0$}}
\put(6.3,1.1){\mbox{$\mathfrak{f}^n_1$}}
\put(6,1){\vector(-1,1){0.9}}
\put(5,2){\circle{0.1}}
\put(4.5,1.8){\mbox{$\mathfrak{f}^n_{0'}$}}

\put(8,1){\mbox{$F(\mathfrak{M}^k)\ = $}}
\put(11,0){\vector(-1,1){0.9}}
\put(11,0){\vector(1,1){0.9}}
\put(12,1){\vector(0,1){0.9}}
\put(10,1){\circle{0.1}}
\put(12,2){\circle{0.1}}
\put(11,0){\circle{0.1}}
\put(12,1){\circle{0.1}}

\put(9.5,2){\mbox{$code(\mathfrak{f}^k_{0'})$}}
\put(10.5,1){\mbox{$\mathfrak{f}^n_{1}0\cdots01$}}
\put(11.3,0){\mbox{$\mathfrak{f}^n_20\cdots0$}}
\put(9.7,1.4){\mbox{$code(\mathfrak{f}^k_{0''})$}}
\put(11.4,2.1){\mbox{$code(\mathfrak{f}^k_{0})$}}

\put(12,1){\vector(-1,1){0.9}}
\put(11,2){\circle{0.1}}

\end{picture}\\
\caption{} \label{k}
\end{figure}
\noindent Then we need to extend the definition of $G(\mathfrak{M}^m)$ with:

\underline{$\mathfrak{G}_2$- and $\mathfrak{G}_2$-models.}
 We only deal with the worst case. Suppose that  $\mathfrak{M}^m=(\mathfrak{G}_3,\{\mathfrak{f}^m_0,\mathfrak{f}^m_{0'},\mathfrak{f}^m_1,\mathfrak{f}^m_2\})$ is a p-irreducible model over $\mathfrak{G}_3$ (see   Figure \ref{k}) non-equivalent with any $F(\mathfrak{M}^k)$, but $(\mathfrak{M}^m)_1\thicksim F(\mathfrak{M}^k)$, where $\mathfrak{M}^k$ is a model over $\mathfrak{F}_2$. Then $(\mathfrak{M}^m)_1\!\!\upharpoonright n = (\mathfrak{L}_2,\{\mathfrak{f}^n_0,\mathfrak{f}^n_1\})\thicksim\sigma(\mathfrak{M}^k)$ for some $\mathfrak{f}^n_0,\mathfrak{f}^n_1$  where $\thicksim$ cannot be replaced with $=$. Let $\mathfrak{M}^k=(\mathfrak{F}_2,\{\mathfrak{f}^k_0,\mathfrak{f}^k_{0'},\mathfrak{f}^k_1\})$ which means $\sigma(\mathfrak{L}_1,\mathfrak{f}^k_0)=\sigma(\mathfrak{L}_1,\mathfrak{f}^k_{0'})=(\mathfrak{L}_1,
 \mathfrak{f}^n_0)$ and let $G((\mathfrak{M}^m)_{0'})=\sigma(\mathfrak{L}_1,\mathfrak{g}^k_{0})=(\mathfrak{L}_1,\mathfrak{f}^n_{0'}), $ for some $\mathfrak{g}^k_{0},\mathfrak{f}^n_{0'}$. Then, for some $\mathfrak{f}^n_2$,  we have
 \begin{figure}[H]\label{pi}
\unitlength1cm
\begin{picture}(3,2)
\thicklines

\put(8.5,1){$\thicksim$}
\put(11,0){\vector(-1,1){0.9}}
\put(11,0){\vector(1,1){0.9}}
\put(12,1){\vector(0,1){0.9}}
\put(10,1){\circle{0.1}}
\put(12,2){\circle{0.1}}
\put(11,0){\circle{0.1}}
\put(12,1){\circle{0.1}}
\put(11.3,0){\mbox{$\mathfrak{f}^n_2$}}
\put(9.7,1.2){\mbox{$\mathfrak{f}^n_{0'}$}}
\put(12.1,2){\mbox{$\mathfrak{f}^n_0$}}
\put(12.3,1.1){\mbox{$\mathfrak{f}^n_1$}}

\put(6,0){\vector(-1,1){0.9}}
\put(6,0){\vector(1,1){0.9}}
\put(7,1){\vector(0,1){0.9}}
\put(7,1){\vector(-1,1){0.9}}
\put(5,1){\circle{0.1}}
\put(7,2){\circle{0.1}}
\put(6,2){\circle{0.1}}
\put(6,0){\circle{0.1}}
\put(7,1){\circle{0.1}}
\put(6.3,0){\mbox{$\mathfrak{f}^n_2$}}
\put(4.8,1.2){\mbox{$\mathfrak{f}^n_{0'}$}}
\put(7.1,2.1){\mbox{$\mathfrak{f}^n_0$}}
\put(7.3,1){\mbox{$\mathfrak{f}^n_1$}}
\put(6.3,2.1){\mbox{$\mathfrak{f}^n_{0}$}}

\put(0,1){$\sigma\Bigl($}
\put(2,0){\vector(-1,1){0.9}}
\put(2,0){\vector(1,1){0.9}}
\put(3,1){\vector(0,1){0.9}}
\put(3,1){\vector(-1,1){0.9}}
\put(1,1){\circle{0.1}}
\put(3,2){\circle{0.1}}
\put(2,2){\circle{0.1}}
\put(2,0){\circle{0.1}}
\put(3,1){\circle{0.1}}
\put(1.3,0){\mbox{$\mathfrak{f}^k_2$}}
\put(0.7,1.2){\mbox{$\mathfrak{g}^k_{0}$}}
\put(3.1,2.1){\mbox{$\mathfrak{f}^k_0$}}
\put(3.3,1){\mbox{$\mathfrak{f}^k_1$}}
\put(2.3,2.1){\mbox{$\mathfrak{f}^k_{0'}$}}
\put(3.6,1){$\Bigr)\quad =$}
\end{picture}\\
\caption{} \label{M2}
\end{figure}
\noindent and the  $n$-model over $\mathfrak{G}_3$ may be  $G(\mathfrak{M}^m)$; it is equivalent with a $\sigma$-model over $\mathfrak{G}_2$.\footnote{The  problem would rise if we tried in the same way show ${\mathsf L}(\mathfrak G_3)$ is finitary, we have no $\sigma$-model over $\mathfrak{G}_3$ equivalent with $G(\mathfrak{M}^m)$. So, ${\mathsf L}(\mathfrak G_3)$ is nullary and the above may be seen as  Ghilardi's Theorem 9, p.112, \cite{Ghi5}.}
The rest of the proof remains the same.

In the same way one shows ${\mathsf L}(\mathfrak G_2)$ has finitary unification; from the proof  that ${\mathsf L}(\mathfrak G_2)\cap{\mathsf L}(\mathfrak C_{5})$ has finitary unification, it suffices to remove $\mathfrak C_{5},\mathfrak R_{2},\mathfrak Y_{2}$ and $\mathfrak Y_{3}$. Then we can  even simplify our reasoning and delate $\mathfrak{g}^n_1\mathfrak{h}^n_1\dots\mathfrak{g}^n_{2^n}\mathfrak{h}^n_{2^n}$
from $code(\mathfrak{f}^k)$. Thus,  we can take $m=n+2$ and define $F(\mathfrak{M}^k)$  using $m$-valuations $\mathfrak{f}^n11$, $\mathfrak{f}^n01$, or $\mathfrak{f}^n00$. The suffices $11$, $01$ and $00$ are necessary.\footnote{Using them we avoid problems that could rise  p-morphisms mentioned in Example \ref{Kost}, the first counterexample. For the second type of p-morphisms, mentioned in the Example, we have no such remedia. The whole proof is to be shown that for these particular frames we can deal somehow with p-morphims which collapse $\mathfrak F_2$ onto $\mathfrak L_2$.}

We need to show that ${\mathsf L}(\mathfrak G_2)$ does not have projective approximation. Consider the substitution
$\sigma\colon\{x_1\}\to \mathsf{Fm}^2$ such that
$\sigma(x_1)=\neg\neg x_1\land \bigl(x_2\lor(x_2\to x_1\lor\neg x_1)\bigr).$ Any $\sigma$-model is equivalent with one of the  1-models:
\begin{figure}[H]
\unitlength1cm
\begin{picture}(3,1.5)
\thicklines

\put(0,0){\circle{0.1}}
\put(0.3,0){\mbox{$1$}}

\put(1.5,0){\circle{0.1}}
\put(1.8,0){\mbox{$0$}}

\put(12,0){\vector(0,1){0.9}}
\put(12,1){\circle{0.1}}
\put(12,0){\circle{0.1}}
\put(12.3,0){\mbox{$0$}}
\put(12.3,1){\mbox{$1$}}

\put(3,1){\circle{0.1}}
\put(4,0){\circle{0.1}}
\put(5,1){\circle{0.1}}
\put(4.3,0){\mbox{$0$}}
\put(3.3,1){\mbox{$0$}}
\put(5.3,1){\mbox{$1$}}
\put(4,0){\vector(1,1){0.9}}
\put(4,0){\vector(-1,1){0.9}}

\put(8,0){\vector(-1,1){0.9}}
\put(8,0){\vector(1,1){0.9}}
\put(9,1){\vector(0,1){0.9}}
\put(9,1){\vector(-1,1){0.9}}
\put(7,1){\circle{0.1}}
\put(9,2){\circle{0.1}}
\put(8,2){\circle{0.1}}
\put(8,0){\circle{0.1}}
\put(9,1){\circle{0.1}}
\put(8.3,0){\mbox{$0$}}
\put(6.7,1.1){\mbox{$1$}}
\put(9.1,1.9){\mbox{$1$}}
\put(9.4,1){\mbox{$1$}}
\put(8.3,1.9){\mbox{$1$}}
\put(10.5,0.5){\mbox{$\thicksim$}}
\end{picture}
\caption{}\label{tt}
\end{figure}
\noindent The above  model over $\mathfrak L_2$ is not a $\sigma$-model  but is equivalent with some $\sigma$-model over  $\mathfrak G_2$. Suppose ${\mathsf L}(\mathfrak G_2)$ has projective approximation and let $G\colon \mathbf M^1\to \mathbf M^1$ be the retraction  given by Theorem \ref{retraction}. We  have $G(\circ\ 0)=\circ\ 0$ and $G(\circ\ 1)=\circ\ 1$.
But if

\unitlength1cm
\begin{picture}(5,2.5)
\thicklines
\put(0,1.4){$G\bigl($}
\put(0.8,2){\circle{0.1}}
\put(0.8,1){\circle{0.1}}
\put(1.1,1){\mbox{$0$}}
\put(1.1,2){\mbox{$1$}}
\put(0.8,1){\vector(0,1){0.9}}
\put(1.5,1.4){$\bigr)=$}

\put(2.2,2){\circle{0.1}}
\put(2.2,1){\circle{0.1}}
\put(2.4,1){\mbox{$0$}}
\put(2.4,2){\mbox{$1$}}
\put(2.2,1){\vector(0,1){0.9}}
\put(2.7,1.4){, then}

\put(4,1.1){$G\Bigl($}

\put(6,0){\vector(-1,1){0.9}}
\put(6,0){\vector(1,1){0.9}}
\put(7,1){\vector(0,1){0.9}}
\put(7,1){\vector(-1,1){0.9}}
\put(5,1){\circle{0.1}}
\put(7,2){\circle{0.1}}
\put(6,2){\circle{0.1}}
\put(6,0){\circle{0.1}}
\put(7,1){\circle{0.1}}
\put(6.3,0){\mbox{$0$}}
\put(4.7,1.1){\mbox{$0$}}
\put(7.1,2.1){\mbox{$1$}}
\put(7.3,1){\mbox{$0$}}
\put(6.3,2.1){\mbox{$1$}}
\put(7.7,1.1){\mbox{$\Bigr)\ =$}}

\put(10,0){\vector(-1,1){0.9}}
\put(10,0){\vector(1,1){0.9}}
\put(11,1){\vector(0,1){0.9}}
\put(11,1){\vector(-1,1){0.9}}
\put(11,1){\circle{0.1}}
\put(11,2){\circle{0.1}}
\put(10,2){\circle{0.1}}
\put(10,0){\circle{0.1}}
\put(11,1){\circle{0.1}}
\put(9.5,0){\mbox{$0$}}
\put(9.5,1){\mbox{$0$}}
\put(11.1,2.1){\mbox{$1$}}
\put(11.3,1){\mbox{$0$}}
\put(10.3,2.1){\mbox{$1$}}

\end{picture}\\
which cannot happen as the last model is not equivalent to any model in
 Figure \ref{tt}. Thus, we get a contradiction.

  In the same way one shows ${\mathsf L}(\mathfrak G_2)\cap{\mathsf L}(\mathfrak C_{5})$ does not have projective approximation.
\end{proof}
We can also show  that  ${\mathsf L}(\mathfrak G_1)$ (see Figure \ref{GF}) has finitary unification; it suffices only to extend the above reasoning admitting  the frames $\mathfrak G_1,\mathfrak G_{3\mathfrak L_2}$ and $\mathfrak G_{3\mathfrak F_2}$ (see Figure \ref{MNU}).  But
 a detailed proof would take much time and place (and, above all, would be boring). We think that original Ghilardi's reasoning, see \cite{Ghi5}, is conscious and elegant and, therefore, we resign from any alternative proof.
\section{Unification Types.}\label{UT}

\subsection{Unitary Unification.}\label{UU}
By Theorem \ref{wpl}, the logic $\mathsf L(\mathfrak R_2)$ of rhombus (see figure \ref{FRF}) has unitary unification.
We have shown this directly, using Theorem \ref{main}, in \cite{dkw}. We cannot have $m=n$ in this case. In other words, unification in $\mathsf L(\mathfrak R_2)$ is unitary but it is not true that every unifiable formula in $n$-variables has a mgu in $n$-variables. A unifiable $A(x_1,\dots,x_n)$ has a mgu in $m$-variables where ($m$ is defined in the proof  and) $n<m$.

Indeed, let
 $A= x_1\lor  x_2\lor(\neg x_1\land\neg x_2)$ and $\varepsilon_1,\varepsilon_2,\varepsilon_3$ be the following  unifiers for $A$:
$$\varepsilon_1(x_1)=\top\ ,\ \varepsilon_1(x_2)= x_2\quad;\quad \varepsilon_2(x_1)= x_1\ ,\ \varepsilon_2(x_2)=\top\quad;\quad\varepsilon_3(x_1)=\varepsilon_3(x_2)=\bot.$$
Let $\varepsilon\colon\{x_1,x_2\}\to\mathsf{Fm}^2$ be a mgu for $A$ in any intermediate logic ${\mathsf L}\subseteq\mathsf L(\mathfrak R_2)$.  Then for some  $\alpha_i\colon\{x_1,x_2\}\to\mathsf{Fm}^2,$ where $i=1,2,3$, we have $\alpha_i\circ\varepsilon=_{\sf L}\varepsilon_i$, and hence the  diagram in Figure \ref{r1} commutes (up to $\thicksim$) where $\mathbf{M}^{2}=\mathbf{M}^{2}(sm(\mathfrak R_2))=\mathbf{M}^{2}(\{\mathfrak L_1,\mathfrak L_2,\mathfrak R_2\})$:
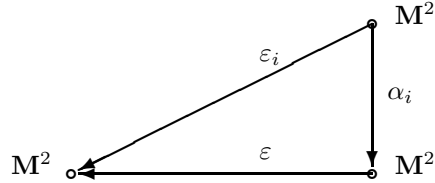
\begin{figure}[H]
\unitlength1cm
\begin{picture}(5,2)
\thicklines

\put(8,2){\vector(0,-1){1.9}}

\put(8,2){\vector(-2,-1){3.9}}
\put(8,0){\vector(-1,0){3.9}}
\put(8,0){\circle{0.1}}
\put(4,0){\circle{0.1}}
\put(8,2){\circle{0.1}}
\put(8.3,2){\mbox{$\mathbf{M}^{2}$}}
\put(3.2,0){\mbox{$\mathbf{M}^2$}}
\put(8.3,0){\mbox{$\mathbf{M}^{2}$}}
\put(8.2,1){\mbox{$\alpha_i$}}
\put(6.5,0.2){\mbox{$\varepsilon$}}
\put(6.5,1.5){\mbox{$\varepsilon_{i}$}}

\end{picture}
\caption{Commutative  Diagram}\label{r1}
\end{figure}
\noindent Consider all $2$-models over one-element frame $\mathfrak L_1$. They are $(\mathfrak L_1,00),(\mathfrak L_1,01)$, $(\mathfrak L_1,10)$ and $(\mathfrak L_1,11)$. Each of them is a $\varepsilon_i$-model for some $i$ and hence they are $\varepsilon-$models. The mapping
$\varepsilon\colon\mathbf{M}^2\to\mathbf{M}^2$ must be one-to-one on $\mathfrak L_1$-models.

Consider the following $2$-models over $\mathfrak L_2$ (a two-element chain):\\

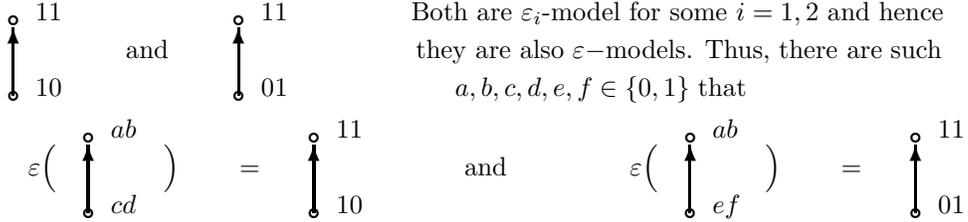
   \unitlength1cm
\begin{picture}(0,1.1)
\thicklines
\put(0,0){\circle{0.1}}
\put(0.3,1){\mbox{$11$}}
\put(0,0){\vector(0,1){0.9}}
\put(0.3,0){\mbox{$10$}}
\put(0,1){\circle{0.1}}
\put(1.5,0.5){\mbox{and}}
\put(3.3,1){\mbox{$11$ \qquad\qquad \ Both are $\varepsilon_i$-model for some $i=1,2$ and hence }}
\put(3,0){\circle{0.1}}
\put(3.6,0.5){\mbox{\qquad\qquad \ \ \ they are also $\varepsilon-$models. Thus, there are such}}
\put(3.3,0){\mbox{$01$ \qquad\qquad \qquad  $a,b,c,d,e,f \in \{0, 1\}$  that}}
\put(3,0){\vector(0,1){0.9}}
\put(3,1){\circle{0.1}}

\end{picture} \\

   \unitlength1cm
\begin{picture}(0,1.1)
\thicklines

\put(0.2,0.5){\mbox{$\varepsilon\Bigl($}}
\put(1,0){\circle{0.1}}
\put(1.3,1){\mbox{$ab$}}
\put(1,0){\vector(0,1){0.9}}
\put(1.3,0){\mbox{$cd$}}
\put(1,1){\circle{0.1}}
\put(2,0.5){\mbox{$\Bigr)\qquad=$}}
\put(4,0){\circle{0.1}}
\put(4.3,1){\mbox{$11$}}
\put(4,0){\vector(0,1){0.9}}
\put(4.3,0){\mbox{$10$}}
\put(4,1){\circle{0.1}}

\put(6,0.5){\mbox{and}}
\put(8.2,0.5){\mbox{$\varepsilon\Bigl($}}
\put(9,0){\circle{0.1}}
\put(9.3,1){\mbox{$ab$}}
\put(9,0){\vector(0,1){0.9}}
\put(9.3,0){\mbox{$ef$}}
\put(9,1){\circle{0.1}}
\put(10,0.5){\mbox{$\Bigr)\qquad=$}}
\put(12,0){\circle{0.1}}
\put(12.3,1){\mbox{$11$}}
\put(12,0){\vector(0,1){0.9}}
\put(12.3,0){\mbox{$01$}}
\put(12,1){\circle{0.1}}

\end{picture}\\

 \noindent
and $ab\not=cd\not=ef\not=ab$. If we consider the following $\mathfrak R_2$-model and its $\varepsilon$-image:\\
 \unitlength1cm
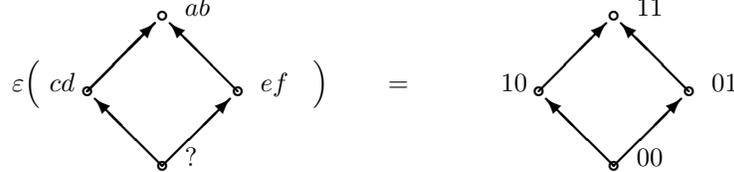
\begin{picture}(3,2.2)
\thicklines
\put(2,1){$\varepsilon\Bigl($}
\put(4,0){\vector(-1,1){0.9}}
\put(4,0){\vector(1,1){0.9}}
\put(5,1){\vector(-1,1){0.9}}
\put(3,1){\circle{0.1}}
\put(4,2){\circle{0.1}}
\put(4,0){\circle{0.1}}
\put(5,1){\circle{0.1}}
\put(4.3,0){\mbox{$?$}}
\put(2.5,1){\mbox{$cd$}}
\put(4.3,2){\mbox{$ab$}}
\put(5.3,1){\mbox{$ef$}}
\put(3,1){\line(1,1){0.9}}
\put(3,1){\vector(1,1){0.9}}
\put(6,1){$\Bigr)$}
\put(7,1){$=$}
\put(10,0){\vector(-1,1){0.9}}
\put(10,0){\vector(1,1){0.9}}
\put(11,1){\vector(-1,1){0.9}}
\put(9,1){\circle{0.1}}
\put(10,2){\circle{0.1}}
\put(10,0){\circle{0.1}}
\put(11,1){\circle{0.1}}
\put(10.3,0){\mbox{$00$}}
\put(8.5,1){\mbox{$10$}}
\put(10.3,2){\mbox{$11$}}
\put(11.3,1){\mbox{$01$}}
\put(9,1){\vector(1,1){0.9}}
\end{picture}\\
\noindent we conclude there is a $\varepsilon$-model which is not a model for $A$ (whatever   $?$ is). Thus, $\varepsilon$ is not a unifier for $A$, a contradiction.

No unifier $\varepsilon\colon\{x_1,x_2\}\to\mathsf{Fm}^2$ for $A$ is more general any  $\varepsilon_i$, for  $i=1,2,3$. Clearly, using filtering unifications, we can find a unifier for $A$ which is more general than any $\varepsilon_i$, we can even find in this way a mgu for $A$. But filtering unifications require additional variables and
thus, we get a mgu $\varepsilon\colon\{x_1,x_2\}\to\mathsf{Fm}^m$, where $m>2$.

\begin{theorem}\label{u1} If an intermediate logic ${\mathsf L}\subseteq\mathsf L(\mathfrak R_2)$ has unitary unification, then  there are  {\sf L}-unifiable formulas  such that all their mgu's introduce new variables.
\end{theorem}
We do not  assume here that {\sf L} is locally tabular.  Now, we can apply a particular splitting in the lattice of intermediate logics.
It is known (see Rautenberg \cite{Rautenberg}) that $(\mathsf L (\mathfrak R_2),\mathsf{LC})$ is a splitting pair for all extensions of $\mathsf{KC}$. It follows that for each {\sf L} extending $\mathsf{KC}$,  either $\mathsf L \supseteq  \mathsf{LC}$,  or $\mathsf L \subseteq  \mathsf L(\mathfrak R_2)$. Recall also that, if $\mathsf{LC}\subseteq  \mathsf L $, then  {\sf L} is one of  G\"odel logics that is {\sf L =LC} or  $\mathsf{L} = \mathsf L(\mathfrak{L}_n)$, for some $n\geq 1$.
Thus, we arrive to the following dichotomy.
\begin{corollary}\label{dichotomy}
For any intermediate logic {\sf L}  with unitary unification, either {\sf L} is one of the G\"odel logics and enjoys projective unification (and all unifiable formulas have mgu's preserving their variables) or there are  {\sf L}-unifiable formulas such that all their mgu's introduce new variables (that is mgu's do not preserve variables of unifiable formulas).
\end{corollary}
We do  not deny the existence of unifiable (non-projective) formulas with mgu's preserving their variables. We claim that, if a non-projective logic has unitary unification, then there must occur unifiable formulas all their mgu's introduce new variables.

\begin{corollary}\label{up1} An  intermediate logic {\sf L} has projective unification if and only if  {\sf L} has unitary unification and  projective approximation.  \end{corollary}
\begin{proof} Suppose that {\sf L} has unitary unification and  projective approximation. Then one notices that each unifiable formula $A$ has a one-elements projective approximation $\Pi(A)$. But it means that any unifiable $A$ has a mgu preserving its variables. By Corollary \ref{dichotomy}, we conclude {\sf L} has projective unification.
\end{proof}
 Corollary \ref{up1} does not hold for transitive modal logics; there are unitary transitive modal logics with  projective approximation that are non-projective. But for reflexive and transitive modal logics one can prove the counterpart of  Corollary \ref{up1}.

Other consequences of  Theorem \ref{u1} are given in  Section \ref{HFU}.

\subsection{Infinitary Unification}\label{IU}

Until now, no   intermediate (nor modal) logic  with infinitary unification has  been found, see  \cite{Ghi5,uni}.
Using Theorem \ref{main} we show that
\begin{theorem}\label{niu3} Any locally tabular intermediate logic does not have infinitary unification.
\end{theorem}
\begin{proof} Let {\sf L} be a locally tabular intermediate logic and suppose that unification is not finitary (nor unitary) in {\sf L}. Then
$$\exists_{n>0} \forall_{m>0} \exists_{k>0} \exists_{\sigma\colon\{x_1,\dots,x_n\}\to \mathsf{Fm^k}} \forall_{\tau\colon\{x_1,\dots,x_n\}\to \mathsf{Fm^m}} \ (\tau(A_\sigma)\in\mathsf{L} \Rightarrow \tau\not\preccurlyeq_{\sf L}\sigma).\leqno(\star)$$
Thus, $n> 0$ is given. Let us define a sequence of integers $n=m_0<m_1<m_2\cdots$ and substitutions $\sigma_1,\sigma_2,\dots$ such that, for each $i>0$,\\
(1) $\sigma_i\colon\{x_1,\dots,x_n\}\to \mathsf{Fm}^{m_i}$ and $\tau\not\preccurlyeq_\mathsf{L}\sigma_i$ if $\tau\colon\{x_1,\dots,x_n\}\to \mathsf{Fm}^{m_{i-1}}$ and $\tau(A_{\sigma_i})\in\mathsf{L}$; \\
 (2)  $\sigma_i(x_{i_1})\land\cdots\land\sigma_i(x_{i_s})$ is {\sf L}--projective, for any $\{i_1,\dots,i_s\}\subseteq \{1,\dots,n\}$;\\
(3) $\sigma_i$-models  are $\land$-closed: $\mathfrak{M}^n\land\mathfrak{N}^n=(W,R,w_0,\{\mathfrak{f}_w^n\land\mathfrak{g}_w^n\}_{w\in W})$ is a $\sigma_i$-model, where
$(\mathfrak{f}_w^n\land\mathfrak{g}_w^n)(x_j)=\mathfrak{f}_w^n(x_j)\land\mathfrak{g}_w^n(x_j)$ for each $w\in W$ and each $j=1,\dots,n$,  if $\mathfrak{M}^n=(W,R,w_0,\{\mathfrak{f}_w^n\}_{w\in W})$ and  $\mathfrak{N}^n=(W,R,w_0,\{\mathfrak{g}_w^n\}_{w\in W})$ are $\sigma_i$-models.

 Let $\mathfrak{M}^n\leq\mathfrak{N}^n$ mean $\mathfrak{f}_w^n\leq\mathfrak{g}_w^n$ for each $w\in W$, where  the order between the valuations is  the product order. Then $\mathfrak{M}^n\land\mathfrak{N}^n\leq\mathfrak{M}^n$ and $\mathfrak{M}^n\land\mathfrak{N}^n\leq\mathfrak{N}^n$.

Our definition is inductive with respect to $i$. Suppose that $m_{i}$ and $\sigma_{i}$ (if $i>0$) are given. Then we apply $(\star)$, where $m=m_{i}$, to get $k$ and $\sigma$ fulfilling (1). We do not take $m_{i+1}=k$ and $\sigma_{i+1}=\sigma$ as we  need  (2)--(3) to be fulfilled.

For (2), let us  define $\pi\colon\{x_1,\dots,x_n\}\to \mathsf{Fm^{k+n}}$ taking
$$\pi(x_j)=(x_{k+j}\leftrightarrow\sigma(x_j)),\qquad\mbox{for each $j=1,\dots,n$}.$$
Note that $x_{k+j}$ does not occur in $\sigma(x_j)$ and hence, to show (2) for $\sigma_{i+1}=\pi$, Lemma \ref{niu1} applies.
We still have (1) (for $\sigma_{i+1}=\pi$) as taking $x_{k+1}\slash\top\cdots x_{k+n}\slash \top$ we get  $\pi\preccurlyeq\sigma$.
But we do not take  $m_{i+1}=k+n$ and $\sigma_{i+1}=\pi$ as we  need (3).

For (3), let $\nu_r\colon\mathsf{Var}\to\mathsf{Var}$, for any $r>0$, be given by $\nu_r(x_i)=x_{i+r}$ for each $i$. We  define $m_{i+1}=s(k+n)$ and, for each $j=1,\dots,n$, let
$$\sigma_{i+1}(x_j)=\pi(x_j)\land \nu_{n+k}(\pi(x_j))\land\nu_{n+k}^2(\pi(x_j))\land\dots\land \nu_{n+k}^{s-1}(\pi(x_j)), \leqno (\star\star)$$
where $s$ is (sufficiently big and is) specified below. We have
 $\sigma_{i+1}=\bigwedge_{j=0}^{s-1}(\nu^j_{n+k}\circ\pi)$.
Using the inverse mapping $\nu^{-1}$ one shows that $\sigma_{i+1}\preccurlyeq\pi$ and hence (1) holds. For (2), we can use Lemma \ref{niu1}. Let us prove (3).

Suppose there are given $(n+k)$-models $\mathfrak{M}_0^{k+n},\dots,\mathfrak{M}_{s-1}^{k+n}$ having the same frame (and root). Let $\mathfrak{M}_i^{k+n}=(W,R,w_0,\{\mathfrak{f}_{i\ w}^{k+n}\}_{w\in W})$, for $i=0,\dots,s-1$, and suppose that the valuations
$\mathfrak{f}_{0\ w}^{k+n},\dots,\mathfrak{f}_{s-1\ w}^{k+n}$, for each $w\in W$, are given by binary strings (of the same length $k+n$). Take the concatenation of the strings
$$\mathfrak{f}_{w}^{s(k+n)}=\mathfrak{f}_{0\ w}^{k+n}\cdots\mathfrak{f}_{s-1\ w}^{k+n}, \qquad\mbox{for each $w\in W$}.$$
Then we get an $s(k+n)$-model $\mathfrak{M}^{s(k+n)}=(W,R,w_0,\{\mathfrak{f}_{w}^{s(k+n)}\}_{w\in W})$, called {\it the concatenation } of $\mathfrak{M}_0^{k+n},\dots,\mathfrak{M}_{s-1}^{k+n}$
for which
$$\sigma_{i+1}(\mathfrak{M}^{s(k+n)})=\pi(\mathfrak{M}^{k+n}_1)\land \pi(\mathfrak{M}^{k+n}_2)\land\dots\land \pi(\mathfrak{M}^{k+n}_{s-1}).$$
Obviously, each $s(k+n)$-model $\mathfrak{M}^{s(k+n)}$ can be received as a `concatenation' of its $(k+n)$-fragments, and hence each $\sigma_{i+1}$-model is a conjunction of some $\pi$-models. Since there is finitely many $n$-models ($\sigma_{i+1}$- and $\pi$-models are $n$-models), we can get as  $\sigma_{i+1}$-models all conjunction of $\pi$-models if $s$ is sufficiently big. It means that $\sigma_{i+1}$-models are closed under conjunction  if $s$ is  big enough, see $(\star\star)$. We get (3).

We have  substitutions $\sigma_1,\sigma_2,\dots$ (and  integers $m_0,m_1,\dots$)  fulfilling (1)-(3). Let us note that these conditions are also fulfilled by any subsequence of $\sigma_1,\sigma_2,\dots$. Since $\sigma_i$-models, for each $i$, are $n$-models and there is only finitely many $n$-models, we can find a subsequence $\sigma_{i_1},\sigma_{i_2},\dots$ such that each $\sigma_{i_k}$ has the same set of models. Thus,   we can assume that the  sequence $\sigma_1,\sigma_2,\dots$  fulfills   \\ \indent (4) $\sigma_i(\mathbf{M}^{m_i})=\sigma_j(\mathbf{M}^{m_j})$ for each $i,j\geq 1$.\\
 \noindent The next (and  crucial) step in our argument is to show that\\ \indent (5) for each $i>0$, there is a number $r>0$ such that $\bigwedge_{j=0}^{r-1}(\nu^j_{m_{i}}\circ\sigma_{i})\preccurlyeq\sigma_{i+1}.$\\
Let us show there is a mapping $F\colon\mathbf{M}^{m_{i+1}}\to\mathbf{M}^{rm_i}$ fulfilling (see Theorem \ref{nsigmai}): \\
(i) $\mathfrak{M}^{m_{i+1}}$ and  $F(\mathfrak{M}^{m_{i+1}})$ have the same frame, for each $\mathfrak{M}^{m_{i+1}}$;\\
(ii) $F((\mathfrak{M}^{m_{i+1}})_w)\thicksim(F(\mathfrak{M}^{m_{i+1}}))_w$, \ for each $w\in W$ ($W$ is the domain of $\mathfrak{M}^{m_{i+1}}$);\\
 (iii) If $\mathfrak{N}^{m_{i+1}}\thicksim\mathfrak{M}^{m_{i+1}}$, then $F(\mathfrak{N}^{m_{i+1}})\thicksim F(\mathfrak{M}^{m_{i+1}}$);\\
 such that the following diagram commutes (up to $\thicksim$)
 \begin{figure}[H]\label{}
\unitlength1cm
\begin{picture}(5,2)
\thicklines

\put(8,2){\vector(0,-1){1.9}}

\put(8,2){\vector(-2,-1){3.9}}
\put(8,0){\vector(-1,0){3.9}}
\put(8,0){\circle{0.1}}
\put(4,0){\circle{0.1}}
\put(8,2){\circle{0.1}}
\put(8.3,2){\mbox{$\mathbf{M}^{m_{i+1}}$}}
\put(3.2,0){\mbox{$\mathbf{M}^n$}}
\put(8.3,0){\mbox{$\mathbf{M}^{rm_i}$}}
\put(8.1,1){\mbox{$F$}}
\put(5.3,0.2){\mbox{$\bigwedge_{j=0}^{r-1}(\nu^j_{m_{i}}\circ\sigma_{i})$}}
\put(5.7,1.2){\mbox{$\sigma_{i+1}$}}

\end{picture}
\caption{}\label{rr}
\end{figure}

\noindent To specify the number $r$, let us assume that the sequence:
$$\mathfrak{M}_0^{m_{i+1}},\mathfrak{M}_1^{m_{i+1}}\dots,\mathfrak{M}_{r-1}^{m_{i+1}}\leqno (\star\star\star)$$
contains all p-irreducible (!) $m_{i+1}$-models.

Suppose that we have   defined $F_j\colon\mathbf{M}^{m_{i+1}}\to\mathbf{M}^{m_i}$, for any $j=0,1,\dots,r-1$ fulfilling \\(i) $\mathfrak{M}^{m_{i+1}}$ and  $F_j(\mathfrak{M}^{m_{i+1}})$ have the same frame, for each $\mathfrak{M}^{m_{i+1}}$;\\
(ii) $F_j((\mathfrak{M}^{m_{i+1}})_w)=(F_j(\mathfrak{M}^{m_{i+1}}))_w$, \ for each $w\in W$;\\
 (iii) If $\mathfrak{N}^{m_{i+1}}\thicksim\mathfrak{M}^{m_{i+1}}$, then $F_j(\mathfrak{N}^{m_{i+1}})\thicksim F_j(\mathfrak{M}^{m_{i+1}})$\\
 and let $F(\mathfrak{M}^{m_{i+1}})$ be the concatenation of the models $F_0(\mathfrak{M}^{m_{i+1}}),\dots,F_{r-1}(\mathfrak{M}^{m_{i+1}})$. We could not claim the following diagram commutes, for any $j$,

\begin{figure}[H]
\unitlength1cm
\begin{picture}(5,2)
\thicklines

\put(8,2){\vector(0,-1){1.9}}

\put(8,2){\vector(-2,-1){3.9}}
\put(8,0){\vector(-1,0){3.9}}
\put(8,0){\circle{0.1}}
\put(4,0){\circle{0.1}}
\put(8,2){\circle{0.1}}
\put(8.3,2){\mbox{$\mathbf{M}^{m_{i+1}}$}}
\put(3.2,0){\mbox{$\mathbf{M}^n$}}
\put(8.3,0){\mbox{$\mathbf{M}^{m_i}$}}
\put(8.1,1){\mbox{$F_j$}}
\put(6.3,0.2){\mbox{$\sigma_{i}$}}
\put(5.7,1.2){\mbox{$\sigma_{i+1}$}}

\end{picture}
\caption{}
\end{figure}

\noindent as it would give us $\sigma_{i+1}\preccurlyeq\sigma_i$  contradicting (1). But we  have
$$\sigma_i(F_k(\mathfrak{M}_k^{m_{i+1}}))=\sigma_{i+1}(\mathfrak{M}_k^{m_{i+1}})\quad \mbox{and} \quad \sigma_i(F_j(\mathfrak{M}_k^{m_{i+1}}))\geq \sigma_{i+1}(\mathfrak{M}_k^{m_{i+1}}),\leqno(\mbox{iv})$$ for each $k,j\in\{0,\dots,r-1\}$. It  means   the diagram in the Figure \ref{rr} commutes as
$$\bigwedge_{j=0}^{r-1}\nu^j_{m_{i}}\circ\sigma_{i}(F(\mathfrak{M}_k^{m_{i+1}}))=
\bigwedge_{j=0}^{r-1}\sigma_i(F_j(\mathfrak{M}_k^{m_{i+1}}))=\sigma_i(F_k(\mathfrak{M}_k^{m_{i+1}}))=
\sigma_{i+1}(\mathfrak{M}_k^{m_{i+1}});$$
and we know each $\mathfrak{M}^{m_{i+1}}\in \mathbf{M^{m_{i+1}}}$ is equivalent with some $\mathfrak{M}_k^{m_{i+1}}$, so by (iii) we get
$$\bigwedge_{j=0}^{r-1}\nu^j_{m_{i}}\circ\sigma_{i}(F(\mathfrak{M}^{m_{i+1}}))\thicksim
\sigma_{i+1}(\mathfrak{M}^{m_{i+1}}).$$
We could show (i)-(iii) for $F$ using  the fact these conditions  are fulfilled by each $F_j$. But such an argument would be too complicated and we could make it easier by the exact definition of a substitution $\alpha\colon\{x_1,\dots,x_{rm_i}\}\to \mathsf{Fm^{m_{i+1}}}$ such that $H_\alpha\thicksim F$. Then (i)-(iii) follow from Lemma \ref{nsigmai}. For any $F_j$ (where $j=0,\dots,r-1$), we should  have $F_j\thicksim H_{\alpha_j}$ for some $\alpha_j\colon\{x_1,\dots,x_{m_i}\}\to \mathsf{Fm^{m_{i+1}}}$. Then we define $\alpha$ as a disjoint union of the $\alpha_j$'s. More specifically:
$$\alpha(x_{l+jm_i})=\alpha_j(x_l), \quad\mbox{for each $l=1,\dots,m_i$}.$$
Obviously, $\alpha(\mathfrak{M}^{m_{i+1}})$ is the concatenation of  $\alpha_0(\mathfrak{M}^{m_{i+1}}),\dots,\alpha_{r-1}(\mathfrak{M}^{m_{i+1}})$ for each $m_{i+1}$-model $\mathfrak{M}^{m_{i+1}}$.
There remains to define $F_j$, for $j=0,\dots,r-1$.

Let us define any $F_j\colon\mathbf{M}^{m_{i+1}}\to\mathbf{M}^{m_i}$ as a partial mapping; its domain $D(F_j)$ is an up-ward closed subset of $\mathbf{M}^{m_{i+1}}$, which means
$$\mbox{if} \quad \mathsf{Th}(\mathfrak{N}^{m_{i+1}}) \subseteq\mathsf{Th}(\mathfrak{M}^{m_{i+1}})\quad \mbox{and } \quad \mathfrak{N}^{m_{i+1}}\in D(F_j),\quad \mbox{then }\quad \mathfrak{M}^{m_{i+1}}\in D(F_j).$$
Obviously, $F_j$ should also fulfill (i)-(iv), for models in $D(F_j)$. Then, step by step, we extend the domain $D(F_j)$ to the whole set $\mathbf{M}^{m_{i+1}}$; our definition of $F_j(\mathfrak{M}^{m_{i+1}})$ is inductive with respect to depth of $\mathfrak{M}^{m_{i+1}}$.

(A). By (4), $\sigma_{i+1}(\mathfrak{M}_j^{m_{i+1}})=\sigma_{i}(\mathfrak{M}^{m_{i}})$, for some $\mathfrak{M}^{m_{i}}\in\mathbf{M}^{m_{i}}$. Thus, we take
$$F_j(\mathfrak{M}_j^{m_{i+1}})=\mathfrak{M}^{m_{i}}$$
which gives $\sigma_i(F_j(\mathfrak{M}_j^{m_{i+1}}))=\sigma_{i+1}(\mathfrak{M}_j^{m_{i+1}})$ and this guarantee the (first part of the) condition (iv). The models $\sigma_{i+1}(\mathfrak{M}_j^{m_{i+1}})$, $\sigma_{i}(\mathfrak{M}^{m_{i}})$, $\mathfrak{M}_j^{m_{i+1}}$, $\mathfrak{M}^{m_{i}}$ have the same frame $(W,R,w_0)$ and  only valuations could be different. To get (ii) we should take $$F_j((\mathfrak{M}_j^{m_{i+1}})_w)=(\mathfrak{M}^{m_{i}})_w, \quad \mbox{for each } w\in W.$$
According to Theorem \ref{pM6},
$(\mathfrak{M}^{m_{i+1}})_w$, for each $w\in W$, is p-irreducible and hence for each $\mathfrak{N}^{m_{i+1}}$ equivalent with $(\mathfrak{M}_j^{m_{i+1}})_w$ there is exactly one p-morphism (see Theorem \ref{lf3i}) $p\colon\mathfrak{N}^{m_{i+1}}\to(\mathfrak{M}_j^{m_{i+1}})_w$. To satisfy (iii), we should take
$$F_j(\mathfrak{N}^{m_{i+1}})=p^{-1}((\mathfrak{M}^{m_{i}})_w)$$
where $p^{-1}((\mathfrak{M}^{m_{i}})_w)$ is the only $n$-model on $(W,R,w_0)$ such that the mapping $p$ is a p-morphism $p\colon p^{-1}((\mathfrak{M}^{m_{i}})_w)\to (\mathfrak{M}^{m_{i}})_w$  of $n$-models.
   Since the valuations are preserved by p-morphisms,
$\sigma_i(F_j(\mathfrak{N}^{m_{i+1}}))=\sigma_{i+1}(\mathfrak{N}^{m_{i+1}})$ which  guarantees the second part of (iv) if $\mathfrak{N}^{m_{i+1}}=\mathfrak{M}_k^{m_{i+1}}$, for some $k$. Our definition of the (partial) mapping $F_j$ is completed; its domain  is an upset.

(B). We have no problems to  define $F_j(\mathfrak{M}^{m_{i+1}})$ for any $m_{i+1}$-model $\mathfrak{M}^{m_{i+1}}$ over one-element frame (assuming the model does not belong to $D(F_j)$ by (A)). By  (4), it  suffices  to define $F_j(\mathfrak{M}^{m_{i+1}})$ in a such way that $\sigma_i(F_k(\mathfrak{M}^{m_{i+1}}))=\sigma_{i+1}(\mathfrak{M}^{m_{i+1}})$ for each $m_{i+1}$-model over one-element frame. Each $\mathfrak{N}^{m_{i+1}}=(W,R,w_0,\{\mathfrak{f}_w^{m_{i+1}}\}_{w\in W})$ equivalent with a model with one-element frame has  $\mathfrak{f}_w^{m_{i+1}}=\mathfrak{f}_u^{m_{i+1}}$, for each $u,w\in W$, and hence we can define
$F_j(\mathfrak{N}^{m_{i+1}})$ preserving $\sigma_i(F_k(\mathfrak{N}^{m_{i+1}}))=\sigma_{i+1}(\mathfrak{N}^{m_{i+1}})$. The conditions (i)-(iv) are  fulfilled and $D(F_j)$ is an upset by Theorem \ref{pat}.

(C). If all $\mathfrak{M}_k^{m_{i+1}}$'s belong to $D(F_j)$ we have done. Suppose that some $\mathfrak{M}_k^{m_{i+1}}$ does not belong to $D(F_j)$. We can assume that $\mathfrak{M}_k^{m_{i+1}}=(W,R,w_0,\{\mathfrak{f}_w^{m_{i+1}}\}_{w\in W})$ and $(\mathfrak{M}_k^{m_{i+1}})_w\in D(F_j)$ for each  $w\not=w_0$. Thus, we have $F_j((\mathfrak{M}_k^{m_{i+1}})_w)$ for each $w\not=w_0$ and we need to define $F_j(\mathfrak{M}_k^{m_{i+1}})$. In other words, an ${m_{i}}$-model $\mathfrak{M}^{m_{i}}=(W,R,w_0,\{\mathfrak{g}_w^{m_{i}}\}_{w\in W})$ is given such that $(\mathfrak{M}^{m_{i}})_w=F_j((\mathfrak{M}_k^{m_{i+1}})_w)$ for each $w\not=w_0$, what is $\mathfrak{g}_{w_0}^{m_{i}}$ does not matter,  we need its variant $\mathfrak{M}_0^{m_{i}}=(W,R,w_0,\{\mathfrak{f}_w^{m_{i}}\}_{w\in W})$ such that $\sigma_i(\mathfrak{M}_0^{m_{i}})\geq \sigma_{i+1}(\mathfrak{M}_k^{m_{i+1}})$; then we can take $F_j(\mathfrak{M}_k^{m_{i+1}})=\mathfrak{M}_0^{m_{i}}$ fulfilling all requirements (except for (iii)).
Let $x\in\{x_1,\dots,x_{n}\}$. We have $$
 \mathfrak{M}_k^{m_{i+1}}\Vdash_w\sigma_{i+1}(x)\quad\Rightarrow\quad \mathfrak{M}^{m_{i}}\Vdash_w\sigma_i(x), \mbox{for each }\ w\not=w_0$$ and want  a variant $\mathfrak{M}_0^{m_{i}}$ of $\mathfrak{M}^{m_{i}}$ such that:
 $$
 \mathfrak{M}_k^{m_{i+1}}\Vdash_{w_0}\sigma_{i+1}(x)\quad\Rightarrow\quad \mathfrak{M}_0^{m_{i}}\Vdash_{w_0}\sigma_i(x).$$
 If $\mathfrak{M}_k^{m_{i+1}})\not\Vdash_w\sigma_{i+1}(x)$ for some $w\not=w_0$ the implication holds. Thus, it suffices to consider the set $\{j_1,\dots,j_s\}\subseteq \{1,\dots,n\}$ containing all $x$'s such that
  $$\mathfrak{M}^{m_{i}}\Vdash_{w}\sigma_i(x), \mbox{ for each} \ w\not=w_0. $$
  By (2), $\sigma_i(x_{j_1})\land\cdots\land\sigma_i(x_{j_s})$ is {\sf L}--projective and hence, by Theorem \ref{niu2}, there is a variant $\mathfrak{M}_0^{m_{i}}$ of $\mathfrak{M}^{m_{i}}$ such that
  $$\mathfrak{M}_0^{m_{i}}\Vdash_{w_0}\sigma_i(x_{j_1})\land\cdots\land\sigma_i(x_{j_s}).$$
 The definition of $F_j(\mathfrak{M}_k^{m_{i+1}})$ is completed. There remains to add that all $m_{i+1}$-models $\mathfrak{N}^{m_{i+1}}$ with $\mathsf{Th}(\mathfrak{N}^{m_{i+1}}) \subset\mathsf{Th}(\mathfrak{M}_k^{m_{i+1}})$ are included in $D(F_j)$ by Lemma \ref{pat}. If $ \mathfrak{N}^{m_{i+1}}\thicksim\mathfrak{M}_k^{m_{i+1}}$, then there is a p-morphism $p\colon\mathfrak{N}^{m_{i+1}}\to\mathfrak{M}_k^{m_{i+1}}$ and hence we can take $F_j(\mathfrak{N}^{m_{i+1}})=p^{-1}(\mathfrak{M}_0^{m_{i}})$. So, we  claim  $D(F_j)$ remains an upset.\\

 We have shown (5) and can get to the proof of our theorem. Suppose that {\sf L} has infinitary unification and  let $A=A_{\sigma_1}$. It follows from (4) that $A=A_{\sigma_i}$, for each $i$, and hence all $\sigma_i$'s are unifiers for $A$. By (1), $A$ cannot have finitary unification in $\mathsf L$ as there is no unifier of $A$ which would be more general then all $\sigma_i$'s. Let $\Sigma$ be a minimal complete set of unifiers for $A$. Then $\tau\preccurlyeq\sigma_1$
for some $\tau\colon\{x_1,\dots,x_n\}\to \mathsf{Fm^m}$ in $\Sigma$. Thus, there is a number $i$ such that $\tau\preccurlyeq\sigma_i$ and $\tau\not\preccurlyeq\sigma_{i+1}$. By $\tau\preccurlyeq\sigma_i$, we also get
$\nu^j_{m}\circ\tau\preccurlyeq\nu^k_{m_{i}}\circ\sigma_i$, for any $j,k\geq 0$; it does not matter if $m=m_i$ or not. Thus,
$$\bigwedge_{j=0}^{r-1}(\nu^j_{m}\circ\tau)\preccurlyeq\bigwedge_{j=0}^{r-1}(\nu^j_{m_{i}}\circ\sigma_{i})$$
  But $\bigwedge_{j=0}^{r-1}(\nu^j_{m}\circ\tau)\preccurlyeq\tau$ and hence $\tau\preccurlyeq\bigwedge_{j=0}^{r-1}(\nu^j_{m}\circ\tau)$ as $\bigwedge_{j=0}^{r-1}(\nu^j_{m}\circ\tau)$ is a unifier for $A$, by (3), and $\tau\in\Sigma$. Thus, we get $\tau\preccurlyeq\sigma_{i+1}$, by (5), which is a contradiction.
\end{proof}
Let {\sf L} be a locally tabular intermediate logic and suppose we have shown, using Theorem \ref{main} or \ref{main2}, that unification in {\sf L} is not finitary (nor unitary). It would mean, by the above Theorem \ref{niu3}, that {\sf L} has nullary unification. In the following Section, we prove in this way that very many intermediate logics has nullary unification.

\subsection{Nullary Unification}\label{NUni} It is known that $\mathsf L({\mathfrak G_3})$  and $\mathsf L({\mathfrak G_3}+)$ (see Figure  \ref{GF})  have nullary unification, see {\it Introduction}. In \cite{dkw}, we proved that unification in the modal version of these logics  is nullary. Below we present the intuitionistic version of our argument
 \begin{theorem}\label{F6m}
The logics $\mathsf L({\mathfrak G_3})$  and $\mathsf L({\mathfrak G_3}+)$ have nullary unification.\end{theorem}
\begin{proof} Let $\mathbf{F}=\{\mathfrak L_1,\mathfrak L_2,\mathfrak L_3,\mathfrak F_{2},\mathfrak G_{3}\}$. Then $\mathbf{F}=sm(\mathfrak G_{3})$.
Assume  ${\mathsf L}(\mathbf F)$ has finitary unification.   By Theorem \ref{main}, for every $n\geq 1$  there is a number $m\geq 1$ such that for any $\sigma\colon\{x_1\}\to \mathsf{Fm}^k$  there are mappings  $G:\mathbf{M}^m\to\mathbf{M}^1$ and $F:\mathbf{M}^k\to\mathbf{M}^m$  fulfilling the conditions (i)-(v). Let $n=1\leq m<k$ and $\sigma\colon\{x_1\}\to\mathsf{Fm}^k$ be as follows
 $$\sigma(x_1)=\neg\neg (\bigvee_{i=1}^k x_i)\land\bigwedge_{i=1}^k(\neg\neg x_i\lor\neg x_i).$$
 \unitlength1cm
\begin{picture}(3,2)
\thicklines

\put(0,1){Take any  $\mathfrak{M}^k$ over $\mathfrak G_3$:}
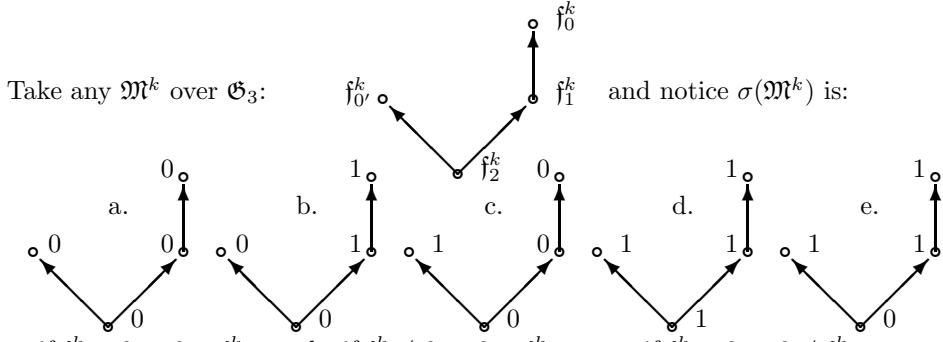
\put(8,1){and notice   $\sigma(\mathfrak{M}^k)$ is:}
\put(6,0){\vector(-1,1){0.9}}
\put(6,0){\vector(1,1){0.9}}
\put(7,1){\vector(0,1){0.9}}
\put(5,1){\circle{0.1}}
\put(7,2){\circle{0.1}}
\put(6,0){\circle{0.1}}
\put(7,1){\circle{0.1}}
\put(6.3,0){\mbox{$\mathfrak{f}_2^k$}}
\put(4.5,1){\mbox{$\mathfrak{f}_{0'}^k$}}
\put(7.3,2){\mbox{$\mathfrak{f}_0^k$}}
\put(7.3,1){\mbox{$\mathfrak{f}_1^k$}}

\end{picture}

\unitlength1cm
\begin{picture}(3,2)
\thicklines

\put(1,1.5){a.}
\put(1,0){\vector(-1,1){0.9}}
\put(1,0){\vector(1,1){0.9}}
\put(2,1){\vector(0,1){0.9}}
\put(0,1){\circle{0.1}}
\put(2,2){\circle{0.1}}
\put(1,0){\circle{0.1}}
\put(2,1){\circle{0.1}}
\put(1.3,0){\mbox{$0$}}
\put(0.2,1){\mbox{$0$}}
\put(1.7,2){\mbox{$0$}}
\put(1.7,1){\mbox{$0$}}

\put(3.5,1.5){b.}
\put(3.5,0){\vector(-1,1){0.9}}
\put(3.5,0){\vector(1,1){0.9}}
\put(4.5,1){\vector(0,1){0.9}}
\put(2.5,1){\circle{0.1}}
\put(4.5,2){\circle{0.1}}
\put(3.5,0){\circle{0.1}}
\put(4.5,1){\circle{0.1}}
\put(3.8,0){\mbox{$0$}}
\put(4.2,1){\mbox{$1$}}
\put(4.2,2){\mbox{$1$}}
\put(2.7,1){\mbox{$0$}}

\put(6,1.5){c.}
\put(6,0){\vector(-1,1){0.9}}
\put(6,0){\vector(1,1){0.9}}
\put(7,1){\vector(0,1){0.9}}
\put(5,1){\circle{0.1}}
\put(7,2){\circle{0.1}}
\put(6,0){\circle{0.1}}
\put(7,1){\circle{0.1}}
\put(6.3,0){\mbox{$0$}}
\put(5.3,1){\mbox{$1$}}
\put(6.7,2){\mbox{$0$}}
\put(6.7,1){\mbox{$0$}}

\put(8.5,1.5){d.}
\put(8.5,0){\vector(-1,1){0.9}}
\put(8.5,0){\vector(1,1){0.9}}
\put(9.5,1){\vector(0,1){0.9}}
\put(7.5,1){\circle{0.1}}
\put(9.5,2){\circle{0.1}}
\put(8.5,0){\circle{0.1}}
\put(9.5,1){\circle{0.1}}
\put(8.8,0){\mbox{$1$}}
\put(7.8,1){\mbox{$1$}}
\put(9.2,2){\mbox{$1$}}
\put(9.2,1){\mbox{$1$}}

\put(11,1.5){e.}
\put(11,0){\vector(-1,1){0.9}}
\put(11,0){\vector(1,1){0.9}}
\put(12,1){\vector(0,1){0.9}}
\put(10,1){\circle{0.1}}
\put(12,2){\circle{0.1}}
\put(11,0){\circle{0.1}}
\put(12,1){\circle{0.1}}
\put(11.3,0){\mbox{$0$}}
\put(10.3,1){\mbox{$1$}}
\put(11.7,2){\mbox{$1$}}
\put(11.7,1){\mbox{$1$}}
\end{picture}

\noindent
a. 	if $\mathfrak{f}_0^k=0\cdots0=\mathfrak{f}_{0'}^k$;\qquad
b.	 if $\mathfrak{f}_{0}^k\not=0\cdots0=\mathfrak{f}_{0'}^k$;\qquad
c.	 if $\mathfrak{f}_{0}^k=0\cdots0\not=\mathfrak{f}_{0'}^k$;\\
d. if $\mathfrak{f}_{0'}^k=\mathfrak{f}_0^k\not=0\cdots0$;\qquad
e.	if $0\cdots0\not=\mathfrak{f}_{0}^k\not=\mathfrak{f}_{0'}^k\not=0\cdots0$.\\
\noindent Thus, each $\sigma$-model is equivalent with a model of the depth $\leq 2$ and there are four p-irreducible $1$-models equivalent with some $\sigma$-models. By (v)
 $$G(F(\circ \ \ 0\cdots0))=G(\circ \ \mathfrak{g}^m)) =\sigma(\circ\ \ 0\cdots0)=\circ\ 0,$$ for some  $\mathfrak{g}^m$.
Since $m<k$,
one can find  $\mathfrak{f}^k\not=\mathfrak{g}^k$ such that $F(\circ\ \mathfrak{f}^k)=\circ \ \mathfrak{f}^m=F(\circ\ \mathfrak{g}^k)$, and $G(\circ \ \mathfrak{f}^m)=\circ  \ 1$, for some $\mathfrak{f}^m$.
By the characterization of all $\sigma$-models\\

\begin{center}
	$G\Bigl($
	\begin{minipage}[c][10mm][b]{10mm}
		\begin{picture}(3,2)
		\thicklines
		\linethickness{0.3mm}
		\put(0.2,0){\circle{0.1}}
		\put(0.5,0){\mbox{$?$}}
		\put(0.2,1){\circle{0.1}}
		\put(0.5,1){\mbox{$\mathfrak{f}^m$}}
		\put(0.2,0){\vector(0,1){0.9}}
		\end{picture}	
	\end{minipage}
	$\Bigr) \quad = \quad$
	\begin{minipage}[c][10mm][b]{10mm}
		\begin{picture}(3,2)
		\thicklines
		\linethickness{0.3mm}
		\put(0.2,0){\circle{0.1}}
		\put(0.5,0){\mbox{$0$}}
		\put(0.2,1){\circle{0.1}}
		\put(0.5,1){\mbox{$1$}}
		\put(0.2,0){\vector(0,1){0.9}}
		\end{picture}	
	\end{minipage}	
\quad or \quad
$G\Bigl($
\begin{minipage}[c][10mm][b]{10mm}
	\begin{picture}(3,2)
	\thicklines
	\linethickness{0.3mm}
	\put(0.2,0){\circle{0.1}}
	\put(0.5,0){\mbox{$?$}}
	\put(0.2,1){\circle{0.1}}
	\put(0.5,1){\mbox{$\mathfrak{f}^m$}}
	\put(0.2,0){\vector(0,1){0.9}}
	\end{picture}	
\end{minipage}
$\Bigr) \quad = \quad$
\begin{minipage}[c][10mm][b]{10mm}
	\begin{picture}(3,2)
	\thicklines
	\linethickness{0.3mm}
	\put(0.2,0){\circle{0.1}}
	\put(0.5,0){\mbox{$1$}}
	\put(0.2,1){\circle{0.1}}
	\put(0.5,1){\mbox{$1$}}
	\put(0.2,0){\vector(0,1){0.9}}
	\end{picture}	
\end{minipage}
\end{center}$ $\\

\noindent for any $m$-valuation $?$. But if the first had happened, we would have \\

\begin{center}
	$G\Bigl($
	\begin{minipage}[c][20mm][b]{25mm}
		\begin{picture}(3,2)
		\thicklines
		\linethickness{0.3mm}
		\put(1,0){\vector(-1,1){0.9}}
		\put(1,0){\vector(1,1){0.9}}
		\put(2,1){\vector(0,1){0.9}}
		\put(0,1){\circle{0.1}}
		\put(2,2){\circle{0.1}}
		\put(1,0){\circle{0.1}}
		\put(2,1){\circle{0.1}}
		\put(1.3,0){$0\cdots0$}
		\put(.5,1){\mbox{$\mathfrak{g}^m$}}
		\put(2.3,2){\mbox{$\mathfrak{f}^m$}}
		\put(2.3,1){?}
		\end{picture}	
	\end{minipage}
	$\Bigr) \quad = \quad$
	\begin{minipage}[c][20mm][b]{25mm}
		\begin{picture}(3,2)
		\thicklines
		\linethickness{0.3mm}
		\put(1,0){\vector(-1,1){0.9}}
		\put(1,0){\vector(1,1){0.9}}
		\put(2,1){\vector(0,1){0.9}}
		\put(0,1){\circle{0.1}}
		\put(2,2){\circle{0.1}}
		\put(1,0){\circle{0.1}}
		\put(2,1){\circle{0.1}}
		\put(1.3,0){$0$}
		\put(.5,1){\mbox{$0$}}
		\put(2.3,2){\mbox{$1$}}
		\put(2.3,1){$0$}
		\end{picture}	
	\end{minipage}	
\end{center}$ $\\

\noindent which would contradict (iv) as no $\sigma$-model is equivalent to the $1$-model on the right hand side of the above equation. We conclude that \\

\begin{center}
	$G\Bigl($
	\begin{minipage}[c][10mm][b]{10mm}
		\begin{picture}(3,2)
		\thicklines
		\linethickness{0.3mm}
		\put(0.2,0){\circle{0.1}}
		\put(0.5,0){\mbox{$?$}}
		\put(0.2,1){\circle{0.1}}
		\put(0.5,1){\mbox{$\mathfrak{f}^m$}}
		\put(0.2,0){\vector(0,1){0.9}}
		\end{picture}	
	\end{minipage}
	$\Bigr) \quad = \quad$
	\begin{minipage}[c][10mm][b]{10mm}
		\begin{picture}(3,2)
		\thicklines
		\linethickness{0.3mm}
		\put(0.2,0){\circle{0.1}}
		\put(0.5,0){\mbox{$1$}}
		\put(0.2,1){\circle{0.1}}
		\put(0.5,1){\mbox{$1$}}
		\put(0.2,0){\vector(0,1){0.9}}
		\end{picture}	
	\end{minipage}	
$\quad \thicksim \quad \circ \ 1$
\end{center}$ $\\
and hence

\begin{center}
	$G\Bigl( F\Bigl($
	\begin{minipage}[c][20mm][b]{27mm}
		\begin{picture}(3,2)
		\thicklines
		\linethickness{0.3mm}
		\put(1,0){\vector(-1,1){0.9}}
		\put(1,0){\vector(1,1){0.9}}
		\put(2,1){\vector(0,1){0.9}}
		\put(0,1){\circle{0.1}}
		\put(2,2){\circle{0.1}}
		\put(1,0){\circle{0.1}}
		\put(2,1){\circle{0.1}}
		\put(1.3,0){$0\cdots0$}
		\put(.5,1){\mbox{$\mathfrak{f}^k$}}
		\put(2.3,2){\mbox{$\mathfrak{g}^k$}}
		\put(2.3,1){$\mathfrak{g}^k$}
		\end{picture}	
	\end{minipage}
	$\Bigr) \Bigr) \quad = \quad G\Bigl($
	\begin{minipage}[c][20mm][b]{27mm}
		\begin{picture}(3,2)
		\thicklines
		\linethickness{0.3mm}
		\put(1,0){\vector(-1,1){0.9}}
		\put(1,0){\vector(1,1){0.9}}
		\put(2,1){\vector(0,1){0.9}}
		\put(0,1){\circle{0.1}}
		\put(2,2){\circle{0.1}}
		\put(1,0){\circle{0.1}}
		\put(2,1){\circle{0.1}}
		\put(1.3,0){?}
		\put(.5,1){\mbox{$\mathfrak{f}^m$}}
		\put(2.3,2){\mbox{$\mathfrak{f}^m$}}
		\put(2.3,1){$\mathfrak{f}^m$}
		\end{picture}	
	\end{minipage}
$\Bigr)\thicksim \circ \ 1$
\end{center}$ $\\

But this is in contradiction with e. \\

\begin{center}
	$\sigma\Bigl($
	\begin{minipage}[c][20mm][b]{27mm}
		\begin{picture}(3,2)
		\thicklines
		\linethickness{0.3mm}
		\put(1,0){\vector(-1,1){0.9}}
		\put(1,0){\vector(1,1){0.9}}
		\put(2,1){\vector(0,1){0.9}}
		\put(0,1){\circle{0.1}}
		\put(2,2){\circle{0.1}}
		\put(1,0){\circle{0.1}}
		\put(2,1){\circle{0.1}}
		\put(1.3,0){$0\cdots0$}
		\put(.5,1){\mbox{$\mathfrak{f}^k$}}
		\put(2.3,2){\mbox{$\mathfrak{g}^k$}}
		\put(2.3,1){$\mathfrak{g}^k$}
		\end{picture}	
	\end{minipage}
	$\Bigr) \quad = \quad $
	\begin{minipage}[c][20mm][b]{27mm}
		\begin{picture}(3,2)
		\thicklines
		\linethickness{0.3mm}
		\put(1,0){\vector(-1,1){0.9}}
		\put(1,0){\vector(1,1){0.9}}
		\put(2,1){\vector(0,1){0.9}}
		\put(0,1){\circle{0.1}}
		\put(2,2){\circle{0.1}}
		\put(1,0){\circle{0.1}}
		\put(2,1){\circle{0.1}}
		\put(1.3,0){$0$}
		\put(.5,1){\mbox{$1$}}
		\put(2.3,2){\mbox{$1$}}
		\put(2.3,1){$1$}
		\end{picture}	
	\end{minipage}
\end{center}
Let $\mathbf{F}=\{\mathfrak L_1,\mathfrak L_2,\mathfrak L_3,\mathfrak L_{4},\mathfrak R_{2},\mathfrak G_{3}+\}$. Then $\mathbf{F}=sm(\mathfrak G_{3}+)$. Suppose that ${\mathsf L}(\mathbf F)$ has unitary unification.  Take $n=2$. By Theorem \ref{main},  there is a number $m\geq 1$ such that for any $\sigma\colon\{x_1,x_2\}\to \mathsf{Fm}^{k+1}$  there are mappings  $G:\mathbf{M}^m\to\mathbf{M}^2$ and $F:\mathbf{M}^{k+1}\to\mathbf{M}^m$  fulfilling the conditions (i)-(v). Let $k>m$ and
$$ \begin{array}{rl}
     \sigma(x_1)=& x_1\\
     \sigma(x_2)= & \Bigl(\bigl(( \bigvee_{i=2}^{k+1}x_i)\to x_1\bigr)\to x_1\Bigr) \ \land \ \bigwedge_{i=2}^{k+1}\Bigl(\bigl((x_i\to x_1)\to x_1\bigr)\lor (x_i\to x_1)\Bigr).
   \end{array}
  $$
If we take $\alpha\circ\sigma$, where $\alpha:x_1\slash\bot$, we  get $\sigma$ as used  for $\mathfrak G_{3}$ (there is only a shift of variables from $x_1\dots x_k$ to $x_2\dots x_{k+1}$).
If $x_1$ is false at the top element of any $2$-model $\mathfrak M^2$ over $\mathfrak G_{3}+$, then $ \sigma(x_2)$ is true at the model and hence $\sigma(\mathfrak M^2)$ reduces to a model over $\mathfrak L_1$. Decapitation  of the top element (and other elements at which $x_1$ is true) give us models over $\mathfrak G_{3}$. Then we can argue as in the case of $\mathsf L(\mathfrak{G}_3).$

We have relatively many $\sigma$-models of the depth $\leq 2$ but there are only two p-irreducible equivalents of $\sigma$-models of the depth $\geq 3$. They are
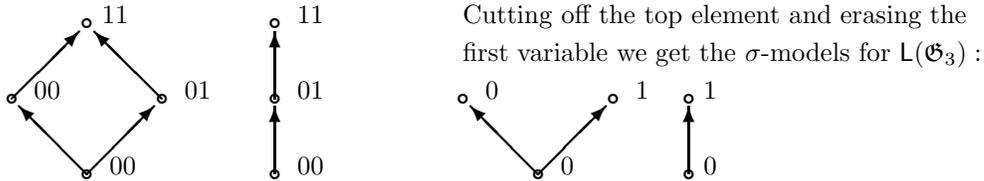
\begin{figure}[H]
 \unitlength1cm
\begin{picture}(3,2)
\thicklines

\put(2,1){\vector(-1,1){0.9}}
\put(0,1){\vector(1,1){0.9}}
\put(1,0){\vector(-1,1){0.9}}
\put(1,0){\vector(1,1){0.9}}
\put(0,1){\circle{0.1}}
\put(1,2){\circle{0.1}}
\put(1,0){\circle{0.1}}
\put(2,1){\circle{0.1}}
\put(1.3,0){\mbox{$00$}}
\put(0.3,1){\mbox{$00$}}
\put(1.2,2){\mbox{$11$}}
\put(2.3,1){\mbox{$01$}}

\put(3.5,0){\vector(0,1){0.9}}
\put(3.5,0){\circle{0.1}}
\put(3.5,1){\vector(0,1){0.9}}
\put(3.5,1){\circle{0.1}}
\put(3.5,2){\circle{0.1}}
\put(3.8,0){\mbox{$00$}}
\put(3.8,1){\mbox{$01$}}
\put(3.8,2){\mbox{$11$}}

\put(6,2){Cutting off the top element and erasing the}
\put(6,1.5){first variable we get the $\sigma$-models
for $\mathsf L(\mathfrak{G}_3):$}

\put(7,0){\vector(-1,1){0.9}}
\put(7,0){\vector(1,1){0.9}}
\put(6,1){\circle{0.1}}

\put(7,0){\circle{0.1}}
\put(8,1){\circle{0.1}}
\put(7.3,0){\mbox{$0$}}
\put(6.3,1){\mbox{$0$}}

\put(8.3,1){\mbox{$1$}}

\put(9,0){\vector(0,1){0.9}}
\put(9,0){\circle{0.1}}

\put(9,1){\circle{0.1}}

\put(9.2,0){\mbox{$0$}}
\put(9.2,1){\mbox{$1$}}


\end{picture}
\end{figure}

\noindent Let $F(\circ\ 1\cdots1)=\circ \ \mathfrak h^m$. Since $\sigma(\circ\ 1\cdots1)=\circ \  11$, we have $G(\circ\ 1\cdots1)=\circ \ 11$, by (v). We also have $\mathfrak g^m,\mathfrak f^m$ and $\mathfrak g^k\not=\mathfrak f^k$ such that
\begin{figure}[H]
 \unitlength1cm
\begin{picture}(3,1)
\thicklines

\put(0,0.5){\mbox{$F\bigl($}}
\put(0.6,0){\vector(0,1){0.9}}
\put(0.6,0){\circle{0.1}}
\put(0.6,1){\circle{0.1}}
\put(0.8,0){\mbox{$0\cdots0$}}
\put(0.8,1){\mbox{$1\cdots1$}}
\put(1.8,0.5){\mbox{$\bigr) \ =$}}
\put(3,0){\vector(0,1){0.9}}
\put(3,0){\circle{0.1}}
\put(3,1){\circle{0.1}}
\put(3.2,0){\mbox{$\mathfrak g^m$}}
\put(3.2,1){\mbox{$\mathfrak h^m$}}
\put(4,0.5){,}

\put(5,0.5){\mbox{$F\bigl($}}
\put(6,0){\vector(0,1){0.9}}
\put(6,0){\circle{0.1}}
\put(6,1){\circle{0.1}}
\put(6.2,1){\mbox{$1\cdots1$}}
\put(6.2,0){\mbox{$0\mathfrak f^k$}}
\put(7.3,0.5){\mbox{$\bigr) \ =$}}

\put(8.2,0.5){\mbox{$F\bigl($}}
\put(9,0){\vector(0,1){0.9}}
\put(9,0){\circle{0.1}}
\put(9,1){\circle{0.1}}
\put(9.2,0){\mbox{$0\mathfrak g^k$}}
\put(9.2,1){\mbox{$1\cdots1$}}
\put(10.3,0.5){\mbox{$\bigr) \ =$}}
\put(11.5,0){\vector(0,1){0.9}}
\put(11.5,0){\circle{0.1}}
\put(11.5,1){\circle{0.1}}
\put(11.7,0){\mbox{$\mathfrak f^m$}}
\put(11.7,1){\mbox{$\mathfrak h^m$}}

\end{picture}
\end{figure}
\begin{figure}[H]
 \unitlength1cm
\begin{picture}(3,1)
\thicklines

\put(0,1){Then by (v)}

\put(3.8,0.5){\mbox{$\bigr) \ =$}}
\put(3,0){\vector(0,1){0.9}}
\put(3,0){\circle{0.1}}
\put(3,1){\circle{0.1}}
\put(3.2,0){\mbox{$\mathfrak g^m$}}
\put(3.2,1){\mbox{$\mathfrak h^m$}}

\put(2,0.5){ $G\bigl($}
\put(5,0){\vector(0,1){0.9}}
\put(5,0){\circle{0.1}}
\put(5,1){\circle{0.1}}
\put(5.2,1){\mbox{$11$}}
\put(5.2,0){\mbox{$00$}}
\put(7,0.5){,}

\put(8.2,0.5){\mbox{$G\bigl($}}
\put(9,0){\vector(0,1){0.9}}
\put(9,0){\circle{0.1}}
\put(9,1){\circle{0.1}}
\put(9.2,0){\mbox{$0\mathfrak f^m$}}
\put(9.2,1){\mbox{$\mathfrak h^m$}}
\put(10.3,0.5){\mbox{$\bigr) \ =$}}
\put(11.5,0){\vector(0,1){0.9}}
\put(11.5,0){\circle{0.1}}
\put(11.5,1){\circle{0.1}}
\put(11.7,0){\mbox{$01$}}
\put(11.7,1){\mbox{$11$}}

\end{picture}
\end{figure}
\noindent By (i)--(ii), for any $m$-valuation $?$, we have\\
\begin{figure}[H]
 \unitlength1cm
\begin{picture}(3,2)
\thicklines
		\put(1.2,0){\circle{0.1}}
		\put(1.5,0){\mbox{$?$}}
		\put(1.2,1){\circle{0.1}}
        \put(1.2,2){\circle{0.1}}
		\put(1.5,1){\mbox{$\mathfrak{f}^m$}}
		\put(1.2,0){\vector(0,1){0.9}}
       \put(1.2,1){\vector(0,1){0.9}}
       \put(1.5,2){\mbox{$\mathfrak{h}^m$}}
	 \put(2,1){$\Bigr) \ = $}
 \put(0,1){$G\Bigl($}
	
		\put(3.2,0){\circle{0.1}}
		\put(3.5,0){\mbox{$00$}}
		\put(3.2,1){\circle{0.1}}
		\put(3.5,1){\mbox{$01$}}
		\put(3.2,0){\vector(0,1){0.9}}
		\put(3.2,2){\circle{0.1}}
		\put(3.5,2){\mbox{$11$}}
		\put(3.2,1){\vector(0,1){0.9}}
\put(5,1){\quad or \quad}
\put(7,1){$G\Bigl($}

		\put(8.2,0){\circle{0.1}}
		\put(8.5,0){\mbox{$?$}}
		\put(8.2,1){\circle{0.1}}
        \put(8.2,2){\circle{0.1}}
		\put(8.5,1){\mbox{$\mathfrak{f}^m$}}
		\put(8.2,0){\vector(0,1){0.9}}
       \put(8.2,1){\vector(0,1){0.9}}
       \put(8.5,2){\mbox{$\mathfrak{h}^m$}}
	 \put(9,1){$\Bigr) \ = $}
	
		\put(10.2,0){\circle{0.1}}
		\put(10.5,0){\mbox{$01$}}
		\put(10.2,1){\circle{0.1}}
		\put(10.5,1){\mbox{$01$}}
		\put(10.2,0){\vector(0,1){0.9}}
		\put(10.2,2){\circle{0.1}}
		\put(10.5,2){\mbox{$11$}}
		\put(10.2,1){\vector(0,1){0.9}}

\end{picture}
\end{figure}

\begin{figure}[H]
 \unitlength1cm
\begin{picture}(3,2.4)
\thicklines
\put(0,1){Thus, either \qquad	$G\Bigl(\quad $}
	
		\put(5,0){\vector(-1,1){0.9}}
		\put(5,0){\vector(1,1){0.9}}
		\put(6,1){\vector(0,1){0.9}}
        \put(4,1){\vector(1,2){0.9}}
        \put(6,2){\vector(-1,1){0.9}}
		\put(4,1){\circle{0.1}}
		\put(6,2){\circle{0.1}}
        \put(5,3){\circle{0.1}}
		\put(5,0){\circle{0.1}}
		\put(6,1){\circle{0.1}}
		\put(5.3,0){$0\cdots0$}
         \put(5.3,3){$\mathfrak h^m$}
		\put(4.5,1){\mbox{$\mathfrak{g}^m$}}
		\put(6.3,2){\mbox{$\mathfrak{f}^m$}}
		\put(6.3,1){?}
	\put(7,1)
	{$\Bigr) \ =$}
	
         \put(8,1){\vector(1,2){0.9}}
        \put(10,2){\vector(-1,1){0.9}}
		 \put(9,3){\circle{0.1}}
		\put(10,2){\circle{0.1}}
		\put(9,0){\vector(-1,1){0.9}}
		\put(9,0){\vector(1,1){0.9}}
		\put(10,1){\vector(0,1){0.9}}
		\put(9,1){\circle{0.1}}
		\put(10,2){\circle{0.1}}
		\put(9,0){\circle{0.1}}
		\put(10,1){\circle{0.1}}
		\put(9.3,0){$00$}
           \put(9.3,3){$11$}
		\put(8.5,1){\mbox{$00$}}
		\put(10.3,2){\mbox{$01$}}
		\put(10.3,1){$00$}
		
\end{picture}
\end{figure}

\noindent which would contradict (iv) as no $\sigma$-model is equivalent to the $2$-model on the right hand side of the above equation, or we would have a contradiction:

\begin{figure}[H]
 \unitlength1cm
\begin{picture}(0,2)
\thicklines
		\put(0,0){\vector(0,1){0.9}}
		\put(0,1){\vector(0,1){0.9}}
		\put(0,1){\circle{0.1}}
		\put(0,2){\circle{0.1}}
\put(0,0){\circle{0.1}}
		
		\put(0.3,0){$00$}
		\put(0.5,1){\mbox{$01$}}
		\put(0.3,2){\mbox{$11$}}
		\put(1,1){$\thicksim$}

	\put(1.5,1){$G\Bigl( F\Bigl($}
	
		\put(4,0){\vector(-1,1){0.9}}
		\put(4,0){\vector(1,1){0.9}}
		\put(5,1){\vector(-1,1){0.9}}
           \put(3,1){\vector(1,1){0.9}}
		\put(3,1){\circle{0.1}}
		\put(4,2){\circle{0.1}}
		\put(4,0){\circle{0.1}}
		\put(5,1){\circle{0.1}}
		\put(4.3,0){$0\cdots0$}
		\put(3.5,1){\mbox{$\mathfrak{f}^k$}}
		\put(4.3,2){\mbox{$1\cdots1$}}
		\put(5.3,1){$\mathfrak{g}^k$}

	\put(6,1){$\Bigr)\Bigr) \  = \ G\Bigl($}
	
		\put(9,0){\vector(-1,1){0.9}}
		\put(9,0){\vector(1,1){0.9}}
		\put(10,1){\vector(-1,1){0.9}}
        \put(8,1){\vector(1,1){0.9}}
		\put(8,1){\circle{0.1}}
		\put(9,2){\circle{0.1}}
		\put(9,0){\circle{0.1}}
		\put(10,1){\circle{0.1}}
		\put(9.3,0){?}
		\put(8.5,1){\mbox{$\mathfrak{f}^m$}}
		\put(9.3,2){\mbox{$\mathfrak{h}^m$}}
		\put(10.3,1){$\mathfrak{f}^m$}

\put(10.6,1){$\Bigr)\ \thicksim $}

\put(12,0){\vector(0,1){0.9}}
		\put(12,1){\vector(0,1){0.9}}
		\put(12,1){\circle{0.1}}
		\put(12,2){\circle{0.1}}
\put(12,0){\circle{0.1}}
		
		\put(12.3,0){$01$}
		\put(12.5,1){\mbox{$01$}}
		\put(12.3,2){\mbox{$11$}}
\end{picture}
\end{figure}
\end{proof}

By \cite{Ghi5},  there are infinitely many logics in which unification is not finitary, see Figure \ref{NU}, and hence there  are infinitely many  logics with nullary unification. We can  produce other infinite families of logics with nullary unification, see Figure \ref{li}.
 \begin{figure}[H]
\unitlength1cm
\thicklines
\begin{picture}(0,3)

\put(7,0){\vector(-1,1){0.9}}
\put(7,0){\vector(1,1){0.9}}
\put(8,1){\vector(-1,1){0.9}}
\put(6,1){\circle{0.1}}
\put(7,2){\circle{0.1}}
\put(7,0){\circle{0.1}}
\put(8,1){\circle{0.1}}
\put(6,1){\vector(1,1){0.9}}
\put(7,3){\circle{0.1}}
\put(8,3){\circle{0.1}}
\put(7,2){\vector(0,1){0.9}}
\put(7,2){\vector(-1,1){0.9}}
\put(6,3){\circle{0.1}}
\put(9,3){\circle{0.1}}
\put(5,3){\circle{0.1}}
\put(10,3){\circle{0.1}}
\put(4,3){\circle{0.1}}
\put(7,2){\vector(1,1){0.9}}
\put(7,2){\vector(-2,1){1.9}}
\put(7,2){\vector(2,1){1.9}}
\put(7,2){\vector(-3,1){2.9}}
\put(7,2){\vector(3,1){2.9}}
\put(9.5,3){\circle{0.1}}
\put(4.5,3){\circle{0.1}}
\put(10.5,3){\circle{0.1}}
\put(4,3){\circle{0.1}}
\put(3.5,3){\circle{0.1}}
\end{picture}
\caption{$\mathfrak F_2+\mathfrak F_s.$}\label{li}
\end{figure}
\begin{theorem}\label{infty}
The logic $\mathsf L(\mathfrak F_2+\mathfrak F_s)$ has nullary unification, for any $s\geq 1$:
\end{theorem}
\begin{proof} Let $\mathbf F=sm(\{\mathfrak F_2+\mathfrak F_s\})$ and suppose that $\mathsf L(\mathbf F)$ has finitary unification. Note that $\{\mathfrak L_1,\mathfrak L_2,\mathfrak L_3,\mathfrak L_4,\mathfrak R_2,\mathfrak R_2+\}\subseteq \mathbf F$. Take $n=3$ and let $m\geq 3$ be given by Theorem \ref{main}.
Let $k> m$ and $\sigma\colon\{x_1,x_2,x_3\}\to \mathsf{Fm}^{k+2}$ be a substitution defined as follows

 $\begin{array}{rl}
                       \qquad\qquad     \sigma(x_1)= & x_1\\
                            \sigma(x_2)= & x_2 \\
                            \sigma(x_3)= & \bigwedge_{i=3}^{k+2}\Bigl((x_i\to x_1)\lor((x_i\to x_1)\to x_1)\Bigr).
                          \end{array}$

\noindent There are mappings $G:\mathbf{M}^m\to\mathbf{M}^3$ and $F:\mathbf{M}^{k+2}\to\mathbf{M}^m$  fulfilling  the conditions (i)-(v). Since $k> m$, there are $\mathfrak f^{k}\not=\mathfrak g^{k}$ such that (for some $\mathfrak f^m, \mathfrak g^m$)\\

   \unitlength1cm
\begin{picture}(0,1.1)
\thicklines

\put(0,0.5){\mbox{$F\Bigl($}}
\put(1,0){\circle{0.1}}
\put(1.3,1){\mbox{$1\cdots1$}}
\put(1,0){\vector(0,1){0.9}}
\put(1.3,0){\mbox{$01\mathfrak f^{k}$}}
\put(1,1){\circle{0.1}}
\put(2.5,0.5){\mbox{$\Bigr)= F\Bigl($}}
\put(4,0){\circle{0.1}}
\put(4.3,1){\mbox{$1\cdots1$}}
\put(4,0){\vector(0,1){0.9}}
\put(4.3,0){\mbox{$01\mathfrak g^{k}$}}
\put(4,1){\circle{0.1}}
\put(5.3,0.5){\mbox{$\Bigr)= $}}

\put(6,0){\circle{0.1}}
\put(6.3,1){\mbox{$\mathfrak g^m$}}
\put(6,0){\vector(0,1){0.9}}
\put(6.3,0){\mbox{$\mathfrak f^m$}}
\put(6,1){\circle{0.1}}
\put(7.3,0.5){\mbox{,  }}

\put(8,0.5){\mbox{$G\Bigl($}}
\put(9,0){\circle{0.1}}
\put(9.3,1){\mbox{$\mathfrak g^m$}}
\put(9,0){\vector(0,1){0.9}}
\put(9.3,0){\mbox{$\mathfrak f^m$}}
\put(9,1){\circle{0.1}}
\put(10,0.5){$\Bigr)=$}
\put(11.3,1){\mbox{$111$  }}
\put(11,0){\circle{0.1}}
\put(11.4,0.5){.\qquad But}
\put(11.3,0){\mbox{$011$}}
\put(11,0){\vector(0,1){0.9}}
\put(11,1){\circle{0.1}}

\end{picture} \\

 \unitlength1cm
\begin{picture}(3,2.2)
\thicklines
\put(0,1){$\sigma\Bigl($}
\put(2,0){\vector(-1,1){0.9}}
\put(2,0){\vector(1,1){0.9}}
\put(3,1){\vector(-1,1){0.9}}
\put(1,1){\circle{0.1}}
\put(2,2){\circle{0.1}}
\put(2,0){\circle{0.1}}
\put(3,1){\circle{0.1}}
\put(2.3,0){\mbox{$010\cdots0$}}
\put(0.5,1.3){\mbox{$01\mathfrak f^{k}$}}
\put(2.3,2){\mbox{$1\cdots1$}}
\put(2.9,1.3){\mbox{$01\mathfrak g^{k}$}}
\put(1,1){\vector(1,1){0.9}}

\put(4,1){$\Bigr)$}
\put(4.5,1){$=$}
\put(7,0){\vector(-1,1){0.9}}
\put(7,0){\vector(1,1){0.9}}
\put(8,1){\vector(-1,1){0.9}}
\put(6,1){\circle{0.1}}
\put(7,2){\circle{0.1}}
\put(7,0){\circle{0.1}}
\put(8,1){\circle{0.1}}
\put(7.3,0){\mbox{$010$}}
\put(5.4,1){\mbox{$011$}}
\put(7.3,2){\mbox{$111$}}
\put(8.3,1){\mbox{$011$}}
\put(6,1){\line(1,1){0.9}}
\put(6,1){\vector(1,1){0.9}}

\put(9,1){$\thicksim$}
\put(10.3,2){\mbox{$111$  }}
\put(10,0){\circle{0.1}}
\put(10,1){\circle{0.1}}
\put(10.3,0){\mbox{$010$}}
\put(10,0){\vector(0,1){0.9}}
\put(10,1){\vector(0,1){0.9}}
\put(10,1){\circle{0.1}}
\put(10.3,1){\mbox{$011$. \ Then}}
\end{picture}\\

 \unitlength1cm
\begin{picture}(3,2.2)
\thicklines
\put(0,1){$F\Bigl($}
\put(2,0){\vector(-1,1){0.9}}
\put(2,0){\vector(1,1){0.9}}
\put(3,1){\vector(-1,1){0.9}}
\put(1,1){\circle{0.1}}
\put(2,2){\circle{0.1}}
\put(2,0){\circle{0.1}}
\put(3,1){\circle{0.1}}
\put(2.3,0){\mbox{$010\cdots0$}}
\put(0.6,1.4){\mbox{$01\mathfrak f^{k}$}}
\put(2.3,2){\mbox{$1\cdots1$}}
\put(2.9,1.3){\mbox{$01\mathfrak g^{k}$}}
\put(1,1){\vector(1,1){0.9}}

\put(4,1){$\Bigr)$}
\put(4.5,1){$=$}
\put(7,0){\vector(-1,1){0.9}}
\put(7,0){\vector(1,1){0.9}}
\put(8,1){\vector(-1,1){0.9}}
\put(6,1){\circle{0.1}}
\put(7,2){\circle{0.1}}
\put(7,0){\circle{0.1}}
\put(8,1){\circle{0.1}}
\put(7.3,0){\mbox{$\mathfrak h^m$}}
\put(5.4,1){\mbox{$\mathfrak f^m$}}
\put(7.3,2){\mbox{$\mathfrak g^m$}}
\put(8.3,1){\mbox{$\mathfrak f^m$}}
\put(6,1){\line(1,1){0.9}}
\put(6,1){\vector(1,1){0.9}}

\put(9,1){$\thicksim$}
\put(10.3,2){\mbox{$\mathfrak g^m$  }}
\put(10,0){\circle{0.1}}
\put(10,1){\circle{0.1}}
\put(10.3,0){\mbox{$\mathfrak h^m$}}
\put(10,0){\vector(0,1){0.9}}
\put(10,1){\vector(0,1){0.9}}
\put(10,1){\circle{0.1}}
\put(10.3,1){\mbox{$\mathfrak f^m$, }}
\end{picture}\\

 \unitlength1cm
\begin{picture}(3,2.2)
\thicklines
\put(0,1){for some $\mathfrak h^m$, and hence}

\put(6,1){$G\Bigl($}
\put(7.3,2){\mbox{$\mathfrak g^m$  }}
\put(7,0){\circle{0.1}}
\put(7,1){\circle{0.1}}
\put(7.3,0){\mbox{$\mathfrak h^m$}}
\put(7,0){\vector(0,1){0.9}}
\put(7,1){\vector(0,1){0.9}}
\put(7,1){\circle{0.1}}
\put(7.3,1){\mbox{$\mathfrak f^m\ \Bigr) $}}

\put(9,1){$=$}
\put(10.3,2){\mbox{$111$  }}
\put(10,0){\circle{0.1}}
\put(10,1){\circle{0.1}}
\put(10.3,0){\mbox{$010$}}
\put(10,0){\vector(0,1){0.9}}
\put(10,1){\vector(0,1){0.9}}
\put(10,1){\circle{0.1}}
\put(10.3,1){\mbox{$011$. }}
\end{picture}\\

 \unitlength1cm
\begin{picture}(3,2.2)
\thicklines
\put(0,1){On the other hand, we have }

\put(6,1){$\sigma\Bigl($}
\put(7.3,2){\mbox{$1\cdots1$  }}
\put(7,0){\circle{0.1}}
\put(7,1){\circle{0.1}}
\put(7.3,0){\mbox{$00\mathfrak h^{k}$}}
\put(7,0){\vector(0,1){0.9}}
\put(7,1){\vector(0,1){0.9}}
\put(7,1){\circle{0.1}}
\put(7.3,1){\mbox{$01\mathfrak f^{k}\ \Bigr) $}}

\put(9,1){$=$}
\put(10.3,2){\mbox{$111$  }}
\put(10,0){\circle{0.1}}
\put(10,1){\circle{0.1}}
\put(10.3,0){\mbox{$001$}}
\put(10,0){\vector(0,1){0.9}}
\put(10,1){\vector(0,1){0.9}}
\put(10,1){\circle{0.1}}
\put(10.3,1){\mbox{$011$\ . Thus, }}
\end{picture}\\
 \unitlength1cm
\begin{picture}(3,2.5)
\thicklines
\put(0,1){$F\Bigl($}
\put(1.3,2){\mbox{$1\cdots1$  }}
\put(1,0){\circle{0.1}}
\put(1,1){\circle{0.1}}
\put(1.3,0){\mbox{$00\mathfrak h^{k}$}}
\put(1,0){\vector(0,1){0.9}}
\put(1,1){\vector(0,1){0.9}}
\put(1,1){\circle{0.1}}
\put(1.3,1){\mbox{$01\mathfrak f^{k}\ \Bigr) $}}

\put(3,1){$=$}
\put(4.3,2){\mbox{$\mathfrak g^{m}$  }}
\put(4,0){\circle{0.1}}
\put(4,1){\circle{0.1}}
\put(4.3,0){\mbox{$\mathfrak k^{m}$}}
\put(4,0){\vector(0,1){0.9}}
\put(4,1){\vector(0,1){0.9}}
\put(4,1){\circle{0.1}}
\put(4.3,1){\mbox{$\mathfrak f^{m}$\ , for some $\mathfrak k^{m}$, and }}

\put(8,1){$G\Bigl($}
\put(9.3,2){\mbox{$\mathfrak g^{m}$  }}
\put(9,0){\circle{0.1}}
\put(9,1){\circle{0.1}}
\put(9.3,0){\mbox{$\mathfrak k^{m}$}}
\put(9,0){\vector(0,1){0.9}}
\put(9,1){\vector(0,1){0.9}}
\put(9,1){\circle{0.1}}
\put(9.3,1){\mbox{$\mathfrak f^{m}\ \Bigr) $}}

\put(11,1){$=$}
\put(12.3,2){\mbox{$111$  }}
\put(12,0){\circle{0.1}}
\put(12,1){\circle{0.1}}
\put(12.3,0){\mbox{$001$}}
\put(12,0){\vector(0,1){0.9}}
\put(12,1){\vector(0,1){0.9}}
\put(12,1){\circle{0.1}}
\put(12.3,1){\mbox{$011.$}}

\end{picture}\\

\unitlength1cm
\thicklines
\begin{picture}(0,3.5)

\put(3,0){\vector(-1,1){0.9}}
\put(3,0){\vector(1,1){0.9}}
\put(4,1){\vector(-1,1){0.9}}
\put(2,1){\circle{0.1}}
\put(3,2){\circle{0.1}}
\put(3,0){\circle{0.1}}
\put(4,1){\circle{0.1}}
\put(2,1){\vector(1,1){0.9}}

\put(3,2){\vector(0,1){0.9}}
\put(0,3){We conclude}
\put(1,1){$G\Bigl(\mathfrak h^{m}$}
\put(4.2,1){$\mathfrak k^{m}\ \Bigr)$}
\put(3.3,0){$0\cdots0$}
\put(3.2,2){$\mathfrak f^{m}$}
\put(3.2,3){$\mathfrak g^{m}$}

\put(8,0){\vector(-1,1){0.9}}
\put(8,0){\vector(1,1){0.9}}
\put(9,1){\vector(-1,1){0.9}}
\put(7,1){\circle{0.1}}
\put(8,2){\circle{0.1}}
\put(8,0){\circle{0.1}}
\put(9,1){\circle{0.1}}
\put(7,1){\vector(1,1){0.9}}

\put(8,2){\vector(0,1){0.9}}
\put(5.5,1){$= \quad 010$}
\put(9.2,1){$001$}
\put(8.3,0){$000$}
\put(8.2,2){$011$ \quad \quad and note that}
\put(8.2,3){$111$}

\put(12.3,3){\mbox{$111$  }}
\put(12,1){\circle{0.1}}
\put(12,2){\circle{0.1}}
\put(12.3,1){\mbox{$010$}}
\put(12,1){\vector(0,1){0.9}}
\put(12,2){\vector(0,1){0.9}}
\put(12,2){\circle{0.1}}
\put(12.3,2){\mbox{$011$}}
\end{picture}\\

\noindent is a generated submodel of the above $\mathfrak F_2+\mathfrak L_2$-model. But the whole model is p-irreducible and hence the  submodel on $\mathfrak L_3$ must be equivalent, by (iv), to a $\sigma$-model on $+\mathfrak F_r$, for some $r\geq 1$. Thus, we get a model on $+\mathfrak F_r$ such that

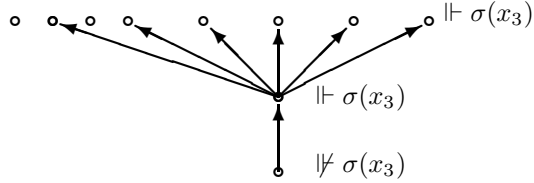
\unitlength1cm
\thicklines
\begin{picture}(0,2.5)

\put(7,1){\circle{0.1}}
\put(7,0){\circle{0.1}}
\put(7,2){\circle{0.1}}
\put(8,2){\circle{0.1}}
\put(9,2){\circle{0.1}}
\put(7,0){\vector(0,1){0.9}}
\put(7,1){\circle{0.1}}
\put(7,1){\vector(0,1){0.9}}
\put(7,1){\vector(-1,1){0.9}}
\put(6,2){\circle{0.1}}
\put(5,2){\circle{0.1}}
\put(4,2){\circle{0.1}}
\put(7,1){\vector(1,1){0.9}}
\put(7,1){\vector(-2,1){1.9}}
\put(7,1){\vector(2,1){1.9}}
\put(7,1){\vector(-3,1){2.9}}
\put(9.2,2){$\Vdash\sigma(x_3)$}
\put(7.5,0.9){$\Vdash\sigma(x_3)$}
\put(7.5,0){$\not\Vdash\sigma(x_3)$}
\put(4.5,2){\circle{0.1}}
\put(4,2){\circle{0.1}}
\put(3.5,2){\circle{0.1}}
\end{picture}\\

\noindent which is impossible.\end{proof}
In the same way one shows that $L(\mathfrak F_r+\mathfrak F_s)$ has nullary unification, for  $r\geq 2$, $s\geq 1$. More than in the number (of logics with nullary unification) we are interested in  their location in the lattice of all intermediate logics. In particular, we would like to know if they can be   put apart from logics with finitary\slash unitary unification; just as logics with finitary  unification are distinguished from unitary logics, by Theorem \ref{kc}.
\begin{theorem}\label{L7} The logic  ${\mathsf L}(\{\mathfrak R_2,\mathfrak F_{2}\})={\mathsf L}(\mathfrak R_2)\cap{\mathsf L}(\mathfrak F_{2})$  has nullary unification.
\end{theorem}
\begin{proof}
Suppose that ${\mathsf L}(\{\mathfrak R_2,\mathfrak F_{2}\})$ has finitary unification and note that
$$sm(\{\mathfrak R_2,\mathfrak F_{2}\})=\{\mathfrak L_1,\mathfrak L_{2},\mathfrak L_3,\mathfrak F_{2},\mathfrak R_2\}.$$
Let $n=2$. According to Theorem \ref{main}, there is a number $m\geq 2$ such that for every $\sigma\colon\{x_1,x_2\}\to\mathsf{Fm}^k$ there are mappings $G:\mathbf{M}^m\to\mathbf{M}^2$ and $F:\mathbf{M}^k\to\mathbf{M}^m$  fulfilling the conditions (i)-(v).
 Take any $k>m$ and let $\sigma\colon\{x_1,x_2\}\to\mathsf{Fm}^k$ be  as follows:
  $$\sigma(x_1)=\bigwedge_{i=1}^k(\neg x_i\lor\neg\neg x_i) \quad \mbox{ and } \quad \sigma(x_2)=\bigwedge_{i=1}^k(\neg\neg x_i\to  x_i).$$
We have the following p-irreducible $\sigma$-models, on   $\mathfrak R_2$ and $\mathfrak F_{2}$, correspondingly:\\

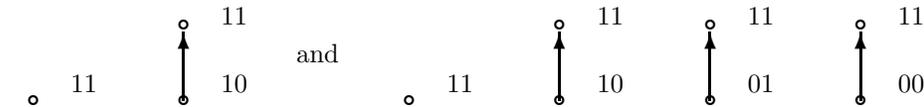
  		\unitlength1cm
\begin{picture}(1,1.5)
\thicklines
\put(0,0){\circle{0.1}}
		\put(0.5,0.1){\mbox{$11$}}
		\put(2,0){\circle{0.1}}
		\put(2.5,0.1){\mbox{$10$}}
		\put(2,1){\circle{0.1}}
		\put(2.5,1){\mbox{$11$}}
       	\put(2,0){\vector(0,1){0.9}}

 \put(3.5,0.5){\mbox{and}}

	\put(5,0){\circle{0.1}}
		\put(5.5,0.1){\mbox{$11$}}
		\put(7,0){\circle{0.1}}
		\put(7.5,0.1){\mbox{$10$}}
		\put(7,1){\circle{0.1}}
		\put(7.5,1){\mbox{$11$}}
       	\put(7,0){\vector(0,1){0.9}}

       \put(9,0){\circle{0.1}}
		\put(9.5,0.1){\mbox{$01$}}
		\put(9,1){\circle{0.1}}
		\put(9.5,1){\mbox{$11$}}
       	\put(9,0){\vector(0,1){0.9}}

       \put(11,0){\circle{0.1}}
		\put(11.5,0.1){\mbox{$00$}}
		\put(11,1){\circle{0.1}}
		\put(11.5,1){\mbox{$11$}}
       	\put(11,0){\vector(0,1){0.9}}
	
	\end{picture}	\\
As $k>m$,  there are  $\mathfrak{f}^k\not=\mathfrak{g}^k$ such that $F_k(\mathfrak L_1, \mathfrak{f}^k)=F_k(\mathfrak L_1, \mathfrak{g}^k)=(\mathfrak L_1, \mathfrak{g}^m)$, for some $\mathfrak{g}^m$. Let $\mathfrak{f}^k\not={0\cdots0}$. Then

		\unitlength1cm
\begin{picture}(1,1.5)
\thicklines
\put(0,0.5){\mbox{$\sigma\Bigl($}}
	\put(0.8,0){\circle{0.1}}
		\put(1,0.1){\mbox{$0\cdots0$}}
		\put(0.8,1){\circle{0.1}}
		\put(1,1){\mbox{$\mathfrak{f}^k$}}
       	\put(0.8,0){\vector(0,1){0.9}}
\put(2,0.5){\mbox{$\Bigr)=$}}
       \put(3,0){\circle{0.1}}
		\put(3.5,0.1){\mbox{$10$}}
		\put(3,1){\circle{0.1}}
		\put(3.5,1){\mbox{$11$}}
       	\put(3,0){\vector(0,1){0.9}}

 \put(4.5,0.5){\mbox{and}}
\put(5.5,0.5){\mbox{$\sigma\bigl($}}
	
		\put(7.5,0){\circle{0.1}}
		\put(8,0.1){\mbox{${0\cdots0}$}}
		\put(6.5,1){\circle{0.1}}
         \put(8.5,1){\circle{0.1}}
		\put(6,1){\mbox{$\mathfrak{f}^k$}}
        \put(8.8,1){\mbox{$\mathfrak{g}^k$}}
		\put(7.5,0){\vector(1,1){0.9}}
         \put(7.5,0){\vector(-1,1){0.9}}
\put(9.2,0.5){\mbox{$\bigr)\thicksim$}}	

   \put(10.5,0){\circle{0.1}}
		\put(11,0.1){\mbox{$00$ \ or \ $01$}}
		\put(10.5,1){\circle{0.1}}
		\put(11,1){\mbox{$11$}}
       	\put(10.5,0){\vector(0,1){0.9}}

	\end{picture}	

		\unitlength1cm
\begin{picture}(1,1.5)
\thicklines
\put(0,1){Let}
\put(0,0.5){\mbox{$F\Bigl($}}
	\put(0.8,0){\circle{0.1}}
		\put(1,0.1){\mbox{$0\cdots0$}}
		\put(0.8,1){\circle{0.1}}
		\put(1,1){\mbox{$\mathfrak{f}^k$}}
       	\put(0.8,0){\vector(0,1){0.9}}
\put(2,0.5){\mbox{$\Bigr)=$}}
       \put(3,0){\circle{0.1}}
		\put(3.5,0.1){\mbox{$\mathfrak{h}^m_1$}}
		\put(3,1){\circle{0.1}}
		\put(3.5,1){\mbox{$\mathfrak{f}^m$}}
       	\put(3,0){\vector(0,1){0.9}}

 \put(4.5,0.5){\mbox{and}}
\put(5.5,0.5){\mbox{$F\bigl($}}
	
		\put(7.5,0){\circle{0.1}}
		\put(8,0.1){\mbox{${0\cdots0}$}}
		\put(6.5,1){\circle{0.1}}
         \put(8.5,1){\circle{0.1}}
		\put(6,1){\mbox{$\mathfrak{f}^k$}}
        \put(8.8,1){\mbox{$\mathfrak{g}^k$}}
		\put(7.5,0){\vector(1,1){0.9}}
         \put(7.5,0){\vector(-1,1){0.9}}
\put(9.2,0.5){\mbox{$\bigr)\thicksim$}}	

   \put(10.5,0){\circle{0.1}}
		\put(11,0.1){\mbox{$\mathfrak{h}^m_2$}}
		\put(10.5,1){\circle{0.1}}
		\put(11,1){\mbox{$\mathfrak{f}^m$}}
       	\put(10.5,0){\vector(0,1){0.9}}

	\end{picture}	\\

\noindent for some $m$-valuations $\mathfrak{h}^m_{1},\mathfrak{h}^m_{2}$. Then, by (v),

		\unitlength1cm
\begin{picture}(1,1.5)
\thicklines
\put(0,0.5){\mbox{$G\Bigl($}}
	\put(0.8,0){\circle{0.1}}
		\put(1,0.1){\mbox{$\mathfrak{h}^m_1$}}
		\put(0.8,1){\circle{0.1}}
		\put(1,1){\mbox{$\mathfrak{f}^m$}}
       	\put(0.8,0){\vector(0,1){0.9}}
\put(2,0.5){\mbox{$\Bigr)=$}}
       \put(3,0){\circle{0.1}}
		\put(3.5,1){\mbox{$11$}}
		\put(3,1){\circle{0.1}}
		\put(3.5,0){\mbox{$10$}}
       	\put(3,0){\vector(0,1){0.9}}

 \put(4.5,0.5){\mbox{and}}
\put(5.5,0.5){\mbox{$G\bigl($}}
	
		  \put(6.5,0){\circle{0.1}}
		\put(6.8,0.1){\mbox{$\mathfrak{h}^m_2$}}
		\put(6.5,1){\circle{0.1}}
		\put(6.8,1){\mbox{$\mathfrak{f}^m$}}
       	\put(6.5,0){\vector(0,1){0.9}}
\put(7.5,0.5){\mbox{$\bigr)=$}}	

   \put(8.5,0){\circle{0.1}}
		\put(9,1){\mbox{$11$}}
		\put(8.5,1){\circle{0.1}}
		\put(9,0){\mbox{$01$ \ or \ $00$}}
       	\put(8.5,0){\vector(0,1){0.9}}

	\end{picture}

$$
\unitlength1cm
\begin{picture}(6,2.1)
\thicklines
\put(0,2){Thus,}
\put(3,2){\circle{0.1}}
\put(2,1){\circle{0.1}}
\put(4,1){\circle{0.1}}
\put(3,0){\circle{0.1}}
\put(4.2,0.9){\mbox{$\mathfrak{h}^m_{2}\quad \Bigr)\quad=$}}
\put(0.5,0.9){\mbox{$G \Bigl(\quad\mathfrak{h}^m_{1}$}}
\put(3.3,0){\mbox{${0\cdots0}$}}
\put(2.2,2){\mbox{$\mathfrak{f}^m$}}
\put(2,1){\vector(1,1){0.9}}
\put(4,1){\vector(-1,1){0.9}}
\put(3,0){\vector(1,1){0.9}}
\put(3,0){\vector(-1,1){0.9}}

\put(8,2){\circle{0.1}}
\put(7,1){\circle{0.1}}
\put(9,1){\circle{0.1}}
\put(8,0){\circle{0.1}}
\put(9.2,0.9){\mbox{$01$ \ or \ $00$}}
\put(6.5,0.9){\mbox{$10$}}
\put(8.3,0){\mbox{$00$}}
\put(8.2,2){$11$ }
\put(7,1){\vector(1,1){0.9}}
\put(9,1){\vector(-1,1){0.9}}
\put(8,0){\vector(1,1){0.9}}
\put(8,0){\vector(-1,1){0.9}}
\end{picture}\qquad\qquad\qquad\qquad\qquad\qquad\qquad\qquad\qquad\qquad\qquad\qquad$$
The above $2$-model on $\mathfrak R_2$  is not equivalent to any $\sigma$-model which contradics (iv).
\end{proof}

Thus, there are  logics with finitary unification the intersection of whose has nullary unification. We could give many  examples  of such logics.
\begin{theorem}\label{L8} The logic  ${\mathsf L}(\mathfrak G_2)\cap{\mathsf L}(\mathfrak F_{3})$ (see Figure \ref{GF} and \ref{8fames}) has nullary unification.
\end{theorem}
\begin{proof} Let  $\mathsf{L}={\mathsf L}(\mathfrak G_2)\cap{\mathsf L}(\mathfrak F_{3})$ and  $\mathbf F=\{\mathfrak L_1, \mathfrak L_2, \mathfrak L_3,\mathfrak F_2,+\mathfrak F_2, \mathfrak G_3, \mathfrak G_2, \mathfrak F_{3}\}$. Then $\mathbf F= sm(\{\mathfrak G_2,\mathfrak F_3\})$ and $\mathsf{L}={\mathsf L}(\mathbf F)$.  Suppose that ${\mathsf L}$ has finitary unification and $n=1$. By Theorem \ref{main},  there is a number $m\geq 1$ such that for any $\sigma\colon\{x_1,\dots,x_n\}\to \mathsf{Fm}^k$  there are  $G:\mathbf{M}^m\to\mathbf{M}^1$ and $F:\mathbf{M}^k\to\mathbf{M}^m$  fulfilling the conditions (i)-(v). Let

$A(x_1,x_2)= (\neg x_1\lor\neg\neg x_1)\lor (\neg x_2\lor\neg\neg x_2)\lor\neg(x_1\leftrightarrow x_2)\lor \neg\neg(x_1\leftrightarrow x_2).$\\
The formula is valid in  $\mathfrak F_2$; to falsify it one needs three end-elements above the root, labeled with distinct valuation. Thus, $A$ is false in the following models over $\mathfrak F_{3}$
\begin{figure}[H]
\unitlength1cm
\begin{picture}(1,1)
\thicklines
\put(0,1){\circle{0.1}}
\put(1,0){\circle{0.1}}
\put(2,1){\circle{0.1}}
\put(1,1){\circle{0.1}}
\put(1.3,0){\mbox{$00$}}
\put(0.3,1){\mbox{$11$}}
\put(1.3,1){\mbox{$01$}}
\put(2.3,1){\mbox{$10$}}
\put(1,0){\vector(1,1){0.9}}
\put(1,0){\vector(-1,1){0.9}}
\put(1,0){\vector(0,1){0.9}}

\put(3.5,1){\circle{0.1}}
\put(4.5,0){\circle{0.1}}
\put(5.5,1){\circle{0.1}}
\put(4.5,1){\circle{0.1}}
\put(4.8,0){\mbox{$00$}}
\put(3.8,1){\mbox{$11$}}
\put(4.8,1){\mbox{$01$}}
\put(5.8,1){\mbox{$00$}}
\put(4.5,0){\vector(1,1){0.9}}
\put(4.5,0){\vector(-1,1){0.9}}
\put(4.5,0){\vector(0,1){0.9}}

\put(7,1){\circle{0.1}}
\put(8,0){\circle{0.1}}
\put(9,1){\circle{0.1}}
\put(8,1){\circle{0.1}}
\put(8.3,0){\mbox{$00$}}
\put(7.3,1){\mbox{$11$}}
\put(8.3,1){\mbox{$00$}}
\put(9.3,1){\mbox{$10$}}
\put(8,0){\vector(1,1){0.9}}
\put(8,0){\vector(-1,1){0.9}}
\put(8,0){\vector(0,1){0.9}}

\put(10,1){\circle{0.1}}
\put(11,0){\circle{0.1}}
\put(12,1){\circle{0.1}}
\put(11,1){\circle{0.1}}
\put(11.3,0){\mbox{$00$}}
\put(10.3,1){\mbox{$00$}}
\put(11.3,1){\mbox{$01$}}
\put(12.3,1){\mbox{$10$}}
\put(11,0){\vector(1,1){0.9}}
\put(11,0){\vector(-1,1){0.9}}
\put(11,0){\vector(0,1){0.9}}

\end{picture}
\end{figure}
\noindent
Let $ m<k=2l$ and  $\sigma\colon\{x_1\}\to\mathsf{Fm^k}$ be defined by
$$ \sigma(x_1)\quad=\quad\neg\neg\Bigl(\bigvee_{i=1}^kx_i\Bigr)\ \land \ \bigwedge_{i=1}^lA(x_{2i-1},x_{2i}).$$
We have the following p-irreducible  $\sigma$-models (or equivalents of $\sigma$-models):
 $$ \circ \ 0 \qquad \circ 1\qquad\unitlength1cm
\begin{picture}(1,1.1)
\thicklines

\put(0,1){\circle{0.1}}
\put(0.3,1){\mbox{$0$}}
\put(2,1){\circle{0.1}}
\put(2.3,1){\mbox{$1$}}
\put(1,0){\circle{0.1}}
\put(1.3,0){\mbox{$0$}}
\put(1,0){\vector(1,1){0.9}}
\put(1,0){\vector(-1,1){0.9}}

\put(3,0){\circle{0.1}}
\put(3.3,0){\mbox{$0$\ \ \ Note ${\sigma}(\circ \ \mathfrak{f}^k)=\circ \ 0$  iff  $\mathfrak{f}^k=\underbrace{0\cdots 0}_{\mbox{$k$-times}}$.}}
\put(3,0){\line(0,1){0.9}}
\put(3,1){\circle{0.1}}
\put(3.8,0.5){\mbox{but it is equivalent with some $\sigma$-model. }}
\put(3.3,1){\mbox{$1$}\quad The last $1$-model is not any $\sigma$-model}
\put(3,0){\vector(0,1){0.9}}\end{picture}\qquad\qquad\qquad\qquad\qquad\qquad\qquad\qquad\qquad\qquad\qquad\qquad$$

 \noindent
Let $F(\circ \ 0\cdots 0)=\circ  \ \mathfrak{g}^m$. Since $m<k$, there are $\mathfrak{f}^k_0\not=\mathfrak{f}^k_{0'}$ such that\\ $F(\circ \ \mathfrak{f}^k_0)=F(\circ \ \mathfrak{f}^k_{0'})$. Then $\mathfrak{f}^k_0(j)\not=\mathfrak{f}^k_{0'}(j)$ for some $j\leq k$. We do not know if  $j$ is odd or even but for pairs of bits we have \ $\mathfrak{f}^k_0(2i-1)\ \mathfrak{f}^k_0(2i)\ \not=\ \mathfrak{f}^k_{0'}(2i-1)\ \mathfrak{f}^k_{0'}(2i)$ \ for some $i\leq l$ (where $2i-1=j$ or $2i=j$). Moreover, we can find a $k$-valuation $\mathfrak{f}^k_{0''}$ such that
$$\mathfrak{f}^k_0(2i-1)\ \mathfrak{f}^k_{0}(2i)\quad\not=\quad\mathfrak{f}^k_{0''}(2i-1)\ \mathfrak{f}^k_{0''}(2i)\quad\not=\quad\mathfrak{f}^k_{0'}(2i-1)\ \mathfrak{f}^k_{0'}(2i)$$
and $\mathfrak{f}^k_{0''}(2i-1)\ \mathfrak{f}^k_{0''}(2i)\ \not=\ 00$.
Let $F(\circ \ \mathfrak{f}^k_0)=F(\circ \ \mathfrak{f}^k_{0'})=\circ  \ \mathfrak{f}^m_0$ and $F(\circ \ \mathfrak{f}^k_{0''})=\circ  \ \mathfrak{f}^m_{0'}$. We  have $\mathfrak{f}^m_0\not=\mathfrak{g}^m$ and $\mathfrak{f}^m_{0'}\not=\mathfrak{g}^m$ as, by (v), $G(\circ \ \mathfrak{f}^m_0)= G(\mathfrak{f}^m_{0'})=\circ \ 1$ and $G(\circ \ \mathfrak{g}^m)=\circ \ 0$.

By (ii), (iv), and the above characterization of all $\sigma$-models, we  conclude that

		\unitlength1cm
\begin{picture}(1,1.5)
\thicklines
\put(0,0.5){\mbox{$G_{} \Bigl($}}
		\put(2,0){\circle{0.1}}
		\put(1.5,0){\mbox{$?$}}
		\put(1,1){\circle{0.1}}
         \put(3,1){\circle{0.1}}
		\put(1.3,1){\mbox{$\mathfrak{f}^m_{0'}$}}
        \put(2.5,1){\mbox{$\mathfrak{f}^m_0$}}
		\put(2,0){\vector(1,1){0.9}}
         \put(2,0){\vector(-1,1){0.9}}
		\put(3.2,0.5){\mbox{$\Bigr) \quad \thicksim\quad$}}
	
		\put(4.7,0){\circle{0.1}}
		\put(5,0){\mbox{$0$}}
		\put(4.7,1){\circle{0.1}}
		\put(5,1){\mbox{$1$}}
		\put(4.7,0){\vector(0,1){0.9}}
	
 \put(6,0.5){\mbox{or}}

	\put(7,0.5){\mbox{$G \Bigl($}}
		\put(9,0){\circle{0.1}}
		\put(8.5,0){\mbox{$?$}}
		\put(8,1){\circle{0.1}}
         \put(10,1){\circle{0.1}}
		\put(8.3,1){\mbox{$\mathfrak{f}^m_{0'}$}}
        \put(9.5,1){\mbox{$\mathfrak{f}^m_0$}}
		\put(9,0){\vector(1,1){0.9}}
         \put(9,0){\vector(-1,1){0.9}}
	
	\put(10.2,0.5){\mbox{$\Bigr) \quad \thicksim\quad$}}
	
		\put(11.7,0){\circle{0.1}}
		\put(12,0){\mbox{$1$}}
		\put(11.7,1){\circle{0.1}}
		\put(12,1){\mbox{$1$}}
		\put(11.7,0){\vector(0,1){0.9}}
	\end{picture}	\\

\noindent (whatever ? is). If the first had happened, we would have \\

		\begin{picture}(5,2)
		\thicklines
		\linethickness{0.3mm}
        \put(0,1){\mbox{$G_{} \Bigl($}}
		\put(2,0){\vector(-1,1){0.9}}
		\put(2,0){\vector(1,1){0.9}}
		\put(3,1){\vector(1,1){0.9}}
        \put(3,1){\vector(-1,1){0.9}}
		\put(1,1){\circle{0.1}}
		\put(2,2){\circle{0.1}}
        \put(4,2){\circle{0.1}}
		\put(2,0){\circle{0.1}}
		\put(3,1){\circle{0.1}}
		\put(2.3,0){?}
		\put(1.5,1){\mbox{$\mathfrak{g}^m$}}
		\put(3.3,2){\mbox{$\mathfrak{f}^m_0$}}
        \put(2.3,2){\mbox{$\mathfrak{f}^m_{0'}$}}
		\put(3.3,1){?}
        \put(4.5,1){\mbox{$\Bigr) \quad = \quad$}}
		
		\put(7,0){\vector(-1,1){0.9}}
		\put(7,0){\vector(1,1){0.9}}
		\put(8,1){\vector(1,1){0.9}}
        \put(8,1){\vector(-1,1){0.9}}
		\put(6,1){\circle{0.1}}
		\put(7,2){\circle{0.1}}
        \put(9,2){\circle{0.1}}
		\put(7,0){\circle{0.1}}
		\put(8,1){\circle{0.1}}
		\put(7.3,0){$0$}
		\put(6.5,1){\mbox{$0$}}
		\put(9.3,2){\mbox{$1$}}
          \put(7.3,2){\mbox{$1$}}
		\put(8.3,1){$0$}
		\end{picture}

\noindent which would contradict (iv) as no $\sigma$-model is equivalent to the $1$-model on the right hand side of the above equation. We conclude that \\

		\unitlength1cm
\begin{picture}(1,1.5)
\thicklines
\put(0,0.5){\mbox{$G_{} \Bigl($}}
		\put(2,0){\circle{0.1}}
		\put(1.5,0){\mbox{$?$}}
		\put(1,1){\circle{0.1}}
         \put(3,1){\circle{0.1}}
		\put(1.3,1){\mbox{$\mathfrak{f}^m_{0'}$}}
        \put(2.5,1){\mbox{$\mathfrak{f}^m_0$}}
		\put(2,0){\vector(1,1){0.9}}
         \put(2,0){\vector(-1,1){0.9}}
	\put(3.2,0.5){\mbox{$\Bigr) \quad =\quad$}}
	
		\put(5.2,0){\circle{0.1}}
		\put(4.8,0){\mbox{$1$}}
		\put(4.2,1){\circle{0.1}}
        \put(6.2,1){\circle{0.1}}
		\put(6.3,1){\mbox{$1$}}
        \put(4.3,1){\mbox{$1$}}
		\put(5.2,0){\vector(1,1){0.9}}
       \put(5.2,0){\vector(-1,1){0.9}}
	\put(8,0.5){\mbox{and hence}}
	\end{picture}

	\unitlength1cm
\begin{picture}(1,2)

\put(0,0.5){\mbox{$G_{} \Bigl( F_{} \Bigl($}}
\put(1.7,1){\circle{0.1}}
\put(2.7,0){\circle{0.1}}
\put(3.7,1){\circle{0.1}}
\put(2.7,1){\circle{0.1}}
\put(3,0){\mbox{$?$}}
\put(2,1){\mbox{$\mathfrak{f}^k_{0''}$}}
\put(3,1){\mbox{$\mathfrak{f}^k_{0'}$}}
\put(4,1){\mbox{$\mathfrak{f}^k_{0}$}}
\put(2.7,0){\vector(1,1){0.9}}
\put(2.7,0){\vector(-1,1){0.9}}
\put(2.7,0){\vector(0,1){0.9}}
\put(4.3,0.5){\mbox{$ \Bigr)\Bigr) =\ G\Bigl($}}

\put(6,1){\circle{0.1}}
\put(7,0){\circle{0.1}}
\put(8,1){\circle{0.1}}
\put(7,1){\circle{0.1}}
\put(7.3,0){\mbox{$?$}}
\put(6.3,1){\mbox{$\mathfrak{f}^m_{0'}$}}
\put(7.3,1){\mbox{$\mathfrak{f}^m_{0}$}}
\put(8.3,1){\mbox{$\mathfrak{f}^m_{0}$}}
\put(7,0){\vector(1,1){0.9}}
\put(7,0){\vector(-1,1){0.9}}
\put(7,0){\vector(0,1){0.9}}
\put(8.8,0.5){\mbox{$\Bigr) \ =$}}

\put(10,1){\circle{0.1}}
\put(11,0){\circle{0.1}}
\put(12,1){\circle{0.1}}
\put(11,1){\circle{0.1}}
\put(11.3,0){\mbox{$1$}}
\put(10.3,1){\mbox{$1$}}
\put(11.3,1){\mbox{$1$}}
\put(12.3,1){\mbox{$1$}}
\put(11,0){\vector(1,1){0.9}}
\put(11,0){\vector(-1,1){0.9}}
\put(11,0){\vector(0,1){0.9}}

\end{picture}\\

But this is in contradiction with (v) as

	\unitlength1cm
\begin{picture}(1,2)

\put(0.5,0.5){\mbox{$H_{\sigma} \Bigl($}}
\put(1.7,1){\circle{0.1}}
\put(2.7,0){\circle{0.1}}
\put(3.7,1){\circle{0.1}}
\put(2.7,1){\circle{0.1}}
\put(3,0){\mbox{$?$}}
\put(2,1){\mbox{$\mathfrak{f}^k_{0''}$}}
\put(3,1){\mbox{$\mathfrak{f}^k_{0'}$}}
\put(4,1){\mbox{$\mathfrak{f}^k_{0}$}}
\put(2.7,0){\vector(1,1){0.9}}
\put(2.7,0){\vector(-1,1){0.9}}
\put(2.7,0){\vector(0,1){0.9}}
\put(4.3,0.5){\mbox{$ \Bigr)\quad = $ }}

\put(6,1){\circle{0.1}}
\put(7,0){\circle{0.1}}
\put(8,1){\circle{0.1}}
\put(7,1){\circle{0.1}}
\put(7.3,0){\mbox{$0$}}
\put(6.3,1){\mbox{$1$}}
\put(7.3,1){\mbox{$1$}}
\put(8.3,1){\mbox{$1$}}
\put(7,0){\vector(1,1){0.9}}
\put(7,0){\vector(-1,1){0.9}}
\put(7,0){\vector(0,1){0.9}}
\end{picture}
\end{proof}
\noindent
 There are, as well, families of  logics with finitary unification closed under intersections. For instance,  ${\mathsf L}(\mathfrak G_2)\cap{\mathsf L}(\mathfrak C_{5})$ has finitary unification, see Theorem \ref{L8i}. The frames $\mathfrak C_{5}$ and $\mathfrak F_{3}$ are quite the same; the logics ${\mathsf L}(\mathfrak C_{5})$ and ${\mathsf L}(\mathfrak F_{3})$ have finitary unification (the second one  even has  projective approximation, see Theorem \ref{lmk}) but it does not mean they generate the same unification types in combination with other frames.

\newpage

\begin{lemma}\label{pil}  If $d\geq 1$, a class  {\bf F} of  finite frames and $\mathfrak F\in\mathbf F$ are such that\\
 \indent (1) $d(\mathfrak F)=d$ \ and \ $d(\mathfrak G)\leq d+1$, for each  $\mathfrak G\in \mathbf F$;\\
 \indent (2) $\mathfrak G\equiv +\mathfrak F$ \ or \ $\mathfrak G\in \mathbf F$, for each p-morphic image $\mathfrak G$ of $+\mathfrak F$;\\
 \indent (3) $\mathsf L({\mathbf F})$  is locally tabular and has finitary\slash unitary unification,\\ then $\mathsf L({\mathbf F}\cup \{+\mathfrak F\})$ is locally tabular and has  finitary\slash unitary unification, as well.\end{lemma}
\begin{proof}  $\mathsf L({\mathbf F}\cup \{+\mathfrak F\})$ is locally tabular, by Corollary \ref{fp}. Assume  ${\mathbf F}={ sm({\mathbf F})}$ and  $+\mathfrak F\not\in \mathbf F$. Take $\mathbf{ M}^k=\mathbf{ M}^k({\mathbf F})$ and $\mathbf{ N}^k=\mathbf{ M}^k(sm({\mathbf F}\cup \{+\mathfrak F\}))$, for each $k\geq 0$. By (2), $sm({\mathbf F}\cup \{+\mathfrak F\})\setminus\mathbf F$ is not empty but it contains  only  isomorphic copies of $+\mathfrak F$.

Let $n\geq 1$ and $\sigma\colon\{x_1,\dots,x_n\}\to \mathsf{Fm^k}$.
There is a number $m\geq 1$ such that  there are  mappings $F\colon \mathbf{ M}^k \to \mathbf{ M}^m$ and $G\colon \mathbf{ M}^m\to \mathbf{ M}^n$ fulfilling the conditions (i)-(v) (of Theorem \ref{main}). We take $m'=m+n+d+1$ and would like to find  mappings  $F'\colon \mathbf{ N}^k_{ir}\to \mathbf{ N}^{m'}$  and  $G'\colon \mathbf{ N}^{m'}_{ir}\to \mathbf{ N}^n$  fulfilling the conditions (i)-(v) of Theorem \ref{main2}.

Let $\mathfrak{M}^k=(W,R,w_0,\{\mathfrak{f}^k_w\}_{w\in W})\in \mathbf{M}^k_{ir}$ and
 $F(\mathfrak{M}^k)=(W,R,w_0,\{\mathfrak{f}^m_w\}_{w\in W})$.  We take $F'(\mathfrak{M}^k)=(W,R,w_0,\{\mathfrak{g}^{m'}_w\}_{w\in W})$, where
$ \mathfrak{g}^{m'}_w=\mathfrak{f}^m_w1\dots1\underbrace{0\dots0}_{d(w)-1}1$.

\noindent We need to define $F'(\mathfrak{M}^k)$ for $\mathfrak{M}^k\in\mathbf{ N}^k_{ir}\setminus\mathbf{M}^k$. Let $\mathfrak{M}^k\in\mathbf{ N}^k_{ir}\setminus\mathbf{M}^k$. Then $\mathfrak{M}^k$ is a $k$-model over  $+\mathfrak{F}$. Let $w_0$ be the root of $+\mathfrak{F}$ and $u$ its immediate successor (the root of $\mathfrak{F}$). Then $(\mathfrak{M}^k)_u$ is a p-irreducible model in $\mathbf{ M}^k$, see Theorem \ref{pM6}, and hence we have  $(F'(\mathfrak{M}^k))_u$. Thus, we only need the valuation  $\mathfrak g^{m'}_{w_0}$ at $w_0$ (to get $F'(\mathfrak{M}^k)$).  Let
$\mathfrak h^{n}_{w_0}$ be the valuation at $w_0$ in $\sigma(\mathfrak{M}^k)$. We take $\mathfrak{g}^{m'}_{w_0}=\underbrace{0\dots0}_{m}\mathfrak{h}^n_{w_0}1\underbrace{0\dots0}_{d}.$

The definition of $F'\colon {\mathbf N}^k_{ir}\to \mathbf{N}^{m+n+d+1}$ is  completed and the conditions (i)-(ii) of Theorem \ref{main2} are obviously fulfilled, as far as only $F'$ is concerned. As concerns (iii), if $i\colon \mathfrak{M}^k\to\mathfrak{N}^k$ is an isomorphism of $k$-models in $\mathbf{ M^k}$, then $i(w)$ and $w$ (for any $w$ in the domain of $\mathfrak{M}^k$) have the same depth and hence $ \mathfrak{g}^{m'}_w= \mathfrak{g}^{m'}_{i(w)}$ which shows that $i$ is also an isomorphism  $F'(\mathfrak{M}^k)$ and $F'(\mathfrak{N}^k)$.

The definition of  $G'\colon \mathbf{ N}^{m'}_{ir}\to \mathbf{ N}^{n}$ involves three cases.\\
(A) Suppose that  $\mathfrak{M}^{m'}\in\mathbf{M}^{m'}_{ir}$ is not equivalent with $F'(\mathfrak{N}^{k})$, for any $\mathfrak{N}^k\in\mathbf{ N}^k_{ir}\setminus\mathbf{M}^k$. Then we take
$G'(\mathfrak{M}^{m'})=G(\mathfrak{M}^{m'}\!\!\!\upharpoonright m)$. The mapping $F'$ preserves the depth of the frame.  Thus, $d(F'(\mathfrak{N}^{k}))=d+1$ for any $\mathfrak{N}^k\in\mathbf{ N}^k_{ir}\setminus\mathbf{M}^k$ and it  means that any generated submodel  of $\mathfrak{M}^{m'}$ is not equivalent with $F'(\mathfrak{N}^{k})$, for any $\mathfrak{N}^k\in\mathbf{ N}^k_{ir}\setminus\mathbf{M}^k$, either. So, we can  show (ii). The conditions (i) and (iii) are  obvious.
We also have $G'(\mathfrak{M}^{m'})\thicksim
\sigma(\mathfrak{M}^{k})$, for some $\mathfrak{M}^{k}$, as $G$ fulfills (iv).

To show the condition  (v) (of Theorem \ref{main2}), let us suppose  $\mathfrak{M}^{m'}\thicksim F'(\mathfrak{M}^{k})$, for some $\mathfrak{M}^k\in\mathbf{M}^k$. Then $\mathfrak{M}^{m'}\!\!\!\!\upharpoonright m\thicksim F'(\mathfrak{M}^{k})\!\!\!\!\upharpoonright m=F(\mathfrak{M}^{k})$ and hence we have $G'(\mathfrak{M}^{m'})=G(\mathfrak{M}^{m'}\!\!\!\!\upharpoonright m)=G(F(\mathfrak{M}^{k}))\thicksim\sigma(\mathfrak{M}^{k}).$\\
(B) Suppose $\mathfrak{N}^{m'}\in\mathbf{N}^{m'}_{ir}\setminus\mathbf{M}^{m'}$ is not equivalent with $F'(\mathfrak{N}^{k})$, for any $\mathfrak{N}^k\in\mathbf{ N}^k_{ir}\setminus\mathbf{M}^k$. Then $\mathfrak{N}^{m'}$ is a model over $+\mathfrak F$. Let $w_0$ be the root of $\mathfrak{F}$ and $u$ its immediate successor; that is  $u$ is the root of $\mathfrak F$. By (A), we have $G'((\mathfrak{N}^{m'})_u)=G((\mathfrak{N}^{m'}\!\!\!\upharpoonright m)_u)$. Let $\mathfrak f^{n}_{u}$ be the valuation at the root $u$ of $G'((\mathfrak{N}^{m'})_u)$. To extend $G'((\mathfrak{N}^{m'})_u)$ to a model over $+\mathfrak F$ , it suffices to define  the valuation $\mathfrak f^{n}_{w_0}$ at $w_0$. Regardless of the choice,   the conditions (i)-(iii) of Theorem \ref{main2} are always fulfilled. We take $\mathfrak f^{n}_{w_0}=\mathfrak f^{n}_{u}$. Then
$G'(\mathfrak{N}^{m'})\thicksim G'((\mathfrak{N}^{m'})_u)\thicksim\sigma(\mathfrak{M}^{k})$, for some $\mathfrak{M}^{k}$, which means the condition (iv) of Theorem \ref{main2} is also fulfilled and (v) is irrelevant in this case.\\
(C)
Suppose that $\mathfrak{M}^{m'}$ is p-irreducible and \  $\mathfrak{M}^{m'}\thicksim F'(\mathfrak{N}^{k})$, for some $\mathfrak{N}^k\in\mathbf{ N}^k_{ir}\setminus\mathbf{M}^k$. Then there is a p-morphism $p\colon F'(\mathfrak{N}^{k})\to \mathfrak{M}^{m'}$. We can assume $\mathfrak{N}^{k}$ is p-irreducible and $\mathfrak{N}^{k}$ is a model over $+\mathfrak F$.
Let $w_0$ be the root of $\mathfrak{N}^{k}$ and $u$ its immediate successor. Then $p(w_0)$ is the root of $\mathfrak{M}^{m'}$ and $p(u)$ its only immediate successor (in $\mathfrak{M}^{m'}$).  By the definition of $F'$, we have  $F'((\mathfrak{N}^{k})_u)\!\!\upharpoonright m=F((\mathfrak{N}^{k})_u)$ and hence
$G'((\mathfrak{M}^{m'})_{p(u)})=G(((\mathfrak{M}^{m'})_{p(u)})\!\!\upharpoonright m) \thicksim G((F'((\mathfrak{N}^{k})_u))\!\!\upharpoonright m)=G(F((\mathfrak{N}^{k})_u))\thicksim( \sigma(\mathfrak{N}^{k}))_u.$

To get $G'(\mathfrak{M}^{m'})$ we need the valuation $\mathfrak f^{n}_{p(w_0)}$ at $p(w_0)$ (in $G'(\mathfrak{M}^{m'})$). Since p-morphisms preserve the valuations we must have (in $\mathfrak{M}^{m'}$ at $p(w_0)$) the valuation
$\mathfrak{g}^{m'}_{w_0}={0\dots0}\mathfrak{h}^n_{w_0}1{0\dots0},$ where $\mathfrak h^{n}_{w_0}$ is the valuation at $w_0$ in $\sigma(\mathfrak{N}^k)$. It means, in particular, that for each $\mathfrak{N}^{k}$ such that $\mathfrak{M}^{m'}\thicksim F'(\mathfrak{N}^{k})$ we have the same valuation $\mathfrak h^{n}_{w_0}$  at $w_0$ in $\sigma(\mathfrak{N}^k)$. Then we take $\mathfrak f^{n}_{p(w_0)}=\mathfrak h^{n}_{w_0}$ to complete our definition of $G'(\mathfrak{M}^{m'})$. Since the valuation $\mathfrak f^{n}_{p(u)}$ at $p(u)$ (in $G'(\mathfrak{M}^{m'})$) must be $\mathfrak h^{n}_{u}$  (the valuation at $u$ in $\sigma(\mathfrak{M}^k)$), the monotonicity condition is fulfilled and our definition is correct. One checks (i)-(iii) of Theorem \ref{main2} and $G'(\mathfrak{M}^{m'})\thicksim\sigma(\mathfrak{N}^k)$ which gives (iv)-(v).
\end{proof}
Our lemma is not sufficient to show  $\mathsf L(+\mathfrak F )$ has finitary\slash unitary unification if $\mathsf L(\mathfrak F)$ does; it applies only to some  frames $\mathfrak F$. But using it, we can show that $\mathsf L(\{\mathfrak F,\mathfrak L_n\} )$ has finitary\slash unitary unification, for any $n\geq 1$, if $\mathsf L(\mathfrak F)$ does. We have $\mathsf L(\{\mathfrak F,\mathfrak L_n\})=\mathsf L(\mathfrak F)\cap\mathsf L(\mathfrak L_n)$ and $\mathsf L(\mathfrak L_n)$ has projective unification. Extending our argument we get:
\begin{theorem}\label{itpr} If {\sf L} is a locally tabular intermediate logics with finitary\slash unitary unification and {\sf L'} is projective, then their intersection $\mathsf L\cap\mathsf L'$ has finitary\slash unitary unification.
\end{theorem}

It follows from Theorem \ref{L8} that we cannot replace the assumption '$\mathsf{L'}$ is projective` with '$\mathsf{L'}$ has projective approximation`.This does not exclude that the intersection of two logics with projective approximation is a logic with finitary unification.

\begin{theorem}\label{L10}    ${\mathsf L}(\mathfrak Y_{2})$, \ ${\mathsf L}(\mathfrak Y_{2}+)$, \ ${\mathsf L}(\mathfrak Y_{3})$, and ${\mathsf L}(\mathfrak Y_{3}+)$ (see Figure \ref{nu}) have nullary unification.
\end{theorem}
\begin{figure}[H]
\unitlength1cm
\thicklines
\begin{picture}(0,2.5)

\put(0,0){$\mathfrak{Y}_2$:}
\put(1,0){\vector(-1,1){0.9}}
\put(1,0){\vector(1,1){0.9}}
\put(2,1){\vector(1,1){0.9}}
\put(0,1){\circle{0.1}}
\put(1,2){\circle{0.1}}
\put(1,0){\circle{0.1}}
\put(2,1){\circle{0.1}}
\put(3,2){\circle{0.1}}
\put(0,1){\vector(1,1){0.9}}
\put(2,1){\vector(-1,1){0.9}}

\put(2.8,0){$\mathfrak{Y}_2+$:}
\put(4.3,0){\vector(-1,1){0.9}}
\put(4.3,0){\vector(1,1){0.9}}
\put(5.3,1){\vector(1,1){0.9}}
\put(3.3,1){\circle{0.1}}
\put(4.3,2){\circle{0.1}}
\put(4.3,0){\circle{0.1}}
\put(5.3,1){\circle{0.1}}
\put(6.3,2){\circle{0.1}}
\put(3.3,1){\vector(1,1){0.9}}
\put(5.3,1){\vector(-1,1){0.9}}
\put(5.3,3){\circle{0.1}}
\put(4.3,2){\vector(1,1){0.9}}
\put(6.3,2){\vector(-1,1){0.9}}

\put(6.2,0){$\mathfrak{Y}_3$:}
\put(6.7,2){\circle{0.1}}
\put(8.7,2){\circle{0.1}}
\put(6.7,1){\circle{0.1}}
\put(8.7,1){\circle{0.1}}
\put(7.7,0){\circle{0.1}}
\put(6.7,1){\vector(0,1){0.9}}
\put(8.7,1){\vector(0,1){0.9}}
\put(6.7,1){\vector(2,1){1.9}}
\put(8.7,1){\vector(-2,1){1.9}}
\put(7.7,0){\vector(1,1){0.9}}
\put(7.7,0){\vector(-1,1){0.9}}

\put(9.5,0){${\mathfrak{Y}_3}+$:}
\put(10,2){\circle{0.1}}
\put(12,2){\circle{0.1}}
\put(10,1){\circle{0.1}}
\put(12,1){\circle{0.1}}
\put(11,0){\circle{0.1}}
\put(11,3){\circle{0.1}}
\put(10,1){\vector(0,1){0.9}}
\put(12,1){\vector(0,1){0.9}}
\put(10,1){\vector(2,1){1.9}}
\put(12,1){\vector(-2,1){1.9}}
\put(11,0){\vector(1,1){0.9}}
\put(11,0){\vector(-1,1){0.9}}
\put(10,2){\vector(1,1){0.9}}
\put(12,2){\vector(-1,1){0.9}}
\end{picture}
\caption{Frames of Logics with Nullary Unification.}\label{nu}
\end{figure}
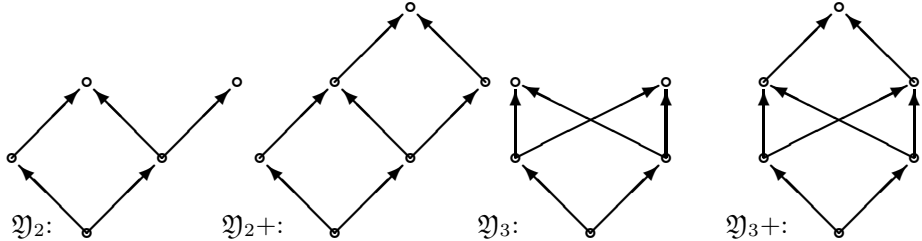\vspace{-0.5cm}
\begin{proof} Let $\mathbf{F}=\{\mathfrak L_1,\mathfrak L_2,\mathfrak L_3,\mathfrak F_{2},\mathfrak R_{2},+\mathfrak F_{2},\mathfrak Y_{2}\}$. Then $\mathbf{F}=sm(\mathfrak Y_{2})$.
Assume  ${\mathsf L}(\mathbf F)$ has finitary unification.  Let $n=1$. By Theorem \ref{main},  there is a number $m\geq 1$ such that for any $\sigma\colon\{x_1\}\to \mathsf{Fm}^k$  there are mappings  $G:\mathbf{M}^m\to\mathbf{M}^1$ and $F:\mathbf{M}^k\to\mathbf{M}^m$  fulfilling the conditions (i)-(v). Take any $k>m$ and let $$ \sigma(x_1)=\bigl(\neg\neg \bigvee_{i+1}^kx_i\bigr) \ \land \ \bigwedge_{i=1}^k(\neg\neg x_i\lor\neg x_i). $$
There are only four p-irreducible $1$-models equivalent with some $\sigma$-models. They are:
 $$ \circ \ 0 \qquad \circ 1\qquad\unitlength1cm
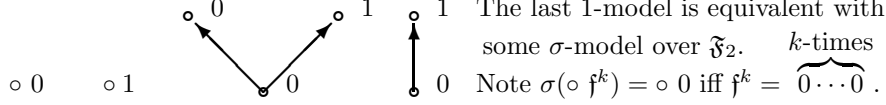
\begin{picture}(1,1.1)
\thicklines

\put(0,1){\circle{0.1}}
\put(0.3,1){\mbox{$0$}}
\put(2,1){\circle{0.1}}
\put(2.3,1){\mbox{$1$}}
\put(1,0){\circle{0.1}}
\put(1.3,0){\mbox{$0$}}
\put(1,0){\vector(1,1){0.9}}
\put(1,0){\vector(-1,1){0.9}}

\put(3,0){\circle{0.1}}
\put(3.3,0){\mbox{$0$\ \ \ Note $\sigma(\circ \ \mathfrak{f}^k)=\circ \ 0$  iff  $\mathfrak{f}^k=\overbrace{0\cdots 0}^{\mbox{$k$-times}}$.}}
\put(3,1){\circle{0.1}}
\put(3.8,0.5){\mbox{ some $\sigma$-model over $\mathfrak{F}_2$. }}
\put(3.3,1){\mbox{$1$}\quad The last $1$-model is equivalent with}
\put(3,0){\vector(0,1){0.9}}\end{picture}\qquad\qquad\qquad\qquad\qquad\qquad\qquad\qquad\qquad\qquad\qquad\qquad$$
Let $F(\circ \ 0\cdots0)=\circ \ \mathfrak g^m$, for some $\mathfrak g^m$. Since $k>m$, there are $k$-valuations  $\mathfrak{f}^k\not=\mathfrak{g}^k$ such that $F(\circ \ \mathfrak{f}^k)=F(\circ \ \mathfrak{g}^k)=\circ\ \mathfrak{f}^m$, for some $\mathfrak{f}^m$, and $\mathfrak{f}^k\not={0\cdots0}\not=\mathfrak{g}^k$. By (v),  $G(\circ \ \mathfrak g^m)=\circ \ 0$ and $G(\circ \ \mathfrak f^m)=\circ \ 1$.
Thus, by (ii) and (iv), for any $m$-valuation $?$

		\unitlength1cm
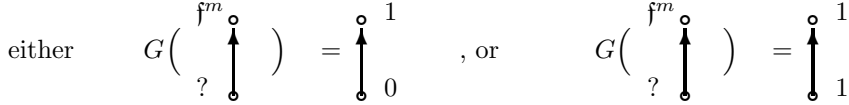
\begin{picture}(1,1.5)
\thicklines
\put(0,0.5){\mbox{either}}
\put(1.8,0.5){\mbox{$G_{} \Bigl($}}
		\put(3,0){\circle{0.1}}
		\put(2.5,0){\mbox{$?$}}
		         \put(3,1){\circle{0.1}}
		\put(2.5,1){\mbox{$\mathfrak{f}^m$}}
        		\put(3,0){\vector(0,1){0.9}}

	\put(3.5,0.5){\mbox{$\Bigr) \quad =\quad$}}
	
		\put(4.7,0){\circle{0.1}}
		\put(5,0){\mbox{$0$}}
		\put(4.7,1){\circle{0.1}}
		\put(5,1){\mbox{$1$}}
		\put(4.7,0){\vector(0,1){0.9}}
	 \put(6,0.5){\mbox{, or}}

	\put(7.8,0.5){\mbox{$G \Bigl($}}
		\put(9,0){\circle{0.1}}
		\put(8.5,0){\mbox{$?$}}
		\put(9,1){\circle{0.1}}
         		\put(8.5,1){\mbox{$\mathfrak{f}^m$}}
		\put(9,0){\vector(0,1){0.9}}

	\put(9.5,0.5){\mbox{$\Bigr) \quad =\quad$}}
	
		\put(10.7,0){\circle{0.1}}
		\put(11,0){\mbox{$1$}}
		\put(10.7,1){\circle{0.1}}
		\put(11,1){\mbox{$1$}}
		\put(10.7,0){\vector(0,1){0.9}}
	\end{picture}

\noindent However, if the first happened, we would have \\

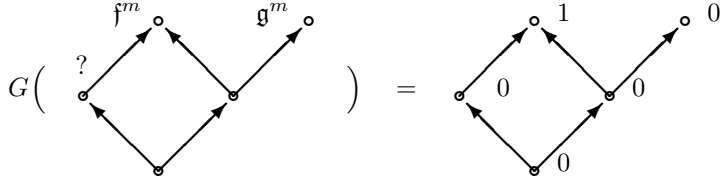
		\begin{picture}(5,2)
		\thicklines
		\linethickness{0.3mm}
        \put(0,1){\mbox{$G_{} \Bigl($}}
		\put(2,0){\vector(-1,1){0.9}}
		\put(2,0){\vector(1,1){0.9}}
		\put(3,1){\vector(1,1){0.9}}
        \put(3,1){\vector(-1,1){0.9}}
         \put(1,1){\vector(1,1){0.9}}
		\put(1,1){\circle{0.1}}
		\put(2,2){\circle{0.1}}
        \put(4,2){\circle{0.1}}
		\put(2,0){\circle{0.1}}
		\put(3,1){\circle{0.1}}
		\put(0.9,1.3){\mbox{$?$}}
		\put(3.3,2){\mbox{$\mathfrak{g}^m$}}
        \put(1.4,2){\mbox{$\mathfrak{f}^m$}}
        \put(4.5,1){\mbox{$\Bigr) \quad = \quad$}}
		
		\put(7,0){\vector(-1,1){0.9}}
		\put(7,0){\vector(1,1){0.9}}
        \put(6,1){\vector(1,1){0.9}}
		\put(8,1){\vector(1,1){0.9}}
        \put(8,1){\vector(-1,1){0.9}}
		\put(6,1){\circle{0.1}}
		\put(7,2){\circle{0.1}}
        \put(9,2){\circle{0.1}}
		\put(7,0){\circle{0.1}}
		\put(8,1){\circle{0.1}}
		\put(7.3,0){$0$}
		\put(6.5,1){\mbox{$0$}}
		\put(9.3,2){\mbox{$0$}}
          \put(7.3,2){\mbox{$1$}}
		\put(8.3,1){$0$}
		\end{picture}

\noindent which would contradict (iv) as the received $1$-model is  not equivalent to any $\sigma$-model. We conclude the second must happen and then we get the contradiction:\\

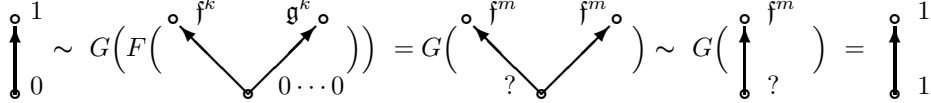
		\unitlength1cm
\begin{picture}(1,1)
\thicklines
\put(0,0){\circle{0.1}}
		\put(0.2,0){\mbox{$0$}}
		\put(0,1){\circle{0.1}}
		\put(0.2,1){\mbox{$1$}}
		\put(0,0){\vector(0,1){0.9}}

\put(0.5,0.5){\mbox{$\thicksim \ G \Bigl(F \Bigl($}}
		\put(3.1,0){\circle{0.1}}
		\put(3.5,0){\mbox{$0\cdots0$}}
		\put(2.1,1){\circle{0.1}}
         \put(4.1,1){\circle{0.1}}
		\put(2.4,1){\mbox{$\mathfrak{f}^k$}}
        \put(3.6,1){\mbox{$\mathfrak{g}^k$}}
		\put(3.1,0){\vector(1,1){0.9}}
         \put(3.1,0){\vector(-1,1){0.9}}
	
	\put(4.4,0.5){\mbox{$\Bigr)\Bigr) \ =\ $}}

	\put(5.4,0.5){\mbox{$G \Bigl($}}
		\put(7,0){\circle{0.1}}
		\put(6.5,0){\mbox{$?$}}
		\put(6,1){\circle{0.1}}
         \put(8,1){\circle{0.1}}
		\put(6.3,1){\mbox{$\mathfrak{f}^m$}}
        \put(7.5,1){\mbox{$\mathfrak{f}^m$}}
		\put(7,0){\vector(1,1){0.9}}
         \put(7,0){\vector(-1,1){0.9}}
	
	\put(8.2,0.5){\mbox{$\Bigr)  \thicksim $}}
	\put(9,0.5){\mbox{$G \Bigl($}}
		\put(9.7,0){\circle{0.1}}
		\put(10,0){\mbox{$?$}}
		\put(9.7,1){\circle{0.1}}
		\put(10,1){\mbox{$\mathfrak{f}^m$}}
		\put(9.7,0){\vector(0,1){0.9}}
\put(10.6,0.5){\mbox{$\Bigr) \ =\ $}}

		\put(11.7,0){\circle{0.1}}
		\put(12,0){\mbox{$1$}}
		\put(11.7,1){\circle{0.1}}
		\put(12,1){\mbox{$1$}}
		\put(11.7,0){\vector(0,1){0.9}}
	\end{picture}	\\

If $\mathbf{F}=\{\mathfrak L_1,\mathfrak L_2,\mathfrak L_3,\mathfrak F_{2},\mathfrak R_{2},+\mathfrak F_{2},\mathfrak Y_{3}\}$, then $\mathbf{F}=sm(\mathfrak Y_{3})$. Suppose that ${\mathsf L}(\mathbf F)$ has finitary unification.  Take $n=2$. By Theorem \ref{main},  there is a number $m\geq 1$ such that for any $\sigma\colon\{x_1,x_2\}\to \mathsf{Fm}^k$  there are mappings  $G:\mathbf{M}^m\to\mathbf{M}^2$ and $F:\mathbf{M}^k\to\mathbf{M}^m$  fulfilling the conditions (i)-(v). Take any $k>m$ and let
$$ \sigma(x_1)=\bigwedge_{i=1}^k(\neg\neg x_i\lor\neg x_i)\quad\mbox{and}\quad \sigma(x_2)=\bigwedge_{i=1}^k(\neg\neg x_i\lor\neg x_i)\rightarrow\Bigl(\bigvee_{i=1}^kx_i\lor\neg \bigvee_{i=1}^kx_i\Bigr). $$Note that $ \sigma(x_1)\land\sigma(x_2)$ is valid in  $\mathfrak L_{1}$. Let us prove $ \sigma$ is a unifier for $x_1\lor x_2$ (in ${\mathsf L}(\mathfrak Y_{3})$). Namely, let be given any $k$-model over $\mathfrak Y_{3}$:
$$
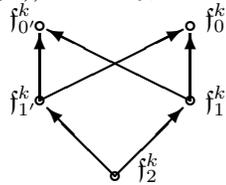
\unitlength1cm
\begin{picture}(6,2.1)
\thicklines

\put(1,2){\circle{0.1}}
\put(3,2){\circle{0.1}}
\put(1,1){\circle{0.1}}
\put(3,1){\circle{0.1}}
\put(2,0){\circle{0.1}}
\put(3.2,0.9){\mbox{$\mathfrak{f}^k_{1}$}\quad\ satisfied only at the end nodes $0$ and $0'$. But it means that}
\put(0.6,0.9){\mbox{$\mathfrak{f}^k_{1'}$}}
\put(2.3,0){\mbox{$\mathfrak{f}^k_{2}$}}
\put(0.6,2){\mbox{$\mathfrak{f}^k_{0'}$}}
\put(3.2,2){\mbox{$\mathfrak{f}^k_{0}$} \quad There are two cases to consider. If $\mathfrak{f}^k_{0'}=\mathfrak{f}^k_{0'}$, then $\sigma(x_1)$ is  }
\put(4,1.5){\mbox{true in the model. Suppose that $\mathfrak{f}^k_{0'}\not=\mathfrak{f}^k_{0'}$. Then $\sigma(x_1)$ is  }}
\put(3.9,0.4){\mbox{ $\sigma(x_2)$ is true in the model. In any case, $\sigma(x_1\lor x_2)$ is }}
\put(3.9,-0.1){\mbox{ satisfied at any node of the model.}}

\put(1,1){\vector(0,1){0.9}}
\put(3,1){\vector(0,1){0.9}}
\put(1,1){\vector(2,1){1.9}}
\put(3,1){\vector(-2,1){1.9}}
\put(2,0){\vector(1,1){0.9}}
\put(2,0){\vector(-1,1){0.9}}
\end{picture}\qquad\qquad\qquad\qquad\qquad\qquad\qquad\qquad\qquad\qquad\qquad\qquad$$
As $m<k$, there are $k$-valuations  $\mathfrak{f}^k\not=\mathfrak{g}^k$ such that $F(\circ \ \mathfrak{f}^k)=F(\circ \ \mathfrak{g}^k)=\circ\ \mathfrak{g}^m$, for some $\mathfrak{g}^m$. Let $\mathfrak{f}^k\not=\underbrace{0\cdots0}_{\mbox{$k$-times}}$ and  consider the following $k$-models over $\mathfrak{F}_2$:

		\unitlength1cm
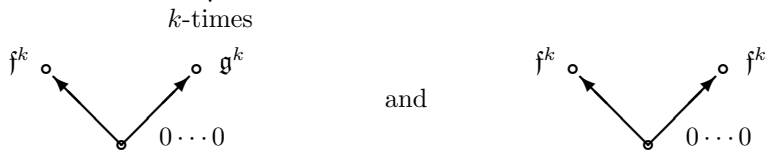
\begin{picture}(1,1.5)
\thicklines

		\put(2,0){\circle{0.1}}
		\put(2.5,0){\mbox{${0\cdots0}$}}
		\put(1,1){\circle{0.1}}
         \put(3,1){\circle{0.1}}
		\put(0.5,1){\mbox{$\mathfrak{f}^k$}}
        \put(3.3,1){\mbox{$\mathfrak{g}^k$}}
		\put(2,0){\vector(1,1){0.9}}
         \put(2,0){\vector(-1,1){0.9}}

 \put(5.5,0.5){\mbox{and}}

		\put(9,0){\circle{0.1}}
		\put(9.5,0){\mbox{${0\cdots0}$}}
		\put(8,1){\circle{0.1}}
         \put(10,1){\circle{0.1}}
		\put(7.5,1){\mbox{$\mathfrak{f}^k$}}
        \put(10.3,1){\mbox{$\mathfrak{f}^k$}}
		\put(9,0){\vector(1,1){0.9}}
         \put(9,0){\vector(-1,1){0.9}}
	
	\end{picture}	

\noindent Note that $\sigma(x_1)$ is false in the first model and $\sigma(x_2)$ in the second. Let

		\unitlength1cm
\begin{picture}(1,1.5)
\thicklines
\put(-0.3,0.5){\mbox{$F \Bigl($}}
		\put(1.5,0){\circle{0.1}}
		\put(1.8,0){\mbox{${0\cdots0}$}}
		\put(0.5,1){\circle{0.1}}
         \put(2.5,1){\circle{0.1}}
		\put(0.8,1){\mbox{$\mathfrak{f}^k$}}
        \put(2,1){\mbox{$\mathfrak{g}^k$}}
		\put(1.5,0){\vector(1,1){0.9}}
         \put(1.5,0){\vector(-1,1){0.9}}
	\put(2.8,0.5){\mbox{$\Bigr)=$}}
	\put(4.5,0){\circle{0.1}}
		\put(4.9,0){\mbox{$\mathfrak{f}^m_{1}$}}
		\put(3.5,1){\circle{0.1}}
         \put(5.5,1){\circle{0.1}}
		\put(3.8,1){\mbox{$\mathfrak{g}^m$}}
        \put(5,1){\mbox{$\mathfrak{g}^m$}}
		\put(4.5,0){\vector(1,1){0.9}}
         \put(4.5,0){\vector(-1,1){0.9}}
 \put(5.7,0.5){\mbox{;}}

	\put(6.2,0.5){\mbox{$F \Bigl($}}
		\put(8,0){\circle{0.1}}
		\put(8.3,0){\mbox{${0\cdots0}$}}
		\put(7,1){\circle{0.1}}
         \put(9,1){\circle{0.1}}
		\put(7.3,1){\mbox{$\mathfrak{f}^k$}}
        \put(8.5,1){\mbox{$\mathfrak{f}^k$}}
		\put(8,0){\vector(1,1){0.9}}
         \put(8,0){\vector(-1,1){0.9}}
		\put(9.2,0.5){\mbox{$\Bigr)=$}}
	\put(11,0){\circle{0.1}}
		\put(11.3,0){\mbox{$\mathfrak{f}^m_{1'}$}}
		\put(10,1){\circle{0.1}}
         \put(12,1){\circle{0.1}}
		\put(10.3,1){\mbox{$\mathfrak{g}^m$}}
        \put(11.5,1){\mbox{$\mathfrak{g}^m$}}
		\put(11,0){\vector(1,1){0.9}}
         \put(11,0){\vector(-1,1){0.9}}
	\end{picture}	\\

\noindent for some $m$-valuations $\mathfrak{f}^m_{1},\mathfrak{f}^m_{1'}$. Then, by (v), we get

		\unitlength1cm
\begin{picture}(1,1.5)
\thicklines
\put(-0.3,0.5){\mbox{$G \Bigl($}}
		\put(1.5,0){\circle{0.1}}
		\put(1.9,0.1){\mbox{$\mathfrak{f}^m_{1}$}}
		\put(0.5,1){\circle{0.1}}
         \put(2.5,1){\circle{0.1}}
		\put(0.8,1){\mbox{$\mathfrak{g}^m$}}
        \put(2,1){\mbox{$\mathfrak{g}^m$}}
		\put(1.5,0){\vector(1,1){0.9}}
         \put(1.5,0){\vector(-1,1){0.9}}
	\put(2.8,0.5){\mbox{$\Bigr)=$}}
	\put(4.5,0){\circle{0.1}}
		\put(4.9,0){\mbox{$01$}}
		\put(3.5,1){\circle{0.1}}
         \put(5.5,1){\circle{0.1}}
		\put(3.8,1){\mbox{$11$}}
        \put(5,1){\mbox{$11$}}
		\put(4.5,0){\vector(1,1){0.9}}
         \put(4.5,0){\vector(-1,1){0.9}}
 \put(5.7,0.5){\mbox{;}}

	\put(6,0.5){\mbox{$G \Bigl($}}
		\put(8,0){\circle{0.1}}
		\put(8.4,0.1){\mbox{$\mathfrak{f}^m_{1'}$}}
		\put(7,1){\circle{0.1}}
         \put(9,1){\circle{0.1}}
		\put(7.3,1){\mbox{$\mathfrak{g}^m$}}
        \put(8.5,1){\mbox{$\mathfrak{g}^m$}}
		\put(8,0){\vector(1,1){0.9}}
         \put(8,0){\vector(-1,1){0.9}}
		\put(9.2,0.5){\mbox{$\Bigr)=$}}
	\put(11,0){\circle{0.1}}
		\put(11.3,0){\mbox{$10$}}
		\put(10,1){\circle{0.1}}
         \put(12,1){\circle{0.1}}
		\put(10.3,1){\mbox{$11$}}
        \put(11.5,1){\mbox{$11$}}
		\put(11,0){\vector(1,1){0.9}}
         \put(11,0){\vector(-1,1){0.9}}
	\end{picture}	\\

Thus, we have
$$
\unitlength1cm
\begin{picture}(6,2.1)
\thicklines

\put(2,2){\circle{0.1}}
\put(4,2){\circle{0.1}}
\put(2,1){\circle{0.1}}
\put(4,1){\circle{0.1}}
\put(3,0){\circle{0.1}}
\put(4.2,0.9){\mbox{$\mathfrak{f}^m_{1}\quad \Bigr)\quad=$}}
\put(0.5,0.9){\mbox{$G \Bigl(\quad\mathfrak{f}^m_{1'}$}}
\put(3.3,0){\mbox{${0\cdots0}$}}
\put(1.6,2){\mbox{$\mathfrak{g}^m$}}
\put(4.2,2){$\mathfrak{g}^m$ }
\put(2,1){\vector(0,1){0.9}}
\put(4,1){\vector(0,1){0.9}}
\put(2,1){\vector(2,1){1.9}}
\put(4,1){\vector(-2,1){1.9}}
\put(3,0){\vector(1,1){0.9}}
\put(3,0){\vector(-1,1){0.9}}

\put(7,2){\circle{0.1}}
\put(9,2){\circle{0.1}}
\put(7,1){\circle{0.1}}
\put(9,1){\circle{0.1}}
\put(8,0){\circle{0.1}}
\put(9.2,0.9){\mbox{$01$}}
\put(6.5,0.9){\mbox{$10$}}
\put(8.3,0){\mbox{$00$}}
\put(6.6,2){\mbox{$11$}}
\put(9.2,2){$11$ }
\put(7,1){\vector(0,1){0.9}}
\put(9,1){\vector(0,1){0.9}}
\put(7,1){\vector(2,1){1.9}}
\put(9,1){\vector(-2,1){1.9}}
\put(8,0){\vector(1,1){0.9}}
\put(8,0){\vector(-1,1){0.9}}
\end{picture}\qquad\qquad\qquad\qquad\qquad\qquad\qquad\qquad\qquad\qquad\qquad\qquad$$
which  contradicts (iv). Indeed, according to (iv), $G(\mathfrak{M}^m)$ should be equivalent to a $\sigma$-model, for any $m$-model $\mathfrak{M}^m$.
But $\sigma$ is a unifier for $x_1\lor x_2$ and hence $x_1\lor x_2$ should be true in the received $2$-model which is not the case.\\

Let us consider $\mathsf L(\mathfrak{Y}_3+)$. We have $sm(\mathfrak Y_{3}+)=\{\mathfrak L_1,\mathfrak L_2,\mathfrak L_3,\mathfrak L_4,\mathfrak R_{2},+\mathfrak R_{2},\mathfrak R_{2}+,\mathfrak Y_{3}+\}$.  Assume that $\mathsf L(\mathfrak{Y}_3+)$ has finitary unification. Take $n=3$ and use Theorem \ref{main} to get a number $m\geq 1$, mappings  $G:\mathbf{M}^m\to\mathbf{M}^2$ and $F:\mathbf{M}^k\to\mathbf{M}^m$  fulfilling the conditions (i)-(v), where $k>m$ and  $\sigma\colon\{x_1,x_2,x_3\}\to \mathsf{Fm}^k$ is as follows
$$ \sigma(x_1)=x_1 \qquad \mbox{and} \qquad \sigma(x_2)=\bigwedge_{i=2}^k\Bigl(\bigl(((x_i\to x_1)\to x_1)\bigr)\lor (x_i\to x_1)\Bigr)\qquad \mbox{and} $$ $$ \sigma(x_3)=\bigwedge_{i=2}^k\Bigl(\bigl(((x_i\to x_1)\to x_1)\bigr)\lor (x_i\to x_1)\Bigr)\rightarrow\Bigl(\bigvee_{i=2}^kx_i\lor\bigl((\bigvee_{i=2}^kx_i)\to x_1\bigr)\Bigr). $$
Note that if we take $\alpha\circ\sigma$, where $\alpha:x_1\slash\bot$ we get a slightly modified substitution $\sigma$ as used above for $\mathfrak Y_{3}$ (we would have $x_1=\bot$ and $x_2,x_3$ substituted instead of $x_1,x_2$). If $\sigma(x_1)$ (that is $x_1$) is true at any node of any $k$-model $\mathfrak M_k$ over $\mathfrak{Y}_3+$, then so are $\sigma(x_2)$ and $\sigma(x_3)$. Thus, if we characterize all p-irreducible $k$-models, we get  $\circ \ 111$ and $\circ \ 011$, as models over $\mathfrak L_1$, and all $\sigma$-models as we got for $\mathfrak Y_{3}$ (where valuations take the form $0ij$ instead of $ij$) with additional top node where the valuation is $111$. It means that $\sigma$ is a unifier for $x_2\lor x_3$ (instead of $x_1\lor x_2$ as we had in $\mathfrak Y_{3}$). Moreover, repeating the above argument we get two $m$-models over $\mathfrak R_2$ such that\\

		\unitlength1cm
\begin{picture}(1,2)
\thicklines
\put(-0.3,1){\mbox{$G \Bigl($}}
        \put(1.5,2){\circle{0.1}}
		\put(1.5,0){\circle{0.1}}
		\put(1.9,0.1){\mbox{$0\mathfrak{f}^m_{1}$}}
		\put(0.5,1){\circle{0.1}}
         \put(2.5,1){\circle{0.1}}
		\put(0.8,1){\mbox{$0\mathfrak{g}^m$}}
        \put(1.7,1){\mbox{$0\mathfrak{g}^m$}}
		\put(1.5,0){\vector(1,1){0.9}}
         \put(1.5,0){\vector(-1,1){0.9}}
         \put(0.5,1){\vector(1,1){0.9}}
         \put(2.5,1){\vector(-1,1){0.9}}
         \put(1.6,2){\mbox{$1\cdots1$}}

	\put(2.8,1){\mbox{$\Bigr)=$}}
	\put(4.5,0){\circle{0.1}}
		\put(4.9,0){\mbox{$001$}}
		\put(3.5,1){\circle{0.1}}
         \put(5.5,1){\circle{0.1}}
		\put(3.8,1){\mbox{$011$}}
        \put(4.8,1){\mbox{$011$}}
		\put(4.5,0){\vector(1,1){0.9}}
         \put(4.5,0){\vector(-1,1){0.9}}
         	\put(3.5,1){\vector(1,1){0.9}}
         \put(5.5,1){\vector(-1,1){0.9}}
         \put(4.5,2){\circle{0.1}}
         \put(4.9,2){\mbox{$111$}}
 \put(5.7,1){\mbox{;}}

	\put(6.2,1){\mbox{$G \Bigl($}}
       \put(8,2){\circle{0.1}}
		\put(8,0){\circle{0.1}}
        \put(8.4,2.1){\mbox{$1\cdots1$}}
		\put(8.4,0.1){\mbox{$0\mathfrak{f}^m_{1'}$}}
		\put(7,1){\circle{0.1}}
         \put(9,1){\circle{0.1}}
		\put(7.3,1){\mbox{$0\mathfrak{g}^m$}}
        \put(8.1,1){\mbox{$0\mathfrak{g}^m$}}
		\put(8,0){\vector(1,1){0.9}}
         \put(8,0){\vector(-1,1){0.9}}
         \put(7,1){\vector(1,1){0.9}}
         \put(9,1){\vector(-1,1){0.9}}
		\put(9.2,1){\mbox{$\Bigr)=$}}

    \put(11,2){\circle{0.1}}
	\put(11,0){\circle{0.1}}
       \put(11.3,2){\mbox{$111$}}
		\put(11.3,0){\mbox{$010$}}
		\put(10,1){\circle{0.1}}
         \put(12,1){\circle{0.1}}
		\put(10.3,1){\mbox{$011$}}
        \put(11.3,1){\mbox{$011$}}
		\put(11,0){\vector(1,1){0.9}}
         \put(11,0){\vector(-1,1){0.9}}
         \put(10,1){\vector(1,1){0.9}}
         \put(12,1){\vector(-1,1){0.9}}
	\end{picture}	

$$\unitlength1cm
\begin{picture}(6,2.8)
\thicklines
\put(0,3){Thus,}
\put(2,2){\circle{0.1}}
\put(4,2){\circle{0.1}}
\put(2,1){\circle{0.1}}
\put(4,1){\circle{0.1}}
\put(3,0){\circle{0.1}}
\put(3,3){\circle{0.1}}
\put(4.2,0.9){\mbox{$0\mathfrak{f}^m_{1}$}}
\put(4.2,1.4){\mbox{$\qquad \Bigr)\quad=$}}
\put(1.4,0.9){\mbox{$0\mathfrak{f}^m_{1'}$}}
\put(0.7,1.4){\mbox{$G \Bigl($}}
\put(3.3,0){\mbox{${0\cdots0}$}}
\put(3.3,3){\mbox{${1\cdots1}$}}
\put(1.4,2){\mbox{$0\mathfrak{g}^m$}}
\put(4.2,2){$0\mathfrak{g}^m$ }
\put(2,1){\vector(0,1){0.9}}
\put(4,1){\vector(0,1){0.9}}
\put(2,1){\vector(2,1){1.9}}
\put(4,1){\vector(-2,1){1.9}}
\put(3,0){\vector(1,1){0.9}}
\put(3,0){\vector(-1,1){0.9}}
\put(2,2){\vector(1,1){0.9}}
\put(4,2){\vector(-1,1){0.9}}

\put(8,3){\circle{0.1}}
\put(7,2){\circle{0.1}}
\put(9,2){\circle{0.1}}
\put(7,1){\circle{0.1}}
\put(9,1){\circle{0.1}}
\put(8,0){\circle{0.1}}
\put(9.2,0.9){\mbox{$001$}}
\put(6.2,0.9){\mbox{$010$}}
\put(8.3,0){\mbox{$000$}}
\put(6.2,2){\mbox{$011$}}
\put(9.2,2){$011$ }
\put(8.3,3){\mbox{$111$}}
\put(7,1){\vector(0,1){0.9}}
\put(9,1){\vector(0,1){0.9}}
\put(7,1){\vector(2,1){1.9}}
\put(9,1){\vector(-2,1){1.9}}
\put(8,0){\vector(1,1){0.9}}
\put(8,0){\vector(-1,1){0.9}}
\put(7,2){\vector(1,1){0.9}}
\put(9,2){\vector(-1,1){0.9}}
\end{picture}\qquad\qquad\qquad\qquad\qquad\qquad\qquad\qquad\qquad\qquad\qquad\qquad$$
which  would mean that $\sigma$ is not a unifier for $x_2\lor x_3$, a contradiction.\\

In the case of  $\mathsf L(\mathfrak{Y}_2+)$, we have $sm(\mathfrak Y_{2}+)=\{\mathfrak L_1,\mathfrak L_2,\mathfrak L_3,\mathfrak L_4,\mathfrak R_{2},+\mathfrak R_{2},\mathfrak R_{2}+,\mathfrak G_{3}+\}$. We take $n=2$ and consider
the substitution $\sigma\colon\{x_1,x_2\}\to \mathsf{Fm}^k$ defined as follows
$$ \begin{array}{rl}
     \sigma(x_1)=& x_1\\
     \sigma(x_2)= & \Bigl(\bigl(( \bigvee_{i=2}^kx_i)\to x_1\bigr)\to x_1\Bigr) \ \land \ \bigwedge_{i=2}^k\Bigl(\bigl((x_i\to x_1)\to x_1\bigr)\lor (x_i\to x_1)\Bigr).
   \end{array}
  $$
If we take $\alpha\circ\sigma$, where $\alpha:x_1\slash\bot$ we (almost) get the substitution $\sigma$ as used  for $\mathfrak Y_{2}$.
If $x_1$ is false at the top element of any $2$-model $\mathfrak M^2$ over $\mathfrak Y_{2}+$, then $ \sigma(x_2)$ is true at the model and hence $\sigma(\mathfrak M^2)$ reduces to a model over $\mathfrak L_1$. Then we can  repeat our argument we have used for $\mathsf L(\mathfrak{Y}_2)$ (similarly as we did in the case of $ \mathfrak{Y}_3+$ where we reduced our argument to $ \mathfrak Y_3$).
\footnote{
We think that  one can show in this way that adding top element to $\mathfrak F$, for any finite frame $\mathfrak F$, does not improve the unification type   if $\mathsf L(\mathfrak{F})$ has nullary unification.}
\end{proof}
Nodes of the graph in Figure \ref{ki} represent  p-morphic images of  $\mathfrak G_1$, see Figure \ref{GF} (or \ref{NU} for $\mathfrak G_{3\mathfrak L_2}$ and $\mathfrak G_{3\mathfrak F_2}$), and edges p-morphisms between the  frames.   We omit, as usual, p-morphisms that are compositions of  other p-morphisms. The graph represents also all (consistent) $H$-complete extensions of $\mathsf H_3\mathsf B_2$; each node is  the logic $\mathsf L(\mathfrak{F})$ of the  frame and the edges mean inclusions. Logics with nullary unification are denoted by a black square. There are 14 logics and 5 of them have nullary unification.

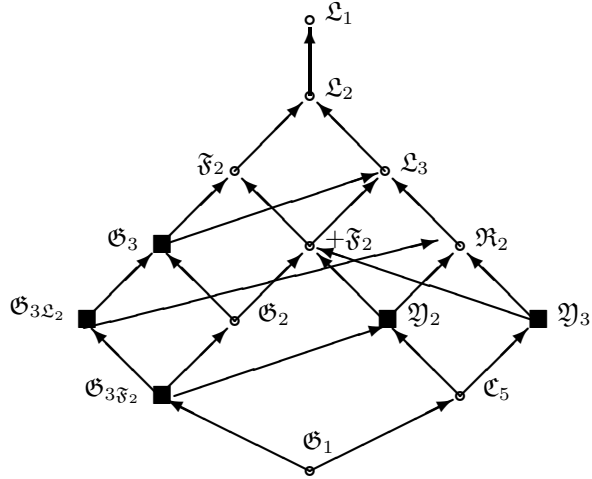
\begin{figure}[H]
\unitlength1cm
\thicklines
\begin{picture}(6,6)
\put(6,0){\circle{0.1}}
\put(5.9,0.3){$\mathfrak G_1$}
\put(3.9,0.9){$\blacksquare$}
\put(3,1){$\mathfrak G_{3\mathfrak F_2}$}
\put(8,1){\circle{0.1}}
\put(8.3,1){$\mathfrak C_{5}$}
\put(2.9,1.9){$\blacksquare$}
\put(2,2.1){$\mathfrak G_{3\mathfrak L_2}$}
\put(5,2){\circle{0.1}}
\put(5.3,2){$\mathfrak G_{2}$}
\put(6.9,1.9){$\blacksquare$}
\put(7.3,2){$\mathfrak Y_{2}$}
\put(8.9,1.9){$\blacksquare$}
\put(9.3,2){$\mathfrak Y_{3}$}
\put(3.9,2.9){$\blacksquare$}
\put(3.3,3){$\mathfrak G_{3}$}
\put(6,3){\circle{0.1}}
\put(6.2,3){$+\mathfrak  F_2$}
\put(8,3){\circle{0.1}}
\put(8.2,3){$\mathfrak R_{2}$}
\put(5,4){\circle{0.1}}
\put(4.5,4){$\mathfrak F_2$}
\put(7,4){\circle{0.1}}
\put(7.2,4){$\mathfrak L_{3}$}
\put(6,5){\circle{0.1}}
\put(6.2,5){$\mathfrak L_{2}$}
\put(6,6){\circle{0.1}}
\put(6.2,6){$\mathfrak L_{1}$}
\put(6,0){\vector(-2,1){1.9}}
\put(6,0){\vector(2,1){1.9}}
\put(4,1){\vector(-1,1){0.9}}
\put(4,1){\vector(1,1){0.9}}
\put(8,1){\vector(-1,1){0.9}}
\put(8,1){\vector(1,1){0.9}}
\put(3,2){\vector(1,1){0.9}}
\put(3,1.9){\vector(4,1){4.7}}
\put(5,2){\vector(1,1){0.9}}
\put(5,2){\vector(-1,1){0.9}}
\put(7,2){\vector(1,1){0.9}}
\put(7,2){\vector(-1,1){0.9}}
\put(9,2){\vector(-1,1){0.9}}
\put(9,2){\vector(-3,1){2.9}}
\put(4,3){\vector(1,1){0.9}}
\put(4,3){\vector(3,1){2.9}}
\put(6,3){\vector(1,1){0.9}}
\put(6,3){\vector(-1,1){0.9}}
\put(8,3){\vector(-1,1){0.9}}
\put(5,4){\vector(1,1){0.9}}
\put(7,4){\vector(-1,1){0.9}}
\put(6,5){\vector(0,1){0.9}}

\put(4.2,1){\vector(3,1){2.8}}

\end{picture}
\caption{Reducts of $\mathfrak G_1$. }\label{ki}
\end{figure}

\noindent Let us agree that there is a chaos in the picture. Extensions of finitar\slash unitary logics may have nullary unification; intersections of some finitary logics are nullary. The chaos   increases  if we add intersections of the logics, see Figure  \ref{ti}. We can only try to identify maximal logics with nullary unification.

\subsection{Hereditary Finitary Unification}\label{HFU}

Let us say that unification in a logic {\sf L} is {\it hereditary finitary} if any extension of {\sf L}, including {\sf L} itself, has finitary (or unitary) unification. Unification in  {\sf L} is {\it hereditary unitary} if any extension of {\sf L} has unitary unification and {\sf L} has {\it hereditary projective approximation} if any its extension has projective approximation.
By Citkin \cite{Tsitkin},
\begin{theorem}\label{Tsit} $(\mathsf L(\mathfrak C_4),\mathsf L(\chi (\mathfrak C_4))$ (where  \ $\mathfrak C_4={\mathfrak G_3}+$, \ see Figure \ref{GF} and \ref{TF}) is a splitting pair in the lattice of intermediate logics (see Theorem \ref{Jankov}) and the logic $\mathsf L(\chi (\mathfrak C_4))$ is locally tabular. Thus, if an intermediate logic {\sf L} omits ${\mathfrak G_3}+$, then {\sf L} is locally tabular.\end{theorem}
We know that $(\mathsf L (\mathfrak C_3),\mathsf {SL})$ (where $\mathsf {SL}$ is Scott logic, see Figure \ref{ILs}, and ${\mathfrak C_3}={\mathfrak G_3}$)   is a splitting pair in extensions of  {\sf INT}.  It means that logics omitting ${\mathfrak C_3}$  may be not locally tabular, like $\mathsf{SL}$ (the unification type of $\mathsf{SL}$ is not known, we conjecture it is $\omega$).

 By Theorem \ref{lmk},  all logics determined by frames $\mathbf H_{pa}$ in Figure \ref{hpa} (which are  $\mathfrak L_d+\mathfrak F_m$,  with $m,d\geq 0$, and we agree that $\mathfrak L_0+\mathfrak F_m=\mathfrak F_m$, and $\mathfrak F_1=\mathfrak L_2$ and $\mathfrak F_0=\mathfrak L_1$) have  projective approximation.
\begin{theorem}\label{ipa} For any intermediate logic {\sf L}, the following conditions are equivalent:\\ (i) {\sf L} has hereditary projective approximation;\\(ii) $\mathsf L(\mathbf H_{pa})\subseteq \mathsf L$;\\
(iii) $\mathsf L$  omits $\mathfrak G_3$ and $\mathfrak R_2$;\\
(iv)  $ \chi ({\mathfrak  G_3}),  \chi ({\mathfrak R_2}) \in  \mathsf L$.
\end{theorem}
\begin{proof} (ii)$\Rightarrow$(i).  Since $sm(\mathbf H_{pa})\subseteq\mathbf H_{pa}$, then all extensions of $\mathsf L(\mathbf H_{pa})$ have projective approximation, by Theorem \ref{lf7}. Thus, $\mathsf L(\mathbf H_{pa})$ and any its extension  enjoy hereditary projective approximation.\\ (i)$\Rightarrow$(iii). We know that  $\mathsf L(\mathfrak G_3)$ has  nullary unification (see Theorem \ref{F6m}). Thus, if $\mathfrak G_3$ were a frame of {\sf L}, the logic {\sf L} would have a nullary extension. Therefore {\sf L} must omit $\mathfrak G_3$. It must also omit  $\mathfrak R_2$ as $\mathsf L(\mathfrak R_2)$ is unitary and non-projective, hence it cannot have projective approximation, by Corollary \ref{up1}.\\
  (iii)$\Rightarrow$(ii). Assume that $\mathsf L$ omits $\mathfrak G_3$ and $\mathfrak R_2$. If {\sf L} omits $\mathfrak R_2$, then it also  omits
 ${\mathfrak G_3}+$ and hence {\sf L} is locally tabular, by Theorem \ref{Tsit}. Thus, we can assume {\sf L}={\sf L}({\bf F}) for some family {\bf F} of finite frames.

For each $\mathfrak F\in\mathbf F$, let $\mathfrak F^\star$ denote its p-morphic image resulting by gluing all its end elements (the relation of gluing end elements is obviously a bisimulation of ${\mathfrak F}$). Thus, $\mathsf L(\{\mathfrak F^\star\colon\mathfrak F\in \mathbf F\})$ is an extension of $\mathsf L$ and it must be an extension of {\sf KC}.  As {\sf{L}(\bf F)} omits $\mathfrak R_2$, we infer $\mathsf L(\{\mathfrak F^\star\colon\mathfrak F\in \mathbf F\})$ is a logic with projective unification\footnote{We use again the known fact (see \cite{Rautenberg}) that $(\mathsf L (\mathfrak R_2),\mathsf{LC})$ is a splitting pair for all extensions of $\mathsf{KC}$.}. It means $\mathfrak F^\star$, for each $\mathfrak F\in \mathbf F$, must be a chain. Since {\sf{L}(\bf F)} omits $\mathfrak G_3$, we conclude each element of {\bf F} must be of the form $\mathfrak L_d+\mathfrak F_m$ for some $m,d\geq 0$.\\
(iii)$\Leftrightarrow$(iv) by Theorem \ref{Jankov}.\end{proof}
Projective approximation is not hereditary,  intuitionistic logic {\sf INT} enjoys this property, see \cite{Ghi2}. We cojecture that projective approximation is  hereditary in locally tabular logics. The above theorem does not settle the question if there are locally tabular logics with projective approximation which are not extensions of   $\mathsf L(\mathbf H_{pa})$.
\begin{corollary}\label{ipa2}
$\mathsf L(\mathbf H_{pa})$ is the least intermediate logic with hereditary projective approximation and   $\mathsf L(\mathbf H_{pa}) = \mathsf{L}(\{\chi ({\mathfrak G_3}),  \chi ({\mathfrak R_2}) \})$.
\end{corollary}

By Theorem \ref{wpl}, we know that all logics determined by frames $\mathbf H_{un}$ in Figure \ref{hpa} (which are  $\mathfrak L_{d}+\mathfrak R_{m}$,  with $d,m\geq 0$, and we agree that $\mathfrak L_{0}+\mathfrak R_{m}=\mathfrak R_{m}$ and $\mathfrak R_{0}=\mathfrak L_{1}$) have hereditary unitary unification.
\begin{theorem}\label{iun} For any intermediate logic {\sf L} the following conditions are equivalent:\\(i) {\sf L} has hereditary unitary unification;\\ (ii) $\mathsf L(\mathbf H_{un})\subseteq \mathsf L$;\\
(iii)  $\mathsf L$ omits  the frames $\mathfrak R_2+$, $\mathfrak G_3+$ and $\mathfrak F_2$;\\
(iv)  $ \chi ({\mathfrak  G_3+}),  \chi ({\mathfrak R_2+}),  \chi ({\mathfrak F_2}) \in \mathsf L$.
\end{theorem}
\begin{proof} (i)$\Leftarrow$(ii)  follows from Theorem \ref{wpl}; note that the class $\mathbf H_{un}$ is closed under p-morphic images and generated subframes, and then apply Theorem \ref{lf7}.\\  (i)$\Rightarrow$(iii). Since $\mathfrak F_2+\mathfrak F_1=\mathfrak F_2+\mathfrak L_2=\mathfrak R_2+$ and $\mathsf L(\mathfrak F_2+\mathfrak F_1)$ has nullary unification (by Theorem \ref{infty}), then {\sf L} omits $\mathfrak R_2+$. We know that $\mathsf L(\mathfrak G_3+)$ has  nullary unification (see \cite{Ghi5}) and $\mathsf L(\mathfrak F_2)$ is not unitary, hence {\sf L} must omit $\mathfrak G_3+$ and $\mathfrak F_2$.\\ (ii)$\Leftarrow$(iii).  Assume that  $\mathsf L$ omits $\mathfrak R_2+$, $\mathfrak G_3+$ and $\mathfrak F_2$. Since  {\sf L} omits
 ${\mathfrak G_3}+$, then  {\sf L} is locally tabular, by Theorem \ref{Tsit}. Thus,  {\sf L}={\sf L}({\bf F}) for some family {\bf F} of finite frames such that {\it sm}({\bf F})={\bf F}. Since {\sf L} omits
 ${\mathfrak F_2}$, all frames in {\bf F} are {\sf KC}-frames.

Let $\mathfrak F=(W,R,w_0)\in\mathbf F$ and $d(\mathfrak F)=d$. We define a p-morphism $p\colon\mathfrak F\to\mathfrak L_d$ by taking $p(w)=d_{\mathfrak F}(w)$, for any $w\in W$. There is only one element in $W$ of the depth 1, let it be denoted by $w_1$, and there is only one element of the depth $d$ (which is $w_0$). Every $w\in W$ sees $w_1$, that is $wRw_1$.

Let $d(w_2)=i$ and $d(w_3)=i-1$, for some $w_2,w_3\in W$ and $i>1$. We prove  $w_2$ sees $w_3$. Suppose it does not. By the definition of the depth function, there is an element $w_4\in W$ such that $d(w_4)=i-1$ and $w_2$ sees $w_4$. One can take a bisimulation gluing all elements of the depth $<i-1$ (with $w_1$). Since {\bf F} is closed under p-morphic images, we could assume that the received (after gluing all elements of the depth $<i-1$) frame is our  $\mathfrak F$. Thus, we have $w_1,w_2,w_3\in W$ which could give (as we shall see) $\mathfrak G_3+$ (as a p-morphic image of $\mathfrak F$) if we add the root $w_0$, see the first frame in Figure \ref{ppp}.

\begin{figure}
\unitlength1cm
\thicklines
\begin{picture}(0,3)
\put(5.5,0){$w_0$}
\put(6.5,3){$w_1$}
\put(4.5,1){$w_2$}
\put(7.5,2){$w_3$}
\put(5.5,2){$w_4$}
\put(6.5,1){$w_5$}
\put(6,0){\vector(-1,1){0.9}}
\put(6,0){\vector(1,1){0.9}}
\put(7,1){\vector(1,1){0.9}}
\put(5,1){\circle{0.1}}
\put(8,2){\circle{0.1}}
\put(6,0){\circle{0.1}}
\put(7,1){\circle{0.1}}
\put(6,2){\circle{0.1}}
\put(5,1){\vector(1,1){0.9}}
\put(7,1){\vector(-1,1){0.9}}
\put(7,3){\circle{0.1}}
\put(6,2){\vector(1,1){0.9}}
\put(8,2){\vector(-1,1){0.9}}

\put(10.5,0){$w_0$}
\put(10.5,3){$w_1$}
\put(9.5,1){$w_2$}
\put(11.5,2){$w_3$}
\put(9.5,2){$w_4$}
\put(11.5,1){$w_5$}
\put(10,2){\circle{0.1}}
\put(12,2){\circle{0.1}}
\put(10,1){\circle{0.1}}
\put(12,1){\circle{0.1}}
\put(11,0){\circle{0.1}}
\put(11,3){\circle{0.1}}
\put(10,1){\vector(0,1){0.9}}
\put(12,1){\vector(0,1){0.9}}
\put(10,1){\vector(2,1){1.9}}
\put(12,1){\vector(-2,1){1.9}}
\put(11,0){\vector(1,1){0.9}}
\put(11,0){\vector(-1,1){0.9}}
\put(10,2){\vector(1,1){0.9}}
\put(12,2){\vector(-1,1){0.9}}

\put(0.4,2){$w_4$}
\put(2.4,2){$w_3$}
\put(0.4,1){$w_2$}
\put(1.3,3){$w_1$}
\put(1.3,0){$w_0$}
\put(2,0){\vector(-1,1){0.9}}
\put(2,0){\vector(1,2){0.9}}
\put(1,1){\vector(0,1){0.9}}
\put(1,2){\circle{0.1}}
\put(3,2){\circle{0.1}}
\put(2,0){\circle{0.1}}
\put(1,1){\circle{0.1}}
\put(2,3){\circle{0.1}}
\put(3,2){\vector(-1,1){0.9}}
\put(1,2){\vector(1,1){0.9}}
\end{picture}
\caption{}\label{ppp}
\end{figure}
\noindent We can use {\it sm}({\bf F})={\bf F} again  to assume no other element of $W$ (only $w_0$) sees both $w_2$ and $w_3$; if not we would replace $\mathfrak F$ with its  subframe generated by an $R$-maximal element that sees both $w_2$ and $w_3$. There could be other elements in $\mathfrak F$ but they all could be identified (by a bisimulation) with one of  $\{w_0,w_1,w_2,w_3,w_4\}$. For instance, all elements $w\in W$ that see only $w_1$ (and do no see $\{w_0,w_2,w_3,w_4\}$) could be identified with $w_1$, etc. This identification  gives us $\mathfrak G_3+$ as a p-morphic image of $\mathfrak F$ (the first frame in Figure \ref{ppp}) provided that there is no element $w_5\in W$ which sees both $w_3$ and $w_4$  but does not see $w_2$. If there is such an element we would get the second frame in Figure \ref{ppp}. But this is $\mathfrak Y_2+$ and  $\mathfrak G_3+$ is its p-morphic image. Thus, in each case, $\mathfrak G_3+$ would belong to $\mathbf F$ which could not be the case.

Thus, we have shown that any element of the depth $i$ sees any element of the depth $i-1$ in $\mathfrak F$. Suppose that we have two (or more) elements of the depth $i$ (let it be $w_2$ and $w_5$) and two (or more) elements of the depth $i-1$ (denoted by $w_3$ and $w_4$). Then we would get (as a generated subframe of a p-morphic image of $\mathfrak F$) the third frame in Figure \ref{ppp}. Thus, $\mathfrak Y_2+$ would belong to {\bf F} which is impossible as {\sf L} omits $\mathfrak R_2+$.

We conclude that each frame $\mathfrak F\in \mathbf F$ must be of the form $\mathfrak F_{n_1}+\cdots+\mathfrak F_{n_s}+$, for some $n_1,\dots,n_s\geq 0$. But  {\sf L} omits $\mathfrak R_2+$ and hence $n_1=n_2=\cdots=n_{s-1}=1$ and, consequently,  each  $\mathfrak F\in \mathbf F$ must be of the form $\mathfrak L_{d}+\mathfrak R_{s}$, for some $d,s\geq 0$. Thus, we have $\mathbf F \subseteq\mathbf H_{un}$ and $\mathsf L(\mathbf H_{un})\subseteq \mathsf L$ as required.\end{proof}

\begin{corollary}\label{iun2}
$\mathsf L(\mathbf H_{un})$ is the least intermediate logic with hereditary unitary unification and we have $\mathsf L(\mathbf H_{un}) = \mathsf{L}(\{\chi ({\mathfrak G_3+}),  \chi ({\mathfrak R_2+}),  \chi ({\mathfrak F_2}) \}) =\mathsf L( \mathsf{KC} \cup \{\chi ({\mathfrak G_3+}), \chi ({\mathfrak R_2+}) \})$.
\end{corollary}
$\mathbf H_{pa}\cap\mathbf H_{un}$  consists of chains only. This is in accordance with Corollary \ref{up1} saying that unitary logics with projective approximation are projective.
\begin{theorem}\label{ihf} For any intermediate logic {\sf L}, the following conditions are equivalent:\\ (i) {\sf L} has hereditary finitary unification;\\ (ii) {\sf L} has hereditary projective approximation or {\sf L} has hereditary unitary unification;\\
(iii) $\mathsf L$ omits $\mathfrak R_{2}+, \mathfrak G_3, \mathfrak G_{3}+$ and one of the frames $\{\mathfrak F_2,\mathfrak R_2\}$;\\
(iv) $\chi ({\mathfrak G_3}), \chi ({\mathfrak G_3+}),  \chi ({\mathfrak R_2+})\in L$ and, either  $\chi ({\mathfrak R_2}) \in L$,  or  $\chi ({\mathfrak F_2}) \in L$.
\end{theorem}
\begin{proof} If $\mathsf L$ has hereditary finitary unification, then $\mathsf L$  omits $\mathfrak R_{2}+$ (by Theorem \ref{infty}), $\mathsf L$ omits $\mathfrak G_3$ and $\mathfrak G_{3}+$ (as $\mathsf L(\mathfrak G_3)$ and $\mathsf L(\mathfrak G_{3}+)$ are  nullary, by Theorem \ref{F6m}) and it omits one of the frames $\{\mathfrak F_2,\mathfrak R_2\}$, by Theorem  \ref{L7}. Thus, {\sf L} has hereditary projective approximation or  hereditary unitary unification (depending on whether {\sf L} omits $\mathfrak F_2$ or $\mathfrak R_2$), by Theorem \ref{ipa} and \ref{iun}.
\end{proof}

All pretabular intermediate logics have   hereditary finitary unification. More specifically,   {\sf LC} enjoys projective unification, {\sf LJ} has hereditary projective approximation and {\sf LH} has hereditary unitary unification.

\begin{corollary}\label{ihf2}
An intermediate logic {\sf L} has hereditary finitary unification iff either $\mathsf L(\mathbf H_{un})\subseteq \mathsf L$ or $\mathsf L(\mathbf H_{pa})\subseteq \mathsf L$. Thus, there are exactly two minimal intermediate logics in the family of logics enjoying hereditary finitary unification.
\end{corollary}

\begin{corollary}\label{L11} There are exactly four maximal intermediate logics with nullary unification. They are: \ ${\mathsf L}(\mathfrak R_{2}+)$, \ ${\mathsf L}(\mathfrak R_{2})\cap{\mathsf L}(\mathfrak F_{2})$, \ ${\mathsf L}(\mathfrak G_{3})$ \ ,\ ${\mathsf L}(\mathfrak G_{3}+)$; see Figure \ref{ci}\footnote{The second graph represents an unrooted frame.}.
\end{corollary}

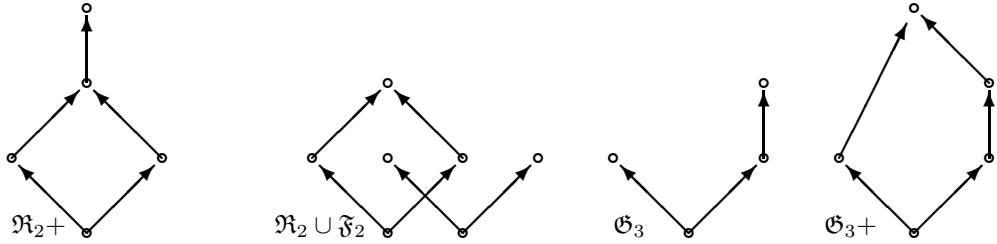
\begin{figure}[H]
\unitlength1cm
\thicklines
\begin{picture}(0,2.5)
\put(0,0){$ \mathfrak R_{2}+$}
\put(1,0){\vector(-1,1){0.9}}
\put(1,0){\vector(1,1){0.9}}
\put(2,1){\vector(-1,1){0.9}}
\put(0,1){\circle{0.1}}
\put(1,2){\circle{0.1}}
\put(1,0){\circle{0.1}}
\put(2,1){\circle{0.1}}
\put(0,1){\vector(1,1){0.9}}
\put(1,3){\circle{0.1}}
\put(1,2){\vector(0,1){0.9}}

\put(3.5,0){$ \mathfrak R_{2}\cup\mathfrak F_{2}$}
\put(4,1){\circle{0.1}}
\put(5,0){\circle{0.1}}
\put(5,1){\circle{0.1}}
\put(5,2){\circle{0.1}}
\put(6,1){\circle{0.1}}
\put(6,0){\circle{0.1}}
\put(7,1){\circle{0.1}}
\put(5,0){\vector(-1,1){0.9}}
\put(5,0){\vector(1,1){0.9}}
\put(6,0){\vector(-1,1){0.9}}
\put(6,0){\vector(1,1){0.9}}
\put(4,1){\vector(1,1){0.9}}
\put(6,1){\vector(-1,1){0.9}}

\put(8,0){$\mathfrak{G}_3$}
\put(9,0){\vector(-1,1){0.9}}
\put(9,0){\vector(1,1){0.9}}
\put(10,1){\vector(0,1){0.9}}
\put(8,1){\circle{0.1}}
\put(10,2){\circle{0.1}}
\put(9,0){\circle{0.1}}
\put(10,1){\circle{0.1}}

\put(10.8,0){${\mathfrak{G}_3}+$}
\put(12,0){\vector(-1,1){0.9}}
\put(12,0){\vector(1,1){0.9}}
\put(13,1){\vector(0,1){0.9}}
\put(11,1){\circle{0.1}}
\put(13,2){\circle{0.1}}
\put(12,0){\circle{0.1}}
\put(13,1){\circle{0.1}}
\put(12,3){\circle{0.1}}
\put(13,2){\vector(-1,1){0.9}}
\put(11,1){\vector(1,2){0.9}}
\end{picture}
\caption{Frames of Maximal Logics with Nullary Unification.}\label{ci}
\end{figure}
\noindent Similarly as in Theorem \ref{kc}, we can use a particular para-splitting\footnote{A `join' of two join-splittings given by Theorems  \ref{ipa} and \ref{iun}, that is $(\{{\mathsf L}(\mathfrak R_{2}+), {\mathsf L}(\mathfrak F_{2}), {\mathsf L}(\mathfrak G_{3}+)\},\mathsf L(\mathbf H_{un}))$ and $( \{ {\mathsf L}(\mathfrak R_{2}), {\mathsf L}(\mathfrak G_{3})\},\mathsf L(\mathbf H_{pr}))$.}of the lattice of extensions of {\sf INT}, given by the pair $$(\{{\mathsf L}(\mathfrak R_{2}+), {\mathsf L}(\mathfrak R_{2})\cap{\mathsf L}(\mathfrak F_{2}), {\mathsf L}(\mathfrak G_{3}),{\mathsf L}(\mathfrak G_{3}+)\}\ ,\ \{\mathsf L(\mathbf H_{un}),\mathsf L(\mathbf H_{pr})\})$$ to summon up our results.
\begin{corollary} For each intermediate logic {\sf L},  either {\sf L} includes $\mathsf L(\mathbf H_{un})$, or  $\mathsf L(\mathbf H_{pa})$, or {\sf L} is included in one of the logics $\{{\mathsf L}(\mathfrak R_{2}+), {\mathsf L}(\mathfrak R_{2})\cap{\mathsf L}(\mathfrak F_{2}), {\mathsf L}(\mathfrak G_{3}),{\mathsf L}(\mathfrak G_{3}+)\}$. If {\sf L} includes $\mathsf L(\mathbf H_{un})$ its unification is unitary, if  {\sf L} includes $\mathsf L(\mathbf H_{pa})$ it has projective approximation. If {\sf L} is included in one of the logics  $\{{\mathsf L}(\mathfrak R_{2}+), {\mathsf L}(\mathfrak R_{2})\cap{\mathsf L}(\mathfrak F_{2}), {\mathsf L}(\mathfrak G_{3}),{\mathsf L}(\mathfrak G_{3}+)\}$ its unification type is not determined but {\sf L} has an extension with nullary unification.\end{corollary}
\begin{corollary}\label{pafiun}
It is decidable whether a recursive intermediate logic enjoys hereditary projective approximation, hereditary unitary unification  or hereditary finitary unification.
\end{corollary}

We claim that most of intermediate logics has nullary unification. To give some evidence supporting our claim let us consider the lattice of all extensions of $\mathsf H_3\mathsf B_2$ ($=\mathsf L(\mathfrak G_1)$).
 Figure \ref{ki} contains all H-complete extensions of $\mathsf H_3\mathsf B_2$.
Adding all intersections of the logics (included in Figure \ref{ki}) we get the lattice in  Figure \ref{ti} below. There are 42 logics and   31 of them have nullary unification. Logics with hereditary finitary unification (there are 7 such logics) are located at the top of the picture. The logic $\mathsf L(\mathbf H_{un})$ is represented by $\mathsf L(\mathfrak R_2)$, and $\mathsf L(\mathbf H_{pa})$  by $\mathsf L(+\mathfrak F_2)$. We have  two (out of four) maximal logics with nullary unification in the picture: $\mathsf L({\mathfrak{G}_3})$ and $\mathsf L( \mathfrak R_{2})\cap\mathsf L(\mathfrak F_{2})$.  The appearance of 4 logics with finitary unification, that is  $\mathsf L(\mathfrak G_{1}),\mathsf L(\mathfrak G_{3}$), $\mathsf L(\mathfrak C_5)$ and $\mathsf L(\mathfrak G_{3})\cap \mathsf L(\mathfrak C_5)$ (see Theorem \ref{c5}, Theorem \ref{L8i} and comments after the theorem), is quite  mysterious.  They cannot be put apart from    nullary ones by their location in the lattice of all logics and they are, rather, isolated points among the majority of those with nullary unification. Unexpectedly, there are no logics with unitary unification, but the extensions of $\mathsf L(\mathfrak R_2)$, which is due (we think) to the fact that our approach is not enough representative. It would be much better if we had considered extensions of $\mathsf H_4\mathsf B_2$ instead, but this was quite beyond our capacity.
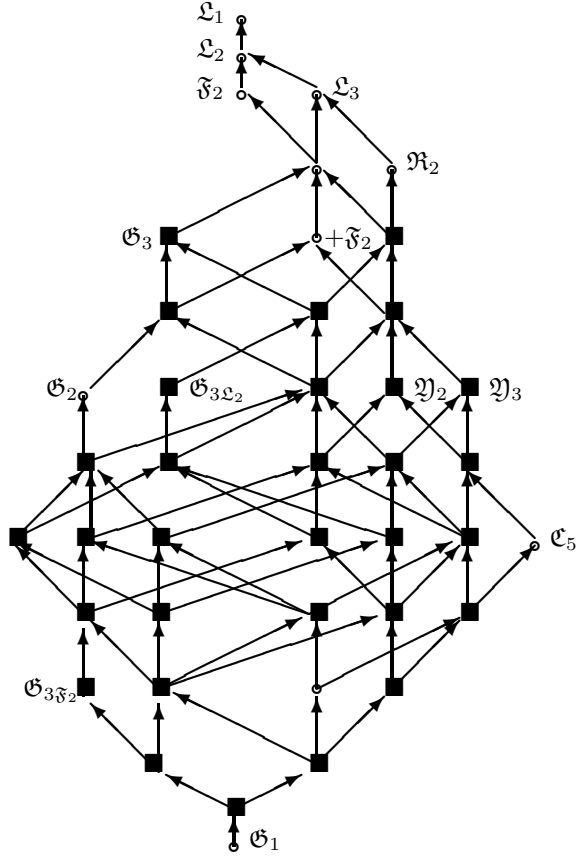
\begin{figure}[H]
\unitlength1cm
\thicklines
\begin{picture}(6,12)
\put(6.1,11){\circle{0.1}}
\put(6.1,10.5){\circle{0.1}}
\put(6.1,10){\circle{0.1}}
\put(7.1,10){\circle{0.1}}
\put(7.1,9){\circle{0.1}}
\put(8.1,9){\circle{0.1}}
\put(7.1,8.1){\circle{0.1}}
\put(4,6){\circle{0.1}}
\put(10,4){\circle{0.1}}
\put(6,0){\circle{0.1}}
\put(5.5,11){$\mathfrak L_1$}
\put(5.5,10.5){$\mathfrak L_2$}
\put(5.5,10){$\mathfrak F_2$}
\put(7.3,10){$\mathfrak L_3$}
\put(8.3,9){$\mathfrak R_2$}
\put(7.2,8){$+\mathfrak F_2$}
\put(4.5,8){$\mathfrak G_3$}
\put(5.4,6){$\mathfrak G_{3\mathfrak L_2}$}
\put(3.5,6){$\mathfrak G_2$}
\put(8.4,6){$\mathfrak Y_2$}
\put(9.4,6){$\mathfrak Y_3$}
\put(10.2,4){$\mathfrak C_5$}
\put(3.2,2){$\mathfrak G_{3\mathfrak F_2}$}
\put(6.2,0){$\mathfrak G_{1}$}
\put(5.9,0.4){$\blacksquare$}
\put(4.8,1){$\blacksquare$}
\put(7,1){$\blacksquare$}
\put(4.9,2){$\blacksquare$}
\put(3.9,2){$\blacksquare$}
\put(8,2){$\blacksquare$}
\put(7.1,2.1){\circle{0.1}}
\put(7,3){$\blacksquare$}
\put(4.9,3){$\blacksquare$}
\put(3.9,3){$\blacksquare$}
\put(8,3){$\blacksquare$}
\put(9,3){$\blacksquare$}
\put(8,4){$\blacksquare$}
\put(3,4){$\blacksquare$}
\put(7,4){$\blacksquare$}
\put(4.9,4){$\blacksquare$}
\put(3.9,4){$\blacksquare$}
\put(9,4){$\blacksquare$}
\put(8,5){$\blacksquare$}
\put(7,5){$\blacksquare$}
\put(5,5){$\blacksquare$}
\put(3.9,5){$\blacksquare$}
\put(9,5){$\blacksquare$}
\put(8,6){$\blacksquare$}
\put(7,6){$\blacksquare$}
\put(5,6){$\blacksquare$}
\put(9,6){$\blacksquare$}
\put(8,7){$\blacksquare$}
\put(7,7){$\blacksquare$}
\put(5,7){$\blacksquare$}
\put(8,8){$\blacksquare$}
\put(5,8){$\blacksquare$}
\put(6,0){\vector(0,1){0.4}}
\put(6,0.5){\vector(-2,1){0.9}}
\put(6,0.5){\vector(2,1){0.9}}
\put(5,1){\vector(-1,1){0.9}}
\put(7.1,1.1){\vector(1,1){0.9}}
\put(5,1){\vector(0,1){0.9}}
\put(7.1,1.1){\vector(0,1){0.9}}
\put(7.1,1.1){\vector(-2,1){1.9}}
\put(4,2){\vector(0,1){0.9}}
\put(5,2.1){\vector(-1,1){0.9}}
\put(5,2.1){\vector(0,1){0.9}}
\put(5,2.1){\vector(2,1){1.9}}
\put(5,2.1){\vector(3,1){2.9}}
\put(7.1,2.1){\vector(0,1){0.9}}
\put(7.1,2.1){\vector(2,1){1.9}}
\put(8.1,2.1){\vector(0,1){0.9}}
\put(8.1,2.1){\vector(1,1){0.9}}
\put(4,3.1){\vector(3,1){2.9}}
\put(4,3.1){\vector(0,1){0.9}}
\put(4,3.1){\vector(-1,1){0.9}}
\put(5,3.1){\vector(3,1){2.9}}
\put(5,3.1){\vector(0,1){0.9}}
\put(5,3.1){\vector(-2,1){1.9}}
\put(7,3.1){\vector(2,1){1.9}}
\put(7,3.1){\vector(-2,1){1.9}}
\put(7,3.1){\vector(-3,1){2.9}}
\put(8.1,3.1){\vector(1,1){0.9}}
\put(8.1,3.1){\vector(0,1){0.9}}
\put(8.1,3.1){\vector(-1,1){0.9}}
\put(9.1,3.1){\vector(1,1){0.9}}
\put(9.1,3.1){\vector(0,1){0.9}}
\put(3.1,4.1){\vector(1,1){0.9}}
\put(3.1,4.1){\vector(2,1){1.9}}
\put(4.1,4.1){\vector(0,1){0.9}}
\put(4.1,4.1){\vector(3,1){2.9}}
\put(5.1,4.1){\vector(-1,1){0.9}}
\put(5.1,4.1){\vector(3,1){2.9}}
\put(7.1,4.1){\vector(0,1){0.9}}
\put(7.1,4.1){\vector(-2,1){1.9}}
\put(8.1,4.1){\vector(0,1){0.9}}
\put(8.1,4.1){\vector(-3,1){2.9}}
\put(9.1,4.1){\vector(0,1){0.9}}
\put(9.1,4.1){\vector(-1,1){0.9}}
\put(9.1,4.1){\vector(-2,1){1.9}}
\put(10,4.1){\vector(-1,1){0.9}}
\put(4,5.1){\vector(0,1){0.9}}
\put(4,5.1){\vector(3,1){2.9}}
\put(5.1,5.1){\vector(0,1){0.9}}
\put(5.1,5.1){\vector(2,1){1.9}}
\put(7.1,5.1){\vector(0,1){0.9}}
\put(7.1,5.1){\vector(1,1){0.9}}
\put(8.1,5.1){\vector(-1,1){0.9}}
\put(8.1,5.1){\vector(1,1){0.9}}
\put(9.1,5.1){\vector(0,1){0.9}}
\put(9.1,5.1){\vector(-1,1){0.9}}
\put(4.1,6.1){\vector(1,1){0.9}}
\put(5.1,6.1){\vector(2,1){1.9}}
\put(7.1,6.1){\vector(-2,1){1.9}}
\put(7.1,6.1){\vector(1,1){0.9}}
\put(7.1,6.1){\vector(0,1){0.9}}
\put(8.1,6.1){\vector(0,1){0.9}}
\put(9.1,6.1){\vector(-1,1){0.9}}

\put(5.1,7.1){\vector(0,1){0.9}}
\put(5.1,7.1){\vector(2,1){1.9}}
\put(7.1,7.1){\vector(1,1){0.9}}
\put(7.1,7.1){\vector(-2,1){1.9}}
\put(8.1,7.1){\vector(0,1){0.9}}
\put(8,7.1){\vector(-1,1){0.9}}
\put(5.1,8.1){\vector(2,1){1.9}}
\put(7.1,8.1){\vector(0,1){0.9}}
\put(8.1,8.1){\vector(-1,1){0.9}}
\put(8.1,8.1){\vector(0,1){0.9}}
\put(7.1,9.1){\vector(-1,1){0.9}}
\put(7.1,9.1){\vector(0,1){0.9}}
\put(8.1,9.1){\vector(-1,1){0.9}}

\put(6.1,10.1){\vector(0,1){0.4}}
\put(7.1,10.1){\vector(-2,1){0.9}}
\put(6.1,10.6){\vector(0,1){0.4}}

\end{picture}

\caption{Extensions of  $\mathsf L(\mathfrak G_1)=\mathsf H_3\mathsf B_2$. }\label{ti}
\end{figure}
\noindent To show that some of the  logics $\mathsf L(\mathbf F)$ (assuming that {\bf F} is a family of frames for $\mathsf H_3\mathsf B_2$ and {\bf F}={\it sm}({\bf F}))  have nullary unification, we need results of the present paper (and these of \cite{Ghi5, dkw}), as well as their proofs. Thus, if $\mathfrak G_3\in \mathbf F$ but neither $\mathfrak G_2$ nor $\mathfrak G_{\mathfrak F_2}$ are  in {\bf F}, we can argue as in the proof of Theorem \ref{F6m}, showing that $\mathsf L(\mathfrak G_3)$ has nullary unification. If $\mathfrak Y_2\in \mathbf F$ but $\mathfrak C_5\not\in \mathbf F$, we can use the proof of Theorem \ref{F6m} concerning $\mathsf L(\mathfrak Y_2)$. If  $\mathfrak R_2$ and $\mathfrak F_2$ are in $\mathbf F$ but $\mathfrak Y_2$ is not in $\mathbf F$, we can use fragments of the proof of Theorem  \ref{L7} concerning  $\mathsf L(\{\mathfrak F_2,\mathfrak R_2\})$.

There are, however, some cases for which it happens that no our  proof fits.
 These are  $\mathsf L(\{\mathfrak G_2,\mathfrak C_5,\mathfrak G_{}\})$ and $\mathsf L(\{\mathfrak C_5,\mathfrak G_{3\mathfrak F_2}\})$. Then we   define $\sigma\colon\{x_1,x_2\}\to \mathsf{Fm}^k$ as follows
$$
   \sigma(x_1)=  \neg\neg x_1\land\bigwedge_{i=1}^k(\neg x_{2i}\lor\neg\neg x_{2i}) \quad \mbox{and}\quad
  \sigma(x_2)=  \neg x_1\land\bigwedge_{i=2}^{k+1}(\neg x_{2i-1}\lor\neg\neg x_{2i-1});
$$
we characterize all $\sigma$-models and easily get to a contradiction assuming that (any of the logics) has  finitary unification.

We claim that similar results given for intermediate logics in the present paper, can be reproven for transitive modal logics.

\bibliographystyle{plain}


 \end{paper}
\end{document}